\documentclass[twoside,11pt]{article}

\usepackage[T1]{fontenc}
\usepackage{pifont}

\usepackage[top=1in,bottom=1.25in, left=1in, right=1in]{geometry}
\usepackage[unicode=true,psdextra]{hyperref}
\usepackage{amsmath}
\usepackage[font=footnotesize]{caption}
\usepackage{tocloft} 
\usepackage{appendix}
\usepackage{algorithm}
\usepackage{algpseudocode}
\usepackage{amssymb,amsthm, mathrsfs}
\usepackage{subcaption}
\usepackage{array}
\usepackage{booktabs}
\usepackage{enumerate}
\usepackage{enumitem} 
\usepackage{graphicx}
\usepackage[dvipsnames]{xcolor}
\usepackage{tikz}
\usepackage{pgf}
\usetikzlibrary{arrows, positioning, patterns, patterns.meta}
\usepackage{tikz-cd}
\usepackage{authblk}
\usepackage{natbib}
\usepackage{bbm}
\bibpunct{(}{)}{;}{a}{}{,}
\usepackage{stackengine}
\usepackage{mathabx}
\usepackage{mystyle} 

\title{Estimating Multiple Weighted Networks with Node-Sparse Differences and Shared Low-Rank Structure}
\author[1]{Hao Yan}
\author[1]{Keith Levin}
\affil[1]{Department of Statistics, University of Wisconsin--Madison 

\texttt{\{hyan84,kdlevin\}@wisc.edu}}
\date{}

\begin{document}
\maketitle




\begin{abstract}
    We study the problem of modeling multiple symmetric, weighted networks defined on a common set of nodes, where networks arise from different groups or conditions.
    We propose a model in which each network is expressed as the sum of a shared low-rank structure and a node-sparse matrix that captures the differences between conditions.
    This formulation is motivated by practical scenarios, such as in connectomics, where most nodes share a global connectivity structure while only a few exhibit condition-specific deviations. 
    We develop a multi-stage estimation procedure that combines a spectral initialization step, semidefinite programming for support recovery, and a debiased refinement step for low-rank estimation.
	We establish minimax-optimal guarantees for recovering the shared low-rank component under the row-wise $\ell_{2,\infty}$ norm and elementwise $\ell_{\infty}$ norm, as well as for detecting node-level perturbations under various signal-to-noise regimes. 
    We demonstrate that the availability of multiple networks can significantly enhance estimation accuracy compared to single-network settings.
    Additionally, we show that commonly-used methods such as group Lasso may provably fail to recover the sparse structure in this setting, a result which might be of independent interest.
\end{abstract}

\section{Introduction} \label{sec:intro}

Networks are a fundamental data object in modern statistical applications, arising in domains as varied as neuroscience \citep{sporns2011human, bassett2017network, guha2024bayesian}, economics \citep{easley2010networks, diebold2014network} and social science \citep{granovetter1973strength, borgatti2009network}, to name just a few.
While modeling a single network has received extensive attention in recent decades--- e.g., through Stochastic Block Models \citep[SBM;][]{holland1983stochastic, airoldi2008mixed, karrer2011stochastic}, latent space models such as Hoff models \cite{hoff2002latent}, Random Dot Product Graphs \citep[RDPG;][]{YouSch2007,SusTanFisPri2012,athreya2018statistical} and graphons \citep{Lovasz2012}--- modeling multiple networks \citep[also known as multilayer networks][]{kivela2014multilayer} has recently been of growing interest to the network analysis community.
In many practical scenarios, one observes multiple networks on a common set of nodes, with individual networks corresponding to different experimental conditions, populations, or time points.
For example, functional magnetic resonance imaging (fMRI) data yield brain networks in which nodes represent brain regions.
These regions can be consistently aligned across subjects, enabling comparisons of brain connectivity across different populations, disease statuses, ages, or cognitive tasks \citep{vaiana2020multilayer}.
Other examples include temporal gene co-expression networks \citep{bakken2016comprehensive} and international food and agriculture trade networks \citep{de2015structural}.
A central goal of inference in these settings is to identify both the structure that is shared across conditions and the systematic differences between them. 

We briefly review recent developments in the statistical analysis of multiple networks.
\cite{levin2017central} proposed the omnibus embedding method for jointly embedding multiple networks sharing a common low-rank mean structure. 
Subsequently, \citet{levin2022recovering} extended this framework to accommodate heterogeneous edge noise across graphs while maintaining a shared low-rank structure in expectation.
Moving beyond shared mean structure, \cite{lei2023bias} analyzed a multilayer SBM (MLSBM), where each network shares a common community structure but permits heterogeneous edge probabilities.
\cite{lei2024computational} further investigated the fundamental limits of community detection and parameter estimation in MLSBM.  
As an extension of the MLSBM, \cite{nielsen2018multiple} considered a model with multiple RDPGs sharing a common eigenspace. 
Building on this idea, \cite{arroyo2021inference} assumed that the adjacency matrices share a common invariant subspace in expectation, and \cite{xie2024bias} proposed a bias-corrected joint spectral embedding method tailored for this setting.

All of the above works assume a shared subspace structure underlying the edge probabilities.
Relaxing this assumption, \cite{macdonald2022multiness} considered collections of networks following the generalized RDPG model \citep{rubin2022statistical}, under the assumption that the latent position matrix of each network is a concatenation of shared and individual latent position matrices.
Several works have also explored incorporating auxiliary covariates into multilayer network models \citep[e.g.,][]{macdonald2022latent,zhang2023generalized}.
More broadly, various extensions of other models to the multilayer setting have been proposed; see, for example, \cite{gollini2016joint, chen2017multilayer, salter2017latent, d2019latent, sosa2022latent, he2023semiparametric}.

This paper addresses the problem of modeling symmetric, weighted networks observed on a common set of nodes. 
We introduce a model in which networks from different groups are expressed as the sum of a shared low-rank matrix, capturing global latent structure, and a node-sparse perturbation that accounts for localized structural differences that are unique to each group.
The node-level sparsity of these perturbations reflects the belief that interventions or condition-specific effects often impact only a small subset of entities (e.g., individual brain regions in connectomic data or small groups of individuals in social networks), whose connectivity patterns are selectively altered under various conditions.
Such sparsity assumptions are common in network estimation to promote model parsimony.
For example, \cite{paul2020random} proposed a random effect SBM in which each subject-specific community assignment is a random perturbation of a group-level mean structure, with the assumption that most nodes across different networks retain their group-level memberships. 
Similarly, \cite{zhang2023generalized} considered multiple networks with covariate information, and assume sparsity in the coefficient tensor associated with the covariates.  

Existing work on multiplex and multilayer networks often assumes that differences across layers are smooth or low-rank, arising from shared latent spaces with layer-specific embeddings, or via stochastic block models with group-specific cluster parameters \citep{paul2020random,macdonald2022multiness,lei2023bias}.
Due to this low-rank or smooth structure, these frameworks are ill-suited to detect sparse or localized deviations. 
In contrast, our model is specifically designed to capture such heterogeneity, which cannot be adequately represented by global low-rank structure alone.
Besides adding a new model to the multilayer network literature, another primary goal is to provide simple examples where the availability of multiple observations provides non-trivial benefits in estimation over a single observation. 

To formalize our setting, we begin with a simplified two-group model, where each group consists of a single undirected, weighted network with adjacency matrices $\mY^{(0)}$ and $\mY^{(1)}$ respectively.
We model them according to
\begin{equation}\label{eq:base-model}
\mY^{(0)} = \mMstar + \mW^{(0)} \quad \text{ and } \quad
\mY^{(1)} = \mMstar + \mBstar + \mW^{(1)},
\end{equation}
where $\mMstar \in \R^{n\times n}$ is a symmetric rank-$r$ matrix representing the shared latent structure, $\mBstar \in \R^{n\times n}$ is a symmetric matrix capturing the differential connectivity between the two groups, and $\mW^{(0)}, \mW^{(1)} \in \R^{n\times n}$ are independent symmetric noise matrices.  
The assumption that $\mY^{(0)}$ consists solely of the shared component plus noise simplifies identifiability by providing a natural reference point.
This is analogous to the classical ANOVA framework, where one distinguishes a baseline (i.e., control) group from one or more treatment conditions, and one aims to characterize the treatment-induced deviations from the baseline group.

Extending this to a more general setting, we consider two groups of weighted networks. 
Let $\calG_0$ and $\calG_1$ denote index sets,
with adjacency matrices in group $\calG_0$ modeled as
\begin{equation}\label{eq:model:calG0}
    \mYz_i = \mMstar + \mW^{(0)}_i, \quad i \in \calG_0,
\end{equation}
and adjacency matrices of networks in group $\calG_1$ modeled as
\begin{equation}\label{eq:model:calG1}
    \mYo_j = \mMstar + \mBstar_j + \mW^{(1)}_j, \quad j \in \calG_1.
\end{equation}
Here, $\mMstar$ captures the network structure that is shared across all groups, while $\mBstar_j$ encodes a subject-specific perturbation for subject $j$ in group $\calG_1$. 
To enhance interpretability and reflect practical constraints, we assume that $\mBstar_j$ exhibits node-level sparsity, meaning that only a small subset of nodes display altered connectivity patterns under conditions different from the baseline.
We will call this type of sparsity \emph{node-sparsity}, since it captures the node-level differences between treatment groups, and this difference is assumed to be sparse.
The precise definition of node-level sparsity is provided below in Equation~\eqref{eq:group-sparsity} in Section~\ref{sec:setup}. 
The same notion of sparsity is also considered in \cite{mohan2014node} for graphical models and \cite{relion2019network} for network classification. 

The formulation given in Equations~\eqref{eq:model:calG0} and~\eqref{eq:model:calG1} allows for multiple independent network observations per group, enabling the model to capture subject-specific heterogeneity in the group $\calG_1$. 
Within this proposed framework, we are interested in addressing the following key questions:
\begin{itemize}[label=\ding{114}]
    \item How does the availability of multiple observations from $\calG_0$ and/or $\calG_1$ improve the recovery of the shared structure $\mMstar$ or the subject perturbations $\mBstar_j$?
    \item What estimation procedures are effective for recovering the shared low-rank structure $\mMstar$ and the subject-specific perturbations $\mBstar_j$?
\end{itemize}
We now provide a brief overview of our main contributions.
\begin{enumerate}
    \item In Section~\ref{sec:spectral-initialization}, we show that a single observation from the control group $\calG_0$ is sufficient to estimate the shared low-rank structure $\mMstar$ at the minimax-optimal rate under the $\ell_{2,\infty}$ norm.
    \item On the other hand, in Section~\ref{sec:common-structure}, we demonstrate that having multiple observations from $\calG_0$ and/or $\calG_1$ can substantially improve the estimation of $\mMstar$, both in terms of eigenspace estimation under the $\ell_{\infty}$ and $\ell_{2,\infty}$ norms, and in estimating the low-rank signal matrix $\mMstar$ under the $\ell_\infty$ norm.
    \item In Section~\ref{sec:one-step-recovery}, we establish that recovering the perturbation matrices $\mBstar_j$ from a single observation in $\calG_1$ requires strong conditions on the noise distribution. 
    However, when multiple observations share the same $\mBstar_j$, these assumptions can be significantly relaxed, leading to improved recovery guarantees.
\end{enumerate}
These results offer a nuanced understanding of how multiple observations from $\calG_0$ and/or $\calG_1$ contribute to the recovery of both shared and individual components of network structure.

On the algorithmic front, we develop a suite of estimation procedures that adapt to the number of observations in each group. 
These include a spectral initialization step (Section~\ref{sec:spectral-initialization}), semidefinite programming (SDP) approaches for identifying individual perturbations (Sections~\ref{sec:single-sdp} and \ref{sec:multiple-sdp}), and a refinement procedure for accurately recovering the shared low-rank structure (Section~\ref{sec:common-structure}).

In addition to algorithm development, we establish minimax lower bounds for both support recovery and low-rank estimation in our model, and demonstrate that our proposed methods attain these bounds under suitable conditions. 
Moreover, we demonstrate that certain intuitive baseline methods, such as group Lasso, can provably fail to recover the individual perturbations $\mBstar_{j}$, highlighting the need for specialized methods in this setting.

\paragraph{Notation} 

Before proceeding, we pause to establish notation.
Throughout, $c$ and $C$ will denote positive constants whose precise values may change from line to line.
For a positive integer $n$, we write $[n] = \{1,2,\dots,n\}$.
For a set $S$, we use $|S|$ to denote its cardinality.
For a vector $\vv=\left(v_1,v_2, \dots, v_n\right)^\top \in \R^n$, we denote the norms $\|\vv\|_2=\sqrt{\sum_{i=1}^n v_i^2}$ and $\|\vv\|_{\infty}=\max_i\left|v_i\right|$.
We denote the standard basis vectors of $\R^n$ by $\ve_i \in \R^{n}$ for $i \in [n]$ and write $\bbS^{n-1} = \left\{ \vu \in \R^n: \|\vu\|_2 = 1 \right\}$ for the unit sphere.
For a matrix $\mM \in \R^{n\times n}$, we write $\mM_{i,\cdot}$ for its $i$-th row, taken as a row vector, and $\mM_{\cdot,i}$ for its $i$-th column, taken as a column vector.
$\|\mM\|$ denotes its operator norm and $\|\mM\|_{2,\infty} := \max_{i\in[n]}\|\mM_{i,\cdot}\|_2$ denotes its maximum row-wise $\ell_2$ norm.
We write $\mI_{n} \in \R^{n\times n}$ for the $n$-by-$n$ identity matrix and $\mJ$ for the matrix of all ones, with dimensions clear from the context.
We use $\Sym(n)$ to denote the space of symmetric $n$-by-$n$ matrices.
We use both standard Landau notation and asymptotic notation, switching between the two conventions as is convenient.
Given positive functions $f(n)$ and $g(n)$, we write $f(n) \gg g(n)$, $f(n) = \omega(g(n))$ or $g(n) = o( f(n) )$ if $f(n)/g(n) \rightarrow \infty$ as $n \rightarrow \infty$.
We write $f(n) \gtrsim g(n)$, $f(n) = \Omega( g(n) )$ or $g(n) = O( f(n) )$ if for a positive constant $C$, we have $f(n)/g(n) \ge C$ for all sufficiently large $n$.
We write $f(n) = \Theta( g(n) )$ if both $f(n) = O(g(n))$ and $g(n) = O(f(n))$.

\section{Setup} \label{sec:setup}

Recall the model defined in Equations~\eqref{eq:model:calG0} and~\eqref{eq:model:calG1}
\begin{equation*}
    \mYz_i = \mMstar + \mWz_i, \; \text{ for } i \in \calG_0, \quad \text{ and } \quad \mYo_j = \mMstar + \mBstar_j + \mWo_j, \; \text{ for } j \in \calG_1.
\end{equation*}
Throughout, we will impose a set of assumptions on the noise matrices $\mW$.
\begin{assumption} \label{assump:noise}
    The noise matrix $\mW \in \R^{n\times n}$ is symmetric, with independent entries $(W_{ij})_{1\leq i\leq j\leq n}$ on and above its diagonal.
Furthermore, the entries of $W$ satisfy
    \begin{enumerate}
        \item {\bf (Heteroscedasticity.)} 
	For all $1 \le i \le j \le n$,
	$0 < \sigma_{\min}^2 \leq \E W_{ij}^2 = \sigma^2_{ij} \leq \sigma^2$, with $\frac{\sigma^2}{\sigma^2_{\min}} = \Theta(1).$ 
        \item {\bf (Boundness.)}
	For all $1 \le i \le j \le n$, $|W_{ij}| \leq L$ almost surely.
        \item {\bf (Symmetry about zero.)}
	For all $1 \le i \le j \le n$, $W_{ij} \eqdist -W_{ij}$.
    \end{enumerate}
\end{assumption}
This assumption is standard in the literature on low-rank matrix estimation \citep{chen2021asymmetry,fan2022asymptotic,zhou2025deflated}.
The lower bound on the noise variance $\sigma^2_{\min}$ in the heteroscedasticity condition 
is not required for establishing the upper bounds in our analysis. 
Rather, it is included to facilitate the derivation of minimax lower bounds.
The boundedness condition on the noise entries can be relaxed to a high-probability version. 
Specifically, it suffices to assume that $\Pr(|W_{ij}| \leq L) \geq 1 - o(n^{-c})$ for some suitably large constant $c > 0$.
We adopt the deterministic version of this bound for technical simplicity and clarity of exposition.
The symmetry assumption on the noise distribution is also introduced primarily for convenience, as it simplifies the analysis, but is not essential to the main results.

Additionally, we will need a near-homoscedastic condition in some scenarios.
\begin{assumption}[{\bf Near-homoscedasticity}] \label{assump:homoscedastic}
Let $\mSig = (\sigma^2_{ij})_{1\leq i, j \leq n} \in \R^{n\times n}$. There exists a value $\sigma_0 > 0$ such that $\sigma_0 = \Theta(\sigma)$ and 
\begin{equation*}
    \left\|\mSig - \sigma^2_0 \mJ\right\| \leq C \sigma^2 \sqrt{n}. 
\end{equation*}
\end{assumption}

\begin{remark}
    Although Assumption~\ref{assump:homoscedastic} requires the overall behavior of the noise to be close to a homoscedastic noise, it does not require the individual entries of the noise to be homoscedastic.
    Consider a symmetric random matrix $\mSig$ with i.i.d.~sub-Gaussian entries with parameter $\sigma^2$, for example, $\Sigma_{ij} \sim \Unif(0,2)$ for $1 \leq i \leq j \leq n$.
    Suppose further that $\E \Sigma_{ij} = \sigma^2_0$ for some $\sigma_0 \asymp \sigma$.
    then $\E \mSig = \sigma^2_0 \mJ$.
    By the matrix Bernstein inequality, we have $\|\mSig - \E \mSig\|\lesssim \sigma^2 \sqrt{n}$, which satisfies Assumption~\ref{assump:homoscedastic}. 
    This implies that Assumption~\ref{assump:homoscedastic} can be satisfied by heteroscedastic noise.  
\end{remark}

Let $\mMstar = \mUstar \mLambdastar \mU^{\star \top}$ be the spectral decomposition of $\mMstar$, where $\mLambdastar \in \R^{r\times r}$ is a diagonal matrix corresponding to the nonzero eigenvalues of $\mMstar$.
The nonzero elements of $\mLambdastar$ are given by $\lambdastar_1, \lambdastar_2, \dots, \lambdastar_r$ indexed so that $|\lambdastar_1| \geq |\lambdastar_2| \geq \cdots \geq |\lambdastar_r|$. 
Denote the condition number of $\mMstar$ as $\kappa$, which is given by $\kappa = |\lambdastar_1| / |\lambdastar_r|$.
For simplicity of presentation, we assume that $\kappa = O(1)$ and refer the reader to \cite{zhou2025deflated} for tackling high conditional number. 
We define the incoherence of $\mUstar$ to be 
\begin{equation*}
    \mu = \frac{n}{r} \|\mUstar\|_{2,\infty}^2 \in \left[1, \frac{n}{r}\right],
\end{equation*}
We impose the following requirement on $L$ and $\sigma$ introduced in Assumption~\ref{assump:noise}:
\begin{equation} \label{eq:L-sigma}
    \sigma \leq L \leq C \sigma \sqrt{\frac{n}{\mu \log n}},
\end{equation}
where $C > 0$ is a universal constant.
We will be primarily interested in the setting where $L$ and $\sigma$ are of the same order up to logarithmic factors, so that we allow a wider range of $\mu$; but whenever possible, we will extend our results to a more general setting, allowing for the possibility that $L$ is much larger than $\sigma$.
As demonstrated in the proof of Theorem~\ref{thm:delta_row}, Equation~\eqref{eq:L-sigma} enables us to localize the row-wise estimation error of $\mMstar$ and $\mUstar$. Without this assumption, the interplay between high coherence and large magnitude can compromise the accuracy of the entire matrix estimation.

Before we proceed to discuss the estimation procedure, we need to ensure that the model is identifiable.
In particular, the structure of the matrix $\mBstar_l$ requires careful consideration.
For notational convenience, we will omit the subscript $l$ in $\mBstar_l$ and refer to a generic node-sparse signal matrix simply as $\mBstar$.
To formalize the notion of {\em node-sparsity}, we have the folloing definition. 
\begin{definition} \label{def:nodesparse}
    The matrix $\mBstar$ is {\em node-sparse} if there exists a subset of indices $\Istar \subseteq [n]$ with $|\Istar| = m \leq n$, such that
    \begin{equation} \label{eq:group-sparsity}
        \begin{aligned}
            \|\mBstar_{i, \cdot}\|_2 = \|\mBstar_{\cdot, i}\|_2 > 0, &\mbox{ if } i\in \Istar;\\
            \Bstar_{ij} = 0, &\mbox{ for any } (i, j) \in I_{\star}^c \times I_{\star}^c.
        \end{aligned}
    \end{equation}
    We refer to $\Istar$ as the {\em node-level support} of $\mBstar$, representing the subset of nodes involved in nontrivial perturbations in our treatment group(s).
\end{definition}

While the classical definition of support for a matrix typically refers to the set of nonzero entries, we adopt the term node-level support to highlight the row/column sparsity pattern described in Equation~\eqref{eq:group-sparsity}.
More generally, for any symmetric matrix $\mB \in \R^{n \times n}$, we introduce the following definition:
\begin{definition}\label{def:IB}
    For any symmetric matrix $\mB \in \R^{n \times n}$, we define $I_{\mB} \subseteq [n]$ to be the smallest index set such that Equation~\eqref{eq:group-sparsity} holds with $\mB$ in place of $\mBstar$.
\end{definition}

An immediate question is whether this node-level support set $\Istar$ is well-defined. 
Although we require $\Istar$ to be the smallest such set, minimality alone is not sufficient for uniqueness, as multiple sets of the same size could satisfy the condition in Equation~\eqref{eq:group-sparsity}.
To resolve this ambiguity and guarantee identifiability of the node-level support set, we introduce the following assumption.
\begin{assumption} \label{assump:B-identifiable}
    There exists a subset $I \subset [n]$ with $|I|=m$ such that if $i \in I$, then $\mBstar_{i,\cdot}$ has at least $m+1$ nonzero entries.
\end{assumption}
Lemma~\ref{lem:identify-B} shows that under this assumption, the node-level support $\Istar$ is uniquely identified.

\begin{lemma} \label{lem:identify-B}
If $\mBstar$ satisfies Assumption~\ref{assump:B-identifiable}, then the minimal index set $\Istar$ satisfying Equation~\eqref{eq:group-sparsity} exists, is unique, and $\Istar = I$, where $I \subseteq [n]$ is the set guaranteed by Assumption~\ref{assump:B-identifiable}.
\end{lemma}
\begin{proof}
We note that the set $I$ ensured by Assumption~\ref{assump:B-identifiable} satisfies the conditions of Equation~\eqref{eq:group-sparsity}, so $|\Istar| \le |I| = m$ trivially.
For any $i \in I$, Assumption~\ref{assump:B-identifiable} ensures the existence of at least $m + 1$ indices $j_1, j_2, \dots, j_{m+1} \in [n]$ such that $\Bstar_{ij_k} \ne 0$ for all $k \in [m+1]$.
Since $\Bstar_{ij_k} \ne 0$, and Equation~\eqref{eq:group-sparsity} requires $\Bstar_{ij} = 0$ whenever $i, j \notin \Istar$, it must be that either $i \in \Istar$ or $j_k \in \Istar$ for all $k \in [m+1]$.
If $i \notin \Istar$, then at least $m + 1$ distinct indices (namely, $j_1,j_2, \dots, j_{m+1}$) must belong to $\Istar$, contradicting the bound $|\Istar| \le m$.
Therefore, $i \in \Istar$, implying $I \subseteq \Istar$. 
Since $|\Istar| \le |I|$, it follows that $\Istar = I$, completing the proof.
\end{proof}

Assumption~\ref{assump:B-identifiable} is not a necessary condition for the uniqueness of the node-level support set. 
However, our goal here is not to identify the weakest possible conditions, but rather to provide a simple assumption that guarantees identifiability and easy to check on any given $\mBstar$. 
Moreover, Assumption~\ref{assump:B-identifiable} is not particularly restrictive. 
If a row $\mBstar_{i,\cdot}$ satisfies that $\left\|\mBstar_{i,\cdot}\right\|_{0} = O(m)$, then both the $i$-th row and column are sparse, reflecting a form of edge-level sparsity.  
While such sparsity could be leveraged to improve estimation accuracy in principle, our focus in this paper is on identifying and modeling node-level effects.
Thus, we do not pursue additional sparsity assumptions beyond those specified in Equation~\eqref{eq:group-sparsity} in the present work.
The notation and definitions introduced above are also ready to be extended to the more general case considered in Equation~\eqref{eq:model:calG1}.
We denote the node support of $\mBstar_l$ as $I_{l, \star}$ and assume that $|I_{l, \star}| \leq m_l$ for all $l \in \calG_1$. 

\subsection{Overview of the Proposed Algorithm}

With the setup in place, we now provide an overview of our proposed algorithm, which also serves as a roadmap for the remainder of the paper.
The procedure is outlined below in Algorithm~\ref{alg:general} and involves three key subroutines, which we refer to as algorithms $\calA_0, \calA_1$ and $\calA_2$.
Specifically, $\calA_0$ provides an initial estimate of the shared structure $\mMstar$, $\calA_1$ produces an estimated $\Ihat_l$ of node-level support set of the perturbation $\mBstar_l$ for each $l \in \calG_1$, and $\calA_2$ refines the initial estimate of the shared structure $\mMstar$.
A more detailed discussion of each step follows.

In Step~\ref{alg:general:step:A0} of Algorithm~\ref{alg:general}, we obtain an initial estimate $\mM^{(0)}$ of the shared low-rank structure $\mMstar$ using Algorithm $\calA_0$, applied to the control group observations $\{\mYz_l\}_{l \in \calG_0}$.
We specify algorithm $\calA_0$ to be the standard spectral estimator for low-rank matrices.
The resulting estimator $\mM^{(0)}$ is itself low-rank and can be factorized as $\mU \mLambda \mU^\top$.
In Step~\ref{alg:general:step:select}, we identify a subset of nodes $\calU \subseteq [n]$ whose corresponding rows in the singular vector matrix $\mU$ have low coherence. 
The details are provided in Section~\ref{sec:spectral-initialization}.

Next, we restrict the observations $\{\mYo_l\}_{l \in \calG_1}$ to the selected node set $\calU$ and subtract the estimated shared structure $\mM^{(0)}$, yielding residual matrices $\mYtil_l$.
In Step~\ref{alg:general:step:A1}, we apply algorithm $\calA_1$ to the residual matrices $\mYtil_l$ to estimate the node-level support $\Ihat_l$ for each subject in the treatment group.
We analyze multiple possible implementations of $\calA_1$ in Section~\ref{sec:one-step-recovery}. 
Our primary proposal is an estimator based on semi-definite programming (SDP), offering both computational efficiency and minimax optimality guarantees.
In Section~\ref{sec:multiple-sdp}, we further demonstrate that the availability of multiple observations is crucial for relaxing the assumptions required to recover the node-level support $I_{l,\star}$.

As an alternative to SDP, we also consider a method based on the group lasso penalty, described in Section~\ref{subsec:glasso}. 
This approach, originally proposed in \cite{mohan2014node} for graphical models exhibiting sparsity patterns similar to ours, can be computationally more efficient than SDP.
However, its theoretical justification in prior work was limited. 
Specifically, it lacked high-probability guarantees for recovering the support of $\mBstar_l$.
We show that because the group lasso method is a greedy algorithm, it can fail to recover the correct support in certain regimes.

Finally, in Step~\ref{alg:general:step:A2}, we refine the estimate of the shared structure $\mMstar$. 
For each $l \in \calG_1$, we remove the rows and columns corresponding to the estimated node support $\Ihat_l$ from $\mYo_l$, obtaining a refined matrix $\mM_l$. 
We then apply algorithm $\calA_2$ to the collection of $\{\mYz_l\}_{l \in \calG_0}$ and $\{\mM_l\}_{l \in \calG_1}$ to produce the final estimator $\mMhat$.
Algorithm $\calA_2$ can be viewed as a debiased version of the spectral method, where we make use of multiple observations to improve the estimation of $\mMstar$.
Details are provided in Section~\ref{sec:common-structure}. 

\begin{algorithm}[ht]
\caption{General procedure to estimate $\mMstar$ and $\mBstar_l$}\label{alg:general}
\begin{algorithmic}[1]
\Require{$\left\{\mYz_{k}\right\}_{k \in \calG_0}, \left\{\mYo_{l}\right\}_{l \in \calG_1}$; Algorithms $\calA_0, \calA_1$ and $\calA_2$.} 
\Ensure{Estimators $\mMhat$ and $\{\mB_l\}_{l \in \calG_1}$.}
\State \label{alg:general:step:A0} $\mM^{(0)} = \mUz \mLambdaz \mUzt = \calA_0\left(\left\{\mYz_i\right\}_{i \in \calG_0}\right)$.
\State \label{alg:general:step:select} Let 
\begin{equation}\label{eq:calU:define}
    \calU := \left\{i \in [n] : \left\|\mUz_{i,\cdot}\right\|_2 \leq C n^{-1/4} \right\}. 
\end{equation}
Without loss of generality, we assume that $\calU = [\ntil]$ for $\ntil \leq n$. 
\State Set 
\begin{equation}\label{eq:Ytil:define}
\Ytil_{l,ij} \gets Y^{(1)}_{l,ij} - M^{(0)}_{ij}, \mbox{ for } \; i,j \in \calU,
\end{equation} 
where $\mYtil_l \in \R^{\ntil \times \ntil}$. 
\State \label{alg:general:step:A1} Set $\Ihat_{l}$ to be the estimated node support of $\mBstar_l$ given by $\{\Ihat_{l}\}_{l \in \calG_1} \gets \calA_1\left(\left\{\mYtil_l\right\}_{l\in \calG_1}\right)$ 
\State Set the rows and columns in the node support $\Ihat_{l}$ from $\mYo_l$ to $0$ and obtain $\mM_{l}$ as follows: 
\begin{equation*} 
    M_{l,ij} :=
    \begin{cases}
        Y^{(1)}_{l,ij} & \mbox{ if } (i,j) \in \Ihat^c_{l} \times \Ihat^c_{l}, \\ 
        0 & \mbox{ otherwise }.
    \end{cases}
\end{equation*}
\State \label{alg:general:step:A2}
Set $\mMhat \gets \calA_2\left(\left\{\mYz_k\right\}_{k\in\calG_0}, \left\{\mM_l\right\}_{l \in \calG_1}\right)$. 
\end{algorithmic}
\end{algorithm}


We finish this section with an informal statement of our main results, which provides a high-level overview of the theoretical guarantees of Algorithm~\ref{alg:general}.
\begin{theorem} \label{thm:infomal}
    Under suitable conditions on the noise matrices and the parameters of the model in Equations~\eqref{eq:model:calG0} and~\ref{eq:model:calG1}, Algorithm~\ref{alg:general} provides the following guarantees:
    \begin{enumerate}
        \item The spectral initialization procedure $\calA_0$ in Step~\ref{alg:general:step:A0} yields an estimate $\mM^{(0)}$ that estimates $\mMstar$ under the row-wise $\ell_{2,\infty}$ norm at the minimax optimal rate.
        \item The support recovery procedure $\calA_1$ in Step~\ref{alg:general:step:A1} yields an estimate $\Ihat_l$ of the node-level support $I_{l,\star}$ that recovers the true support with high probability under minimal signal-to-noise ratio (SNR) conditions.
        \item The debiased refinement $\calA_2$ in Step~\ref{alg:general:step:A2} yields near-minimax optimal estimates of $\mUstar$ under the row-wise $\ell_{2,\infty}$ norm, and of $\mMstar$ under the elementwise $\ell_{\infty}$ norm.
    \end{enumerate}
\end{theorem}
Theoretical results for the spectral initialization are presented in Theorems~\ref{thm:delta_row} and~\ref{thm:Mstar:2-infty:minimax} in Section~\ref{sec:spectral-initialization}.
Upper bounds for support recovery are established in Theorems~\ref{thm:sdp:L1-bound}, \ref{thm:trunc:sdp}, and~\ref{thm:sdp:multiple}, while the related lower bounds are provided in Theorems~\ref{thm:symmetric:minimax} and~\ref{thm:sdp-minimax} in Section~\ref{sec:one-step-recovery}. 
For the debiased refinement procedure, upper bounds are given in Theorems~\ref{thm:improved:linear-form-bound}, \ref{thm:eigenspace:l2infty}, and~\ref{thm:improved:entrywise}, and a minimax lower bound is presented in Theorem~\ref{thm:rank1:linfty:minimax}, all in Section~\ref{sec:common-structure}.

\section{Spectral Initialization} \label{sec:spectral-initialization}

In this section, we focus on analyzing Steps~\ref{alg:general:step:A0} and~\ref{alg:general:step:select} of Algorithm~\ref{alg:general}. 
We focus on the case where $|\calG_0| = 1$, and denote the only matrix in $\calG_0$ by $\mYz$.
The case where $|\calG_0| > 1$ is discussed in Section~\ref{sec:common-structure}. 
Let the spectral decomposition of $\mYz$ be 
\begin{equation*}
    \mYz = \mUz \mLambdaz \mUzt + \mU_{\perp} \mLambda_{\perp} \mU_{\perp}^\top,
\end{equation*}
where $\mLambdaz = \diag(\lambdaz_1,\lambdaz_2,\dots, \lambdaz_r)$ is the diagonal matrix of the leading $r$ eigenvalues, the columns of $\mUz \in \R^{n\times r}$ are the corresponding $r$ orthonormal eigenvectors of $\mYz$,
and the columns of $\mU_{\perp} \in \R^{n\times (n-r)}$ are the remaining $n-r$ orthonormal eigenvectors.
We take $\mMz = \mUz \mLambdaz \mUzt$ as the initial spectral estimator of $\mMstar$ and denote the estimation error by the residual matrix 
\begin{equation} \label{eq:mDelta:define}
    \mDelta := \mMz - \mMstar \in \R^{n \times n}.
\end{equation}

Adapting the proof of Corollary 4.3 in \cite{chen2021spectral}, we obtain the following refined row-wise and entry-wise bounds.

\begin{theorem} \label{thm:delta_row}
Suppose that $\mYz$ is given by Equation~\eqref{eq:base-model} with $\mW^{(0)}$ satisfying Assumption~\ref{assump:noise} and Equation~\eqref{eq:L-sigma}. 
Assume that $|\lambdastar_r| \geq C \sigma \kappa \sqrt{n \log n }$ for some sufficiently large constant $C$ and $\kappa = O(1)$.
Then with probability at least $1 - O(n^{-6})$, the following bounds hold for the residual matrix $\mDelta$ defined in Equation~\eqref{eq:mDelta:define}:
\begin{enumerate}
    \item \textit{(Row-wise bound)} Uniformly over all $\ell \in [n]$,
  \begin{equation*}
  \left\|\mDelta_{\ld}\right\|_{2} \lesssim \left(\sigma \kappa^2 \sqrt{n}\right) \left\|\mUstarl\right\|_2 + \sigma \kappa \sqrt{r \log n} \lesssim \left(\sigma \sqrt{n}\right) \left\|\mUstarl\right\|_2 + \sigma \sqrt{r \log n},
  \end{equation*}
  \item \textit{(Entry-wise bound)} Uniformly over all $k,\ell \in [n]$:
  \begin{equation*}
  \begin{aligned}
      \left|\mDelta_{\ell k}\right| &\lesssim \left(\sigma \kappa^2 \sqrt{n}\right) \left\|\mUstarl\right\|_2 \left\|\mU^\star_{\kd}\right\|_2 + \sigma \kappa \sqrt{r \log n} \max\left\{\left\|\mUstarl\right\|_2 + \left\|\mU^\star_\kd\right\|_2\right\} + \frac{\sigma \kappa \sqrt{r \log n}}{\lambdastar_r} \\
      &\lesssim (\sigma \sqrt{n}) \left\|\mUstarl\right\|_2 \left\|\mU^\star_{\kd}\right\|_2 + \sigma \sqrt{r \log n} \max\left\{\left\|\mUstarl\right\|_2 + \left\|\mU^\star_\kd\right\|_2\right\}
  \end{aligned}
  \end{equation*}
\end{enumerate}
\end{theorem}
The proof Theorem~\ref{thm:delta_row} is provided in Appendix~\ref{sec:localized}, alongside the proof of an analogous result for subspace estimation, which appears as Theorem~\ref{thm:ustar-row}.

\begin{remark} \label{rem:scales-with-Ustar}
    The bound in Theorem~\ref{thm:delta_row} above reveals that the row-wise $\ell_2$-norm of $\mDelta$ scales with the corresponding row norm of $\mUstar$. 
    The condition in Equation~\eqref{eq:L-sigma} is imposed to prevent estimation errors associated with rows having large $\|\mUstar_{\ell,\cdot}\|_2$ from propagating to those with small row norms.
    The proof of Theorem~\ref{thm:delta_row} traces this cross-contamination to the term $\|\mWz_{\ell,\cdot} \mUstar\|_2$. 
    As an illustration, consider $|\mWz_{\ell,\cdot} \vustar_1|$.
    We have
    $
    |W^{(0)}_{\ell,i} u^\star_{1,i}| \leq L \|\vustar_1\|_{\infty}
    $ 
    for all $i \in [n]$ and 
    $
    \sum_{i=1}^n \E \left(W^{(0)}_{li} u^\star_{1i}\right)^2 \leq \sigma^2.  
    $ 
    Applying Bernstein's inequality highlights that the condition in Equation~\eqref{eq:L-sigma} is essential for obtaining an upper bound for $|\mWz_{\ell,\cdot} \vustar_1|$ that ensures row-wise error localization.
\end{remark}

To recover the node support $\Istar$ of $\mBstar$, it is necessary for all $i \in \Istar$, each $\mBstar_{i,\cdot}$ has sufficiently large norm. 
Specifically, Theorem~\ref{thm:symmetric:minimax} in Section~\ref{sec:minimax-support} establishes that a signal strength of at least $\Omega(\sigma n^{1/4} \log^{1/4} n)$ is required for recovering $i \in \Istar$ with high probability.
However, support recovery becomes more challenging when the rows of the low-rank component $\mMstar$ exihibit high coherence.
According to Theorem~\ref{thm:delta_row}, the row-wise estimation error of $\mMstar$ scales with $\|\mU^\star_{\ell,\cdot}\|_2$, implying that high-coherence rows may suffer from larger estimation errors under spectral initialization.
If the node support of $\mBstar$ overlaps with these high-error rows, the estimation error in $\mMstar$ may exceed the minimal signal strength required for detecting $\mBstar$, making recovery impossible.
More generally, high coherence rows of $\mMstar$ can lead to identifiability issues.
Consider the rank-$4$ matrix $\mMstar = \lambdastar \left(\ve_1 \vu_1^\top + \vu_1 \ve_1^\top + \ve_2 \vu_2^\top + \vu_2 \ve_2^\top\right)$, and let $\mBstar = \ve_1 \vv^\top + \vv \ve_1^\top$, where $v_1 = v_2 = 0$ and $v_{i} \sim \calN(0, \sigma_B^2)$ for $3 \leq i \leq n$.
Under our setup, we observe
\begin{equation*}
\begin{aligned}
    \mY^{(0)} = \mMstar + \mW^{(0)}, \; \mY^{(1)} = \mMstar + \mBstar + \mW^{(1)}.
\end{aligned}
\end{equation*}
If $\mMstar + \mBstar$ still has rank $4$, then we can set $\mM^{(1)} = \mMstar + \mBstar$ and consider
\begin{equation*}
\begin{aligned}
    \mY^{(0)} = \mM^{(1)} + \mW^{(0)}, \; \mY^{(1)} = \mM^{(1)} + \mW^{(1)}.
\end{aligned}
\end{equation*}
If the first two rows of $\mW^{(0)}$ have higher variance than the remaining rows, then unless $\sigma_B = \Theta(\sigma)$, one cannot reliably distinguish the above two settings. 
In this paper, to avoid this possible ambiguity between high coherence rows of $\mMstar$ and the node support of $\mBstar$, we impose the following assumption:
\begin{assumption}\label{assump:no-overlap}
    Define the index set 
    \begin{equation} \label{eq:calUstar:define}
        \calU^\star := \left\{\ell \in [n] : \|\mU^\star_{\ell,\cdot}\|_2 = O(n^{-1/4})\right\} \subset [n] .
    \end{equation}
    The node support set $\Istar$ obeys $\Istar \subset \calU^\star$.
\end{assumption}
Under Assumption~\ref{assump:no-overlap}, we only need to discard the high-coherence rows of $\mMstar$ to ensure that the estimation error does not exceed the minimal signal strength required for support recovery.
Theorem~\ref{thm:ustar-row} in Appendix~\ref{sec:localized} shows that the spectral estimator $\mU$ is sufficient for us to accurately estimate the magnitude of each row norm $\|\mU^\star_{\ell,\cdot}\|_2$ up to the tolerable statistical error rate. 
Consequently, Step~\ref{alg:general:step:select} in Algorithm~\ref{alg:general} reliably identifies and removes the high-error rows of $\mMstar$, which is critical for accurate support recovery. 

\begin{remark} \label{rem:calU:wlog}
Under Assumption~\ref{assump:no-overlap}, the high-coherence rows of $\mMstar$ in $\calU^\star$ can be identified and removed with high probability. 
As a result, they do not impact the analysis of the support recovery problem.
For simplicity of exposition, in Sections~\ref{sec:one-step-recovery} and~\ref{sec:common-structure}, we assume without loss of generality that $\left(\calU^\star\right)^c = \emptyset$, so that $\calU^\star = [n]$ holds with high probability.
\end{remark}

It is worth noting that explicit estimation of $\mMstar$ is not strictly required for recovering $\mBstar$.
For instance, one could instead subtract any matrix $\mY^{(0)}_i$ for $i \in \calG_0$ from the observed matrix $\mY^{(1)}_j$ to cancel out the shared component $\mMstar$, resulting in the residual matrix $\mBstar + \mW^{(1)} - \mW^{(0)}$.
This differencing approach removes the low-rank component without explicitly estimating it.
The benefit of this approach is that it can recover $\Istar$ even when Assumption~\ref{assump:no-overlap} does not hold, but this hinges on the assumption that the noise matrices $\mW^{(0)}$ also satisfy Assumption~\ref{assump:homoscedastic}.
Further discussion of Assumption~\ref{assump:homoscedastic} is provided in the beginning of Section~\ref{sec:minimax-support}. 
In contrast, estimating and subtracting the low-rank component $\mM^{(0)}$ avoids the need to impose Assumption~\ref{assump:homoscedastic} on the control noise $\mW^{(0)}$.
Moreover, the na\"{i}ve differencing approach increases the effective noise level, which can hinder support recovery.
In practice, one can apply both methods on the data as the analysis and methods developed in Section~\ref{sec:one-step-recovery} also apply to both the residual-based and direct ``differencing'' approaches.
For this reason, we focus on the residual-based strategy in what follows.


A natural follow-up question to Theorem~\ref{thm:delta_row} concerns whether having more samples in the control group $\calG_0$ significantly improves the row-wise estimation error rate of $\mMstar$. 
Interestingly, even without any additional correction, the spectral estimator $\mM^{(0)}$ already achieves the minimax-optimal rate for estimating $\mMstar$ under the $\ell_{2,\infty}$ norm.

\begin{theorem}\label{thm:Mstar:2-infty:minimax}
Let 
\begin{equation*}
    \calM_0(\mu, \sigma_{\min}) := \left\{\mMstar = \lambdastar \vustar \vustart : \left\|\vustar\right\|_{\infty} = \sqrt{\frac{\mu}{n}}, |\lambdastar| \geq \sigma_{\min} \sqrt{n} \right\} .
\end{equation*}
Suppose that we observe $N\geq 1$ independent copies of $\mM_l = \mMstar + \mW_l$, where $W_{l,ij} \sim \calN(0, \sigma^2_{l,ij})$ are independent random variables for $1\leq i\leq j \leq n$ and $l \in [N]$.
Then 
\begin{equation*}
    \inf_{\mMhat} \sup_{\mMstar \in \calM_0(\mu, \sigma_{\min})} \E_{\mMstar} \left\|\mMhat - \mMstar \right\|_{2,\infty} \geq \frac{\sigma_{\min}}{27\sqrt{\log 4}} \sqrt{\frac{\mu}{N}} ,
\end{equation*}
where the infimum is over all estimators of $\mMstar$.
\end{theorem}

It is straightforward to extend Theorem~\ref{thm:Mstar:2-infty:minimax} to the case where $\mMstar$ is not necessarily rank-one. 
However, we focus on the rank-one setting for simplicity, as it is sufficient to convey the key message.
When $N=1$, Theorem~\ref{thm:delta_row} implies that the spectral estimator achieves this lower bound up to logarithmic factors, and is thus nearly minimax-optimal under the $(2,\infty)$-norm. 
When $N > 1$, the minimax rate improves with $1/\sqrt{N}$, and is easily attained by averaging the $N$ independent observations $\mM_1, \mM_2, \dots, \mM_N$ and applying the spectral estimator to the result.
In particular, when $N = \Theta(1)$, a single copy from $\calG_0$ suffices to achieve the minimax-optimal rate. 
This observation highlights a key theme of this paper: while the estimation error rate under the $\ell_{2,\infty}$ norm is unaffected when increasing $N$ from $1$ to a larger constant, other settings, such as support recovery for $\mBstar$ and elementwise estimation of $\mMstar$, exhibit a more pronounced distinction between $N=1$ and $N=2$.
Details are given in Section~\ref{sec:multiple-sdp} and Section~\ref{sec:common-structure}, respectively.

\section{\texorpdfstring{Recovery of the Differences $\mBstar$}{Recovery of the Differences B*}}
\label{sec:one-step-recovery}

In this section, we focus on recovering the difference matrix $\mBstar$ from the residual matrix $\mYtil$ defined in Equation~\eqref{eq:Ytil:define}.
We begin by analyzing the case where $|\calG_1| = 1$, and for notational simplicity, we drop all the subscripts and superscripts associated with the group $\calG_1$.
Specifically, in the remainder of this section, we write $\mYtil$, $\mBstar$, and $\mW$ to denote $\mYtil_1$, $\mBstar_1$, and $\mW^{(1)}_1$, respectively.
The more general case where $|\calG_1| > 1$ will be addressed in Section~\ref{sec:multiple-sdp}.

A natural approach to estimating $\mBstar$ and recovering its node support $\Istar$ is to solve the following constrained least squares problem:
\begin{equation} \label{eq:least-square-B} \begin{aligned}
   \arg \min_{\mB, I} & \sum_{1\leq i\leq j \leq n} (\Ytil_{ij} - B_{ij})^2, \\
   \text{s.t. } & B_{ij} = 0 \text{ for any } i, j \in I^c, |I| \leq m
\end{aligned} \end{equation}
where $m$ is an upper bound on the cardinality of $\Istar$.
For theoretical analysis, we assume that the true support size $m = |\Istar|$ is known.
Practical strategies for choosing $m$ are discussed in Section~\ref{sec:m-selection}.
The formulation in Equation~\eqref{eq:least-square-B} is equivalent to selecting a set of indices $I \subseteq [n]$ that minimizes the squared residuals outside the support, i.e.,
\begin{equation*} 
    \Ihat = \arg\min_{I \subseteq [n]} \sum_{i,j \in I^c, i \leq j} \Ytil_{ij}^2, \quad |I| = m.
\end{equation*}
The corresponding estimator $\mBhat$ then takes the form: $B_{ij} = \Ytil_{ij}$ if $i \in \Ihat$ or $j \in \Ihat$, and $B_{ij} = 0$ otherwise.
This problem is equivalent to the $(n-m)$-Lightest Subgraph problem, which is known to be NP-hard in general \citep{watrigant2012ksparse}. 
This NP-hardness in the symmetric setting stands in sharp contrast to the asymmetric case.
In an asymmetric version the problem, where $\mBstar$ is assumed to be row-sparse (or column-sparse), support recovery becomes significantly easier.
In that setting, one can simply rank the row norms of $\mBstar$ and greedily select the top-$m$ rows, making the problem computationally trivial. 
This contrast highlights the added difficulty posed by the symmetric node-sparse structure considered here.

\begin{remark}
    The problem of recovering $\mBstar$ is closely related to the submatrix detection problem, which involves recovering a row- and column-sparse signal matrix from a noisy observation \citep{ma2015computational,CheXu2016,cai2017computational,cai2020statistical,SohWei2025}.
    Much of the existing work in this area focuses on the asymmetric setting with i.i.d.~noise and investigates both statistical and computational limits.
    Notably, these studies identify regimes in which the problem becomes computationally intractable.
    Extending these results to the symmetric setting, which is more relevant to our model, is an interesting direction for future research.
    Related work has also examined the estimation of matrices with either row-sparse or column-sparse structure \citep{klopp2015estimation, obozinski2011support, yang2016rate}. 
    While these problems are similar in spirit, they typically require different techniques and do not directly address the symmetric, node-sparse structure considered here.
\end{remark}

As a first step towards understanding the behavior of our proposed method, we consider a simplified benchmark model
\begin{equation}\label{eq:mYtil:benchmark}
    \mYtil = \mBstar + \mW,
\end{equation}
in which we ignore the estimation error introduced by $\mM^{(0)}$. 
To formalize the minimal signal strength required for recovery, we define
\begin{equation} \label{eq:def:b}
b
= \min_{i \in \Istar}
\left\|
	\left(B_{i j}^{\star}\right)_{j \in
			\left([n] \backslash \Istar \right) \cup\{i\}}
\right\|_2.
\end{equation}
We will discuss the definition of $b$ in more detail below in Section~\ref{sec:minimax-support}.
For now, we simply note that it is the key quantity in determining sufficient conditions for recovering the node support of $\mBstar$, as the following theorem demonstrates.

\begin{theorem} \label{thm:mle-result}
Consider an observed matrix $\mYtil$ from Equation~\eqref{eq:mYtil:benchmark}, where $\mW$ satisfies Assumptions~\ref{assump:noise} and~\ref{assump:homoscedastic}, and $\mBstar$ satisfies Assumption~\ref{assump:B-identifiable}. 
Assume that $n \geq 3m$ and that
\begin{equation}\label{eq:max-bstar}
    \left\|\mBstar\right\|_{\infty} = o\left(\sigma \sqrt{\frac{n}{\log n}}\right).
\end{equation}
With $b$ as defined in Equation~\eqref{eq:def:b},
suppose that
\begin{equation} \label{eq:b-MLE-sufficient}
    b \geq c_0 (\sigma L)^{1/2} \left(n \log n\right)^{1 / 4}
\end{equation}
for a sufficiently large universal constant $c_0$.
Then it holds with probability at least $1-O(n^{-6})$ that the maximum-likelihood estimator $\Ihat$ recovers $\Istar$ exactly.

Furthermore, if for some $k \in [m]$, we have
\begin{equation} \label{eq:b-MLE-partial-sufficient}
b
\geq c_0 (\sigma L)^{1/2} \left(n \log \frac{n}{k} \right)^{1/4},
\end{equation}
then with probability at least $1 - O(n^{-6})$, we have
\begin{equation*}
    |\Ihat \cap \Istar| \geq m - k.
\end{equation*}
\end{theorem}
The proof is provided in Appendix~\ref{sec:mle:proof}. The conditoin in Equation~\eqref{eq:max-bstar} is a technical condition to simplify the analysis. 
In practice, this is a very mild condition, as violating it would imply that some entries of $\mBstar$ are very large, making the support recovery problem trivial.

\subsection{Minimax Risk for Support Recovery}
\label{sec:minimax-support}

Before discussing computationally efficient methods to estimate $\Istar$, we first examine whether the lower bound in Equation~\eqref{eq:b-MLE-sufficient} is optimal in a minimax sense.
To illustrate the necessity of certain assumptions on the noise matrix, we begin with a simple example showing that Assumption~\ref{assump:homoscedastic} is essential when only a single observation of $\mYtil$ is available.
Consider the following hypothesis testing problem:
\begin{equation*}
    \calH_0: \vy = \vb + \vw_1 \quad \text{vs} \quad \calH_1: \vy = \vw_2,
\end{equation*}
where $\vb \sim \calN(0, \sigma^2 \mI_n)$, $\vw_1 \sim \calN(0, \sigma^2 \mI_n)$ and $\vw_2 \sim \calN(0, 2\sigma^2 \mI_n)$ are independent. 
Clearly, $\calH_0$ and $\calH_1$ induce the same probability distribution on $\vy$, making them indistinguishable. 
This example shows that without further structural assumptions on $\mBstar$, or a large signal strength (e.g., $\|\mBstar_{\ell,\cdot}\|_2 \gg \sigma \sqrt{n}$ for $\ell \in \Istar$), recovering $\Istar$ from $\mYtil$ is impossible when Assumption~\ref{assump:homoscedastic} does not hold.
In contrast, we show in Section~\ref{sec:multiple-sdp} that Assumption~\ref{assump:homoscedastic} is not necessary if two independent observations of $\mYtil$ are available.

We now turn to the case where Assumption~\ref{assump:homoscedastic} holds.
To formalize the support recovery problem, define the binary selector vector $\veta^\star \in \{0,1\}^n$ by
\begin{equation}\label{eq:selector:star}
    \eta^\star_i = \indic\{i \in \Istar\} \quad \text{for } i \in [n],
\end{equation}
and let
\begin{equation*}
    \veta_{\sim i}^\star = (\eta^\star_j)_{1\leq j\leq n : j\neq i}.
\end{equation*}
Consider the collection of hypothesis testing problems
\begin{equation*}
    H_{i,0}: \eta_i^\star=0 \quad \text{vs} \quad H_{i,1}: \eta_i^\star = 1, \quad \text{ for } i =1,2,\dots n,
\end{equation*}
and note that this task is strictly harder than solving the corresponding set of conditional testing problems,
\begin{equation*}
    H_{i,0}: \eta_i^\star=0 \quad \text{vs} \quad H_{i,1}: \eta_i^\star = 1, \quad \text{given } \veta_{\sim i}^\star, \quad \text{ for } i=1,2,\dots,n ,
\end{equation*}
where one assumes knowledge of the inclusion or exclusion status of all nodes except node $i$.
Each of these conditional problems is equivalent to testing
\begin{equation*}
    H_{i,0}: (B^\star_{ij})_{j \in ([n]\backslash\Istar)\cup \{i\}} = 0 \quad \text{vs} \quad H_{i,1}: (B^\star_{ij})_{j \in ([n]\backslash\Istar)\cup \{i\}} \neq 0. 
\end{equation*}

This motivates the introduction of the following parameter set, which captures node-sparse symmetric matrices with minimum signal strength:
\begin{equation} \label{eq:calBS:define}
    \calB_S(b, n, m) := \left\{ 
    \mB \in \Sym(n),\; |I_{\mB}| \leq m :
    \begin{array}{cc}
         \|\Tilde{\vb}_i\|_2 \geq b, \|\mB_{i,\cdot}\|_{0} \geq m+1 & \forall i \in I \\
         B_{ij} = 0,  & \forall i, j \in I^c 
    \end{array}  \right\},
\end{equation}
where for any symmetric matrix $\mB \in \R^{n\times n}$, we denote $\Tilde{\vb}_i = (B_{ij})_{j \in [n]\backslash I_{\mB}\cup \{i\}}$ and $I_{\mB}$ is the node support of $\mB$ defined in Definition~\ref{def:IB}.  

Let $\veta \in \{0, 1\}^n$ be the binary support indicator associated with $\mB$, defined analogously to $\veta^\star$ in Equation~\eqref{eq:selector:star}, with $I_{\mB}$ in place of $\Istar$.
For any estimator $\vetahat \in [0,1]^n$ derived from the observation $\mYtil = \mBstar + \mW$, we measure the estimation error using the Hamming distance,
\begin{equation*}
    \|\vetahat - \veta \|_{\rmH} = \sum_{j=1}^n |\hat{\eta}_j - \eta_j| . 
\end{equation*}
To facilitate the comparison between bounds on $\Pr\left(\vetahat \neq \veta\right)$ and bounds on the expected Hamming error $\E_{\mBstar} \|\vetahat - \veta\|_{\rmH}$, we note the following implication: 
if $\Pr(\vetahat \neq \veta^\star) \leq \alpha$ for some $\alpha \in [0,1]$, then by the trivial upper bound $\|\veta^\star - \vetahat\|_{\rmH} \leq n$, we have
\begin{equation*}
\E_{\mBstar} \|\veta^\star - \vetahat\|_{\rmH}
\leq 0 \cdot \Pr\left[ \|\veta^\star - \vetahat \|_{\rmH} = 0 \right]
	+ n ~\Pr\left[ \|\veta^\star - \vetahat\|_{\rmH} \neq 0 \right]
\leq n \alpha. 
\end{equation*}
As a result, upper bounds on the misclassification probability $\Pr\left(\vetahat \neq \veta\right)$ can be translated into corresponding bounds on $\E_{\mBstar} \|\veta - \vetahat\|_{\rmH}$.
Theorem~\ref{thm:symmetric:minimax} provides a minimax lower bound for the expected Hamming error over the parameter set $\calB_S(b,n,m)$.

\begin{theorem}\label{thm:symmetric:minimax}
    Consider observed data $\mYtil = \mBstar + \mW$, where $\mW$ is a symmetric noise matrix with entries $(W_{ij})_{1\leq i\leq j\leq n}$ drawn i.i.d.~from $\calN(0, \sigma^2)$.
    Assume $|\Istar| = m < n/3$.
    Then, there exists a universal constant $C > 0$, such that for sufficiently large $n$, 
    \begin{equation} \label{eq:minimax-lower-bound}
    \begin{aligned}
        \inf_{\vetahat \in [0,1]^{n}} \sup_{\mBstar \in \calB_S(b, n, m)} \E_{\mBstar} \|\veta - \vetahat\|_{\rmH}
        &\geq \max_{m' \in (0, m]}
        4m' \left[\frac{1}{16} f_n\left( \frac{C b^4}{\sigma^4}, m'\right)
        - e^{ \frac{-(m-m^{\prime})^2}{2m} } \right],
    \end{aligned} 
    \end{equation} 
    where for any $a \in \R$, the function $f_n(a, m')$ is defined as 
    \begin{equation*}
        f_n(a, m') = 3-\frac{m^{\prime}}{n}\left(\exp \left\{a-C' n \log \left(\frac{n}{3 m^{\prime}}\right)\right\}+\exp \left(\frac{a}{\log n}-C' n\right)+\exp\left\{ \frac{C' a}{n} \right\} \right)
    \end{equation*}
    for some universal constant $C' > 0$.
\end{theorem}

A proof of Theorem~\ref{thm:symmetric:minimax} is given in Appendix~\ref{apx:minimax}, alongside a proof of the following corollary.

\begin{corollary}\label{cor:minimax-recoverable}
    Under the same setting as Theorem~\ref{thm:symmetric:minimax}, for any $m < n/3$ and $m' \in [1, m]$, there exists a universal constant $C > 0$, independent of all other parameters, 
    such that if
    \begin{equation*}
    b < C \sigma n^{1/4} \log^{1/4} \left(n/3m'\right),
    \end{equation*}
    then it follows that 
    \begin{equation} \label{eq:final-lower-risk}
        \inf_{\vetahat \in [0,1]^{n}} \sup_{\mBstar \in \calB_S(b, n, m)} \E_{\mBstar} \|\vetahat-\veta\|_{\rmH}
        \geq 4m^{\prime} \left[\frac{1}{8}-\exp \left\{-\frac{\left(m-m^{\prime}\right)^2}{2 m}\right\} \right].
    \end{equation}
\end{corollary}

Corollary~\ref{cor:minimax-recoverable} is readily extended to a multiple sample setting, where we observe $N$ independent copies of $\mYtil$.

\begin{theorem}\label{thm:multiple:minimax-recoverable}
    Suppose we observe $N\geq 1$ independent samples of the form 
    $$
    \mYtil_{\ell} = \mBstar + \mW_\ell,~~~\ell = 1,2,\dots,N
    $$ 
    where $\{ \mW_\ell : \ell \in [N] \}$ are independent symmetric noise matrices with entries $(W_{\ell, ij})_{1\leq i\leq j\leq n}$ drawn i.i.d.~from $\calN(0, \sigma^2)$.
    Assume that $m < n/3$.
    Then there exists a universal constant $C > 0$, such that for sufficiently large $n$ and any $m' \in [1, m]$, if
    \begin{equation*}
        b < \frac{C \sigma n^{1/4} \log^{1/4} \left(n/3m'\right)}{\sqrt{N}}, 
    \end{equation*}
    it holds that 
    \begin{equation} \label{eq:final-lower-risk-multiple}
        \inf_{\vetahat \in [0,1]^{n}} \sup_{\mBstar \in \calB_S(b, n, m)} \E_{\mBstar}\|\vetahat-\veta\|_{\rmH}
        \geq 4m' \left[\frac{1}{8}-\exp \left\{-\frac{\left(m-m'\right)^2}{2 m}\right\} \right].
    \end{equation}
\end{theorem}

The proof of Theorem~\ref{thm:multiple:minimax-recoverable} is similar to that of Theorem~\ref{thm:symmetric:minimax}. 
Details are provided in Appendix~\ref{subsec:multiple:minimax-recoverable}.

\begin{remark}\label{rem:mle-snr}
    By Equations~\eqref{eq:final-lower-risk} and~\eqref{eq:final-lower-risk-multiple}, when $m \geq 6$, we may take $m' = 1$ and conclude that, for some constant $c_0$,
    \begin{equation*} \begin{aligned}
        \inf_{\vetahat \in [0,1]^{n}} \sup_{\mBstar \in \calB_S(b, n, m)} \E_{\mBstar}\|\vetahat-\veta\|_{\rmH} &\geq 
        c_0 > 0,
    \end{aligned} \end{equation*}
    provided that $b < C N^{-1/2} \sigma n^{1/4}\log^{1/4} \left(n/3\right)$.
    To express this condition in terms of the signal-to-noise ratio, define
    \begin{equation*}
	\SNR = \frac{ \min_{i \in \Istar} \sqrt{\sum_{j \in ([n]\backslash \Istar)\cup \{i\}} {B^\star_{ij}}^2 }} {\sigma / \sqrt{N}}.
    \end{equation*}
    Our results imply that no estimator can achieve exact recovery of the node support $\Istar$ when 
    \begin{equation*}
        \SNR < C n^{1 / 4} \log ^{1 / 4} n
    \end{equation*}
    for some universal constant $C > 0$.
    We refer to the rate $\Omega(n^{1 / 4} \log ^{1 / 4} n)$ as the {\em exact recovery threshold}.
    
    Moreover, taking $m' = m / 2$ in Equations~\eqref{eq:final-lower-risk} and~\eqref{eq:final-lower-risk-multiple} yields 
    \begin{equation*}
    \inf_{\vetahat \in [0,1]^{n}} \sup_{\mBstar \in \calB_S(b, n, m)}
				\E_{\mBstar}\|\vetahat-\veta\|_{\rmH}
    \geq \frac{1}{8} m,
    \end{equation*}
    provided that $b < C N^{-1/2} \sigma n^{1/4}\log^{1/4} \left(2n/3m\right)$ and $m \geq 25$. 
    Therefore, to recover a {\em non-trivial} fraction of the entries of \( \veta \) correctly, Equation~\eqref{eq:minimax-lower-bound} implies that we require
    \begin{equation*}
        \SNR = \Omega\left( n^{1 / 4} \log ^{1 / 4}\left(n/m\right) \right).
    \end{equation*}
    We refer to this lower bound on the SNR as the {\em partial recovery threshold}.
\end{remark}

Both the exact recovery threshold and the partial recovery threshold can be achieved by the least squares estimator (LSE) when the noise satisfies $L \asymp \sigma$, 
as established in Equation~\eqref{eq:b-MLE-partial-sufficient} in Theorem \ref{thm:mle-result}.
Consequently, the LSE is minimax rate-optimal in recovering $\Istar$ under this regime. 

\subsection{Semidefinite Programming: Single Observation} \label{sec:single-sdp}

Motivated by the optimality of the LSE, we now seek a computationally tractable approximation by formulating a Semidefinite Programming (SDP) relaxation. 
In this section, we continue to focus on the case where $|\calG_1| = 1$. 
Recall from Equation~\eqref{eq:Ytil:define} that $\mYtil$ denotes the residual matrix after subtracting off the initial estimate produced by procedure $\calA_0$. 
Let $\mDeltatil$ denote the estimation error matrix $\mDelta \in \R^{|\calU| \times |\calU|}$ restricted to rows and columns indexed by $\calU$, so that
\begin{equation}
\label{eq:mDeltatil:define}
    \Deltatil_{ij} := \Delta_{ij} \quad \text{for } i, j \in \calU.
\end{equation}
Similarly, let $\mWtilo$ denote the restriction of the noise matrix $\mWo$ to the same index set. 
Following the same reason stated in Remark~\ref{rem:calU:wlog}, without loss of generality, we assume $\calU^\star = [n]$ so that with high probability, we have $\calU = [n]$.
Under this setup, the residual matrix can be decomposed as
\begin{equation*}
\mYtil = \mBstar + \mWtilo + \mDeltatil.
\end{equation*}
We define the incoherence parameter for the remaining rows of $\mUstar$ as
\begin{equation*}
\mut := \max_{i \in \calU} \frac{n}{r} \left\|\mU^\star_{i,\cdot}\right\|^2_2 \lesssim n^{1/2},
\end{equation*}
where the last inequality follows from Step~\ref{alg:general:step:select} in Algorithm~\ref{alg:general}.
Let $\mC := \mYtil \circ \mYtil$ be the entrywise square and let $K = n-m$. 
We consider the following convex relaxation of the LSE as defined in Equation~\eqref{eq:least-square-B}: 
\begin{equation} \label{eq:prime:sdp}
\begin{aligned}
\widehat{\mZ}_{\SDP}=\underset{\mZ\in \R^{n \times n}}{\arg \min } & \langle \mC, \mZ\rangle  \\
\text { s.t. } & \mZ \succeq 0 \\
& Z_{i i} \leq 1, \quad \forall i \in[n] \\
& Z_{i j} \geq 0, \quad \forall i, j \in[n] \\
& \langle \mI, \mZ\rangle=K \\
& \langle \mJ, \mZ\rangle=K^2 .
\end{aligned}
\end{equation}
This SDP formulation is closely related to the one proposed in \cite{hajek2016achieving} for the community detection problem under the SBM with two clusters. 
Similar SDP-based relaxations have been extensively studied in the context of community detection.
See the survey \cite{li2021convex-relax} for a comprehensive overview of convex optimization methods in this domain.

Denote the indicator vector for nodes not in the support set and the corresponding support matrix as
\begin{equation} \label{eq:def:mZstar}
\vnu_\star := \vone_n - \veta^\star
	\quad \text{and} \quad
\mZstar := \vnu_\star \vnu_\star^\top.
\end{equation}
We aim to identify sufficient conditions under which the SDP formulation in Equation~\eqref{eq:prime:sdp} recovers the ground truth matrix $\mZ^\star$.
Given the recovered support encoded by $\mZhat_{\SDP}$, we construct an estimator for $\mBstar$ by masking the residual matrix $\mYtil$:
\begin{equation*}
    \mBhat_{\SDP} := \mZhat_{\SDP} \circ \mYtil. 
\end{equation*}

To understand the theoretical performance of the SDP estimator in Equation~\eqref{eq:prime:sdp}, 
we begin by analyzing a simplified structure for $\mBstar$.
Specifically, we consider the case where   
\begin{equation}\label{eq:B-more-general}
\left|B_{i j}^{\star}\right|= \begin{cases}0, & \text { if } i, j \in {\Istar}^c \\ 
0, & \text { if } i, j \in \Istar \\ 
\beta, & \text { otherwise. }\end{cases}
\end{equation}
A more general structure we will later consider is: 
\begin{equation}\label{eq:B-more-and-more-general}
\left|B_{i j}^{\star}\right| \begin{cases} = 0, & \text { if } i, j \in \Istar^c \\ 
\geq 0, & \text { if } i, j \in \Istar \\ 
\geq \beta, & \text { otherwise. }\end{cases}
\end{equation}
which permits heterogeneous signal strengths within the support of $\mBstar$. 

We quantify the estimation error between $\mZhat$ and $\mZstar$ using the $\ell_1$ norm.
This choice facilitates the analysis of support recovery guarantees for $\Istar$. 
The following key lemma provides a bound on this error.

\begin{lemma} \label{lem:key-lower}
Suppose that $\mBstar$ satisfies Equation~\eqref{eq:B-more-general}.
Let 
\begin{equation} \label{eq:mCo:define}
\mCo = (\mBstar + \mWtilo) \circ (\mBstar + \mWtilo),
\end{equation}
where $\mWtilo$ is a noise matrix satisfying Assumption~\ref{assump:noise}.
If $n > 3m$, 
then for $\mZhat$ and $\mZstar$ defined in Equations~\eqref{eq:prime:sdp:equal} and~\eqref{eq:def:mZstar} respectively, we have 
\begin{equation} \label{eq:sdp-key-upper2}
    \|\mZhat - \mZstar\|_1 \leq 
    \frac{5}{2\beta^2} \left\langle \mZstar-\mZhat ,
			\mC -\E \mCo + \mSig - \sigma_0^2 \mJ\right\rangle,
\end{equation}
where $\mSig$ and $\sigma_0$ are as defined in Assumption~\ref{assump:homoscedastic}.
Furthermore, conditioned on the event $\{\|\mDeltatil\|_{\F} \leq C\sigma \sqrt{n r}\}$ and under Assumption~\ref{assump:homoscedastic}, we have 
\begin{equation*}
\|\mZhat - \mZstar\|_1
\lesssim \frac{1}{\beta^2}
\left| \left\langle\mZstar-\mZhat,
	\mCo - \E \mCo + 2\mDeltatil \circ \mWtilo \right\rangle \right|
+ \frac{\sigma^2 n^{3/2} + \sigma^2 n r}{\beta^2}
+ \frac{\sigma r^{1/2} n^{3/2}}{\beta}. 
\end{equation*}
\end{lemma}
The proof of Lemma~\ref{lem:key-lower} is provided in Appendix~\ref{sec:proof:sdp:key-lower}.
We note that a similar result was established by \cite{fei2018exponential} in the context of community detection under the SBM. 
We will adapt several proof techniques from their work to obtain a sharp bound on the estimation error $\|\mZhat - \mZstar\|_1$ in our setting.
To this end, it suffices to derive a upper bound on the right-hand side of Equation~\eqref{eq:sdp-key-upper2}.


Our first step toward analyzing the estimation error is to establish a preliminary, albeit coarse, upper bound for the matrix inner product using Grothendieck's inequality, 
which states that for any symmetric matrix $\mA \in \R^{n \times n}$,
\begin{equation}\label{eq:Grothendieck}
\sup _{\substack{\mZ \succeq 0, \\ \max_{i \in [n]} Z_{ii} \leq 1}} 
|\langle\mZ, \mA \rangle|
\leq c_0 \max_{\vx, \vy \in\{ \pm 1\}^{n}}
	\left|\sum_{i, j} A_{ij} x_i y_j\right|,
\end{equation}
where $0 < c_0 < 1.783$ is an absolute constant. 
For a detailed reference, see Theorem 3.5.1 in \cite{vershynin2018HDP}.

\begin{lemma} \label{lem:Grothendieck:bound}
Under the same conditions as Lemma \ref{lem:key-lower}, suppose that $\beta \leq L$. 
Conditioned on the event 
\begin{equation} \label{eq:Grothendieck:condition}
\{\|\mDelta\|_\F \leq C\sigma \sqrt{r n}\}
\cap
\{\|\mDelta\|_{\infty} \leq C \sigma\}
\end{equation}
for some constant $C > 0$, we have with probability at least $1 - O(n^{-40})$
\begin{equation*}
\left|\langle \mZhat-\mZstar,
		\mCo - \E \mCo + 2 \mDeltatil \circ \mWtilo\rangle\right|
\leq C' \sigma L n^{3/2}
\end{equation*}
for a constant $C'$ sufficiently large. 
\end{lemma}

A proof is provided in Appendix~\ref{sec:proof:sdp:grothendieck}.
It is reasonable to focus on the regime $\beta \leq L$.
Otherwise, the signal in the entries of $\mYtil$ would dominate the noise, rendering the support recovery problem trivial.
Combining Lemma~\ref{lem:Grothendieck:bound} with Equation~\eqref{eq:sdp-key-upper2} shows that when $\beta = \omega(\sigma^{1/2} L^{1/2} n^{-1/4})$, the estimation error $\|\mZhat - \mZstar\|_1$ satisfies
\begin{equation} \label{eq:weak-control}
\frac{1}{n^2} \left\|\mZhat - \mZstar\right\|_1 = o_{\Pr}(1).
\end{equation}
While this provides a consistency guarantee, the bound in Equation~\eqref{eq:weak-control} is rather weak: it becomes vacuous as soon as $m = o(n)$, since in that case any feasible solution to the SDP in Equation~\eqref{eq:prime:sdp} would trivially satisfy the bound. 
Nonetheless, establishing Equation~\eqref{eq:weak-control} serves as a crucial first step that enables sharper control through more refined analysis.

To improve upon Equation~\eqref{eq:weak-control}, we leverage the low-rank structure of $\mZstar$.
Recall from Equation~\eqref{eq:def:mZstar} that $\mZstar$ is a rank-one matrix with spectral decomposition given by
\begin{equation*}
    \mZstar = K \left(\frac{1}{\sqrt{K}} \vnu_\star\right)
	\left(\frac{1}{\sqrt{K}} \vnu_\star\right)^\top.
\end{equation*}
Motivated by the geometric structure of the manifold of rank-one symmetric matrices, we decompose the key inner product in Lemma~\ref{lem:key-lower} into components along the tangent space at $\mZstar$ and its orthogonal complement.
 For further background on tangent space geometry in low-rank matrix manifolds, see Chapter 7.5 of \citet{boumal2023introduction}.
Specifically, following the approach of \cite{fei2018exponential}, we define 
\begin{equation}\label{eq:S1:define}
S_1 := \left \langle
	\mZstar-\mZhat,
	\calP_T\left(\mC - \E \mCo + \mSig - \sigma_0^2 \mJ\right)
	\right \rangle 
\end{equation}
and
\begin{equation}\label{eq:S2:define}
    S_2 := \left \langle \mZstar-\mZhat,
		\calP_{T^\perp}\left(\mC -\E \mC^{(1)}
		+ \mSig - \sigma_0^2 \mJ\right)
	\right \rangle
\end{equation}
and decompose the key inner product term as
\begin{equation*} 
\begin{aligned}
    & \left \langle \mZstar-\mZhat, \mC - \E \mCo + \mSig - \sigma_0^2 \mJ\right \rangle = S_1 + S_2,
\end{aligned} 
\end{equation*}
where for any matrix $\mA \in \R^{n\times n}$, 
the projections $\calP_{T}(\mA)$ and $\calP_{T^\perp}(\mA)$ are defined as 
\begin{equation} \label{eq:calP:define}
\calP_{T}(\mA) := K^{-1} \mZstar \mA + K^{-1}\mA \mZstar - K^{-2} \mZstar \mA \mZstar,
\quad \text{and} \quad 
\calP_{T^\perp}(\mA) := \mA - \calP_{T}(\mA).
\end{equation}
This decomposition is standard in the low-rank matrix estimation literature.
See, for example, \cite{recht2011simpler} and \cite{negahban2011estimation}. 
The operator $\calP_T$ captures deviations aligned with the column and row space of $\mZstar$, while $\calP_{T^\perp}$ captures orthogonal directions.
Bounding $S_1$ and $S_2$ precisely is key to obtaining a sharp control of the estimation error. 
This is carried out in
Lemmas ~\ref{lem:S2:bound} and~\ref{lem:S1:bound} in Appendices~\ref{subsec:lem:S1:bound} and~\ref{subsec:lem:S2:bound}, respectively.

We are now ready to state a result characterizing the estimation error $\|\mZhat - \mZstar\|_1$ under the SDP relaxation.

\begin{theorem}\label{thm:sdp:L1-bound}
Suppose the noise matrix $\mWtilo$ satisfies Assumptions~\ref{assump:noise} and~\ref{assump:homoscedastic}, as well as Equation~\eqref{eq:L-sigma}. 
Assume the conditions of Theorem~\ref{thm:delta_row} hold. 
Furthermore, suppose $n > 3m$, $\beta \leq L$, and either $\beta = \Omega(\sigma^{1/2} L^{1/2} n^{-1/4})$ or $m = o(n)$. 
If $\mut r^{3/2} \leq c n \beta^2 / \sigma^2$ for some sufficiently small constant $c > 0$, 
then any solution to the SDP in Equation~\eqref{eq:prime:sdp} with $\mBstar$ given by Equation~\eqref{eq:B-more-general} satisfies
\begin{equation} \label{eq:tau-bound}
\|\mZhat - \mZstar\|_1 
\leq
n^2 \exp\left\{ -\frac{C n\beta^4}{\sigma^2 L^2} \right\}
\end{equation}
with probability at least $1 - O(n^{-6})$. 
\end{theorem}


The proof is included in Appendix~\ref{subsec:thm:sdp:L1-bound}.
Next, we extend the results established under the simplified model in Equation~\eqref{eq:B-more-general} to the more general setting of Equation~\eqref{eq:B-more-and-more-general}, 
which allows for heterogeneous signal strengths on the support of $\mBstar$.
To accommodate this extension, we introduce the following condition on the noise distribution:
\begin{assumption}\label{assump:mono}
For any $1 \leq i \leq j\leq n$, $W_{ij}$ is a continuous random variable. 
The density function $f_{ij}$ of $W_{ij}$ is such that $f_{ij}(x) \geq f_{ij}(y)$ for any $0 \leq x \leq y$. 
\end{assumption}
Combined with the symmetric-about-zero conditon in Assumption~\ref{assump:noise}, this monotonicity condition ensures that smaller absolute values of $W_{ij}$ are more likely, which allows us to apply stochastic dominance arguments.
\begin{corollary} \label{cor:sdp-more-general}
Suppose the noise matrix $\mWtilo$ satisfies Assumptions~\ref{assump:noise} and~\ref{assump:homoscedastic}, as well as Equation~\eqref{eq:L-sigma}. 
Assume the conditions of Theorem~\ref{thm:delta_row} hold. 
If $n > 3m$, 
$\mut r^{3/2} \leq c n \beta^2 / \sigma^2$ for some sufficiently small constant $c > 0$ and either $\beta = \Omega(\sigma^{1/2} L^{1/2} n^{-1/4})$ or $m = o(n)$,
then under Assumption~\ref{assump:mono}, 
any solution to the SDP in Equation~\eqref{eq:prime:sdp} with $\mBstar$ from Equation~\eqref{eq:B-more-general} satisfies 
\begin{equation*} 
\|\mZhat - \mZstar\|_1
\leq n^2 \exp\left\{ -\frac{C n\beta^4}{\sigma^2 L^2} \right\}
\end{equation*}
with probability at least $1 - O(n^{-6})$. 
\end{corollary}

When $L/\sigma = \Theta(1)$, as a consequence of Theorem~\ref{cor:sdp-more-general}, 
if $\beta = C\sigma n^{-1/4}\log^{1/4} n$ for some sufficiently large constant $C$, 
any solution $\mZhat$ to the SDP in Equation~\eqref{eq:prime:sdp} with $\mBstar$ from Equation~\eqref{eq:B-more-general} obeys the bound
\begin{equation*}
    \left\|\mZhat - \mZstar\right\|_1 \leq c n
\end{equation*}
with high probability, where $c>0$ is a small constant.
In this case, we can recover the exact support $\Istar$ by sorting the rows of $\mZhat$ according to their sums and selecting the indices corresponding to the $m$ smallest row sums.
Moreover, if $\beta = C\sigma n^{-1/4} \log^{1/4}(n/m)$ for some sufficiently large $C$, then the solution satisfies
\begin{equation*}
\left\|\mZhat - \mZstar\right\|_1 \leq c m n.
\end{equation*}
Under the above setting, for $\mZhat$, at most $2cm$ rows of $i \in \Istar$ have row sum greater than $n/2$, while at most $2cm$ rows $i \in I_{\star}^c$ have row sum less than $n/2$, and we have $|I_{\star}^c \cap \Ihat| + |I_{\star} \cap \Ihat^c| \leq 4cm$. 
Thus, when $c$ is sufficiently small, this guarantees partial recovery of the support $\Istar$.  

\subsubsection{Minimax Recovery for SDP}

Compared to the minimax lower bound in Corollary~\ref{cor:minimax-recoverable}, 
the SDP approach imposes an additional structural requirement, namely, an entrywise lower bound on the signal strength of $\mBstar$.
While it may be possible to derive stronger guarantees for SDP under weaker conditions, we instead focus on a restricted parameter space where this additional assumption holds.
Toward this end, define the restricted class
\begin{equation*}
    \calB'_S(\beta, n, m) := \left\{\mB \in \Sym(n): 
    |I_{\mB}| \leq m,~
    \begin{array}{cc}
	|B_{ij}| \geq \beta, & \forall i \in I_{\mB}, j\in I_{\mB}^c \\
        B_{i j}=0, & \forall i, j \in I_{\mB}^c
    \end{array}\right\},
\end{equation*}
which enforces a uniform lower bound $\beta$ on the signal magnitude between active and inactive nodes.

We show that over the class $\calB'_S(\beta, n, m)$, the SDP estimator achieves the minimax optimal rate for support recovery under Gaussian noise.

\begin{theorem}\label{thm:sdp-minimax}
Suppose we observe $N \geq 1$ independent samples of the form 
\begin{equation*}
    \mYtil_{\ell} = \mBstar + \mW_{\ell}, ~~~\ell=1,2,\dots,N,
\end{equation*}
where $\mW_{\ell}$ are independent symmetric noise matrices with i.i.d. entries $W_{\ell, ij}$ drawn from $\calN(0,\sigma^2)$. 
Then there exists a universal constant $C>0$ such that for sufficiently large $n$ and any $m^{\prime} \in[1, m]$ with $m<n / 3$, if
\begin{equation*}
\beta
< C \sigma n^{-1 / 4} \log^{1 / 4}\left(n /\left(3 m^{\prime}\right)\right),
\end{equation*}
then
\begin{equation*}
    \inf _{\hat{\boldsymbol{\eta}} \in[0,1]^n} \sup _{\mBstar \in \calB'_S(\beta, n, m)} \mathbb{E}_{\mBstar}\|\boldsymbol{\eta}-\hat{\boldsymbol{\eta}}\|_\rmH \geq 4 m^{\prime}\left[\frac{1}{8}-\exp \left(-\frac{\left(m-m^{\prime}\right)^2}{2 m}\right)\right].
\end{equation*}
\end{theorem}

\begin{remark}
Following Remark~\ref{rem:mle-snr}, define
\begin{equation*}
\SNR_{\beta}
:= \frac{\min_{i\in \Istar, j\in I_\star^c} |B_{ij}^\star|}{\sigma / \sqrt{N}}.
\end{equation*}
Theorem~\ref{thm:sdp-minimax} shows that no method can achieve exact recovery  
over $\calB_S'(\beta, m, n)$ when 
\begin{equation*}
    \SNR_\beta < C n^{-1 / 4} \log ^{1 / 4} n.
\end{equation*}
We call $n^{-1 / 4} \log ^{1 / 4} n$ the {\em exact recovery threshold} over $\calB_S'(\beta, m, n)$. 
Following Remark~\ref{rem:mle-snr}, we also call $ n^{-1 / 4} \log ^{1 / 4} (n/m)$ the {\em partial recovery threshold} over $\calB_S'(\beta, m, n)$. 
As shown in Corollary~\ref{cor:sdp-more-general}, both the exact and partial recovery thresholds over $\calB_S'(\beta, m, n)$ can be achieved by SDP. 
\end{remark}

\subsubsection{Truncated SDP}
\label{sec:truncated-sdp}

Compared to the minimax lower bound established in Theorem~\ref{thm:sdp-minimax}, the estimation error bound in Theorem~\ref{thm:sdp:L1-bound} demonstrates that the SDP estimator in Equation~\eqref{eq:prime:sdp} achieves the minimax rate. 
This result relies on two key assumptions: 
(i) the noise is approximately isotropic, 
and (ii) the truncation level satisfies $L \asymp \sigma$.
In this section, we examine the setting where the truncation level is much larger than the noise scale, 
i.e., $L \gg \sigma$. 
We defer analysis of the heteroscedastic case to Section~\ref{sec:multiple-sdp}.

When $L$ is large, 
the observed noise entries $W_{ij}^{(1)}$ as defined in Equation~\eqref{eq:model:calG1} 
may exhibit heavy-tailed behavior, 
potentially following a power law distribution. 
This presents challenges for both the SDP and the LSE, as the $\ell_2$ norm is sensitive to large outliers and lacks robustness under such noise conditions. 
A principled remedy is to replace the squared loss in the least squares formulation with a robust alternative, such as the Huber loss or Jaeckel's dispersion function, and then develop a suitable convex relaxation or other computationally tractable approximation. 
For an overview of such methods, see standard references on robust statistics, e.g., \citet{hettmansperger2010robust, huber2011robust}. 
We leave this direction for future work.
Instead, to demonstrate that our SDP-based approach can be extended to handle some heavy-tailed scenarios, we instead adopt a simple truncation strategy. 

Consider the truncation operator $\calT_{\gamma}(\cdot)$ defined
entrywise as
\begin{equation*}
    \calT_{\gamma}(a) = a \wedge \gamma. 
\end{equation*}
Let $\tau_0^2$ be a threshold that serves as a proxy for the noise variance $\sigma^2$. 
To reduce the impact of extreme noise values, 
we apply truncation to the entrywise square of the residual matrix $\mYtil$, 
and solve the SDP using the truncated matrix
\begin{equation*}
    \mCbar := \calT_{\tau_0^2}(\mYtil \circ \mYtil)
\end{equation*}
in place of $\mC$ in Equation~\eqref{eq:prime:sdp}.

We impose the following regularity condition on the noise distribution.
\begin{assumption} \label{assump:lipschitz}
    For $c \sigma \leq \tau_0 \leq C \sigma$, where $c$ is a sufficiently large constant, we have for any $|x| \leq \frac{\tau_0}{2}$,
    \begin{equation*}
        |f(\tau_0+x) - f(\tau_0)| \leq \gamma x, \quad \text{and } \gamma \leq c' \sigma^{-2}(\tau_0/\sigma)^{-3}
    \end{equation*}
    for some sufficiently small constant $c' > 0$. 
\end{assumption}
\begin{remark}
A distribution $f(x; \mu, \sigma)$ is a \emph{location-scale family} if it can be expressed as
\begin{equation*}
    f(x; \mu, \sigma) = \frac{1}{\sigma} g\left(\frac{x - \mu}{\sigma}\right),
\end{equation*} 
where $g$ is a standardized density function with mean $0$ and variance $1$, $\mu$ is a location parameter and $\sigma$ is a scale parameter.
Assumption~\ref{assump:lipschitz} is satisfied by location-scale family with $g$ that has a small Lipschitz constant.
Indeed, one has 
\begin{equation*}
    \left|f(x_1) - f(x_2)\right| \leq \gamma_0 \left|x_1 - x_2\right| \quad \text{for all } x_1, x_2 \in \R
\end{equation*}
with 
\begin{equation*}
    \gamma_0 = \sigma^{-2} \sup_{z \in \R} \left|g'(z)\right|. 
\end{equation*}
Thus, if the Lipschitz constant of $g$ is sufficiently small, then Assumption~\ref{assump:lipschitz} holds.
For example, consider the $t_3$ distribution, which has density
\begin{equation*}
    f(t) = \frac{2}{\eta \pi \sqrt{3}\left(1+\frac{t^2}{3\eta^2}\right)^2},
\end{equation*}
standard deviation of $\sqrt{3}\eta$ and derivative given by
\begin{equation*}
    f'(t) = \frac{4t}{3\eta^3 \pi\sqrt{3}\left(1+\frac{t^2}{3\eta^2}\right)^3} \leq \frac{t/\eta}{4\eta^2\left(1+\frac{t^2}{3\eta^2}\right)^3} \leq \frac{9}{4\eta^2 (t/\eta)^{4}}.
\end{equation*}
This satisfies Assumption~\ref{assump:lipschitz} with $c_1 = 1/2$ when $t > 9\eta$. By picking a larger $t$, one can then choose a smaller value for $c_1$.  
\end{remark}

To select a suitable truncation level $\tau_0$, 
we require it to lie within a constant factor of the noise standard deviation:
\begin{equation} \label{eq:tau-lower-upper-bound}
c_0 \sigma \leq \tau_0 \leq C_0 \sigma
\end{equation}
for some fixed constants $0 < c_0 < C_0$. 
We obtain $\tau_0$ using a subset of rows from the initial observation matrix $\mYz$. 
Specifically, define the index set
\begin{equation*}
\calS
:= \left\{\ell \in [n] : 
	\sqrt{n} \left\|\mU_{\ell, \cdot}\right\|_2 \leq C \sqrt{\log n} 
	\right\},
\end{equation*}
and denote the submatrix of $\mYz$ corresponding to the rows and columns in $\calS$ as
\begin{equation*}
\mYz_{\calS} := \left(Y^{(0)}_{ij}\right)_{(i,j) \in \calS \times \calS} \in \R^{|\calS| \times |\calS|}. 
\end{equation*}
Using the estimated low-rank component $\mMz$ obtained from $\mYz$,
we define the residual matrix as 
\begin{equation} \label{eq:mWcheck:define}
    \mWcheck := \mYz_{\calS} - \mMhat_{\calS},
\end{equation}
and estimate the truncation threshold $\tau$ via
\begin{equation}\label{eq:tauhat:define}
    \tau := \frac{1}{|\calS|} \left\|\mWcheck\right\|_{\F}.
\end{equation}
The following theorem guarantees that $\tau$ satisfies the required bounds in Equation~\eqref{eq:tau-lower-upper-bound} with high probability. 
A proof is included in Appendix~\ref{sec:proof:thm:tauhat}. 
\begin{theorem}\label{thm:tauhat}
Suppose the observed matrix $\mYz$ follows the model in Equation~\eqref{eq:model:calG1} with $|\calG_1| = 1$ and that the noise matrix $\mWz$ satisfies Assumption~\ref{assump:noise} along with the condition in Equation~\eqref{eq:L-sigma}. 
If $\kappa r = O(1)$, then  
the estimated threshold $\tau$ satisfies the bound in Equation~\eqref{eq:tau-lower-upper-bound} with probability at least $1 - O(n^{-6})$. 
\end{theorem}

With a properly-chosen threshold $\tau$, we now turn to its application in robustifying the SDP estimator.
We truncate the input matrix used in the SDP formulation and establish the following robust recovery guarantee. 
Specifically, we truncate the input matrix in the SDP formulation by applying the entrywise operator $\calT_{\tau^2}$ to $\mYtil \circ \mYtil$, 
resulting in the modified input
$\mCbar := \calT_{\tau^2}(\mYtil \circ \mYtil)$. 
We then solve the SDP in Equation~\eqref{eq:prime:sdp} using $\mCbar$ in place of the original matrix $\mC$.
To establish theoretical guarantees for this truncated formulation, 
we impose an additional condition to control the second-order behavior of the truncated noise matrix.
\begin{assumption}\label{assump:strong:homoscedastic}
For any $\tau$ satisfying $c_0 \sigma \leq \tau \leq C_0 \sigma$, let 
\begin{equation*}
    \mWbar := \calT_{\tau^2}(\mWtilo \circ \mWtilo).
\end{equation*}
There exists a scalar $\sigma_{\tau}^2$ such that 
\begin{equation*}
\left\|\E (\mWbar \circ \mWbar) - \sigma^2_{\tau} \mJ \right\|_{1}
\leq C \sigma^2 n^{3/2}
 \quad \text{and} \quad
\left\|\E (\mWbar \circ \mWbar) - \sigma^2_{\tau} \mJ \right\|
\leq C \sigma^2 \sqrt{n} .
\end{equation*}
\end{assumption}

\begin{remark}
By requiring uniform concentration for a range of $\tau$, Assumption~\ref{assump:strong:homoscedastic} is stronger than Assumption~\ref{assump:homoscedastic}.
It is satisfied, for example, when the majority of entries in $\mWtilo$ have identical variance and only a small subset exhibit larger deviations.
\end{remark}

The next theorem establishes a non-asymptotic error bound for the truncated SDP estimator.
\begin{theorem}\label{thm:trunc:sdp}
Suppose the noise matrix $\mWtilo$ satisfies Assumptions~\ref{assump:noise} and~\ref{assump:homoscedastic}, as well as Equation~\eqref{eq:L-sigma}. 
Assume the conditions of Theorem~\ref{thm:delta_row} hold. 
When $n > 3m$, if $\kappa = O(1)$, $\mut r^{3/2} = O(\sqrt{n})$, and either $\beta = \Omega(\sigma n^{-1/4})$ or $m = o(n)$ holds, 
then under Assumption~\ref{assump:mono}, \ref{assump:lipschitz} and~\ref{assump:strong:homoscedastic}, 
any solution to the SDP in Equation~\eqref{eq:prime:sdp} with $\mBstar$ from Equation~\eqref{eq:B-more-general} 
and $\mCbar$ substituted for $\mC$, 
satisfies
\begin{equation*}
\left\|\mZhat - \mZstar\right\|_1
\leq n^2 \exp\left\{ -\frac{C n \beta^4}{\sigma^4} \right\}
\end{equation*}
with probability at least $1 - O(n^{-6})$. 
\end{theorem}

A proof is provided in Appendix~\ref{subsec:proof:trunc:sdp}.
As a result of Theorem~\ref{thm:trunc:sdp}, we conclude that under some heavy-tailed noise regimes, the exact and partial recovery thresholds can still be achieved by SDP, after suitable entrywise truncation.


\subsection{SDP: Multiple Observations} \label{sec:multiple-sdp}

We now extend the above results to $|\calG_1| > 1$.
In general, if we do not impose any additional assumptions on $\mBstar_j$ for $j \in \calG_1$, we may treat each $j$ to be its own singleton group $\calG_{1,j}$, and apply the methods introduced in the previous section independently to recover the support $\Istar_j$ for each $j \in \calG_1$.
In particular, the SDP-based procedures introduced earlier remain applicable and yield accurate recovery under the same conditions as Theorems~\ref{thm:sdp:L1-bound} and~\ref{thm:trunc:sdp}.

However, when multiple observations share the same underlying difference matrix $\mBstar$, i.e., $\mBstar_j = \mBstar$ for several $j \in \calG_1$, we can refine both the estimation procedure and the accompanying analysis.
In Section~\ref{sec:minimax-support}, we demonstrated that Assumption~\ref{assump:homoscedastic} is essential to attain the minimax lower bound in Theorem~\ref{thm:symmetric:minimax} when $|\calG_1| =1$.
Remarkably, we now demonstrate that this assumption can be relaxed in the presence of multiple independent observations sharing the same $\mBstar$.
Specifically, suppose we are given two independent copies $\mYo_1$ and $\mYo_2$ following Equation~\eqref{eq:model:calG1} with $\mBstar_1 = \mBstar_2 = \mBstar$.
Constructing $\mYtil_1$ and $\mYtil_2$ from Equation~\eqref{eq:Ytil:define} in Algorithm~\ref{alg:general},
we set $\mCtil = \mYtil_1 \circ \mYtil_2$ and run the SDP in Equation~\eqref{eq:prime:sdp} on $\mCtil$ in place of $\mC$. 
Owing to the independence of $\mYtil_1$ and $\mYtil_2$, we have
$\E \mCtil = \mBstar \circ \mBstar$.
This construction effectively removes the bias introduced by heteroscedastic noise, enabling accurate estimation under weaker assumptions.

\begin{theorem}\label{thm:sdp:multiple}
Assume the conditions of Theorem~\ref{thm:delta_row} hold.     
Under Assumption~\ref{assump:noise}, if $\tilde{\mu} r^{3/2} \leqslant c n \beta^2 / \sigma^2$, any solution $\mZhat$ to the SDP in Equation~\eqref{eq:prime:sdp}, when applied to $\mCtil$ in place of $\mC$ and with $\mBstar$ given by Equation~\eqref{eq:B-more-and-more-general}, satisfies 
\begin{equation*}
\left\|\mZhat - \mZstar\right\|_1 
\lesssim
n^2 \max\left\{ \exp\left(-\frac{C n \beta^2}{L^2}\right),
		\exp\left(-\frac{C n \beta^4}{\sigma^4}\right) \right\}
\end{equation*}
with probability at least $1 - O(n^{-6})$. 
\end{theorem}

Theorem~\ref{thm:sdp:multiple} holds under significantly weaker assumptions compared to Theorem~\ref{thm:sdp:L1-bound} and Corollary~\ref{cor:sdp-more-general}. 
In particular, it does not require the near-isotropic noise condition in Assumption~\ref{assump:homoscedastic}, 
nor does it rely on Assumption~\ref{assump:mono} regarding the noise density. 
Moreover, the bound remains valid for any $\mBstar$ that satisfies the general structure in Equation~\eqref{eq:B-more-and-more-general}. 
When the noise parameters are such that $L \asymp \sigma$, 
the SDP based on multiple observations also achieves the exact and partial recovery thresholds established in Theorem~\ref{thm:sdp-minimax}.
Moreover, this approach naturally extends to $N$ independent observations via sample splitting and averaging: divide the $N$ observations into two groups of size $N/2$, compute the empirical averages within each group to obtain $\mYtil_1$ and $\mYtil_2$, and apply the same procedure using the elementwise product $\mCtil = \mYtil_1 \circ \mYtil_2$.
As discussed above in Section~\ref{sec:truncated-sdp}, when $L \gg \sigma$, the performance of the SDP might be suboptimal due to heavy-tailed noise.
We anticipate that a truncation-based approach similar to that presented in Section~\ref{sec:truncated-sdp} can be applied in the present setting, but we leave a thorough investigation of this regime to future work.

\begin{remark}
A similar phenomenon to that described in Theorem~\ref{thm:sdp:multiple}, in which access to multiple networks leads to a nontrivial improvement, has also been observed in \cite{ghoshdastidar2020two}. 
In that work, the authors studied the problem of two-sample graph testing under Bernoulli edge weights, with the goal of determining whether two collections of graphs are drawn from the same distribution.
In contrast, our focus is on weighted networks and recovery of the node support of the difference matrix $\mBstar$.
The underlying intuition is that when only a single observation is available, the difference matrix $\mBstar$ behaves similar to a noise matrix, necessitating stronger assumptions for recovery or testing to be feasible.
By contrast, when multiple observations are available, the signal component $\mBstar$ must remain consistent across them, which enables accurate recovery under substantially weaker conditions.
\end{remark}

\subsection{Model Selection in SDP} \label{sec:m-selection}
One important practical consideration is the selection of the parameter $m$. 
Under a homoscedastic Gaussian model, the entries of the matrix outside the support of $\mBstar$ follow a noncentral chi-squared distribution, specifically $\chi^2_1(\Delta^2_{ij})$, where $\mDelta$ is the residual matrix defined in Equation~\eqref{eq:mDelta:define}. 
Assume without loss of generality that the set $\calU$ in Step~\ref{alg:general:step:select} is $[n]$. 
For all $i \in I_\star^c$, it follows from Theorem~\ref{thm:delta_row} and $\chi^2$-concentration inequality that
\begin{equation*}
\left|\sum_{j\in I_\star^c} \Ytil^2_{ij} - \sigma^2 (n-m)\right|
\leq C \sigma^2 \sqrt{(n-m)\log n} + C \sigma^2 \kappa^4 \mut r + C \sigma^2 r \log n
\end{equation*}
holds with high probability.
When $\kappa^4 \mut r = o(\sqrt{n \log n})$ and $r = o(\sqrt{n /\log n})$, we will thus have that 
\begin{equation*}
-C \sigma^2 \sqrt{(n-m) \log n}
\leq \max_{i \in I_\star^c} \sum_{j\in I_\star^c} \Ytil^2_{ij} - \sigma^2 (n-m)
\leq C \sigma^2 \sqrt{(n-m) \log n}.
\end{equation*}
This suggests using a heuristic threshold of $C \sigma^2 \sqrt{(n - m) \log n}$ to guide the choice of $m$.
A practical selection strategy is as follows.
\begin{algorithm}[H]
\caption{Heuristic Selection of $m$ for SDP}
\begin{algorithmic}[1]
\Require Observed matrix $\mYtil$, noise variance estimate $\widehat{\sigma}^2$, initial guess $m_0$
\State Set $m \gets m_0$
\Repeat
    \State Solve the SDP in Equation~\eqref{eq:prime:sdp} with cardinality parameter $m$, and obtain solution $\mZhat$
    \State Let $\Ihat$ be the set of indices corresponding to the $m$ smallest row sums of $\mZhat$
    \State Compute the maximum row sum in $\Ihat^c$: $s_{\max} \gets \max_{i \in \Ihat^c} \sum_{j\in \Ihat^c} \Ytil^2_{ij}$
    \State Compute heuristic threshold $t \gets c \widehat{\sigma}^2 \sqrt{(n-m)\log n}$ for some constant $c > 0$
    \If{$|s_{\max} - \widehat{\sigma}^2(n-m)| > t$}
        \State Increase $m \gets m + 1$
    \Else
        \State Decrease $m \gets m - 1$
    \EndIf
\Until{ $s_{\max} > t$ and the maximum row sum drops below threshold when $m$ is decreased by one}
\State \Return Final value of $m$
\end{algorithmic}
\end{algorithm}

The above procedure combines an SDP with a hard-thresholding step. While hard-thresholding alone is generally insufficient for recovering the support of $\mBstar$ (see Theorem~\ref{thm:failure}), it remains a useful diagnostic tool for detecting signal components that may be excluded from the estimated support set $\Ihat$. In particular, evaluating the maximum row sum over $\Ihat^c$ can help identify whether the current choice of $m$ is too small. To implement the thresholding rule, one may estimate the noise variance $\widehat{\sigma}^2$ using the same approach employed in constructing the truncation threshold for the SDP, as defined in Equation~\eqref{eq:tauhat:define}.
In this work, we assume that either $m$ is known a priori or there is a preferred choice based on domain knowledge.
A complete development of a data-driven selection procedure is an interesting direction for future research, but we do not pursue it here.

\subsection{Group Lasso} \label{subsec:glasso}
While semidefinite programming offers a principled framework for exact and partial recovery of the support set $\Istar$, it can be computationally intensive for large-scale problems.
Using a general-purpose interior-point method to exactly solve the SDP in Equation~\eqref{eq:prime:sdp} requires a per-iteration runtime complexity of $O(n^6)$ \citep{vandenberghe2005interior}.  
Faster implementations exist, but they still require at least $O(n^4)$ runtime complexity per iteration in our setting \citep{jiang2020faster}.
As discussed earlier in Section~\ref{sec:setup}, 
an alternative and potentially more scalable approach is to exploit the sparsity structure of $\mBstar$ through a structured penalization framework.
In particular, 
the overlapping group Lasso has been widely used in the network analysis literature, including applications in graphical model selection \citep{mohan2014node}, network regression \citep{kessler2022prediction}, and classification \citep{relion2019network}.
On the theoretical side, \cite{mohan2014node} provided sufficient conditions under which the overlapping group Lasso yields node-sparse solutions, though without guarantees on support recovery. 
\cite{relion2019network} later established high-probability support recovery guarantees under the irrepresentability condition, which unfortunately does not hold in our setting.

To formulate the overlapping group Lasso, for each $i \in [n]$,
we define the $i$-th group to be the collection of parameters 
\begin{equation*}
\{B_{ij} : j \in [n]\} \cup \{B_{ji} : j \in [n]\},
\end{equation*}
resulting in $n$ overlapping groups. 
Each parameter appears in two groups, and we maintain the symmetry of $\mB$ via the reparameterization $B_{ij} = v_{ij} + v_{ji}$ for all $i, j \in [n]$.
Applied to an observed matrix $\mY = [y_{ij}]_{1\leq i,j \leq n}$, the overlapping group Lasso solves the following convex optimization problem:
\begin{equation} \label{eq:node-glasso}
\min_{\mV \in \R^{n\times n}}
\frac{1}{2} \sum_{1\leq i \leq j\leq n} (y_{ij} - v_{ij} - v_{ji})^2
+ \lambda \Omega( \mB ),
\end{equation} 
where $\lambda \geq 0$ is the regularization parameter, $\Omega(\mB)$ is the group Lasso penalty defined as
\begin{equation*}
    \Omega(\mB) := \sum_{i=1}^n \sqrt{\sum_{j=1}^n v_{ij}^2}.
\end{equation*}
This formulation has a natural interpretation: each $v_{ij}$ can be viewed as the contribution of node $i$ to edge $(i, j)$. 
Defining
\begin{equation} \label{eq:alpha-def}
\alpha_i := \sqrt{\sum_{j=1}^n v_{ij}^2}, \quad \text{for } i \in [n],
\end{equation}
the group Lasso selects the $i$-th group if and only if $\alpha_i > 0$.
To solve the optimization problem in Equation~\eqref{eq:node-glasso}, 
following \citet{mohan2014node},
we adopt an approach based on the alternating direction method of multipliers approach  \citep[see][for a comprehensive review]{boyd2011distributed}.
Details are provided in Algorithm~\ref{alg:admm-glasso} in Appendix~\ref{sec:glasso}.

We also extend the theoretical analysis of the overlapping group Lasso in our setting. 
A particularly useful structural property that we establish is the existence of a stable solution path:
\begin{center}
     \textit{If $\alpha_i>0$ for some penalty parameter $\lambda_0$, then $\alpha_i>0$ for all $\lambda < \lambda_0$.}
\end{center}
In other words, the overlapping group Lasso behaves like a greedy procedure: 
once a node enters the model (i.e., its group is selected), it remains active as the regularization strength $\lambda$ decreases.
A formal statement and proof of this property can be found in Lemma~\ref{lem:unique} in Appendix~\ref{sec:glasso}.
While not central to our main goal of establishing exact recovery guarantees for $\Istar$, 
this result provides insight into the behavior of group Lasso estimators and may guide practical strategies for algorithm choice. 


Rather than analyzing the full solution path of the group Lasso estimator,
we focus on its performance near the exact recovery threshold.
Specifically, Theorem~\ref{thm:glaso-sufficient-crude} shows that under Gaussian noise and appropriate conditions on the signal strength and sparsity, the group Lasso achieves exact support recovery with high probability.
A proof is given in Appendix~\ref{sec:proof-thm-glaso-sufficient-crude}.

\begin{theorem} \label{thm:glaso-sufficient-crude}
Let $\mYtil$ be the symmetric matrix obtained from Step~\ref{alg:general:step:select} in Algorithm~\ref{alg:general},
where the data-generating model follows Equation~\eqref{eq:model:calG1} with signal matrix $\mBstar$ and noise matrix $\mWo$. 
Assume that $\mWo$ is a symmetric noise matrix with entries $(W^{(1)}_{ij})_{1\leq i\leq j\leq n}$ drawn i.i.d.~from $\calN(0, \sigma^2)$.
Suppose that $n > 3m$, $\kappa^2 \sqrt{\mut r} = O(n^{1/4} \log^{-1/4} n)$, and that
\begin{equation} \label{eq:growth-condition-Bstar} \begin{aligned}
\max_{i \in [m]} \|\mBstar_{i, \cdot} \|_2
&\asymp \min_{i\in [m]} \|\mBstar_{i, \cdot} \|_2
= \Theta(\sigma n^{1/4} \log^{1/4} n), \\
\max _{i \in[m]^c} \left(\sum_{k=1}^m B_{i k}^{\star 2}\right)^{1/2}
&= O(\sigma n^{1/4} \log^{1/4} n) ~~~\text{ and } \\
\min_{i\in [m]} \|\mBstar_{i, \cdot} \|_2^2
&\geq \max_{\ell \in [m]^c} \sum_{k=1}^{m} B_{\ell k}^{\star 2}
	+ C\sigma^2 \sqrt{n \log n}
\end{aligned} \end{equation}
for some suitably large constant $C > 0$.
Then, for sufficiently large $n$, there exists a regularization parameter $\lambda > 0$ such that the group Lasso estimator in Equation~\eqref{eq:node-glasso},
when applied to $\mYtil$, selects exactly the true support set $\Istar$ with probability at least $1 - O(n^{-6})$.
\end{theorem}
Theorem~\ref{thm:glaso-sufficient-crude} characterizes a parameter regime in which group Lasso succeeds with high probability. 
Notably, this regime coincides with the minimax exact recovery threshold described in Remark~\ref{rem:mle-snr}, highlighting the near-optimality of the method.
Interestingly, under the same set of conditions, a much simpler procedure, namely top-$m$ hard thresholding, also achieves exact support recovery.
Theorem~\ref{thm:hard-thresh-suffice} formalizes this observation.
Its proof is given in Appendix~\ref{sec:hard-thresh-suffice}. 

\begin{theorem}\label{thm:hard-thresh-suffice}
Under the same set of conditions as Theorem \ref{thm:glaso-sufficient-crude}, 
the hard thresholding method selects all and only groups from $\Istar$ with probability at least $1 - O(n^{-6})$. 
\end{theorem}

While the preceding results demonstrate that both the group Lasso and hard thresholding can successfully recover the true support under suitable conditions, 
their inherently greedy nature makes them vulnerable in more challenging regimes. 
The following result formalizes a specific setting in which both methods provably fail to achieve exact support recovery.
The proof is given in Appendix~\ref{sec:proof:failure}.

\begin{theorem} \label{thm:failure}
Consider a signal-plus-noise model of the form
\begin{equation*}
    \mY = \mB + \mW,
\end{equation*}
where $\mW$ is a symmetric matrix with i.i.d.~$\calN(0, \sigma^2)$ entries on and above diagonal. 
Let the true node support be $\Istar = [m]$, and define the signal matrix $\mB$ with entries
\begin{equation}\label{eq:Bstar:gl:fail}
    B_{ij} = \begin{cases}
        0 & \text{ if } i, j \in [m] \text{ or } i, j \in [m]^c; \\
        \beta_1 & \text{ if } i = m+1, j \in [m] \text{ or } i \in [m], j=m+1\\
        \beta_2 & \text{ if } i \in [m], j \in [m+1]^c \text{ or } j \in [m], i \in [m+1]^c,
    \end{cases}
\end{equation}
where $\beta_2 = C \sigma n^{-1/4}\log^{1/4} n$ for some constant $C$ sufficiently large.
Then, both the group Lasso and the top-$m$ hard-thresholding approach fail to recover the true support $\Istar$ with high probability whenever
\begin{equation*}
    \beta_1 \geq C\sigma \frac{(n\log n)^{1/4}}{\sqrt{m}}
\end{equation*}
for some constant $C$ sufficiently large.
\end{theorem}

\begin{remark}
When $\beta_1 = \Theta(\sigma)$, 
Theorem~\ref{thm:failure} implies that the group Lasso and hard thresholding methods fail to exactly recover the support $\Istar$ when the sparsity level satisfies $m = \Omega(\sqrt{n})$.
Together with Theorems~\ref{thm:glaso-sufficient-crude} and~\ref{thm:hard-thresh-suffice},
this highlights a nuanced landscape: 
although these methods can succeed under favorable signal-to-noise and sparsity conditions, 
their effectiveness is fundamentally constrained by their greedy selection procedure.
Compared to semidefinite programming (SDP), group Lasso and hard thresholding offer computational advantages. 
However, due to the NP-hardness of the underlying problem, their effective regime is fundamentally restricted.
It is plausible that analogous failure modes may arise in related models, such as the graphical Lasso, and a more thorough investigation into such phenomena is a promising direction for future research.
\end{remark}

\section{Estimating the Common Structure} \label{sec:common-structure}


We now turn to the final stage of the procedure, which refines the estimation of the low-rank signal component $\mMstar$ using the residual data.
Specifically, we analyze Step~\ref{alg:general:step:A2} of Algorithm~\ref{alg:general}, where spectral estimation is performed after we have subtracted off the estimated node-sparse component.
To reduce the bias introduced by node-wise sparsity, we follow the asymmetric rearrangement strategy proposed by \citet{chen2021asymmetry}.
For clarity, we first focus on the simplified case where both $|\calG_0| = 1$ and $|\calG_1| = 1$, 
as described in Equations~\eqref{eq:model:calG0} and~\eqref{eq:model:calG1}.

Let $\Ihat_{1}$ denote the estimated node support obtained from Step~\ref{alg:general:step:A1} of Algorithm~\ref{alg:general}.
We define the matrix $\mM_1$ as the remaining of $\mY^{(1)}_1$ after zeroing out rows and columns corresponding to $\Ihat_1$:
\begin{equation} \label{eq:remove:M1}
    M_{1,ij} :=
    \begin{cases}
        Y^{(1)}_{1,ij} & \mbox{ if } (i,j) \in \Ihat^c_{1} \times \Ihat^c_{1}, \\ 
        0 & \mbox{ otherwise }.
    \end{cases}
\end{equation}
Notably, if $\Istar \subseteq \Ihat_1$, then the entries in $\mM_1$ are independent elements drawn from a low-rank-plus-noise model. 
Define 
\begin{equation} \label{eq:remove:M0}
    M_{0, ij} = \begin{cases}
        Y^{(0)}_{0, ij} & \mbox{ if } (i,j) \in \Ihat^c_{1} \times \Ihat^c_{1} \\
        0 & \mbox{ otherwise. }
    \end{cases}
\end{equation}
The asymmetric rearranged matrix $\mM$ obtained from $\mM_1$ and $\mY_0$ is then given by
\begin{equation*}
    M_{ij} = \begin{cases}
        M_{1,ij} & \mbox{ if } 1 \leq i < j \leq n \\
        M_{0, ij} & \mbox{ if } 1 \leq j \leq i \leq n. 
    \end{cases}
\end{equation*}
Since restricting a low-rank matrix to a submatrix cannot increase its rank, the rearranged matrix $\mM$ still adheres to a low-rank-plus-noise structure.
For notational simplicity, we continue to use $\mMstar$ to denote the low-rank signal component corresponding to the submatrix $\mM$, and denote its spectral decomposition as $\mUstar \mLambdastar \mUstart$. 
See Remark~\ref{rem:rearrangement} for further discussion of this notation.
To differentiate its incoherence parameter from that of the original matrix $\mMstar$, we denote its incoherece parameter as $\mub$ with 
\begin{equation}\label{eq:mub:define}
    \mub := \frac{n}{r} \|\mUstar\|^2_{2, \infty}.
\end{equation}
Given $\Ihat_1 = \Istar$, this construction ensures that the entries of $\mM$ are independent, 
and we have
\begin{equation}\label{eq:M:asym}
    \mM = \mMstar + \mH,
\end{equation}
where $\mH$ is an asymmetric noise matrix with independent entries satisfying the same set of distributional assumptions as those in Assumption~\ref{assump:noise} and Equation~\eqref{eq:L-sigma}. 
In fact, the results in this section holds under weaker assumption than Equation~\eqref{eq:L-sigma}:
\begin{equation}\label{eq:L-sigma:weaker}
    L \ll \sigma \min\left\{\frac{n}{\mub \log^{3/2} n}, \sqrt{\frac{n}{\log^3 n}} \right\}
\end{equation}

Throughout this section, we assume that the true support $\Istar$ is successfully recovered, 
i.e., $\Istar \subseteq \Ihat_1$.
As a reminder, the regime under which recovering $\Istar$ is possible is analyzed in Section~\ref{sec:one-step-recovery} and the conditions under which this recovery is guaranteed are established in Section~\ref{sec:one-step-recovery}.
Since $\Ihat_1$ recover $\Istar$ with high probability, without loss of generality, we can assume that $\Ihat^c_1 = [n]$. 
Following the same reason as given in the beginning of Section~\ref{sec:one-step-recovery}, we also assume that the selected node set $\calU$ from Step~\ref{alg:general:step:select} is equal to $[n]$.

\begin{remark} \label{rem:rearrangement}
    We give another reason for using the same notation $\mMstar$ for the low-rank signal component of the rearranged matrix $\mM$.
    If we have multiple observations in $\calG_0$ and $\calG_1$ with estimated node support $\Ihat_k$ for $k \in \calG_1$, then we can similarly obtain 
    \begin{equation} \label{eq:Mk:define}
        M_{k, ij} = \begin{cases}
            Y^{(0)}_{0, ij} & \mbox{ if } (i,j) \in \Ihat^c_{k} \times \Ihat^c_{k} \\
            * & \mbox{ otherwise. }
        \end{cases}
    \end{equation}
    If $|\calG_0| \geq 2$, then we can divide the matrices in $\calG_0$ into two sets $\calG_{0,1}$ and $\calG_{0,2}$ with size difference at most one. 
    We can also divide matrices $\mM_{k}$ for $k \in \calG_1$ into two sets $\calG_{1,1}$ and $\calG_{1,2}$ with size difference at most one. 
    Define $\Omega_l := \{k \in \calG_{0,l} \cup \calG_{1,l}: M_{k, ij} \neq * \}$ for $l = 1,2$.
    For $l = 1,2$, we then combine $\calG_{0,l}$ with $\calG_{1,l}$ to obtain an averaged matrix $\mM^{\text{avg}}_{l}$ with 
    \begin{equation*}
        M^{\text{avg}}_{l, ij} = \frac{1}{|\Omega_l|} \sum_{k \in \Omega_l} M_{k, ij}.
    \end{equation*}  
    We can then asymmetrically arrange the entries of $\mM^{\text{avg}}_{1}$ and $\mM^{\text{avg}}_{2}$ to obtain the final matrix $\mM$. 
    Under this construction, the obtained matrix $\mM$ indeed has $\mMstar$ as its low-rank signal component. 
\end{remark}

With the above setup, we now pause to provide an overview of the goal of this section.
\begin{itemize}
    \item The primary objective is to construct an improved entry-wise estimator for the low-rank signal matrix $\mMstar$.
    Although Theorem~\ref{thm:delta_row} establishes an entry-wise error bound for the initial spectral estimator $\mM^{(0)}$,
    we show that this estimator is sub-optimal in the minimax sense.
    In Section~\ref{sec:entrywise:estimation}, we propose a refined estimator that achieves the optimal entry-wise rate.
    \item While restricting the original signal matrix to a submatrix alters its eigenspace structure, 
    accurate estimation of the resulting eigenspace remains essential, both as an independent objective and as a key ingredient in achieving optimal entry-wise recovery.
    In Section~\ref{sec:linear-form}, we analyze the estimation error of linear forms of the eigenvectors of $\mMstar$ and derive sharp perturbation bounds.
    We also introduce an improved eigenspace estimator that achieves the minimax-optimal rate.
\end{itemize}

We adopt the estimation method proposed in \cite{cheng2021tackling}, 
which provides sharp entry-wise error bounds for eigenvectors under certain perturbation regimes.
Let the right and left eigenvectors of $\mM$ be
\begin{equation} \label{eq:left-right-eigenvc}
    \mM \vu_l=\lambda_l \vu_l \quad \text { and } \quad \mM^{\top} \vw_l=\lambda_l \vw_l
\end{equation}
for all $l \in [r]$, where $\lambda_1, \lambda_2, \dots, \lambda_r$ are eigenvalues of $\mM$.   
To ensure that the leading eigenvectors of $\mM$ provide a reliable estimate,
we impose a lower bound on the minimal signal strength:
\begin{assumption}\label{assump:ev:lower}
    There exists some sufficiently large constant $C > 0$ such that 
    \begin{equation}\label{eq:lambdastar:min:lower:assump}
        \lambdastar_{\min} \geq C \max \left\{ \sigma \sqrt{n \log^3 n}, L \log^3 n \right\}. 
    \end{equation}
\end{assumption}
To establish fine-grained eigenspace perturbation bounds, we introduce the notion of eigengaps.
For each $l \in [r]$, define the eigengap of the $l$-th eigenvalue of $\mMstar$ as
\begin{equation} \label{eq:eigen-gap:define}
    \delta_l^{\star}:= \begin{cases}
    \min _{1 \leq k \leq r, k \neq l}
	\left|\lambdastar_l -\lambdastar_k\right|, &\mbox{ if } l>1 \\ 
	\infty &\mbox{ otherwise. }\end{cases}
\end{equation} 
We caution the reader not to confuse this quantity with the matrix $\mDelta$ used earlier in Section~\ref{sec:spectral-initialization} 
to denote the estimation error of the low-rank initializer $\mM^{(0)}$.
For notational convenience, we also define 
\begin{equation*}
    \deltastar_{\max} := \max_{l \in [r]} \deltastar_l \quad \text{and} \quad \deltastar_{\min} := \min_{l \in [r]} \deltastar_l.
\end{equation*}
To control the perturbation of individual eigenspaces, 
we require that each one-dimensional signal eigenspace is sufficiently well-separated from the others. 
The following eigengap condition guarantees such separation:
\begin{assumption}\label{assump:eigengap:lower}
There exists a sufficiently large constant $C > 0$ such that 
\begin{equation} \label{eq:gap-condition}
\begin{aligned}
    &\delta^\star_{\min} \geq
    C \kappa^3 r^2 \! \max\!\left\{\! \frac{\mub L\! \log ^5\! n}{n},  \sigma \log^{7/2} n\right\}. 
\end{aligned} \end{equation}
\end{assumption}

\subsection{Linear form of eigenvectors} \label{sec:linear-form}

We begin by considering estimating a linear form $\va^\top \vustar_l$,  
that is, the projection of $\vustar_l$ onto a known direction $\va \in \R^d$. 
This formulation generalizes the entry-wise estimation problem, which corresponds to taking $\va$ to be a standard basis vector.
To estimate this quantity, we adopt the estimator proposed by \cite{cheng2021tackling}, 
\begin{equation}\label{eq:linear-form:est}
\uhat_{\va, l}:=\min \left\{\sqrt{\left|\frac{\left(\va^{\top} \vu_l\right)\left(\va^{\top} \vw_l\right)}{\vw_l^{\top} \vu_l}\right|},\|\va\|_2\right\},
\end{equation}
where $\vu_l$ and $\vw_l$ are the left- and right-eigenvectors of the observed matrix $\mM$, 
defined in Equation~\eqref{eq:left-right-eigenvc}.  
Our analysis builds on a combination of techniques from \cite{yan2025improved} and \cite{cheng2021tackling}, 
and improves the original results in \cite{cheng2021tackling} in terms of the dependence on the incoherence parameter $\mub$ of $\mMstar$ (see Equation~\eqref{eq:mub:define} for the reason to use the notation $\mub$ rather than $\mu$). 
 
Before presenting the formal result, we provide some intuition behind the estimator.
By the standard Neumann expansion (see Lemma~\ref{lem:neumann} in Appendix~\ref{sec:improved}), 
we have
\begin{equation} \label{eq:neumann}
    \vu_l = \sum_{j=1}^r \frac{\lambdastar_j}{\lambda_l} \left(\vustart_j \vu_l\right) \vustar_j + \sum_{j=1}^r \frac{\lambdastar_j}{\lambda_l} \left(\vustart_j \vu_l\right) \left\{\sum_{s=1}^\infty \frac{1}{\lambda_l^s} \mH^s \vustar_j\right\}.
\end{equation}
From Theorem~\ref{thm:eigenvalue:rank-r} in Appendix~\ref{sec:improved}, we know that $\lambda_l \approx \lambdastar_l$.
Therefore, when the condition number $\kappa$ obeys $\kappa = O(1)$, it follows that 
\begin{equation*}
    \vu_l \approx \sum_{j=1}^r \frac{\lambdastar_j}{\lambdastar_l} \left(\vustart_j \vu_l\right) \vustar_j + \sum_{j=1}^r O(1) \cdot \left\{\sum_{s=1}^\infty \frac{1}{\lambda_l^s} \mH^s \vustar_j\right\}.
\end{equation*}
Moreover, Theorem~\ref{thm:yan:main} refines this expansion and gives the sharper bound
\begin{equation}\label{eq:first-order-approx}
    \vu_l \approx \sum_{j=1}^r \frac{\lambdastar_j}{\lambdastar_l} \left(\vustart_j \vu_l\right) \vustar_j + \Otilde\left(\frac{\sigma}{\lambdastar_{\min}}\right),
\end{equation}
where $\Otilde( \cdot )$ denotes that we are suppressing logarithmic factors.
This indicates that $\vu_l$ is well-approximated to first order by a linear combination of the true eigenvectors $\{\vustar_j\}_{j \in [r]}$. 
This accurate first-order expansion guides the development of methods to perform bias-correction when estimating the linear form $\va^\top \vu_l$.  

To formalize the above intuition, we first rederive certain results from \cite{cheng2021tackling}, 
focusing on the empirical right eigenvectors and obtaining an improved dependence on the incoherence parameter $\mub$. 
The proof of the following theorem is provided in Appendix~\ref{sec:proof:thm:eigenvec:l2}.
\begin{theorem}\label{thm:eigenvec:l2}
Consider the model given in Equation~\eqref{eq:M:asym} and suppose that the noise matrix $\mH$ with independent entries satisfies Assumption~\ref{assump:noise} and Equation~\eqref{eq:L-sigma}. 
Assume that $\lambdastar_{\min}$ satisfies Assumption~\ref{assump:ev:lower} and $\deltastar_{\min}$ satisfies Assumption~\ref{assump:eigengap:lower}. 
Then
\begin{enumerate}
    \item ($\ell_2$-perturbation) It holds with probability at least $1 - O(n^{-17})$ that for all $l \in [r]$,
    \begin{equation}\label{eq:l2:perturb:norm}
    \min \left\{\|\vu_l \pm \vustar_l\|_2\right\} \lesssim \kappa \frac{\sigma\sqrt{n \log n}}{\lambdastar_{\min}} + \frac{\sigma}{\left(\delta^\star_l\right)} \kappa^2 r \log^{7/2} n ,
    \end{equation}
    \begin{equation}\label{eq:l2:perturb:inner:diff}
        \left|\vustart_k \vu_l\right| \lesssim \frac{\sigma \sqrt{\kappa^4 r \log^7 n}}{\delta^\star_l},
    \end{equation}
    and
    \begin{equation}\label{eq:l2:perturb:inner:same}
        \left|\vustart_l \vu_l\right| \geq 1 - O\left( \frac{ \sigma^2_{\max} \kappa^4 r^2 \log^7 n}{\left(\delta^\star_l\right)^2} + \frac{\kappa^2 \sigma^2_{\max} n \log n}{\left(\lambdastar_{\min}\right)^2}\right);
    \end{equation}
    \item (perturbation of linear forms) For any fixed vector $\va$ with $\|\va\|_2 = 1$ and any $l \in [r]$, with probability at least $1 - O(n^{-17})$, we have
    \begin{equation}\label{eq:naive:linear-form} 
    \begin{aligned}
        \min \left|\va^\top \! \left(\vu_l \pm \vustar_l\right)\right|
	&\lesssim 
        ~\sigma_{\max } \sqrt{\kappa^2 r \log^7 n}\left\{\frac{1}{\lambda_{\min }^{\star}}+\kappa r \max _{k: k \neq l} \frac{\left|\boldsymbol{a}^{\top} \boldsymbol{u}_k^{\star}\right|}{\left|\lambda_l^{\star}-\lambda_k^{\star}\right|}\right\} \\
        &~~~~~~+\left(\frac{\sigma_{\max }^2 \kappa^4 r^2 \log ^7 n}{\left(\delta_l^{\star}\right)^2}+\frac{\kappa^2 \sigma_{\max }^2 n \log n}{\left(\lambda_{\min }^{\star}\right)^2}\right) \left|\va^\top \vustar_l\right|.
    \end{aligned}
    \end{equation}
    \item ($\ell_\infty$-perturbation) For any $l \in [r]$, with probability at least $1 - O(n^{-17})$, we have
    \begin{equation}\label{eq:naive:entrywise}
    \begin{aligned}
    \min \left\|\vu_l \pm \vustar_l\right\|_{\infty} 
    &\lesssim \sigma_{\max } \sqrt{\kappa^2 r \log^7 n}
	\left\{\frac{1}{\lambda_{\min }^{\star}}
	+\frac{\kappa r }{\delta^\star_l}\sqrt{\frac{\mub}{n}}\right\} \\
    &~~~~~~+ \left(\frac{\sigma_{\max }^2 \kappa^4 r^2 \log ^7 n}{\left(\deltastar_l \right)^2}+\frac{\kappa^2 \sigma_{\max }^2 n \log n}{\left(\lambda_{\min }^{\star}\right)^2}\right) \sqrt{\frac{\mub}{n}}.
    \end{aligned}
    \end{equation}
\end{enumerate}
\end{theorem}

\begin{remark}
By Theorem~\ref{thm:eigenvec:l2}, we have the approximation
\begin{equation*}
    \vustart_k \vu_l \approx \tilde{O}\left(\frac{\sigma}{\left|\lambdastar_l - \lambdastar_k\right|}\right)
\end{equation*}
for $k \neq l$. 
In addition, ignoring the terms involving eigengaps, we also have
\begin{equation*}
    \vustart_l \vu_l \approx 1 - \tilde{O}\left(\frac{\sigma^2 n}{\left(\lambdastar_{\min}\right)^2}\right).
\end{equation*}
We observe that these approximations resemble the asymptotic behavior seen in the BBP phase transition limit \citep{Florent2011eigenvalues}.
Unfortunately, deriving the precise limiting form in our setting is challenging due to the heteroscedasticity in the variance structure of $\mH$.
\end{remark}

Armed with the previous results, we are now ready to establish an error bound for the improved estimator defined in Equation~\eqref{eq:linear-form:est}.
To maintain technical simplicity, we state the following results under $\kappa = O(1)$. 
The precise dependence on $\kappa$ can be found in the Appendix.

\begin{theorem}\label{thm:improved:linear-form-bound}
Suppose that $\kappa = O(1)$ and $r = O(1)$. 
Under the same assumptions as Theorem~\ref{thm:eigenvec:l2}, for any $l \in [r]$, the estimation error of the improved estimator in Equation~\eqref{eq:linear-form:est} is bounded by
\begin{equation} \label{eq:improved:linear-form-bound}
\min \left|\widehat{u}_{\va, l} \pm \va^{\top} \vustar_l \right|
\lesssim
\left\{ \frac{1}{\lambda_{\min }^{\star}}
  + \max _{k: k \neq l} \frac{\left|\va^{\top} \vustar_k \right|}
			{\left|\lambdastar_l-\lambdastar_k\right|}
\right\} \sqrt{ \sigma^2 \log^7 n }
+ \left|\va^{\top} \vustar_l \right|
	\frac{\sigma^2 \log^7 n}{\left(\delta^\star_{l}\right)^2}.
\end{equation}
with probability at least $1 - O(n^{-17})$. 
\end{theorem}

The proof is included in Appendix~\ref{sec:proof:thm:improved:linear-form-bound}. 
To establish the tightness of Theorem~\ref{thm:improved:linear-form-bound}, we have the following minimax result adapted from Theorem 3 in \cite{cheng2021tackling}. 
\begin{theorem}\label{thm:multiple:linear-form:minimax}
Suppose that we observe $N \ge 1$ independent copies of symmetric matrices 
$
\{\mM^{(l)}: l \in [N]\},
$
where $\mM^{(l)} = \mMstar + \mW^{(l)}$.
Suppose that $\{\mW^{(l)}: l \in [N]\}$ has independent Gaussian entries on and above diagonal, with 
\begin{equation*}
    W_{ij}^{(l)} \sim \calN\left(0, \left(\sigma_{ij}^{(l)}\right)^2\right)
\end{equation*}
and $\sigma_{\min} \leq \sigma_{ij}^{(l)} \leq \sigma$ for all $i,j \in [n]$ and $l \in [N]$. 
Let 
\begin{equation*}
    \calM_0(\mMstar) := \left\{\mA \in \bbS^n \mid \operatorname{rank}(\mA) = r, \lambda_i(\mA) = \lambdastar_i \;(i \in [r]), \left\|\mA - \mMstar\right\|_{\F} \leq \frac{\sigma_{\min}}{2}\right\}.
\end{equation*}
Assume that $|\lambdastar_l| \geq 4\sigma_{\min} \sqrt{n / N}$ and $\deltastar_l \geq \sigma_{\min} / \sqrt{N}$.
Then there exists a constant $c>0$ such that 
\begin{equation*}
    \inf_{\uhat_{\va, l}} \sup_{\mA \in \calM_0(\mMstar)} \E \left[ \min \left|\uhat_{\va,l} \pm \va^\top \vu_l(\mA) \right|\right] \geq 
    c \left\{\left|\va^{\top} \vu_l^{\star}\right| \frac{\sigma_{\min }^2}{N\left(\deltastar_l\right)^{2}}
    + \frac{\sigma_{\min}}{\sqrt{N}} \max _{k: k \neq l} \frac{\left|\va^{\top} \vu_k^{\star}\right|}{\left|\lambda_l^{\star}-\lambda_k^{\star}\right|}\right\} .
\end{equation*}
\end{theorem}
The proof largely follows that of Theorem 3 in \cite{cheng2021tackling}, which uses Le Cam's two-point method. 
The main difference is that we have $N$ independent observed matrices, resulting in a factor of $N$ in the upper bound of the KL-divergence between the two hypotheses, which in turn leads to a factor of $1/\sqrt{N}$ in the lower bound.
We omit the proof for the sake of brevity.
Comparing Equation~\eqref{eq:improved:linear-form-bound} with the rate in Theorem~\ref{thm:multiple:linear-form:minimax}, we see that the estimator in Equation~\eqref{eq:linear-form:est} is optimal up to a logarithmic factor under Gaussian noise.

For any $l \in [r]$, taking $\va = \ve_1, \ve_2, \dots, \ve_n$ in Equation~\eqref{eq:improved:linear-form-bound}, the following corollary follows immediately from Theorems~\ref{thm:improved:linear-form-bound} and~\ref{thm:eigenvec:l2}.
\begin{corollary}\label{cor:single:eigenspace}
Suppose that $\kappa r = O(1)$.
Under the same assumptions in Theorem~\ref{thm:eigenvec:l2}, for any $l \in [r]$, the estimation error of the improved estimator in Equation~\eqref{eq:linear-form:est} is bounded by 
\begin{equation} \label{eq:improved:infty:eigenvec:bound}
\min \|\vuhat_{l} \pm \vustar_l\|_{\infty}
\lesssim 
\sigma \left\{\frac{1}{\lambda_{\min }^{\star}} 
    + \frac{\sqrt{\mub/n}}{\delta^\star_l}\right\} \sqrt{\log^7 n}
+ \frac{\sigma^2 \log^7 n}{\left(\delta^\star_{l}\right)^2}
	\sqrt{\frac{\mub}{n}}
\end{equation}
with probability at least $1 - O(n^{-16})$, where $\vuhat_l$ is given by 
\begin{equation}\label{eq:vuhat:def}
    \uhat_{l,i} := \sgn{u_{l, i}} \left|\uhat_{\ve_i, l}\right|, \quad i \in [n]. 
\end{equation}

More generally, writing
\begin{equation}\label{eq:mUhat}
    \mUhat := \left(\vuhat_1, \vuhat_2, \dots, \vuhat_r\right),
\end{equation}
we have that with probability at least $1 - O(n^{-16})$,
\begin{equation}\label{eq:improved:tti:eigenspace:bound}
\min_{\substack{\mS = \diag(s_1, s_2, \dots, s_r), \\s_l \in \{-1,1\}, \; l \in [r]}}
\left\|\mUhat \mS - \mUstar\right\|_{2,\infty} 
\lesssim \sigma \left\{\frac{1}{\lambda_{\min }^{\star}}+ \frac{\sqrt{\mub/n}}{\deltastar_{\min}}\right\} \sqrt{ \log^7 n } .
\end{equation}
\end{corollary}

When the incoherence parameter satisfies $\sqrt{\mub/n}/\deltastar_{\min} \asymp \left(\lambdastar_{\min}\right)^{-1}$, we obtain the row-wise bound
\begin{equation*}
\min_{\substack{\mS = \diag(s_1, s_2, \dots, s_r), \\ s_l \in \{-1,1\}, \; l \in [r]}} \left\|\mUhat \mS - \mUstar\right\|_{2,\infty}
\lesssim \frac{\sigma \sqrt{\log^7 n}}{\lambdastar_{\min}}.
\end{equation*}
This rate matches the known minimax lower-bound up to logarithmic factors, provided that the eigengap is sufficiently large.
See, for example, Remark 7 in \citet{agterberg2024estimating}.

\begin{remark}
When estimating the linear form $\va^\top \vustar_l$ from a single observation $\mM = \mMstar + \mW$, where $\mW$ is a symmetric noise matrix, our analysis reveals that the term $\va^\top \mW \vustar_l$ can introduce a non-negligible bias. 
Due to the heteroscedasticity of the noise, removing this bias without additional samples is challenging.
A similar issue arises when estimating the eigenspace $\mUstar$: the term $\mH^\top \mH$ introduces systematic bias that is difficult to correct in the single-sample setting. 
One possible remedy might involve diagonal deletion techniques, as explored in \citet{zhang2022heteroskedastic}.
In contrast, our goal is to demonstrate that with multiple independent network copies, one can construct estimators that match the minimax optimal rate using conceptually simpler methods and more straightforward analysis.
\end{remark}

If additional independent samples are available, the eigenspace can be estimated more accurately under weaker eigengap conditions.
Suppose that we observe two extra independent copies, $\mM^{(1)}$ and $\mM^{(2)}$, each generated as
\begin{equation*}
    \mM^{(i)} = \mMstar + \mW_{i}, \quad \text{for } i=1,2. 
\end{equation*}
Here, $\mW_i$ are symmetric noise matrices satisfying Assumption~\ref{assump:noise} and Equation~\eqref{eq:L-sigma}. 
This could be obtained from having multiple copies from either $\calG_0$ or $\calG_1$.
Define the matrix
\begin{equation} \label{eq:G:defin}
    \mG := \left(\mLambda^{-1} \mU^\top \mM^{(1)} \mM^{(2)} \mU \mLambda^{-1}\right)^{-1} 
\end{equation}
and its symmetrized version
\begin{equation} \label{eq:Gsymm:def}
    \mG_{\symm} := \frac{1}{2}\left(\mG + \mG^\top\right). 
\end{equation}
Let the spectral decomposition of $\mG_{\symm}$ be
\begin{equation} \label{eq:Gsymm:spec:def}
    \mG_{\symm} = \mGammahat_{\symm} \mSighat^{-2}_{\symm} \mGammahat^\top_{\symm}
\end{equation}
and define the inverse square root matrix
\begin{equation} \label{eq:Psihat:def}
\mPsihat := \mGammahat_{\symm} \mSighat^{-1}_{\symm} \mGammahat^\top_{\symm}.
\end{equation}
We then obtain the following result, the proof of which is provided in Appendix~\ref{sec:proof:thm:eigenspace:l2infty}.

\begin{theorem} \label{thm:eigenspace:l2infty}
Under the same assumptions as Theorem~\ref{thm:eigenvec:l2}, the $\ell_{2,\infty}$ estimation error of the corrected estimator $\mU \mPsihat$ is bounded by
\begin{equation}\label{eq:eigenspace:l2infty}
    \min_{\mO \in \bbO_r} \left\|\mU \mPsihat - \mUstar \mO\right\|_{2,\infty} \lesssim \frac{\sigma \kappa r \sqrt{\log^7 n}}{\lambdastar_{\min}}
\end{equation}
with probability at least $1 - O(n^{-16})$.
\end{theorem}

Notably, by leveraging two additional independent networks and a simple correction step, the estimator $\mU \mPsihat$ achieves a near-optimal $\ell_{2,\infty}$ error rate.
This result holds under significantly milder eigengap conditions compared to single-copy estimators, highlighting the benefit of using additional copies to perform bias-correction in recovering low-rank eigenspace. 


\subsection{Entrywise Estimation}

\label{sec:entrywise:estimation}

We now conclude by addressing the entry-wise estimation of the low-rank signal matrix $\mMstar$.
Using the best rank-$r$ approximation of $\mY_0$, the spectral estimator $\mMhat_{\spec}$ yields an entry-wise estimation error of the form
\begin{equation}
\label{eq:Mstar:spec}
    \left\|\mMhat_{\spec} - \mMstar\right\|_{\infty} \lesssim \sigma \kappa^2 \mub r \sqrt{\frac{\log n}{n}} 
\end{equation}
with probability at least $1 - O(n^{-5})$ (See Corollary 4.3 in \cite{chen2021spectral}). 
This bound has a suboptimal dependence on the incoherence parameter $\mub$. 
We will give more comments on this in Section~\ref{sec:experiments}. 
By using the improved eigenvector estimator $\mUhat$ from Corollary~\ref{cor:single:eigenspace} along with the empirical eigenvalues of $\mM$,
we obtain the following bound on the entry-wise estimation error.
\begin{corollary}\label{cor:single:entrywise}
Let $\mMhat := \mUhat \mLambda \mUhat^\top$, where $\mUhat$ is given by Equation~\eqref{eq:mUhat} and $\mLambda$ is the diagonal matrix of eigenvalues of the observed matrix $\mM$ given by Equation~\eqref{eq:M:asym}.
Suppose that $\kappa r = O(1)$.
Under the same assumptions as Theorem~\ref{thm:eigenvec:l2}, it holds with probability at least $1 - O(n^{-16})$ that
\begin{equation*}
\|\mMhat - \mMstar\|_{\infty}
\lesssim \sigma \sqrt{\frac{\mub \log^{7} n}{n}}
+ \frac{\sigma \lambdastar_{\max} \sqrt{\log^7 n}}{\delta^\star_{\min}}
	\frac{\mub}{n}.
\end{equation*}
\end{corollary}

The proof is stated in Appendix~\ref{sec:proof:cor:single:entrywise}. 
While Corollary~\ref{cor:single:entrywise} shows that accurate entry-wise estimation is achievable, it still requires a sufficiently large eigengap to control the second term in the upper bound.
To further improve the rate, especially under weaker eigengap conditions, we make use of the corrected estimator $\mU \mPsihat$ from Theorem~\ref{thm:eigenspace:l2infty}, which leverages additional independent copies.
The following result provides a sharper entry-wise bound for this improved estimator, provided that such additional independent networks are available.
A proof is given in Appendix~\ref{sec:proof:thm:improved:entrywise}.

\begin{theorem}\label{thm:improved:entrywise}
Let $\mPsihat$ be as defined in Equation~\eqref{eq:Psihat:def}. 
Under the same assumptions as Theorem~\ref{thm:eigenspace:l2infty}, it holds with probability at least $1 - O(n^{-16})$ that
\begin{equation*}
\left\|\mU\mPsihat \mLambda \mPsihat^\top \mU^\top - \mMstar\right\|_{\infty}
\lesssim \sigma \kappa^2 r^2 \sqrt{\frac{\mub \log^7 n}{n}}
	+ \frac{\mub}{n} \frac{ \deltastar_{\max} }{\delta_{\min}^{\star}}
	\sigma \kappa^3 r^3 \sqrt{\log^7 n}.
\end{equation*}
\end{theorem}
This result shows that incorporating additional independent observations enables near-optimal entry-wise recovery of the low-rank structure, 
even under weak eigengap conditions. 
We suspect that the dependence on the eigengap likely reflects a technical artifact of the proof rather than a fundamental limitation, and we conjecture that a more careful analysis may eliminate this dependence altogether.

Finally, we complement our upper bounds with a matching minimax lower bound for the entry-wise estimation error in the rank-one case.
The proof is given in Appendix~\ref{sec:proof:thm:rank1:linfty:minimax}.

\begin{theorem}\label{thm:rank1:linfty:minimax}
For $1 \leq \mu \leq n$, define the parameter set
\begin{equation}\label{eq:calM:def}
\calM(\mu, \sigma_{\min})
= \left\{\mMstar = \lambdastar \vustar \vustart :
	\frac{1}{2}\sqrt{\frac{\mu}{n}} \leq \|\vustar\|_{\infty}
	\geq 2\sqrt{\frac{\mu}{n}}, \;
	\left|\lambdastar\right| \geq \sigma_{\min} \sqrt{n}\right\}.
\end{equation}
Suppose that the observed matrix $\mM$ is given by Equation~\eqref{eq:M:asym}, where $\mH$ has independent Gaussian entries, with 
\begin{equation*}
    H_{ij} \sim \calN(0, \sigma_{ij}^2)
\end{equation*}
and $0 < \sigma_{\min} \leq \sigma_{ij} \leq \sigma$ for all $i,j \in [n]$. 
We have
\begin{equation*}
\inf_{\mMhat \in \R^{n\times n}} ~\sup_{\mMstar \in \calM(\mu, \sigma_{\min})} 
	~\E \left\|\mMhat - \mMstar\right\|_{\infty}
\geq \frac{\sigma_{\min}}{8}\sqrt{\frac{3\mu}{n}}.
\end{equation*}
\end{theorem}

An immediate question is whether access to multiple independent copies of the matrix $\mM$ can lead to a fundamental improvement in the minimax rate for entrywise estimation. 
The following result shows that the benefit of having $N$ independent samples is limited to a standard variance reduction factor of $1 / \sqrt{N}$.
No further gain is possible in terms of dependence on other parameters.
\begin{corollary}\label{cor:more:linfty:minimax}
Suppose that we observe $N \ge 1$ independent matrices $\mM^{(1)}, \mM^{(2)}, \dots, \mM^{(N)}$ from the model specified in Equation~\eqref{eq:M:asym}, where $\{\mH^{(l)}: l \in [N]\}$ has independent Gaussian entries, with 
\begin{equation*}
H^{(l)}_{ij} \sim \calN\left(0, \left(\sigma^{(l)}_{ij}\right)^2\right)
\end{equation*}
and $0 < \sigma_{\min} \leq \sigma^{(l)}_{ij} \leq \sigma$ for all $i,j \in [n]$ and $l \in [N]$. 
Over the set of rank-one signal matrices $\calM(\mu, \sigma_{\min})$ as defined in Equation~\eqref{eq:calM:def}, we have 
\begin{equation*}
\inf_{\mMhat \in \R^{n\times n}} ~\sup_{\mMstar \in \calM(\mu, \sigma_{\min})}
	~\E \left\|\mMhat - \mMstar\right\|_{\infty}
\geq \frac{\sigma_{\min}}{8 \sqrt{N}}\sqrt{\frac{3\mu}{n}}, 
\end{equation*}
where the infimum is over all estimators of $\mMstar$ that take as input the observed matrices $\{\mM^{(l)}\}_{l \in [N]}$.
\end{corollary}


This corollary quantifies the impact of sample size on entrywise estimation: having $N$ independent observations improves the minimax rate by a factor of $1/\sqrt{N}$, reflecting the classical variance reduction from averaging. 
However, this improvement is purely in terms of sample size and does not affect the dependence on the structural parameters $\mu$, $n$, or $\sigma_{\min}$. 
In this sense, additional samples help only in a trivial way, and do not change the intrinsic difficulty of the problem.
Comparing with the upper bound in Theorem~\ref{thm:improved:entrywise}, we see that when $\deltastar_{\max} \asymp \deltastar_{\min}$, our estimator $\mU \mPsihat \mLambda \mPsihat^\top \mU^\top$ achieves a near-optimal rate, up to constant factors.

\section{Numerical Results} \label{sec:experiments}
From our prior analysis, we see that recovering the node support $\Istar$ of $\mBstar$ and estimating the common structure $\mMstar$ are two nearly orthogonal tasks.
In the experiments below, we treat them separately, with Section~\ref{sec:recovering-Bstar:experiments} focusing on the recovery of $\mBstar$ and Section~\ref{sec:estimating-Mstar:experiments} on estimating $\mMstar$. 

\subsection{\texorpdfstring{Recovering the Node Support of $\mBstar$}{Recovering the Node Support of B*}}

\label{sec:recovering-Bstar:experiments} 

In this section, we conduct a series of experiments to evaluate the recovery of $\mBstar$ under various conditions. 
In what follows, the symmetric noise matrix $\mW$ is generated with independent entries $(W_{ij})_{1 \leq i \leq j \leq n}$, where the variances and distribution may vary across entries depending on the experimental setting. 
Unless otherwise specified, we assume that $(W_{ij})_{1 \leq i \leq j \leq n} \overset{\text{i.i.d.}}{\sim} \calN(0,1)$, and set $W_{ji} = W_{ij}$ for $i < j$ to ensure symmetry.
Unless otherwise specified, $\mBstar$ is generated with signal strength $\sigma_B = 2n^{-1/4}\log^{1/4} n$ as follows:
\begin{algorithm}[ht]
    \caption{Generate $\mBstar$}\label{alg:Bstar-gen}
    \begin{algorithmic}[1]
        \Require: Matrix dimension $n$, node-support size $m$, signal strength parameter $\sigma_B$. 
        \Ensure: $\mBstar$
        \State Randomly select a subset $\Istar \subset [n]$ with $|\Istar| = m$.
        \State Construct a matrix $\mB_0 \in \mathbb{R}^{n \times n}$ row-wise:
        \begin{equation*}
            \mB_{0, i, \cdot} = \begin{cases}
                \mathbf{0}_n, & \text{if } i \notin \Istar; \\
                \text{i.i.d.~samples from } \calN(0, \sigma_B^2), & \text{if } i \in \Istar.
            \end{cases}
        \end{equation*}
        \State Set $\mBstar \gets \mB_0 + \mB_0^\top$.
    \end{algorithmic}
\end{algorithm}

We generate a random low-rank singular subspace $\mUstar$ as follows: 
\begin{algorithm}[ht]
    \caption{Generate $\mUstar$}\label{alg:Ustar-gen}
    \begin{algorithmic}[1]
        \Require: Matrix dimension $n$, matrix rank $r$, incoherence parameter $\mu$. 
        \Ensure: $\mUstar$
        \State Set $m \gets \lfloor n / \mu \rfloor$. Initialize $\mUstar \gets \mathbf{0}_{n \times r}$. 
        \State Denote the Stiefel manifold as 
        $$
            \St(n, p)=\left\{\mX \in \R^{n \times p}: \mX^{\top} \mX=\mI_p\right\}
        $$
        Fill the top $m$ rows of $\mUstar$ with a random orthonormal matrix drawn uniformly from $\St(m, r)$, i.e., according to the Haar measure on the Stiefel manifold.
        \State Fill the remaining rows with another random singular subspace in $\St(n-m, r)$.
        \State Normalize each column of $\mUstar$ to unit $\ell_2$ norm.
    \end{algorithmic}
\end{algorithm}

We assume that the node support size $m = |\Istar|$ is known a priori.
For the SDP method, the estimated support $\Ihat$ is obtained by solving the SDP in Equation~\eqref{eq:prime:sdp:equal} and selecting the $m$ rows of $\widehat{\mZ}_{\SDP}$ with the largest $\ell_1$-norm.
For the group Lasso method, $\Ihat$ consists of the $m$ nodes corresponding to the largest values of $\alpha_i$, as defined in Equation~\eqref{eq:alpha-def}.  
Our primary evaluation metric is the \emph{false negative rate (FNR)}, defined as 
\begin{equation*}
    \FNR := 1 - \frac{|\Ihat \cap \Istar|}{|\Istar|}.
\end{equation*}
All experiments are repeated over 30 independent trials per combination of parameters. 
The shaded regions in the figures represent 95\% bootstrap confidence intervals computed using the \texttt{Seaborn} package \citep{Waskom2021} in Python.


Before presenting the numerical results, we briefly discuss the details of solving the SDP in Equation~\eqref{eq:prime:sdp}.
As noted in Section~\ref{subsec:glasso}, the original SDP suffers from scalability issues: na\"{i}ve implementations are inefficient for large-scale problems due to the presence of $\Theta(n^2)$ constraints.
This renders the computation of the Schur complement, which is central to  interior-point methods, prohibitively expensive in both time and space \citep{vandenberghe1996semidefinite}. 

From a theoretical perspective, it is known that if an SDP has $q$ affine constraints and a compact search space, then it admits a global optimum of rank at most $p_0$ with $p_0(p_0+1)/2 \leq q$ \citep{barvinok1995problems, pataki1998rank}. 
This results motivates a simplification of the original SDP by eliminating the inequality constraints and considering the relaxed problem
\begin{equation} \label{eq:prime:sdp:equal}
\begin{aligned}
\widehat{\mZ}_{\SDP}=\underset{\mZ\in \R^{n \times n}}{\arg \min } & \langle \mC, \mZ\rangle  \\
\text { s.t. } & \mZ \succeq 0, \; \langle\mathbf{I}, \mZ\rangle=K, \; \langle\mathbf{J}, \mZ\rangle=K^2,
\end{aligned}
\end{equation}
which reduces the number of constraints from $\Theta(n^2)$ to just $2$.
This simplification significantly improves computational efficiency and enables the use of off-the-shelf SDP solvers on moderately large problems. 
In our experiments, we are able to handle problem sizes up to $n=7500$ using commodity hardware.
Computations were carried out on the UW–Madison Center for High Throughput Computing (CHTC) infrastructure \citep{chtc2006}.

Among the available techniques, one popular approach is the Burer-Monteiro factorization~\citep{burer2003nonlinear}, which reparametrizes the matrix variable as $\mZ=\mX\mX^\top$, where $\mX \in \R^{n \times p}$ is a matrix with $p \ll n$.
This yields a non-convex optimization problem over $\mX$, but under mild regularity assumptions (which are satisfied by the SDP in Equation~\eqref{eq:prime:sdp:equal}), it has been shown that all second-order critical points of the reparametrized problem are globally optimal~\citep[Theorem 2]{boumal2016non}.
The resulting non-convex problem can be solved using Riemannian optimization methods, which are efficient and scalable.
In our experiments, we solve Equation~\eqref{eq:prime:sdp:equal} using the commercial solver \texttt{MOSEK}~\citep{mosek}, which efficiently handles SDPs of this form.

\subsubsection{Effect of Step~\ref{alg:general:step:select} in Algorithm~\ref{alg:general}}

We first examine the necessity of Step~\ref{alg:general:step:select} in Algorithm~\ref{alg:general}.  
Suppose we observe $\mY^{(0)}$ and $\mY^{(1)}$ following Equations~\eqref{eq:model:calG0} and~\eqref{eq:model:calG1}, respectively.  
We vary the incoherence parameter $\mu$ according to
\begin{equation*}
\mu \in \left\{\log n, \sqrt{n / \log n}, \sqrt{n} \log n, n^{3/4} \right\}
\end{equation*}
and let $n$ range from $500$ to $7500$ in increments of $500$.  
We fix the rank at $r = 3$ and the eigenvalues of $\mMstar$ are set to $\lambdastar_i = 3\sqrt{n} + (3-i)\log n$ for $i = 1, 2, 3$.
We estimate $\mMstar$ using the spectral initializer $\mM^{(0)}$ on $\mY^{(0)}$, and form the residual matrix $\widetilde{\mY} = \mY^{(1)} - \mM^{(0)}$.  
We then apply both the SDP and group Lasso methods to $\widetilde{\mY}$, with $m = 10$.  

Figure~\ref{fig:exp1} shows the results of this experiment.
Inspecting the figure, a clear phase transition can be seen to occur near $\mu = \sqrt{n}$.  
When $\mu \ll \sqrt{n}$, both the SDP and group Lasso methods successfully recover the support $\Istar$, and they perform quite similarly, as indicated by the overlapping curves and markers.
On the other hand, when $\mu \gg \sqrt{n}$, the FNRs increase with $n$, indicating a breakdown in support recovery under high incoherence. 
This observation is in accodance with Theorem~\ref{thm:delta_row}, which states that when $\mu \gg \sqrt{n}$, the rowwise estimation error for high-coherence rows exceeds $\Theta(n^{1/4}\log^{1/4} n)$.
In this regime, the signal strength $\sigma_B = \Theta(\sigma n^{-1/4} \log^{1/4} n)$ is insufficient for accurate recovery of $\Istar$. 
These results highlight the importance of Step~\ref{alg:general:step:select}, which removes high-coherence rows prior to applying the spectral initializer.

\begin{figure}
    \centering
    \includegraphics[width=0.7\textwidth]{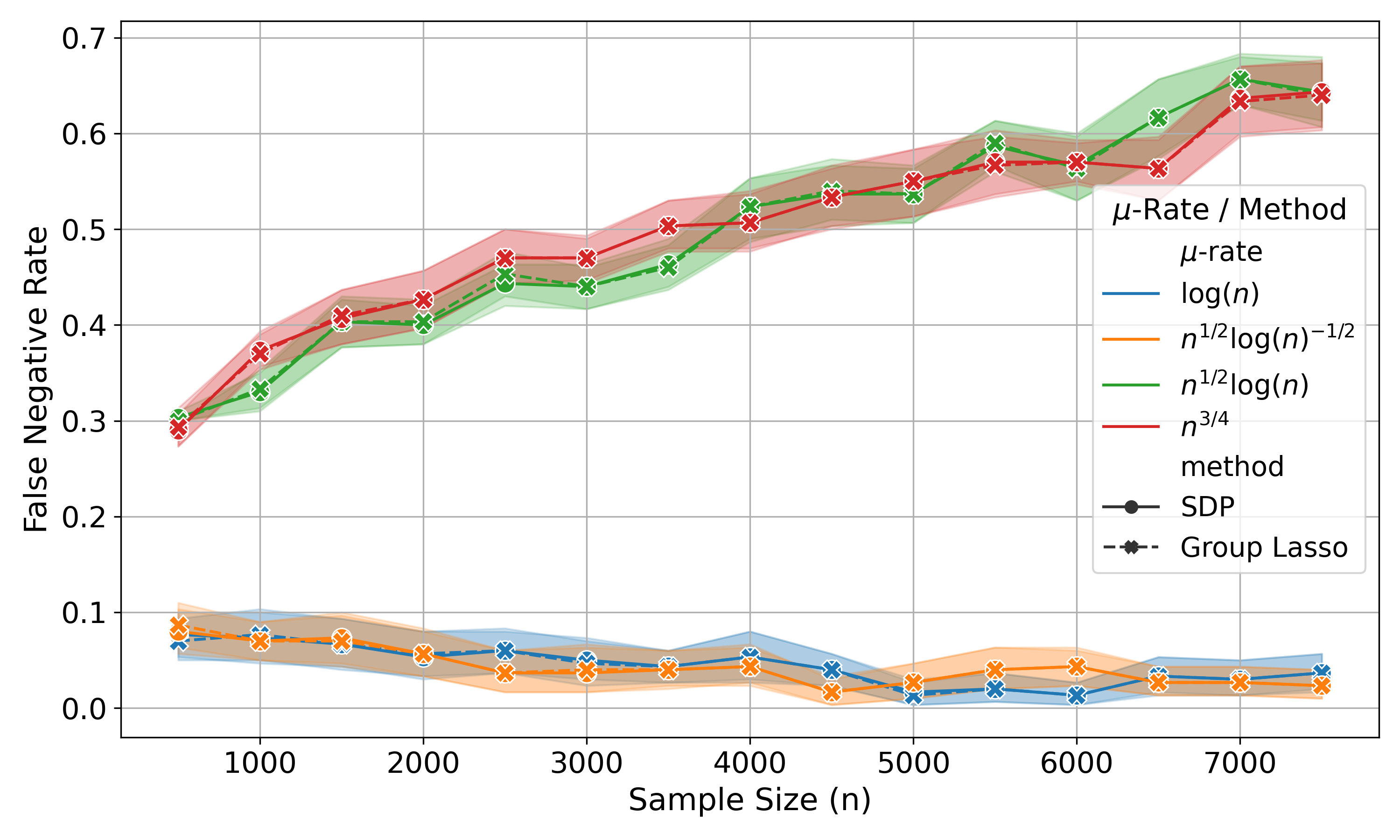}
    \caption{False negative rate of $\Istar$ without Step~\ref{alg:general:step:select} in Algorithm~\ref{alg:general}, across different coherence levels $\mu$.
    Each curve corresponds to a different rate of the incoherence parameter $\mu$, with solid lines for SDP and dashed lines for group Lasso.}
    \label{fig:exp1}
\end{figure}

\subsubsection{Identifying the Recovery Threshold}

We next investigate the effect of the signal strength constant $C$ in the theoretical threshold $\sigma_B = C n^{-1/4} \log^{1/4} n$.  
We first vary $C \in \{0.8, 1.2, 1.6, 2.0, 2.4\}$ and generate $\mBstar$ using Algorithm~\ref{alg:Bstar-gen}.  
To assess recovery under heteroscedastic noise, we sample $W_{ij} \sim \calN(0, \sigma_{ij}^2)$ with $\sigma_{ij}$ uniformly drawn from $[\sigma_{\min}, \sigma_{\max}]$.  
The node support size is set to $m = 10$. 

Figure~\ref{fig:exp7} shows the results of this experiment.
We observe a phase transition at $C \approx 1.6$:  
\begin{enumerate}
    \item When $C \leq 1.2$, the FNR increases with $n$, indicating recovery failure.
    \item When $C \geq 2.0$, the FNR converges to zero.
    \item At the threshold $C = 1.6$, the FNR stabilizes around $0.3$.
\end{enumerate}
These observations align closely with the theoretical predictions in Remark~\ref{rem:mle-snr}, which identifies the signal-to-noise ratio (SNR) threshold
\begin{equation*}
    \SNR = \Omega(n^{1/4} \log^{1/4} n)
\end{equation*}
as the minimum rate required for exact recovery of the node support $\Istar$. 
Below this threshold, no estimator can succeed uniformly over the class of node-sparse signals.
Our empirical findings support this characterization: when $C$ is too small, the effective SNR falls below the exact recovery threshold, and recovery fails.
Both SDP and group Lasso perform similarly in this setting, and are robust to heteroscedasticity when Assumption~\ref{assump:homoscedastic} is satisfied.

Recall from Remark~\ref{rem:mle-snr} that the partial recovery threshold is given by
\begin{equation*}
    \SNR = \Omega\left(n^{1/4} \log^{1/4} \left(\frac{n}{m}\right)\right).
\end{equation*}
When $m = 10$, this threshold has the same asymptotic form as the exact recovery threshold, differing only in the constant factor.
The stablization of the FNR around $0.3$ at $C = 1.6$ provides empirical support for this theoretical prediction, suggesting that $C = 1.6$ lies in the regime of partial recovery.

To investigate the partial recovery threshold in finer detail, we further vary the signal strength constant over a denser grid: $C \in \{1.3, 1.4, 1.5, 1.6, 1.7, 1.8\}$. 
For each combination of sample size $n$, signal strength $C$, and method, we repeat the experiment $100$ times and compute the FNR for each trial.
The results, summarized in Figure~\ref{fig:exp7finer}, show that as $n$ increases, the FNR converges to a stable value for each fixed $C$. 
Notably, stronger signals yield better limiting FNRs: for example, when $C = 1.3$, the FNR levels off around $0.7$, while for $C = 1.8$, it approaches $0.1$.
This trend aligns with the theoretical prediction based on the partial recovery threshold in Remark~\ref{rem:mle-snr}. 



\begin{figure}
    \centering
    \includegraphics[width=0.8\textwidth]{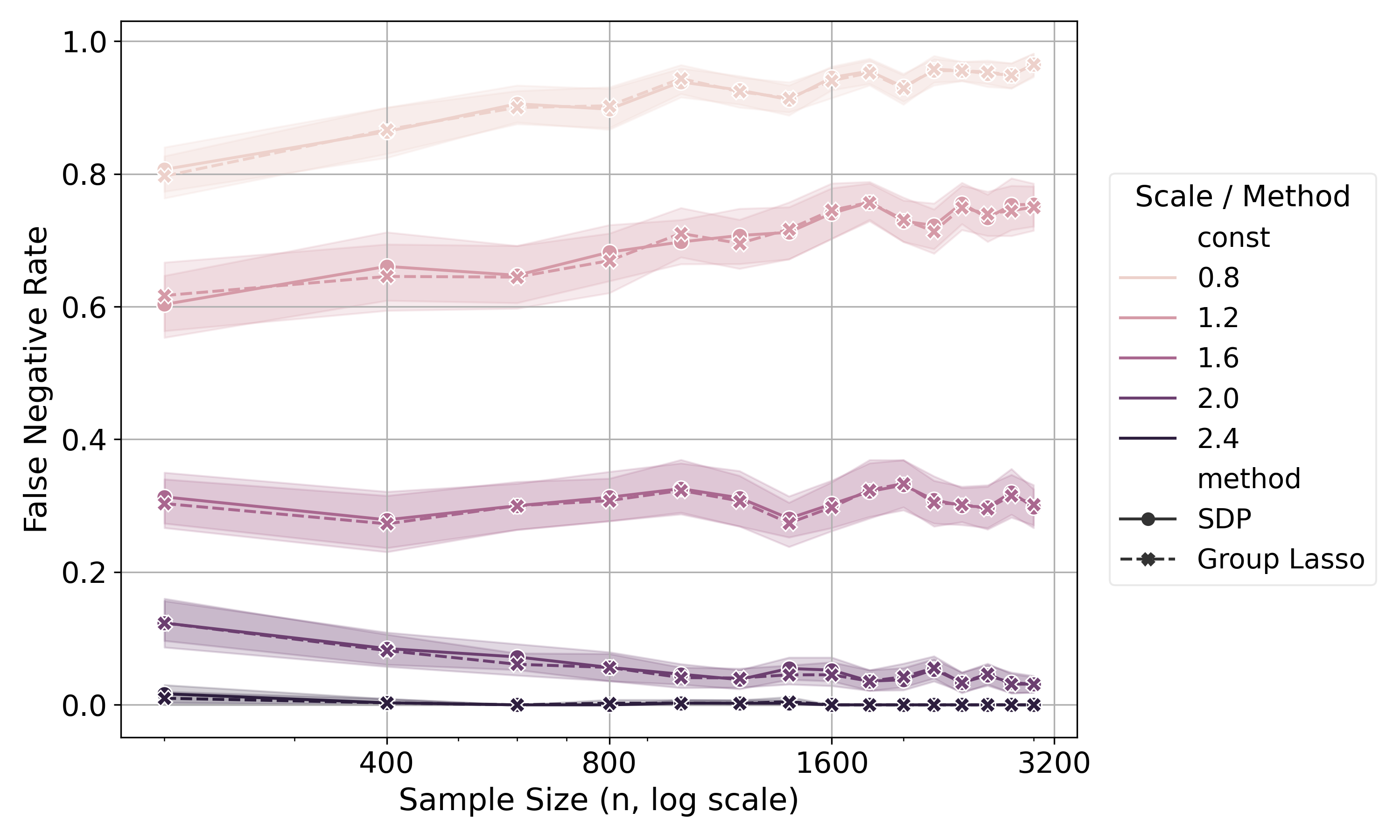}
    \caption{False negative rate of $\Istar$ for SDP and group Lasso under varying signal strength constant $C$ in $\sigma_B = C n^{-1/4} \log^{1/4} n$. 
    Each color corresponds to a different value of the scale parameter $C \in \{0.8, 1.2, 1.6, 2.0, 2.4\}$.
    As $C$ increases, both methods transition from failure to successful recovery. 
    A phase transition occurs near $C = 1.6$: when $C \leq 1.2$, the FNR remains high and increases with $n$; when $C \geq 2.0$, the FNR decays toward zero.
    Both SDP and group Lasso perform similarly in this heteroscedastic setting.}
    \label{fig:exp7}
\end{figure}

\begin{figure}
    \centering
    \includegraphics[width=0.8\textwidth]{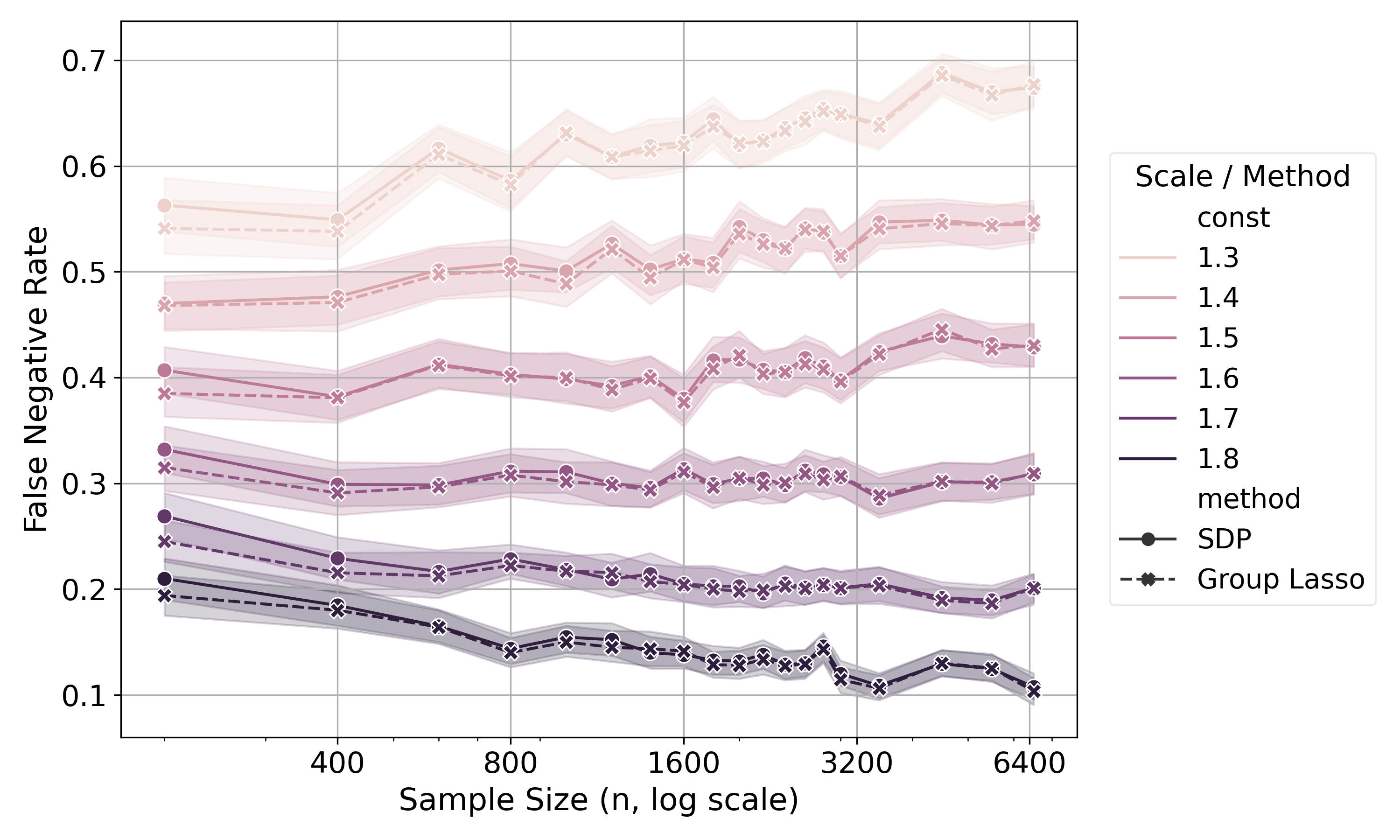}
    \caption{False negative rate of $\Istar$ for SDP and group Lasso under varying signal strength constant $C$ in $\sigma_B = C n^{-1/4} \log^{1/4} n$. 
    Each color corresponds to a different value of the scale parameter $C \in \{1.3, 1.4, 1.5, 1.6, 1.7, 1.8\}$.
    As the matrix dimension $n$ increases, all curves eventually stablize around some constant FNRs, with larger $C$ leading to lower FNRs. 
    Both SDP and group Lasso perform similarly in this heteroscedastic setting.}
    \label{fig:exp7finer} 
\end{figure}

\subsubsection{Failure Case for Group Lasso} 

We now demonstrate a failure mode for group Lasso by constructing a signal matrix $\mBstar$ that satisfies the sparsity assumption but induces misleading group-level structure.
Our construction is similar to Equation~\eqref{eq:Bstar:gl:fail} in Theorem~\ref{thm:failure}.

We set the number of signal nodes to $m = \lfloor 2\sqrt{n} \rfloor$ and define an augmented support set $\Istar\cup J$ of size $m + k$, where $k = \max\{5, \lfloor 0.2 m \rfloor\}$.
We randomly select $m + k$ nodes without replacement and let $\Istar$ denote the first $m$ indices (true signal nodes) and $J$ the remaining $k$ indices (perturbation nodes). 
The signal matrix $\mBstar$ is constructed as follows:
\begin{enumerate}
    \item For each $(i,j) \in \Istar\times J$ (and symmetrically $(j,i)$), set $B^\star_{ij} = \beta_1$, where $\beta_1 = \frac{2.5(n \log n)^{1/4}}{\sqrt{m}}$.
    \item For each $i \in \Istar$ and $j \notin \Istar\cup J$, set $B^\star_{ij} = \beta_2$, where $\beta_2 = 2 n^{-1/4} \log^{1/4} n$.
    \item All other entries of $\mBstar$ are set to zero.
\end{enumerate}

This construction creates two types of signal entries in $\mB$: 
a few relatively large entries, corresponding to entries indexed by an element of $\Istar$ and an element of $J$, and many weaker entries corresponding to entries indexed by an index in $\Istar$ and the rest of the nodes.
Under this setting, although nodes in $\Istar$ have signal strength above the exact recovery threshold given in Remark~\ref{rem:mle-snr}, the group Lasso struggles to separate $J$ and $\Istar$ due to its greedy selection mechanism.

We apply both SDP and group Lasso to the observed matrix $\mY = \mBstar + \mW$ to recover the support $\Istar$.  
As shown in Table~\ref{tab:exp2}, the SDP method consistently achieves exact support recovery across all sample sizes considered, since the signal in $\mBstar$ is well above the exact recovery threshold.
In contrast, group Lasso exhibits persistent recovery failure, with its false negative rate stabilizing around $0.2$.
This value aligns with the fraction $k / (m + k)$ of perturbation nodes, suggesting that group Lasso tends to select the larger-magnitude but spurious connections to $J$ rather than the true support. 
These empirical observations corroborate the theoretical limitations of group Lasso established in Theorem~\ref{thm:failure}.



\begin{table}[ht]
\centering
\caption{False negative rate (FNR) in recovering $\Istar$ of SDP and group Lasso across various sample sizes. 
Each value is reported as the mean (standard deviation) over $50$ repeated trials.}
\setlength{\tabcolsep}{10pt}
\renewcommand{\arraystretch}{1.2}
\begin{tabular}{ccc}
\toprule
\textbf{Sample Size ($n$)} & \textbf{SDP FNR} & \textbf{Group Lasso FNR} \\
\midrule
200  & 0.000 (0.000) & 0.152 (0.026) \\
1200 & 0.000 (0.000) & 0.168 (0.016) \\
2200 & 0.000 (0.000) & 0.180 (0.012) \\
3200 & 0.000 (0.000) & 0.181 (0.009) \\
4200 & 0.000 (0.000) & 0.184 (0.007) \\
5200 & 0.000 (0.000) & 0.184 (0.007) \\
6200 & 0.000 (0.000) & 0.189 (0.006) \\
7200 & 0.000 (0.000) & 0.190 (0.005) \\
\bottomrule
\end{tabular}
\label{tab:exp2}
\end{table}


\subsubsection{Multiple Copies versus Single Copy in SDP}

In Section~\ref{sec:minimax-support}, we established that the SDP method requires Assumption~\ref{assump:homoscedastic} to recover the node support $\Istar$ of $\mBstar$ when there is only one observed copy of $\mBstar$. 
We demonstrate that access to multiple copies allows us to relax Assumption~\ref{assump:homoscedastic} as shown in Theorem~\ref{thm:sdp:multiple}.  
Set the node support size $m = \lceil 2\log n \rceil$ and generate $\mY_1 = \mBstar + \mW_1$ and $\mY_2 = \mBstar + \mW_2$, where noise is constructed as follows:
\begin{enumerate}
    \item For each row $i$, sample $\sigma_i \sim \Unif[\sigma_{\min}, \sigma_{\max}]$, with $\sigma_{\min} = 0.8$ and $\sigma_{\max} = 1.3$.
    \item Set the $i$-th row of $\mSig_0$ to $\sigma_i^2$ and define $\mSig = \mSig_0 + \mSig_0^\top$.
    \item Generate a symmetric matrix $\mW_0$ as before and let $\mW = \mW_0 \circ \mSig$.
\end{enumerate}

This setup yields symmetric but row-heteroscedastic noise.  
We apply two versions of SDP, where for the multiple-copy SDP, the cost matrix is $\mY_1 \circ \mY_2$, and for the single-copy SDP, the cost matrix is $\mY_{\mathrm{avg}} \circ \mY_{\mathrm{avg}}$, where $\mY_{\mathrm{avg}} = (\mY_1 + \mY_2)/2$.

The results in Figure~\ref{fig:exp3} demonstrate that the single-copy SDP fails under row-heteroscedastic noise, with FNR increasing as $n$ grows.
This aligns with the discussion at the beginning of Section~\ref{sec:minimax-support}, confirming that support recovery is ill-posed in the presence of row-wise heteroscedasticity.
By comparison, the multiple-copy SDP accurately recovers $\Istar$ without Assumption~\ref{assump:homoscedastic}, consistent with Theorem~\ref{thm:sdp:multiple}.


\begin{figure}
    \centering
    \includegraphics[width=0.7\textwidth]{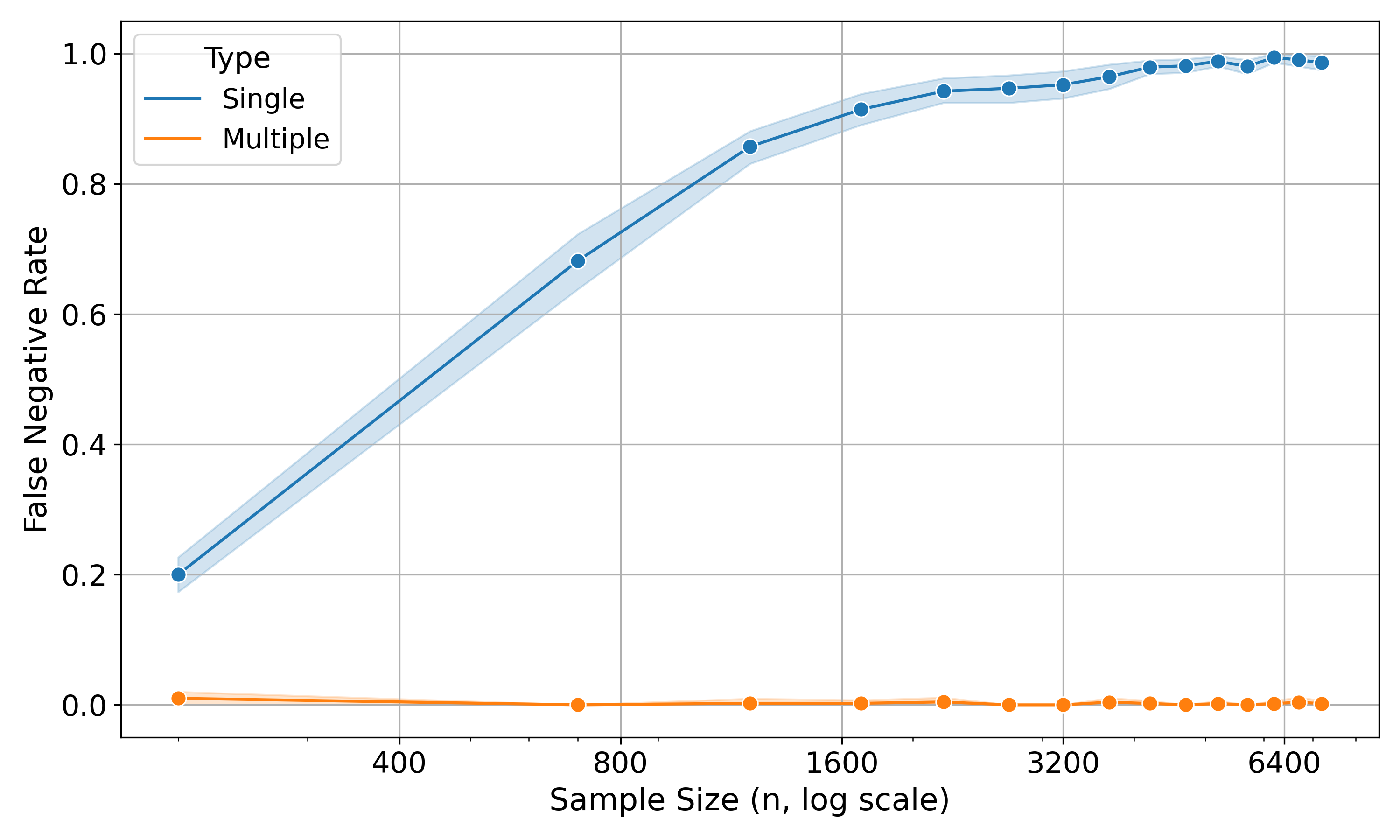}
    \caption{False negative rate for recovering $\Istar$ as a function of sample size $n$ (log scale), comparing two SDP-based methods: single-copy SDP (labeled ``Single'') and multiple-copy SDP (labeled ``Multiple'').  
	}
    \label{fig:exp3}
\end{figure}

\subsubsection{Truncated SDP for Heavy-Tailed Noise}

Finally, we consider a heavy-tailed noise setting where the entries of the noise matrix are drawn i.i.d.~according to a $t_4$ distribution rescaled to have unit variance.
The noise matrix $\mW$ and signal matrix $\mBstar$ are otherwise generated as inthe previous experiments as described in Section~\ref{sec:recovering-Bstar:experiments}.  
We set the node support size $|\Istar|$ to $m = \lceil 2\log n \rceil$ and compare the performance of the vanilla SDP estimator with the truncated version that clips extreme entries, as described in Section~\ref{sec:truncated-sdp}.  
As shown in Figure~\ref{fig:exp4}, the truncated SDP reliably recovers $\Istar$, in agreement with the theoretical guarantees established in Theorem~\ref{thm:trunc:sdp}. 
Specifically, truncated SDP achieves exact support recovery under heavy-tailed noise when the signal-to-noise ratio exceeds the exact recovery threshold $\Theta(n^{1/4} \log^{1/4} n)$ defined in Remark~\ref{rem:mle-snr}. 
In contrast, the vanilla SDP fails under this setting, with its FNR diverging to 1 as $n$ increases.
This behavior also aligns with Theorem~\ref{thm:sdp:L1-bound}, which indicates that when the noise magnitude $L$ significantly exceeds its standard deviation $\sigma$, the vanilla SDP requires a substantially higher SNR to achieve support recovery.


\begin{figure}
    \centering
    \includegraphics[width=0.7\textwidth]{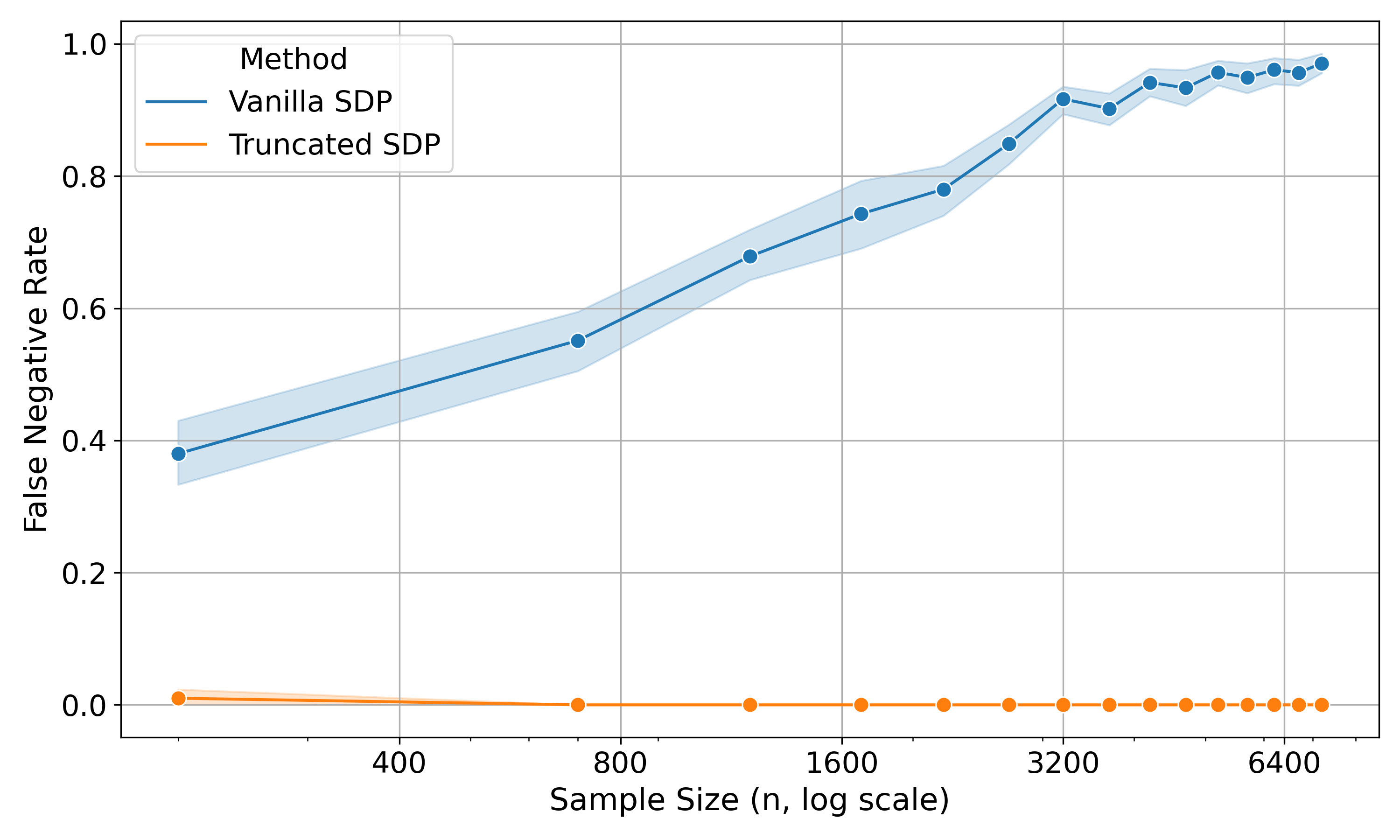} 
    \caption{False negative rate of recovering $\Istar$ under scaled $t_4$-distributed noise as a function of sample size $n$ (log scale). 
    The plot compares two methods: vanilla SDP (blue) and truncated SDP (orange). 
}
    \label{fig:exp4}
\end{figure}
  
\subsection{\texorpdfstring{Improved Estimation of $\mMstar$}{Improved Estimation of M*}} \label{sec:estimating-Mstar:experiments}

We now empirically validate our findings from Section~\ref{sec:common-structure}. 
Throughout this section, we assume access to four observations $\mY_{11}, \mY_{12}, \mY_{21}, \mY_{22}$ generated according to a low-rank-plus-noise model with fixed rank $r=3$. 
The noise matrices are generated as described in Section~\ref{sec:recovering-Bstar:experiments}. 
To construct the estimators, we asymmetrically arrange $\mY_{i1}$ and $\mY_{i2}$ to obtain $\mY_{i}$ for $i = 1,2$.

\subsubsection{\texorpdfstring{Varying $\mu$ and $\lambdastar_{\min}$}{Varying mu and lambda* min}}

We vary $n$ over a range of values from $700$ to $40000$.
The incoherence parameter $\mu$ is set to either $n^{4/5}$ or $n^{5/6}$, and the minimum eigenvalue of the low rank signal matrix $\mMstar$ is chosen as $\lambdastar_{\min} \in \{2.05, 3\} \sqrt{n}$.  
We consider three approaches to estimating $\mUstar$ and consequently $\mMstar$:
\begin{enumerate}
    \item Estimate $\mUstar$ using the bias-corrected eigenvectors $\mUhat$ from Equation~\eqref{eq:mUhat} applied to the average matrix $\mY_{\mathrm{asym}}$ of two assymetrically arranged matrices $\mY_1$ and $\mY_2$.
    Form the estimator $\mMhat_1 := \mUhat \mLambda \mUhat^\top$ as in Corollary~\ref{cor:single:entrywise}, where $\mLambda$ are the eigenvalues of $\mY_{\mathrm{asym}}$. 
    \item Averaging all four matrices $\mY_{ij}$ for $i,j \in [2]$ and compute the top-$r$ eigenspace $\mUhat_{\mathrm{spec}}$ and eigenvalues $\mLambda_{\mathrm{spec}}$ from the spectral decomposition of the average.  
    Then form $\mMhat_{\mathrm{spec}} := \mUhat_{\mathrm{spec}} \mLambda_{\mathrm{spec}} \mUhat_{\mathrm{spec}}^\top$ as the spectral estimator of $\mMstar$.
    \item Use the refined eigenspace estimate $\mU \mPsihat$ from Theorem~\ref{thm:eigenspace:l2infty}, and construct $\mMhat_2 := \mU \mPsihat \mLambda \mPsihat^\top \mU^\top$ as in Theorem~\ref{thm:improved:entrywise}, where $\mU$ consists of the right-eigenvectors of $\mY_1$, and $\mPsihat$ is the bias-correction factor computed from $\mY_{21}$ and $\mY_{22}$ following Equations~\eqref{eq:G:defin} to~\eqref{eq:Psihat:def}. 
\end{enumerate}

The subspace recovery and matrix entrywise recovery results are summarized in Figures~\ref{fig:exp5U} and~\ref{fig:exp5M}, respectively.
We highlight the following observations:
\begin{enumerate}
    \item \textbf{Estimation of $\mUstar$:} 
    The error rate of $\mUhat$ is the worse among all methods, since its error is dominated by the term $\sqrt{\mu / n} / \deltastar_{\min}$ as given in Corollary~\ref{cor:single:eigenspace}, which is significantly larger compared to that of $\mU\mPsihat$ in Theorem~\ref{thm:eigenspace:l2infty} and $\mUhat_{\spec}$ in Theorem~\ref{thm:ustar-row}.  
    When $\mu$ is small, $\mUhat_{\spec}$ achieves the lowest error for smaller matrix dimension $n$. 
    As $n$ increases, the behavior of $\mUhat_{\spec}$ exhibits a clear transition: the error initially decreases with $n$, then flattens, and eventually become worse than that of $\mU \mPsihat$.
    This transition point depends on both $\mu$ and $\lambdastar_{\min}$. 
    For instance, when $\mu = n^{4/5}$ and $\lambdastar_{\min} = 3\sqrt{n}$, the turning point occurs around $n = 8000$, beyond which the error of $\mUhat_{\spec}$ levels off. 
    As $\mu$ increases, the transition shifts to smaller $n$, while increasing $\lambdastar_{\min}$ delays the transition to larger $n$.
    According to Theorem~\ref{thm:ustar-row}, this suggests that for small $n$, the term $\sigma \sqrt{r \log n} / |\lambdastar_{\min}|$ initially dominates the estimation error of $\mUhat_{\mathrm{spec}}$, while the asymptotic term $\sigma \sqrt{\mu r} / |\lambdastar_{\min}|$ becomes dominant only for sufficiently large $n$. 
    \item \textbf{Estimation of $\mMstar$:} 
    We observe a similar trend in the estimation of $\mMstar$.
    The estimation error of $\mMhat_1 = \mUhat \mLambda \mUhat^\top$ increases as $\lambdastar_{\min}$ increases, consistent with Corollary~\ref{cor:single:entrywise}.  
    Among all methods, the spectral estimator $\mMhat_{\mathrm{spec}}$ performs best when $\mu$ is small and $\lambdastar_{\min}$ is large, particularly for small to moderate $n$. 
    As $n$ increases, the error of $\mMhat_{\spec}$ increases, while the error of $\mMhat_2$ remains nearly flat or slightly decreases. 
    Eventually, for large $n$, the error of $\mMhat_{\spec}$ surpasses that of $\mMhat_2$, in line with Theorem~\ref{thm:improved:entrywise}, which shows that $\mMhat_2$ achieves a better error rate than $\mMhat_{\spec}$ for sufficiently large $n$.

\end{enumerate}

\begin{remark}
    The estimators $\mU \mPsihat$ and $\mMhat_2$ are likely suboptimal in terms of the leading constant in their error bounds.
    This inefficiency arises because they do not fully exploit the information available from all four observations: two copies are primarily used for estimation, while the remaining two are incorporated only as a correction term.
    It might be possible to improve the efficiency by combining the spectral estimator with $\mU \mPsihat$ and $\mMhat_2$, which we leave for the future work.
\end{remark}

\begin{figure}
    \centering
    \includegraphics[width=0.7\textwidth]{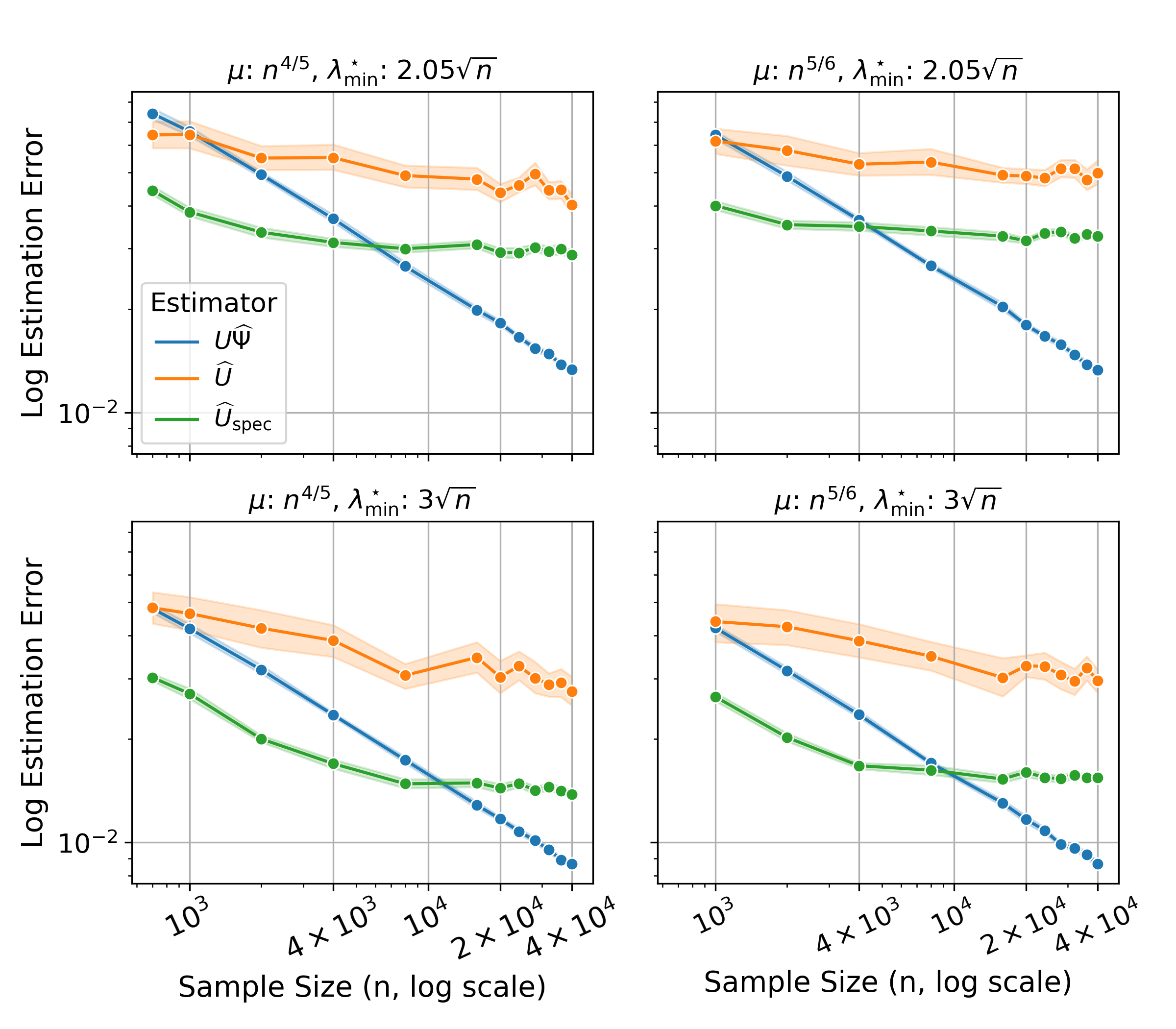}
    \caption{Estimation error of $\mUstar$ under the $\ell_{2,\infty}$ norm across varying coherence $\mu$ and signal strength $\lambdastar_{\min}$. 
    Each subplot corresponds to a specific pair of $(\mu, \lambdastar_{\min})$, with $\mu \in \{n^{4/5}, n^{5/6}\}$ (columns, left to right) and $\lambdastar_{\min} \in \{2.05\sqrt{n}, 3\sqrt{n}\}$ (rows, top to bottom).
    The estimators compared are $\mU \mPsihat$ (blue), $\mUhat$ (orange), and the spectral estimator $\widehat{\mU}_{\mathrm{spec}}$ (green).
    The $y$-axis shows the $\ell_{2,\infty}$ error on a log scale.
    }
    \label{fig:exp5U}
\end{figure}

\begin{figure}
    \centering
    \includegraphics[width=0.7\textwidth]{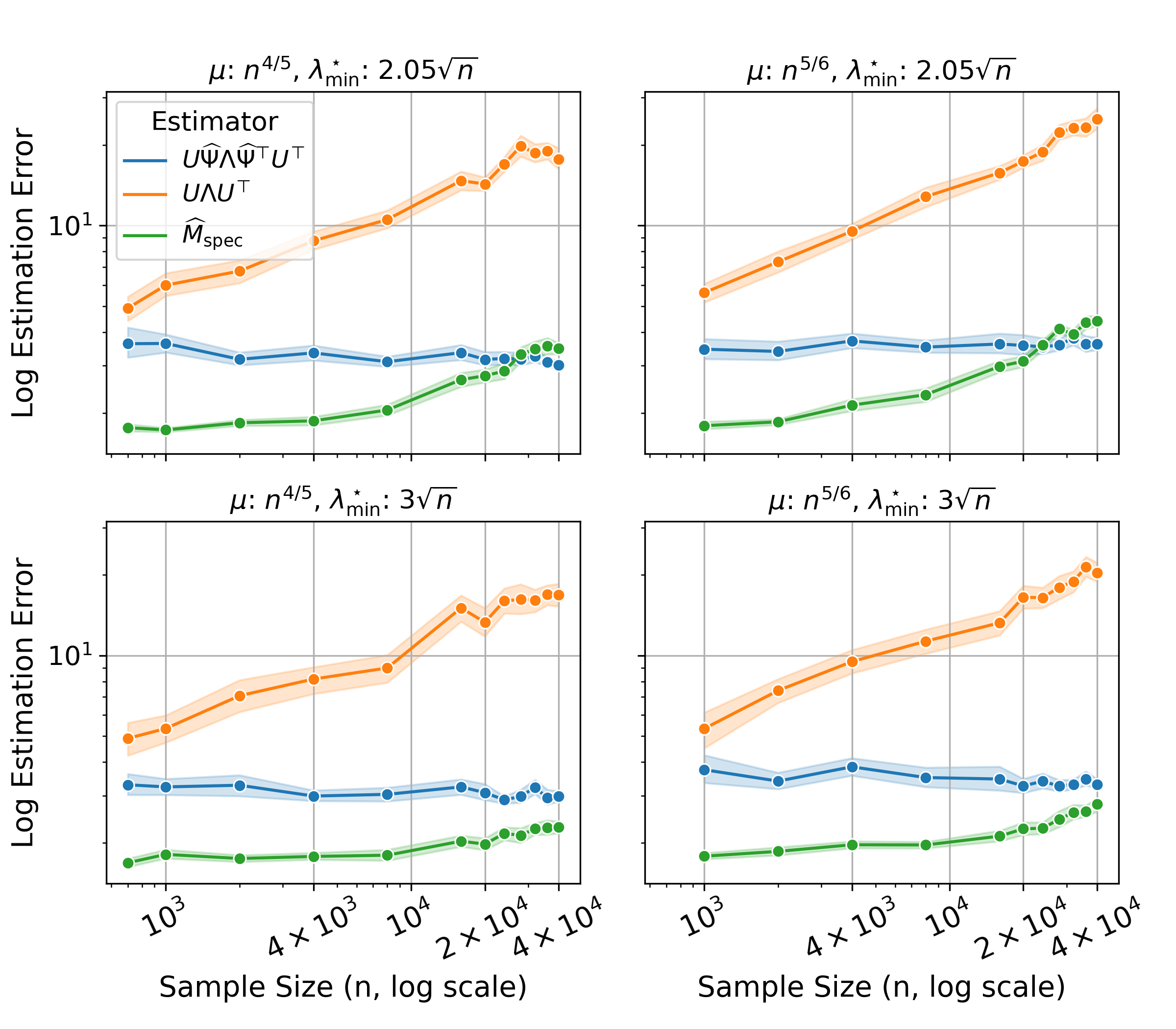}
    \caption{Estimation error of $\mMstar$ under the $\ell_{\infty}$ norm across various regimes of coherence $\mu$ and signal strength $\lambdastar_{\min}$. 
    Each subplot corresponds to a different combination of $(\mu, \lambdastar_{\min})$, where $\mu \in \{n^{4/5}, n^{5/6}\}$ (columns, left to right) and $\lambdastar_{\min} \in \{2.05\sqrt{n}, 3\sqrt{n}\}$ (rows, top to bottom).
    Three estimators are compared: $\mU \mPsihat \mLambda \mPsihat^\top \mU^\top$ (blue), $\mUhat \mLambda \mUhat^\top$ (orange), and the spectral estimator $\widehat{\mM}_{\mathrm{spec}}$ (green).}
    \label{fig:exp5M}
\end{figure}

\subsubsection{Varying Eigengap Ratio}

We conclude by examining the effect of the ratio $\rho := \deltastar_{\max}/\deltastar_{\min}$ on the estimation error rates of $\mMstar$ when using the estimator $\mMhat_2$. 
We fix $\mu = \sqrt{n} \log n$ and $\deltastar_{\min} = \log n$ in the experiments, and vary $\deltastar_{\max} \in \{\log n, n^{1/3}, n^{1/2}\}$.  
The result is shown in Figure~\ref{fig:exp6}. 
Different values of $\rho$ yield overlapping error curves, indicating that changes in $\rho$ have negligible impact on the estimation error rates. 
This empirical finding suggests that the dependence on $\deltastar_{\max}/\deltastar_{\min}$ displayed by the upper bound in Theorem~\ref{thm:improved:entrywise} is likely a technical artifact rather than a fundamental limitation. 
We leave the pursuit of this improvement to future work.

\begin{figure}
    \centering
    \includegraphics[width=0.7\textwidth]{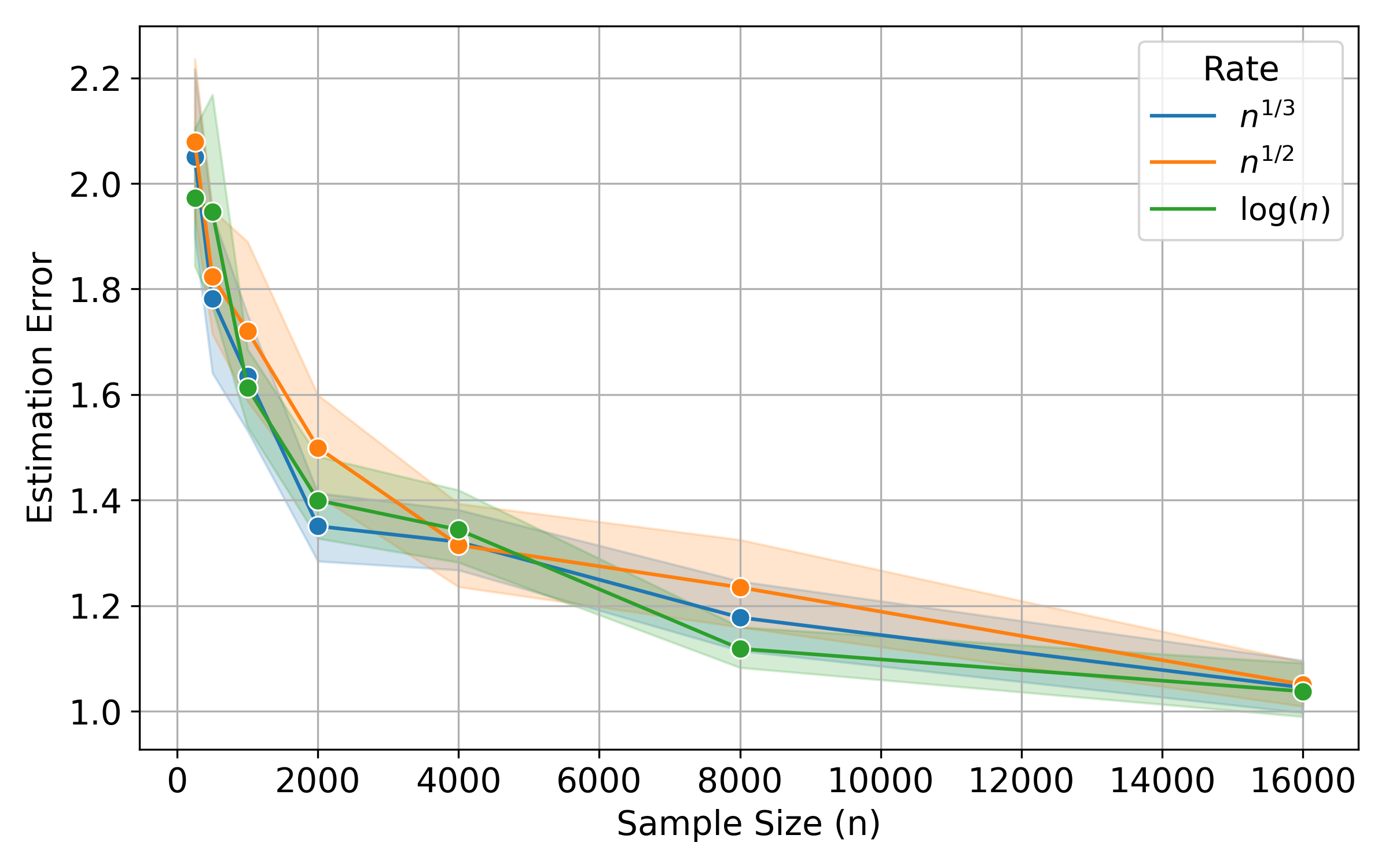}
    \caption{Estimation errors of $\mMhat_2 := \mU \mPsihat \mLambda \mPsihat^\top \mUhat^\top$ as a function of $n$ under the $\ell_{\infty}$ norm. 
    Each curve corresponds to a different scaling rate for the ratio $\rho := \deltastar_{\max} / \deltastar_{\min}$.
    Each point represents the average error over 30 independent trials.}
    \label{fig:exp6}
\end{figure}

\section{Discussion} \label{sec:discussion}

In this paper, we proposed a new statistical framework for analyzing multiple weighted networks that share a low-rank structure but differ through sparse, node-specific perturbations. Our model captures a broad class of practical scenarios, extending beyond traditional multilayer network formulations that primarily address global or smooth variations across conditions.

We introduced a suite of estimation procedures that adapt to the number of observations available per group. Our approach combines spectral initialization, semidefinite programming (SDP) for support recovery, and bias-corrected refinement for improved low-rank estimation. Theoretical results establish minimax optimality under various signal and sparsity regimes, and highlight the interplay between observation multiplicity and statistical efficiency.

We further examined the limitations of commonly used alternatives such as Group Lasso, which fails under certain structured perturbation settings. In contrast, our SDP-based methods exhibit robust support recovery even under heteroscedastic and heavy-tailed noise. Comprehensive numerical experiments corroborate the theory and demonstrate key behaviors, including phase transitions, the role of coherence, and the benefits of bias-correction and multiple network observations.

Several directions for future research also emerge from our findings:

\paragraph{Node-sparse structure estimation.}
First, a deeper understanding of the computational barriers for node-sparse recovery is needed. It remains open whether the limitations of Group Lasso observed here also apply in broader settings, such as covariance estimation.  
Second, while our SDP formulation is based on a least-squares loss, it is not robust to outliers. Incorporating robust loss functions (e.g., Huber or quantile-based objectives) and designing tractable algorithms under such losses would be interesting.
Third, our current framework assumes continuous weights. Extending support recovery guarantees to discrete-weight networks such as binary or count-valued adjacency matrices poses both theoretical and algorithmic challenges.

\paragraph{Low-rank structure estimation.}
The low-rank model for the control group may be oversimplified for real-world data. In many applications, the assumption of a clean low-rank plus noise decomposition may not hold.  
While our bias-corrected estimator achieves a near-optimal rate under the $\ell_{\infty}$ norm, it only outperforms baseline methods when the matrix dimension is sufficiently large. Improving the efficiency of this estimator, especially in terms of its leading constant, would enhance its practical utility in moderate-sized problems.  
Moreover, the fundamental limits of low-rank matrix recovery under $\ell_{\infty}$ and $\ell_{2,\infty}$ norms with only a single observation remain unresolved. In particular, it is unclear whether a coherence-independent estimator exists for subspace recovery in this setting. Developing such an estimator, if it exists, may require substantially more sophisticated techniques.

\medskip
Our results provide a foundation for future work on structure recovery in high-dimensional network models, and contribute to a broader understanding of statistical estimation under sparsity and low-rank constraints.

\newpage


\bibliographystyle{apalike}
\bibliography{bib}

\newpage

\appendix

\section*{Appendix}

\renewcommand{\contentsname}{Table of Contents}
\tableofcontents

\addcontentsline{toc}{section}{Appendix}

\section{Proof of Theorem \ref{thm:mle-result}} \label{sec:mle:proof}

\begin{proof}
To analyze the LSE, the high level idea is to make use of the optimality condition of the LSE that the $\Ihat$ is a solution to the optimization problem in Equation~\eqref{eq:least-square-B}.
Combining with standard concentration inequalities yields the exact and partial recovery results for the LSE. 
Below, we first prove a few concentration inequalities.  

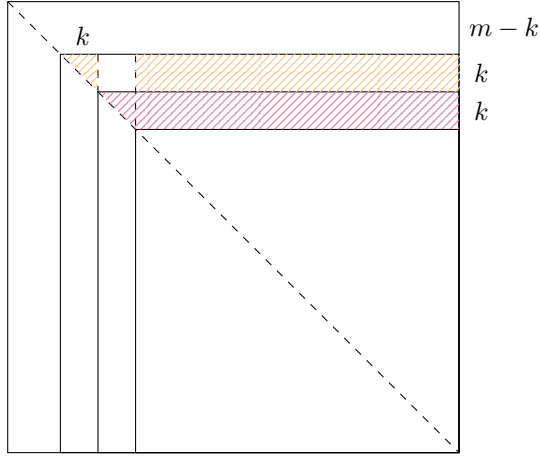
\begin{figure}[h]
    \centering
    \begin{tikzpicture}
        \def\s{6}
        \def\d{0.5}
        \def\o{0.2}

        \draw (0,0) rectangle (\s,\s);  
        \draw (\d+\o,\s-\d-\o) rectangle (\s,0);  
        \draw (2*\d+\o,\s-2*\d-\o) rectangle (\s,0); 
        \draw (3*\d+\o,\s-3*\d-\o) rectangle (\s,0); 
        
        \draw[dashed] (0,\s) -- (\s,0);  
        \draw[dashed] (2*\d+\o,\s-\d-\o) -- (2*\d+\o,\s-2*\d-\o);
        \draw[dashed] (3*\d+\o,\s-\d-\o) -- (3*\d+\o,\s-3*\d-\o);

        \node at (\s+0.6,{\s-0.5*(\d+\o)}) {\small $m - k$};
        \node at (\s+0.3,{\s-(\d+\o)-0.5*\d}) {\small $k$};
        \node at (\s+0.3,{\s-(2*\d+\o)-0.5*\d}) {\small $k$};
        \node at (\d+\o+0.3,{\s-(\d+\o)+0.5*\d}) {\small $k$};

        \fill[pattern=north east lines, pattern color = orange!50] (\d+\o,{\s-(\d+\o)}) -- (2*\d+\o,\s-\d-\o) -- (2*\d+\o,\s-2*\d-\o) -- cycle;
        \fill[pattern=north east lines, pattern color = orange!50] (3*\d+\o,{\s-(\d+\o)}) -- (\s,\s-\d-\o) -- (\s,\s-2*\d-\o) -- (3*\d+\o,\s-2*\d-\o) -- cycle;
        \fill[pattern=north east lines, pattern color = purple!50] (2*\d+\o,{\s-(2*\d+\o)}) -- (\s,\s-2*\d-\o) -- (\s,\s-3*\d-\o) -- (3*\d+\o,\s-3*\d-\o) -- cycle;

    \end{tikzpicture} 
    \caption{The overlap pattern of $I$ and $\Istar$ when $\Istar = [m]$ and $|I \cap \Istar| = m-k$.}
    \label{fig:mle}
\end{figure}

Without loss of generality, we assume that $\Istar = [m]$.
Let $I \subset [n]$ be any set of size $|I|=m$ and $|I\cap \Istar| = m-k$.
A demonstration of the overlapping pattern of the rows and columns indexed by $I$ and $\Istar$ is shown in Figure \ref{fig:mle}. 
Let 
\begin{equation*}
    \calS_{I} := \left\{ (i,j)\in [n]\times[n] : i<j, i \in I \right\}
\end{equation*}
and, for each $k \in [m]$, define
\begin{equation} \label{eq:def:N}
  N_k = k(n-m) + k/2 - k^2/2, 
\end{equation}
which is precisely the cardinality of the set $\calS_{I} \setminus \calS_{\Istar}$ when $|I\cap \Istar| = m-k$.
Under the assumption that $n \geq 3m$, we have that $N_k = \Omega(kn)$ for all $k \in [m]$. 

    We first fix a $k \in [m]$.
     With index set $I$ fixed, we proceed to bound the sum of squares for $W_{ij}$ in the purple shaded region shown in Figure~\ref{fig:mle}, which has $N_k$ entries in total.
It follows from Lemma~\ref{lem:bern-ineq} that for $t_k > 0$, 
\begin{equation*} \begin{aligned}
\Pr\left[\left|
\sum_{(i,j) \in \calS_I - \calS_{\Istar}} (W_{ij}^2 - \E W_{ij}^2)\right|
    \geq t_k \right]
\leq 2 \exp\left( -\frac{1}{2} \min\left\{ \frac{t_k^2}{N_k L^2 \sigma^2},
			\frac{t_k}{L^2} \right\} \right) .
\end{aligned} \end{equation*}
Since there are $\binom{n-m}{k}$ choices for $I \cap {\Istar}^c$, an application of the union bound yields 
\begin{equation*} \begin{aligned}
\Pr \left[ \max_{\substack{I \subset [n] \\ I \cap [m]^c = k}} \left|
    \sum_{(i,j) \in \calS_I \setminus \calS_{\Istar}} \!\!\!\! (W_{ij}^2 - \E W_{ij}^2)\right| \geq t_k \right] 
    &\leq 2 \binom{n-m}{k} \exp\!\left(-\frac{1}{2} \min\left\{ \frac{t_k^2}{N_k L^2 \sigma^2}, \frac{t_k}{L^2} \right\}\right)\\
    &\leq 2\left(\frac{en}{k}\right)^k \exp\!\left(-\frac{1}{2} \min\left\{ \frac{t_k^2}{N_k L^2 \sigma^2}, \frac{t_k}{L^2} \right\}\right)\\
    &\leq 2\exp \! \left( 2k \log \left(\frac{n}{k}\right)
	- \frac{1}{2} \min\!\left\{ \frac{t_k^2}{N_k L^2 \sigma^2},
				\frac{t_k}{L^2} \right\}\right) ,
\end{aligned} \end{equation*}
where the last inequality holds from the fact that $k \le m$ and our assumption that $n > 3m$.
Setting 
\begin{equation*}
    t_k = C \max\left\{L \sigma k \sqrt{n\log\left(\frac{n}{k}\right)}, L^2 k \log \left(\frac{n}{k} \right)\right\},
\end{equation*}
for a suitably large constant $C > 0$, it holds for for any $k \in [m]$ that
\begin{equation} \label{eq:mle:maxbound}
\max_{\substack{I \subset [n] \\ I \cap [m]^c = k}}
	\left|
    	\sum_{(i,j) \in \calS_{I}-\calS_{\Istar}} (W_{ij}^2 - \E W_{ij}^2)
		\right|
	\geq C \max\left\{L \sigma k \sqrt{n\log\left(\frac{n}{k}\right)}, L^2 k \log \left(\frac{n}{k} \right)\right\} 
\end{equation}
holds with probability at most $O(\left(k/n\right)^{7k})$, where we have used Equation~\eqref{eq:def:N} to bound the probability that Equation~\eqref{eq:mle:maxbound} fails to hold.

Observing that
\begin{equation*}
    \sum_{k=1}^m \left(\frac{k}{n}\right)^{7k} = O(n^{-7}),
\end{equation*}
a second union bound over all $k \in [m]$ yields that for all suitably large $n$, with probability at least $1 - O(n^{-7})$, it holds for all $k \in [m]$ that
\begin{equation}\label{eq:chi-sq-upper1}
\begin{aligned}
\max_{I \subset [n], I\cap [m]^c = k}
    \left|
    \sum_{(i,j) \in \calS_{I}-\calS_{\Istar}} (W_{ij}^2 - \E W_{ij}^2)\right|
    \leq C \max\left\{L \sigma k \sqrt{n\log\left(\frac{n}{k}\right)},
			L^2 k \log \left(\frac{n}{k} \right)\right\}. 
\end{aligned} \end{equation}
By the same argument, this time applied to the sum of squares of $W_{ij}$ in the orange shaded region in Figure~\ref{fig:mle}, it holds with probability at least $1-O(n^{-7})$ that for all $k \in [m]$,
\begin{equation}\label{eq:chi-sq-upper2} \begin{aligned}
\max_{I \subset [n], I\cap [m]^c = k}
& \left|\sum_{(i,j) \in \calS_{\Istar}-\calS_{I}} (W_{ij}^2 - \E W_{ij}^2)\right|
\leq C \max\left\{L \sigma k \sqrt{n\log\left(\frac{n}{k}\right)}, L^2 k \log \left(\frac{n}{k} \right)\right\}.
\end{aligned} \end{equation}
 
We now bound the sum of $W_{ij}B^\star_{ij}$ terms corresponding to entries in the orange shaded region shown in Figure~\ref{fig:mle}.
Under the condition in Equation~\eqref{eq:b-MLE-partial-sufficient}, we have 
\begin{equation}\label{eq:b:homo}
\frac{\max_{i,j \in \calS_{\Istar} - \calS_I} |B^\star_{ij}|}
  {\left(\sum_{i,j \in \calS_{\Istar} - \calS_I} B^{\star 2}_{ij}\right)^{1/2}}
= O\left( k^{-1/2} \log^{-1/2} \left(\frac{n}{k}\right) \right). 
\end{equation}
Applying standard sub-exponential tail bounds \citep{vershynin2018HDP}, it holds for any $I \subset [n]$ with $|I \cap \Istar| = k \in [m]$ that 
\begin{equation*} \begin{aligned}
\Pr \left(
	\left|\sum_{(i,j) \in \calS_{\Istar} - \calS_I} \!\!
		B^{\star}_{ij} W_{ij} \right| \geq t \right)
    &\leq 2 \exp\left(-\frac{1}{2} \min\left\{\frac{t^2}{\sigma^2 \sum_{(i,j)\in \calS_{\Istar} - \calS_I} B^{\star 2}_{ij}}, \frac{t}{L \max_{(i,j) \in \calS_{\Istar} - \calS_I} |\Bstar_{ij}|}\right\}\right).
\end{aligned} \end{equation*}
Taking 
\begin{equation*}
t = C\max\left\{\sigma \left(\sum_{(i,j)\in \calS_{\Istar} - \calS_I} B^{\star 2}_{ij}\right)^{1/2} \sqrt{k\log \left(\frac{n}{k}\right)}, L \max_{i,j \in \calS_{\Istar} - \calS_I} |\Bstar_{ij}| \cdot k \log\left(\frac{n}{k}\right)\right\}
\end{equation*}
and applying a union bound over $k \in [m]$ and $I \subset [n]$ such that $|I \cap \Istar| = k \in [m]$, it holds with probability at least $1 - O(n^{-7})$, for all $k \in [m]$ and $I \subseteq [n]$ with $|I \cap \Istar| = k$ that,
\begin{equation}\label{eq:normal-upper} \begin{aligned}
&
\left|
	\sum_{(i,j) \in \calS_{\Istar} - \calS_I} B^{\star}_{ij} W_{ij} 
	 \right|  \\
    &~~\leq C \max\left\{\sigma \left(\sum_{(i,j)\in \calS_{\Istar} - \calS_I} B^{\star 2}_{ij}\right)^{1/2} \sqrt{k\log \left(\frac{n}{k}\right)}, L \max_{(i,j) \in \calS_{\Istar} - \calS_I} |\Bstar_{ij}| \cdot k \log\left(\frac{n}{k}\right)\right\}\\
    &~~\leq C \max\left\{\sigma \left(\sum_{(i,j)\in \calS_{\Istar} - \calS_I} B^{\star 2}_{ij}\right)^{1/2} \sqrt{k\log \left(\frac{n}{k}\right)}, o\left(L \sigma k \sqrt{n \log\left(\frac{n}{k}\right)}\right) \right\}\\
\end{aligned} \end{equation}
where the last inequality follows from the condition in Equation~\eqref{eq:max-bstar}. 


For a matrix $\mB$ and a set $I \subseteq [n]$, we define the matrix $\mB_I = (B_{I,ij})_{1 \leq i,j \leq n} \in \R^{n \times n}$ according to
\begin{equation} \label{eq:matrix-B-I}
    B_{I,ij} = \begin{cases}
        B_{ij} &\mbox{ if } (i,j) \in I \times I,\\
        0 &\mbox{ otherwise.}
    \end{cases}
\end{equation}

For a matrix $\mA$, we let $\calM(\mA)$ denote the masked matrix whose upper triangle is filled with entries of $\mA$ and the lower triangle is all zero.
To establish exact recovery for the LSE $\Ihat$, we start with the basic inequality.
By the optimality of LSE, we have
\begin{equation*}
    \|\calM(\mY_{\hat{I}^c})\|_{\F}^2 = \|\calM(\mB^\star_{\hat{I}^c} + \mW_{\hat{I}^c})\|_{\F}^2 \leq \|\calM(\mY_{I_\star^c})\|_{\F}^2 = \|\calM(\mW_{I_\star^c})\|_{\F}^2.
\end{equation*}
Expanding the squared Frobenius norm and rearranging terms, we have
\begin{equation} \label{eq:mle-basic-ineq}
\begin{aligned}
    \|\calM(\mB^\star_{\hat{I}^c})\|_\F^2 + 2\langle \calM(\mB^\star_{\hat{I}^c}), \calM(\mW_{\hat{I}^c})\rangle \leq \|\calM(\mW_{I_\star^c})\|_{\F}^2 - \|\calM(\mW_{\hat{I}^c})\|_{\F}^2.
\end{aligned} \end{equation}

We will now show that in order for Equation~\eqref{eq:mle-basic-ineq} to hold, we must have $|\hat{I} \cap \Istar|=m$ with high probability under the condition given in Equation~\eqref{eq:b-MLE-sufficient}.
If for some $k \in [m]$, 
$|\hat{I} \cap \Istar| = m-k$, combining Equations~\eqref{eq:chi-sq-upper1}-\eqref{eq:normal-upper} into Equation~\eqref{eq:mle-basic-ineq} yields that with probability at least $1 - O(n^{-7})$, we have 
\begin{equation*} \begin{aligned}
    \|\calM(\mB^\star_{\hat{I}^c})\|_\F^2 &\leq C \sigma \|\calM(\mB^\star_{\hat{I}^c})\|_\F \sqrt{k\log \left(\frac{n}{k}\right)} + C L \sigma k \sqrt{n \log \left(\frac{n}{k}\right)}\\
\end{aligned} \end{equation*}
where the second inequality follows from Equation~\eqref{eq:def:N}. %
Taking square roots, it follows that
\begin{equation*} \begin{aligned}
\|\calM(\mB^\star_{\hat{I}^c})\|_F
&\leq C (L \sigma k)^{1/2} n^{1/4} \log^{1/4} \left(\frac{n}{k}\right).
\end{aligned} \end{equation*}
Noting that by the condition in Equation~\eqref{eq:b-MLE-sufficient}, we have
\begin{equation*}
    \|\calM(\mB^\star_{\hat{I}^c})\|_\F \geq \sqrt{k} b
\end{equation*}
where $b$ is defined in Equation~\eqref{eq:def:b}, we have 
\begin{equation*} \begin{aligned}
    b &\leq 
    C (L \sigma)^{1/2} n^{1/4} \log^{1/4} \left(\frac{n}{k}\right)
\end{aligned} \end{equation*}
Hence, if for some $k_0 > 0$ and a sufficiently large constant $c_0 > 0$,
\begin{equation*}
    b \geq c_0 (L \sigma)^{1/2} n^{1/4} \log^{1/4} \left(\frac{n}{k_0}\right) 
\end{equation*}
the inequality in Equation~\eqref{eq:mle-basic-ineq} 
cannot hold for any $k \geq k_0$ with probability at least $1 - O(n^{-7})$ and therefore, we must have $|\hat{I}\cap \Istar| \geq m - k_0$. 
Taking $k_0 = 1$ yields exact recovery. 

\end{proof}

\section{Proof of Minimax Results for Support Recovery} \label{apx:minimax}

In this section, we provide a proof of Theorem~\ref{thm:symmetric:minimax}, along with the proof of Corollaries~\ref{cor:minimax-recoverable} and~\ref{thm:multiple:minimax-recoverable}.
Several supporting technical lemmas are deferred to Section~\ref{subsec:symmetric:minimax:technical}.

\subsection{Proof of Theorem \ref{thm:symmetric:minimax}}

\begin{proof}
    We start with an overview of the proof strategy.
    Since the minimax risk is lower-bounded by the Bayes risk under any proper prior over $\mB \in \calB_S(b,n,m)$ defined in Equation~\eqref{eq:calBS:define}, 
    it suffices to construct a prior for which the lower bound in Equation~\eqref{eq:minimax-lower-bound} holds.  
    For details on minimax hypothesis testing, we refer the reader to \cite{ingster2003nonparametric}. 
    Our argument proceeds in four steps:
    \begin{itemize}
        \item \textbf{Step 1.} We construct a hierarchical prior over the parameter set $\calB_S(b,n,m)$, and then replace it with a more tractable product prior.
        The Bayes risk under this product prior serves as a lower bound of the minimax risk, and is easier to analyze.
        \item \textbf{Step 2.} We reduce the Bayes risk to sum of expected errors in a sequence of $n$ binary hypothesis testing problems, one for each node, allowing a tractable lower bound on the Bayes risk. 
        \item \textbf{Step 3.} We derive lower bound on the expected testing risk for each hypothesis by analyzing a Gaussian mixture testing problem, where the risk is defined as a weighted combination of the Type I and Type II errors. 
        \item \textbf{Step 4.} Finally, we aggregate the binary hypothesis testing risks to obtain the desired lower bound on the overall Bayes risk. 
    \end{itemize}
 
\paragraph{Step 1: Constructing a prior.}
To define a probability measure over $\calB_S(b, n, m)$, we proceed in two steps. 
First, we sample the binary support indicator vector $\veta \in \{0,1\}^n$ from a prior $\pi_\Theta$, defined in Lemma~\ref{lem:minimax:technical}, with parameter $m' \in (0, m]$. 
Conditioned on $\veta$, for each $i \in [n]$:
\begin{itemize}
    \item If with $\eta_i = 1$, we sample the vector of entries $\{B_{ii}\} \cup \{B_{ij} | j: \eta_j = 0\}$ uniformly from the sphere $b\bbS^{n-\sum_{j=1}^n \eta_j}$.
    For $B_{ij}$ where $i \neq j$ and $\eta_i = \eta_j = 1$, we set $B_{ij}$ to be a small fixed nonzero value, for example, $B_{ij} = \sigma/n$.
    \item If $\eta_i = 0$, for $j \in [n]$ and $\eta_j = 0$, we set $B_{ij} = 0$.
\end{itemize}
By definition of $\calB_S(b, n, m)$, the prior must exclude the set
\begin{equation*}
\calE := \left\{\mB \in \Sym(n), |I_\mB| \leq m :
    		\|\mB_{i,\cdot}\|_0 \leq m
    		\text{ for all } i \in I_\mB \right\},
\end{equation*}
which includes matrices that violate Assumption~\ref{assump:B-identifiable}, thereby making the node support $I_{\mB}$ potentially non-identifiable. 
However, since the event $\calE^c$ occurs almost surely under our construction (since $\|\mB_{i,\cdot}\|_0 > m$ occurs almost surely), including it does not affect the resulting Bayes risk.
Thus, for notational and technical simplicity, we proceed with the sampling distribution defined above.
Under $\calB_S(b,n,m) \setminus \calE$, the support $I_{\mB}$ corresponds to the indices $i$ such that $\eta_i = 1$. 

Although the prior $\pi_{\Theta}$ defines a valid probability distribution over $\calB_S(b,n,m)$, 
computing the Bayes risk under this prior is challenging.
Fortunately, we can construct a simpler distribution that yields a lower bound on the Bayes risk under $\pi_{\Theta}$ that is easier to analyze. 
Let $\pi$ denote the product measure of $n$ independent Bernoulli distributions with shared parameter $m'/n$. 
By Equation~\eqref{eq:independent-lower-bound} in Lemma \ref{lem:minimax:technical}, we have 
\begin{equation}\label{eq:lower-risk-eq1}
\inf_{\vetahat} \E_{\pi_\Theta} \E_{\veta} \|\veta - \vetahat\|_\rmH
\geq \inf_{\vetahat \in [0,1]^n} \E_{\pi} \E_{\veta} \|\vetahat - \veta\|_\rmH
- 4 m^{\prime} \exp \left(-\frac{\left(m-m^{\prime}\right)^2}{2 m}\right),
\end{equation}
which implies that it suffices to analyze the Bayes risk under the simpler product prior $\pi$.
    
Accordingly, we revise our sampling procedure as follows. 
First, sample $\veta$ from $\pi$.
Then, conditioned on $\veta$, draw $\{B_{ii}\} \cup \{B_{ij} | j: \eta_j = 0\}$ from the uniform distribution over $b\bbS^{n-\sum_{j=1}^n \eta_j}$. 
For $i, j \in [n]$ distinct such that $\eta_i = \eta_j = 1$, set $B_{ij}$ to be a small fixed nonzero value such as $\sigma / n$.    
In this way, we are guaranteed that $I_{\mB}$ as well as $\veta$ are almost surely identifiable. 
 
\paragraph{Step 2: Bayes risk reduces to testing $n$ hypotheses.}
To analyze the information available for estimating $\eta_i$, define the following subsets of observed entries:
\begin{equation*}
    \mYtil_{\sim i} = \{\Ytil_{jk}: j\neq i, k\neq i\} \cup \{\Ytil_{ik}: \eta_k = 1, k\neq i\}, \quad \text{and} \quad \vytil_i = \{\Ytil_{ik}: \eta_{k} = 0\} \cup \{\Ytil_{ii}\}.
\end{equation*}
Denote 
\begin{equation}\label{eq:geta:define}
    g(\veta, \mYtil_{\sim i}) := \E \left[ \left. \left|\etahat_i(\mYtil) - \eta_i\right| \right| \veta, \widetilde{\mY}_{\sim i}\right].
\end{equation}
By definition of the Hamming distance, we have 
\begin{equation} \label{eq:lower-risk-eq1-5}
\begin{aligned}
\inf_{\vetahat \in [0,1]^n} \E_{\pi} \E_{\veta} \|\vetahat - \veta\|_\rmH
&\geq 
\sum_{i=1}^n\inf_{\etahat_i\in [0,1]}
	\E_{\pi} \E_{\veta} \left|\etahat_i(\mYtil) - \eta_i\right| \\
&= \sum_{i=1}^n \inf_{\etahat_i\in [0,1]} \E_{\pi}
	\E \left[\E \left. \left[ \left. \left|\etahat_i(\mYtil) - \eta_i\right| \right| \veta, \widetilde{\mY}_{\sim i}\right] \right| \veta \right]\\
&= \sum_{i=1}^n \inf_{\etahat_i\in [0,1]} \E_{\pi}
	\E \left[\left. g(\veta, \mYtil_{\sim i}) \right| \veta \right]. 
\end{aligned} \end{equation}
Recall that under the product measure $\pi$, the indicator $\eta_i$ is independent of $\mYtil_{\sim i}$ and $\veta_{\sim i}$,
and follows a Bernoulli distribution with parameter $m'/n$, denoted $\pi_0$. 
We can rewrite Equation~\eqref{eq:lower-risk-eq1-5} as
\begin{equation*}
\begin{aligned}
    \inf_{\vetahat \in [0,1]^n} \E_{\pi} \E_{\veta} \|\vetahat - \veta\|_\rmH
&\geq \sum_{i=1}^n \inf_{\etahat_i\in [0,1]} \E_{\pi_0^{n-1}} \E_{\pi_0} \left[
	\E \left[\left. g(\veta, \mYtil_{\sim i}) \right| \veta \right] \veta_{\sim i}\right]\\
    &= \sum_{i=1}^n \inf_{\etahat_i\in [0,1]} \E_{\pi_0^{n-1}} \E \left[ \left.
	\E_{\pi_0} \left[\left. g(\veta, \mYtil_{\sim i}) \right| \veta_{\sim i}, \mYtil_{\sim i} \right] \right|\veta_{\sim i}  \right].
\end{aligned}
\end{equation*}
    Interchanging the order of infimum and $\E_{\pi_0^{n-1}} \E$ in the above display and using the definition of $g$ in Equation~\eqref{eq:geta:define} yields that 
\begin{equation} \label{eq:lower-risk-eq2}
\begin{aligned}
    &\inf_{\vetahat \in [0,1]^n} \E_{\pi} \E_{\veta} \|\vetahat - \veta\|_\rmH \\
    &~~\geq 
    \sum_{i=1}^n  \E_{\pi_0^{n-1}} \E \left[\left. \inf_{\etahat_i\in [0,1]} \E_{\pi_0} \left\{ \E \left[ \left. \left|\etahat_i(\mYtil) - \eta_i\right| | \veta, \widetilde{\mY}_{\sim i}\right] \right| \veta_{\sim i}, \widetilde{\mY}_{\sim i} \right\} \right|\veta_{\sim i} \right].
\end{aligned}
\end{equation}
Now, since $\eta_i$ is independent of $\veta_{\sim i}$ and $\mYtil_{\sim i}$, 
and since the only part of $\mYtil$ that depends on $\eta_i$ is $\vytil_i$,
the inner expectation of Equation~\eqref{eq:lower-risk-eq2} reduces to 
\begin{equation} \label{eq:lower-risk-eq2.5}
\begin{aligned}
    &\inf_{\etahat_i\in [0,1]} \E_{\pi_0} \left[ \left. \E \left[ \left|\etahat_i(\widetilde{\vy}_i, \widetilde{\mY}_{\sim i}) - \eta_i\right| | \veta, \widetilde{\mY}_{\sim i}\right] \right| \veta_{\sim i}, \widetilde{\mY}_{\sim i}\right]
    = & \inf_{T_i( \widetilde{\vy}_i ) \in [0,1]} \E_{\pi_0} \left[\E \left[ \left|T_i(\widetilde{\vy}_i) - \eta_i\right| | \veta \right]| \veta_{\sim i}\right],
\end{aligned}
\end{equation}
where the infimum on the right hand side is taken over all test functions $T_i( \widetilde{\vy}_i ) \in [0,1]$ that is a measurable function of $\widetilde{\vy}_{i}$.
Taking Equation~\eqref{eq:lower-risk-eq2.5} into Equation~\eqref{eq:lower-risk-eq2}, we have
\begin{equation} \label{eq:lower-risk-eq2.6}
\begin{aligned}
    &\inf_{\vetahat \in [0,1]^n} \E_{\pi} \E_{\veta} \|\vetahat - \veta\|_\rmH \geq 
    \sum_{i=1}^n  \E_{\pi_0^{n-1}} \left[\inf_{T_i( \widetilde{\vy}_i ) \in [0,1]} \E_{\pi_0} \left[\E \left[ \left|T_i(\widetilde{\vy}_i) - \eta_i\right| | \veta \right]| \veta_{\sim i}\right] \right].
\end{aligned}
\end{equation}

\paragraph{Step 3: Lower bound the expected test errors.}
Recall the sampling scheme for $\veta$ and $\mB$ described in Step 1.
For a given $\veta_{\sim i}$, the right hand side of Equation~\eqref{eq:lower-risk-eq2.5} corresponds to a hypothesis testing problem under a Gaussian mixture model.
Let 
\begin{equation}\label{eq:k:define}
    k = n - \sum_{j\neq i} \eta_j.
\end{equation}
We consider a binary random variable $\eta \in \{0,1\}$ drawn from the prior distribution $\pi_0$, and define a random vector $\vy \in \R^k$ according to the following mixture model:
\begin{equation*}
\vy \eqdist
    \begin{cases}
        \vw & \mbox{ if } \eta = 0,\\
        \vtheta + \vw & \mbox{ if } \eta = 1,
    \end{cases}
\end{equation*}
where $\vtheta \sim \Unif(b\bbS^{k-1})$ and $\vw \sim \calN(0, \sigma^2 \mI_k)$ are independent.
By construction, $\vytil_i$ has the same distribution as $\vy$ and we have 
\begin{equation*}
    \inf_{T_i(\vytil_i) \in [0,1]} \E_{\pi_0} [\E \left[ \left|T_i(\vytil_i) - \eta_i\right| \veta \right] | \veta_{\sim i}] = \inf_{T(\vy) \in [0,1]} \E_{\pi_0} \E_\eta \left[ \left|T(\vy) - \eta\right| \right],
\end{equation*}
where the infimum on the right hand side is taken over all test functions $T( \vy ) \in [0,1]$ that is a measurable function of $\vy$. 
We also use the fact that $\eta$ is independent of $\veta_{\sim i}$ under the product measure $\pi$ to obtain the above equality. 

Combining the above display with Lemma \ref{lem:minimax:risk}, when we choose a sufficiently small $b$, it follows that 
\begin{equation} \label{eq:lower-risk-eq3}
\begin{aligned}
    \inf_{T_i(\vytil_i) \in [0,1]} \E_{\pi_0} [\E \left[ \left|T_i(\vytil_i) - \eta_i\right| \veta \right] | \veta_{\sim i}] &= \inf_{T(\vy) \in [0,1]} \E_{\pi_0} \E_\eta \left[ \left|T(\vy) - \eta\right| \right]\\
    &\geq \frac{m^{\prime}}{n}-\frac{m^{\prime 2}}{4 n^2} \exp \left(\frac{b^4 C}{k \sigma^4}\right)-\frac{3 m^{\prime 2}}{4 n^2},\\
\end{aligned}
\end{equation}
where $C > 0$ is a universal constant. 
As a reminder, $k$ in Equation~\eqref{eq:lower-risk-eq3} is a random quantity defined in Equation~\eqref{eq:k:define}. 

\paragraph{Step 4: Combining the results.}
Combining Equations~\eqref{eq:lower-risk-eq2.6} and~\eqref{eq:lower-risk-eq3},
\begin{equation*}
\inf_{\vetahat \in [0,1]^n} \E_{\pi} \E_{\veta} \|\vetahat - \veta\|_\rmH 
\geq \max_{m' \in (0, m]} \left[ m' - \frac{3m^{\prime 2}}{4n} - \frac{m^{\prime 2}}{4n}\E\left[ \exp\left(\frac{b^4 C}{k \sigma^4}\right) \right]\right] . 
\end{equation*}
Defining $a = C b^4 / \sigma^4$ and applying Lemma~\ref{lem:exp-expect-binom}, we have for sufficiently large $n$ and $3m < n$, 
\begin{equation*} \begin{aligned}
\inf_{\vetahat \in [0,1]^n} & \E_{\pi} \E_{\veta} \|\vetahat - \veta\|_\rmH \\
&\geq \max_{m' \in (0, m]} \frac{m'}{4}
	\left[ 3 - \frac{m^{\prime}}{n}
		\left( \exp \left\{a-C n \log \frac{n}{3 m^{\prime}} \right\} 
			+ \exp \left\{ \frac{a}{\log n}-C n\right\} 
			+ e^{\frac{C a}{n}} \right) \right] .
\end{aligned} \end{equation*}
Recalling the definition of $f_n(a, m')$ from the statement of Theorem~\ref{thm:symmetric:minimax}, it follows that
\begin{equation} \label{eq:lower-risk-indep}
\inf_{\vetahat \in [0,1]^n} \E_{\pi} \E_{\veta} \|\vetahat - \veta\|_\rmH 
\ge \max_{m' \in (0, m]} \left(\frac{m'}{4} f_n(a, m') \right) .
\end{equation}
Thus, combining Equations~\eqref{eq:lower-risk-eq1}-\eqref{eq:lower-risk-indep}, we have
\begin{equation*} \begin{aligned}
\inf_{\vetahat} \E_{\pi_\Theta} \E_{\veta} \|\vetahat - \veta\|_\rmH
&\geq
\inf _{T \in[0,1]} \inf_{\vetahat \in [0,1]^n} 
	\E_{\pi} \E_{\veta} \|\vetahat - \veta\|_\rmH
	- 4 m^{\prime} \exp \left\{ -\frac{\left(m-m^{\prime}\right)^2}{2 m}\right\} \\
&\geq \max_{m' \in (0, m]} \left[ \frac{m'}{4} f_n(a, m') 
	-4 m^{\prime} \exp \left\{ -\frac{\left(m-m^{\prime}\right)^2}{2 m}\right\} \right]
\end{aligned}
\end{equation*}
which completes the proof.
\end{proof}

\subsection{Proof of Corollary~\ref{cor:minimax-recoverable}}
\begin{proof}
To show that exact recovery of $\veta$ is infeasible, it suffices to provide a nontrivial lower bound of the Hamming risk 
\begin{equation*}
\inf_{\vetahat \in [0,1]^n} \E_{\pi} \E_{\veta} \|\vetahat - \veta\|_\rmH.
\end{equation*}
Let $a = C b^4 / \sigma^4$.
For any $m' \in [1, m]$, when 
\begin{equation*}
    a < Cn\log\left(\frac{n}{3m'}\right),
\end{equation*}
we have 
\begin{equation*}
\max\left\{a-C n \log \left(\frac{n}{3 m^{\prime}}\right),~
	\frac{a}{\log n}-C n,~
	\frac{C a}{n}\right\} 
< \log\left(\frac{n}{3m'}\right).
\end{equation*}
Taking exponentials on both sides and applying Theorem~\ref{thm:symmetric:minimax}, the following inequality holds:
\begin{equation} \label{eq:minimax:requirement}
3 - \frac{m^{\prime}}{n}\left( 
	\exp \left\{a-C n \log \frac{n}{3 m^{\prime}} \right\}
	+ \exp \left\{ \frac{a}{\log n}-C n\right\} 
	+\exp\left\{\frac{C a}{n}\right\} \right) > 2.
\end{equation}
Combining this with Equation~\eqref{eq:lower-risk-indep}, we have that
\begin{equation*}
\inf_{\vetahat \in [0,1]^n} \E_{\pi} \E_{\veta} \|\vetahat - \veta\|_\rmH 
> \frac{m'}{2}. 
\end{equation*}
Plugging the above bound into Equation~\eqref{eq:lower-risk-eq1}, we conclude that 
\begin{equation*} \begin{aligned}
    \inf_{\vetahat} \E_{\pi_\Theta} \E_{\veta} \|\veta - \vetahat\|_\rmH
    &\geq 
    \inf_{\vetahat \in [0,1]^n} \E_{\pi} \E_{\veta} \|\vetahat - \veta\|_\rmH - 4 m^{\prime} \exp \left(-\frac{\left(m-m^{\prime}\right)^2}{2 m}\right) \\
    &\geq  \frac{m'}{2} -  4 m^{\prime} \exp \left(-\frac{\left(m-m^{\prime}\right)^2}{2 m}\right)
\end{aligned} \end{equation*}
as claimed.
\end{proof}

\subsection{Proof of Theorem~\ref{thm:multiple:minimax-recoverable}}
\label{subsec:multiple:minimax-recoverable}

\begin{proof}
    Following the same argument as in the proof of Theorem~\ref{thm:symmetric:minimax}, we can construct a prior over $\calB_S(b, n, m)$ and then replace it with a more tractable product prior via Lemma~\ref{lem:minimax:technical}.
    We use the same priors for $\veta$ and $\mB$ in the proof of Theorem~\ref{thm:symmetric:minimax}, but we must consider different observed data.
    Thus, as counterparts to $\mYtil_{\sim i}$ and $\vytil_{i}$ in the proof of Theorem~\ref{thm:symmetric:minimax}, consider
    \begin{equation*}
        \mYtil^{(1:N)}_{\sim i} = \bigcup_{\ell \in [N]} \left\{\Ytil_{\ell, jk}, j\neq i, k \neq i \right\} \cup \left\{\Ytil_{\ell,ik}, : \eta_k = 1, k \neq i\right\}
    \end{equation*}
    and 
    \begin{equation*}
        \vytil^{(1:N)}_{i} = \bigcup_{\ell \in [N]} \left\{\Ytil_{\ell, ik}: \eta_k = 0\right\} \cup \left\{\Ytil_{\ell, ii}\right\},
    \end{equation*}
    respectively. 
    Following the same argument as Step 2 in the proof of Theorem~\ref{thm:symmetric:minimax}, we can show that
    \begin{equation*}
        \inf_{\vetahat \in [0,1]^n} \E_{\pi} \E_{\veta} \|\vetahat - \veta\|_\rmH \geq \sum_{i=1}^n \inf_{\etahat_i\in [0,1]} \E_{\pi} \E_{\veta} \left[\left|\etahat_i(\mYtil_1, \mYtil_2, \cdots, \mYtil_N) - \eta_i\right| \mid \veta_{\sim i}, \mYtil^{(1:N)}_{\sim i} \right]. 
    \end{equation*}
    By arguments similar to those yielding Equations~\eqref{eq:lower-risk-eq2} and~\eqref{eq:lower-risk-eq2.5}, we can once again reduce the Bayes risk to the sum of expected errors in a sequence of $n$ binary hypothesis testing problems, one for each node.
    To specify the hypotheses, for any $i \in [n]$, given $\veta_{\sim i}$, we let $k = n - \sum_{j\neq i} \eta_j$ and 
    \begin{equation*}
        \vy^{(\ell)} = \vtheta + \vw^{(\ell)}, \quad \ell \in [N],
    \end{equation*}
    where $\vw^{(\ell)} \stackrel{i.i.d.}{\sim} \calN(0, \sigma^2 \mI_k)$ for all $\ell \in [N]$.
    Under the null hypothesis, we have $\vtheta = \vo$, while under the alternative hypothesis, $\vtheta \sim \Unif(b\bbS^{k-1})$.
    Then the binary hypothesis testing risk associated with the $i$-th node is
    \begin{equation*}
    \inf_{T \in [0,1]} \E_{\pi_0} \E_{\eta}\left[ 
	\left|T(\vy^{(1)},\vy^{(2)}, \dots, \vy^{(N)}) - \eta\right|\right],
    \end{equation*}
    where $\eta \in \{0,1\}$ is a binary selector drawn from a Bernoulli distribution with parameter $m'/n$, denoted as $\pi_0$.
    
    As seen in the proof of Lemma~\ref{lem:minimax:risk}, the testing risk above can be bounded by the expectation of the likelihood ratio under the null hypothesis.
    This likelihood ratio is given by
    \begin{equation} \label{eq:likelihood:ratio:multiple}
    \begin{aligned}
        L &= \frac{f_1\left(\vy^{(1)}, \vy^{(2)}, \cdots, \vy^{(N)}\right)}{\prod_{\ell=1}^N f_0\left(\vy^{(\ell)}\right)}
        = \frac{\E_{\vtheta \sim \Unif(b\mathbb{S}^{k-1})} \prod_{\ell=1}^N \exp\left\{-\frac{1}{2\sigma^2} \left\|\vy^{(\ell)} - \vtheta\right\|^2_2\right\}}{\prod_{\ell=1}^N \exp\left\{-\frac{1}{2\sigma^2} \left\|\vy^{(\ell)}\right\|_2^2\right\}}\\
        &= \E_{\vtheta \sim \Unif(b\mathbb{S}^{k-1})} \exp\left\{-\frac{N}{2\sigma^2} \left\|\vtheta\right\|_2^2 + \frac{1}{\sigma^2} \sum_{\ell=1}^N \vy^{(\ell)\top} \vtheta\right\}.
    \end{aligned}
    \end{equation}
    Setting $\vybar := \frac{1}{N} \sum_{\ell=1}^N \vy^{(\ell)}$, we have
    \begin{equation*}
        L = \E_{\vtheta \sim \Unif(b\mathbb{S}^{k-1})} \exp\left\{-\frac{N}{2\sigma^2} \left\|\vtheta\right\|_2^2 + \frac{N}{\sigma^2} \vybar^\top \vtheta\right\}.
    \end{equation*}
    Thus, replacing $\sigma^2$ with $\sigma^2/N$ and $\vy$ with $\vybar$ in the proof of Lemma~\ref{lem:likelihood:second-moment}, we obtain that
    \begin{equation*}
        \E_0 L^2 \leq \exp\left\{\frac{C b^4}{k(\sigma/\sqrt{N})^4}\right\}. 
    \end{equation*}
    The remainder of the proof directly follows the same argument as in the proof of Theorem~\ref{thm:symmetric:minimax}.
\end{proof}

\subsection{Technical Lemmas} 
\label{subsec:symmetric:minimax:technical}

\begin{lemma}\label{lem:minimax:technical}
Let $\mYtil_{\ell} = \mB + \mW_{\ell}$ for $\ell \in [N]$,
where $\mB \in \calB_S(b, n, m)$ and $\mW_{\ell}$ are symmetric matrices with $W_{\ell, ij} \iid \calN(0, \sigma^2)$ for all $1 \leq i \leq j \leq n$ and $\ell \in [N]$.
Define the set of sparse binary vectors with at most $m$ non-zero entries as
\begin{equation*}
    \Theta(n, m) = \left\{\vzeta \in \{0,1\}^n : \|\vzeta\|_0 \leq m\right\}.
\end{equation*}
Let $\pi$ denote the product measure of $n$ independent Bernoulli distributions with shared parameter $m'/n$, for some real number $m' \in (0, m]$, 
and let $\pi_{\Theta}$ be the conditional distribution of $\pi$ restricted to $\Theta(n, m)$, i.e., for any subset $E \subset \{0, 1\}^n$,
\begin{equation*}
\pi_\Theta(E)
= \frac{\pi(E \cap \{\vzeta \in \Theta(n, m)\})}{\pi(\vzeta \in \Theta(n, m))}. 
\end{equation*}
Let $\veta \in \{0,1\}^n$ be the binary indicator vector defined as
\begin{equation*}
    \eta_i = \indic\{ i \in I_{\mB} \}, \quad i \in [n]. 
\end{equation*}
Then for any $\vetahat \in \R^n$ that is a measurable function of $(\mYtil_{1},\mYtil_2, \dots, \mYtil_{N})$, 
\begin{equation} \label{eq:independent-lower-bound}
\inf_{\vetahat} \E_{\pi_\Theta} \E_{\veta} \|\veta - \vetahat\|_{\rmH}
\geq \inf_{\vetahat \in [0,1]^n} \E_{\pi} \E_{\veta} \|\vetahat - \veta\|_\rmH
	- 4m'\exp\left(-\frac{(m-m')^2}{2m}\right).
\end{equation}
\end{lemma}
\begin{proof}
The result follows from a minor adaptation of Theorem 1.1 in \cite{butucea2018hamming}. 
Details are omitted for the sake of space.
\end{proof}

\begin{lemma} \label{lem:likelihood:second-moment}
Consider $\vy = \vtheta + \vw$, where $\vw \sim \calN(0, \sigma^2 \mI_k)$ and we test the two hypotheses 
\begin{equation*}
    H_0: \vtheta = \mathbf{0} \quad \text{ v.s. } \quad
    H_1: \vtheta \sim \Unif\left(b \bbS^{k-1}\right).
\end{equation*}
Let $f_i$ be the probability density function of $\vy$ under hypothesis $H_i$ for $i=0,1$ and let $L(\vy)=f_1(\vy) / f_0(\vy)$ denote the likelihood ratio. Then for some universal constant $C$, 
\begin{equation*}
\E_0 L^2 \leq \exp \left\{\frac{C b^4}{k \sigma^4}\right\}.
\end{equation*}
\end{lemma}
\begin{proof}
The likelihood ratio is given by
\begin{equation*} \begin{aligned}
L &= \frac{f_1(\vy)}{f_0(\vy)}
= \frac{\E_{\vtheta \sim \Unif(b\bbS^{k-1})}
	\exp \left\{-\frac{1}{2 \sigma^2}\left\|\vy-\vtheta \right\|^2\right\}}
	{\exp \left\{-\frac{1}{2 \sigma^2}\|\vy\|^2\right\}} \\
&= \E_{\vtheta \sim \Unif(b\bbS^{k-1})}
\exp \left\{-\frac{1}{2 \sigma^2}\left\|\vtheta \right\|^2 
		+ \frac{1}{\sigma^2} \vy^\top\vtheta \right\}\\
&= \exp \left\{-\frac{b^2}{2 \sigma^2}\right\}
	\E_{\vtheta \sim \Unif(b\bbS^{k-1})} 
		\exp\left\{\frac{1}{\sigma^2} \vy^\top\vtheta\right\}.
\end{aligned} \end{equation*}
To compute the second moment of $L$ under the null hypothesis, we have
\begin{equation*} \begin{aligned}
\E_0 L^2 &= \E_0\left[ \exp \left\{-\frac{b^2}{ \sigma^2}\right\}\E_{\vtheta \sim \Unif(b\bbS^{k-1})} \left(\exp \left\{\frac{1}{\sigma^2} \vy^\top\vtheta\right\}\right) \E_{\tilde\vtheta \sim \Unif(b\bbS^{k-1})} \left(\exp \left\{\frac{1}{\sigma^2} \vy^\top\tilde\vtheta\right\}\right) \right]\\
    &= \exp \left\{-\frac{b^2}{ \sigma^2}\right\} \E_0 \left[ 
 \E_{\vtheta, \tilde{\vtheta} \sim \Unif(b\bbS^{k-1})} \exp \left\{\frac{1}{\sigma^2} \vy^\top(\vtheta+\tilde\vtheta)\right\}\right].
\end{aligned}
\end{equation*}
By Fubini's Theorem and properties of the Gaussian moment generating function, it follows that
\begin{equation*}
    \begin{aligned}
        \E_0 L^2 &= \exp \left\{-\frac{b^2}{ \sigma^2}\right\} \E_{\vtheta, \tilde{\vtheta} \sim \Unif(b\bbS^{k-1})} \left[ 
   \E_0 \exp \left\{\frac{1}{\sigma^2} \vy^\top(\vtheta+\tilde\vtheta)\right\}\right]\\
   &= \exp \left\{-\frac{b^2}{ \sigma^2}\right\} \E_{\vtheta, \tilde{\vtheta} \sim \Unif(b\bbS^{k-1})} \left[ \exp\left\{ 
\frac{\|\vtheta+\tilde{\vtheta}\|_2^2}{2\sigma^2} \right\} \right]\\
&= \exp \left\{-\frac{b^2}{ \sigma^2}\right\} \E_{\vtheta, \tilde{\vtheta} \sim \Unif(b\bbS^{k-1})} \left[ \exp\left\{ \frac{b^2}{\sigma^2} + 
\frac{\vtheta^\top\tilde{\vtheta}}{\sigma^2} \right\} \right]\\
&= \E_{\vtheta, \tilde{\vtheta} \sim \Unif(b\bbS^{k-1})} \left[ \exp\left\{
\frac{\vtheta^\top\tilde{\vtheta}}{\sigma^2} \right\} \right] \\
&= \E_{\vv, \tilde{\vv} \sim \Unif(\bbS^{k-1})} \left[ \exp\left\{
\frac{b^2\vv^\top\tilde{\vv}}{\sigma^2} \right\} \right].
\end{aligned} \end{equation*}
Denote $\vx = \sqrt{k} \vv$, where $\vv \sim \Unif(\bbS^{k-1})$. 
By Theorem 3.4.6 in \cite{vershynin2018HDP}, $\vx \sim \Unif(\sqrt{k}\bbS^{k-1})$ is a sub-Gaussian random vector with sub-Gaussian norm bounded by a constant.
Hence, fixing $\tilde{\vv} \in \bbS^{k-1}$, by the property of the moment generating functions of sub-Gaussian random variables, we have
\begin{equation*} \begin{aligned}
\E_{\vx \sim \Unif(\sqrt{k}\bbS^{k-1})} \left[ 
	\exp\left\{ \frac{b^2\vx^\top \tilde{\vv} }{\sqrt{k} \sigma^2} \right\} 
	\right]
&\leq \exp\left\{\frac{b^4\|\tilde\vv\|^2_2C^2}{2k\sigma^4}\right\}
= \exp\left\{\frac{C b^4}{k\sigma^4}\right\},
\end{aligned}
\end{equation*}
where the last equality follows from the fact that $\|\tilde\vv\|^2_2 = 1$ since $\tilde\vv \in \bbS^{k-1}$.
Noting the right hand side does not depend on $\tilde\vv$, we have 
\begin{equation*}
    \E_0 L^2 \leq \exp\left\{\frac{C b^4}{k\sigma^4}\right\} ,
\end{equation*}
as we set out to show.
\end{proof}

\begin{lemma} \label{lem:likelihood:second-moment:beta}
Consider $\vy = \vtheta + \vw$, where $\vw \sim \calN(0, \sigma^2 \mI_k)$ and we test the two hypotheses 
\begin{equation*}
    H_0: \vtheta =\mathbf{0} \quad \text{ v.s. } \quad 
    H_1: \vtheta \sim \beta \Unif\left(\{\pm 1\}^k \right).
\end{equation*}
Let $f_i$ be the probability density of $\vy$ under hypothesis $H_i$ for $i=0,1$ and let $L(\vy)=f_1(\vy) / f_0(\vy)$ denote the likelihood ratio. Then for some universal constant $C$, 
\begin{equation*}
\E_0 L^2 \leq \exp \left\{\frac{\beta^4 C^2}{2 k \sigma^4}\right\}.
\end{equation*}
\end{lemma}
\begin{proof}
    Following the same proof of Lemma \ref{lem:likelihood:second-moment}, we have that 
    \begin{equation*}
    \begin{aligned}
        L &= \frac{\E_{\vtheta \sim \beta \Unif(\{\pm 1\}^k)} \exp\{-\frac{1}{2\sigma^2} \|\vy - \vtheta\|^2 \}}{\exp\{-\frac{1}{2\sigma^2} \|\vy\|^2\} }\\
        &= \exp\left\{-\frac{k\beta^2}{2\sigma^2}\right\}\E_{\vtheta \sim \beta \Unif(\{\pm 1\}^k)} \exp\left\{ \frac{1}{\sigma^2} \vy^\top \vtheta \right\}
    \end{aligned}
    \end{equation*}
    and
    \begin{equation*}
    \begin{aligned}
        \E_0 L^2 = \exp\left\{ -\frac{k\beta^2}{\sigma^2} \right\} \E_0\left[ \E_{\vtheta, \tilde{\vtheta} \sim \beta\Unif(\{\pm 1\}^k)} \exp \left\{\frac{1}{\sigma^2} \vy^\top(\vtheta+\tilde\vtheta)\right\} \right].
    \end{aligned}
    \end{equation*}
    Interchanging the order of the expectation and using properties of the Gaussian moment generating function, we have 
    \begin{equation*}
    \begin{aligned}
        \E_0 L^2 
        &= \exp\left\{ -\frac{k\beta^2}{\sigma^2} \right\} \E_{\vtheta, \tilde{\vtheta} \sim \beta\Unif(\{\pm 1\}^k)} \left[ \exp \left\{\frac{\|\vtheta+\tilde{\vtheta}\|_2^2}{2\sigma^2}\right\} \right]\\
        &= \E_{\vtheta, \tilde{\vtheta} \sim \beta\Unif(\{\pm 1\}^k)} \left[ \exp\left\{
\frac{\vtheta^\top\tilde{\vtheta}}{\sigma^2} \right\} \right].
    \end{aligned}
    \end{equation*}
    By Lemma 3.4.2 in \cite{vershynin2018HDP}, $\vtheta$ is a sub-Gaussian vector with sub-Gaussian norm bounded by $C\beta$ for some constant $C > 0$. Fixing $\tilde{\vtheta}$, we have that the sub-Gaussian norm of $\vtheta^\top\tilde{\vtheta}$ is bounded by $C\beta^2\sqrt{k}$. By the property of the moment generating functions of sub-Gaussian random variables, it follows that
    \begin{equation*}
        \E_{\vtheta \sim \beta\Unif(\{\pm 1\}^k)} \left[ \exp\left\{
\frac{\vtheta^\top\tilde{\vtheta}}{\sigma^2} \right\} \right] \leq \exp\left(\frac{C\beta^4 k}{\sigma^4}\right).
    \end{equation*}
    Noting that the right hand side does not depend on $\tilde{\vtheta}$, we have 
    \begin{equation*}
        \E_0 L^2 \leq \exp \left(\frac{C\beta^4 k}{\sigma^4}\right),
    \end{equation*}
completing the proof.
\end{proof}

\begin{lemma}\label{lem:minimax:risk}
Consider a random vector $\vy \in \R^k$ and a random variable $\eta \in \{0,1\}$ such that $\eta \sim \pi_0$, where $\pi_0$ is a Bernoulli distributions with parameter $m'/n$. 
Assume that when $\eta = 0$, $\vy \sim \calN(0, \sigma^2 \mI_k)$
and when $\eta = 1$, $\vy \eqdist  \vtheta + \vw$, where $\vtheta \sim \Unif(b\bbS^{k-1})$ and $\vw \sim \calN(0, \sigma^2 \mI_k)$.
Suppose that 
\begin{equation} \label{eq:x:in:range}
    \frac{4m'}{n}-\frac{{m^{\prime}}^2}{n^2}
\exp\left\{ \frac{b^4 C^2}{k \sigma^4}\right\} -\frac{3 {m^{\prime}}^2}{n^2} \in [0, 1],
\end{equation}
then
\begin{equation*}
\inf_{T \in [0,1]} \E_{\pi_0}\E_{\eta}\left[\left|T(\vy) - \eta\right| \right]
\geq \frac{m'}{n}-\frac{{m^{\prime}}^2}{4 n^2}
\exp\left\{ \frac{b^4 C^2}{k \sigma^4}\right\} -\frac{3 {m^{\prime}}^2}{4 n^2},
\end{equation*}
where the infimum is taken over all test functions $T(\vy) \in [0,1]$ that is a
measurable function of $\vy$.
\end{lemma}
\begin{proof}
By the definition of $\pi_0$, we can write the risk of $T$ as
\begin{equation*} \begin{aligned}
R(T) = 
    \E_{\pi_0} \E_\eta |T(\vy) - \eta|  &= 
    \frac{m'}{n} \E_1(1 - T(\vy)) + \frac{n - m'}{n} \E_0 T(\vy) .
\end{aligned} \end{equation*}
Consider the simple versus simple hypothesis testing problem
\begin{equation*}
H_0 : \vtheta = \vo \quad \text{ vs. } \quad 
H_1: \vtheta \sim \Unif (b \bbS^{k-1}),
\end{equation*}
with a prior distribution $\Pr[H_0] = 1 - m'/n$ and $\Pr[H_1] = m'/n$.
From the Bayesian version of the Neyman-Pearson Lemma, the optimal test $T^\star$ is given by 
\begin{equation*}
T^\star(\vy)
= \indic \left\{
	\frac{(m'/n) f_1(\vy)}{\left(1 - m'/n \right) f_0(\vy)} > 1
	\right\}, 
\end{equation*}
where $f_i$ is the probability density function of $\vy$ under hypothesis $H_i$ for $i = 0, 1$. 
Let $\alpha = m'/n$ and denote the likelihood ratio as $L(\vy) =  f_1(\vy) / f_0(\vy)$.
Then the minimum risk can be written as 
\begin{equation*} \begin{aligned}
    R(T^\star) &= \alpha + (1 - \alpha) \E_0 T^\star - \alpha \E_1 T^\star\\
    &= \alpha - \int_{\frac{\alpha}{1-\alpha}L > 1} [\alpha f_1 - (1-\alpha)f_0] d\vy\\
    &= \alpha - \int_{\frac{\alpha}{1-\alpha}L > 1} \frac{\alpha f_1 - (1-\alpha)f_0}{(1-\alpha) f_0} \cdot (1 - \alpha) f_0 d\vy\\
    &= \alpha - \int_{\frac{\alpha}{1-\alpha}L > 1} \left(\alpha L - (1 - \alpha)\right) \cdot f_0 d\vy
\end{aligned} \end{equation*}
Noting that 
\begin{equation*} \begin{aligned}
\int \alpha f_1  - (1-\alpha)f_0 d\vy
&= \int_{\frac{\alpha}{1-\alpha}L > 1} \alpha f_1 - (1-\alpha)f_0  d\vy + \int_{\frac{\alpha}{1-\alpha}L < 1} \alpha f_1 - (1-\alpha)f_0 d\vy\\
    &= 2\alpha - 1,
\end{aligned} \end{equation*}
and 
\begin{equation*}
\begin{aligned}
    \E_0 \left|\alpha L - (1 - \alpha)\right| &= \int_{\frac{\alpha}{1-\alpha}L > 1} \alpha f_1 - (1-\alpha)f_0 d\vy - \int_{\frac{\alpha}{1-\alpha}L < 1} \alpha f_1 - (1-\alpha)f_0 d\vy,
\end{aligned}
\end{equation*}
it follows that 
\begin{equation*} \begin{aligned}
R(T^\star)
&= \alpha - \frac{1}{2}(2\alpha - 1 
	+ \E_0 \left|\alpha L - (1 - \alpha)\right|)\\
&=\frac{1}{2} - \frac{1}{2}\E_0 \left|\alpha L - (1 - \alpha)\right|. 
\end{aligned} \end{equation*}
Applying Jensen's inequality,
\begin{equation} \label{eq:minimax:middle-step}
\begin{aligned}
R(T^\star)
&\geq \frac{1}{2}
	- \frac{1}{2} \sqrt{\E_0 \left(\alpha L - (1 - \alpha)\right)^2}\\
&= \frac{1}{2} - \frac{1}{2} \sqrt{\alpha^2 \E_0 L^2 
	- 2(1-\alpha) \alpha \E_0 L + (1-\alpha)^2}.
\end{aligned} \end{equation}
Combining Equation~\eqref{eq:minimax:middle-step} with Lemma \ref{lem:likelihood:second-moment}, we have
\begin{equation*}
\begin{aligned}
R(T^\star)
&\geq \frac{1}{2}
  - \frac{1}{2} \sqrt{\alpha^2 \exp\left( \frac{b^4 C^2}{k\sigma^4}\right) 
  - 2(1-\alpha)\alpha + (1-\alpha)^2}\\
&= \frac{1}{2}
  - \frac{1}{2}\sqrt{1 + \frac{{m^\prime}^2}{n^2}
		\exp\left\{ \frac{b^4 C^2}{k\sigma^4}\right\}
  - \frac{4{m^\prime}}{n} + \frac{3{m^\prime}^2}{n^2}}
\end{aligned} \end{equation*}
Since $1 - \sqrt{1 - x} \geq x/2$ for $x \in [0,1]$, under Equation~\eqref{eq:x:in:range} 
we have
\begin{equation*}
\inf_{T \in [0,1]} R(T)
\ge R(T^\star)
\geq \frac{1}{4} \left(\frac{4m'}{n}
	- \frac{{m^\prime}^2}{n^2} \exp\left( \frac{b^4 C^2}{k\sigma^4}\right) 
	- \frac{3{m^\prime}^2}{n^2}\right),
\end{equation*}
which completes the proof.
\end{proof}

\begin{lemma}\label{lem:minimax:risk:beta}
Consider a random vector $\vy \in \R^k$ and a random variable $\eta \in \{0,1\}$ such that $\eta \sim \pi_0$, where $\pi_0$ is a Bernoulli distributions with parameter $m'/n$.
Assume that we have $\vy \sim \calN(0, \sigma^2 \mI_k)$ when $\eta = 0$ and when $\eta = 1$ we have $\vy \eqdist \vtheta + \vw$, where $\vtheta \sim \beta\Unif(\{\pm 1\}^k)$ and $\vw \sim \calN(0, \sigma^2 \mI_k)$.
Suppose that 
\begin{equation*}
    \frac{4m'}{n}
	-\frac{{m^{\prime}}^2}{n^2}
		\exp \left\{ \frac{C \beta^4 k}{\sigma^4}\right\}
	-\frac{3 {m^{\prime}}^2}{n^2} \in [0,1], 
\end{equation*}
then
\begin{equation*}
\inf_{T\in [0,1]} \E_{\pi_0}\E_{\eta} \left[ \left|T(\vy) - \eta\right| \right]
\geq \frac{m'}{n}
	-\frac{{m^{\prime}}^2}{4 n^2}
		\exp \left\{ \frac{C \beta^4 k}{\sigma^4}\right\}
	-\frac{3 {m^{\prime}}^2}{4 n^2}.  
\end{equation*}
\end{lemma}
\begin{proof}
The result follows from the same argument as in the proof of Lemma~\ref{lem:minimax:risk}, and thus we omit the details.
\end{proof}

\begin{lemma} \label{lem:exp-expect-binom}
For $a > 0$ and a random variable $k$ such that $k-1 \sim \Binom(n-1, m'/n)$ for $3m' < n$, it follows that for $n$ sufficiently large,
\begin{equation*}
\E e^{\frac{a}{k}}
\leq \exp \left\{a-C n \log \left(\frac{n}{3 m^{\prime}}\right)\right\} 
	+ \exp \left\{ \frac{a}{\log n}-C n\right\}
	+ e^{\frac{C a}{n}},
\end{equation*}
where $C > 0$ is a universal constant that does not depend on $a, n$ or $m'$. 
\end{lemma}
\begin{proof} 
Letting $\mu_k = \E (k-1)$, we have that 
\begin{equation} \label{eq:exp-expect-upper-1}
\begin{aligned}
\E\left[ e^\frac{a}{k} \right]
&= \E\left[ \exp\left\{ \frac{a}{k}\right\} \indic \{k-1 \leq \log n\}\right]
 + \E\left[ \exp\left\{\frac{a}{k}\right\} 
		\indic\{\log n \leq k-1 \leq (1-\delta) \mu_k\}\right]\\
    &~~~~~~ + \E\left[ \exp\left\{ \frac{a}{k}\right\}
		\indic\{k-1 \geq (1-\delta) \mu_k\}\right]\\
&\leq e^a \underbrace{\Pr\left(k-1 \leq \log n\right)}_{P_1} 
	+ e^{\frac{a}{1+\log n}}
	\underbrace{\Pr\left(k-1 \leq (1-\delta) \mu_k\right)}_{P_2}
	+ e^{\frac{a}{(1-\delta) \mu_k + 1}}.
\end{aligned} \end{equation}

To bound $P_1$, let $X = \sum_{j\neq i} \eta_j$ and observe that $X \sim \Binom(n-1, m'/n)$.
It follows that 
\begin{equation*}
\begin{aligned}
    P_1 &= \Pr\left(X \geq n-1-\log n\right)
    \leq e^{-\frac{n-1}{n} m'} \left(\frac{e (n-1)m'/n }{n - 1 -\log n}\right)^{n-1-\log n},
\end{aligned} \end{equation*}
where the inequality follows from Chernoff's inequality for large deviations \citep[see Theorem 2.3.1 in][]{vershynin2018HDP}.
For suitably large $n$, we have $n-1-\log n > en/3$, and it follows that
\begin{equation} \label{eq:P1-upper-2}
\begin{aligned}
P_1
&\leq e^{-\frac{n-1}{n} m'} \left(\frac{e m'}{n - 1 -\log n}\right)^{n-1-\log n}
\leq \left(\frac{3m'}{n}\right)^{2n/3} = \exp\left\{- \frac{2n}{3}\log\left(\frac{n}{3m'}\right)\right\}.
\end{aligned}
\end{equation}

To bound $P_2$, since $k-1 \sim \Binom(n-1, 1-m'/n)$, Chernoff's inequality for small deviations \citep[see Exercise 2.3.5 in][]{vershynin2018HDP} yields that
\begin{equation} \label{eq:P2-upper}
\begin{aligned}
P_2
&\leq e^{-\frac{\mu_k \delta^2}{2}} 
= \exp\left\{-\frac{\delta^2}{2} \left(n-1 - \frac{(n-1)m'}{n}\right)\right\} \leq e^{-\delta^2 n / 6},
\end{aligned} \end{equation}
where the second inequality holds for $n > 2$ by our assumption that $m' < n/3$.

Applying Equations~\eqref{eq:P1-upper-2} and \eqref{eq:P2-upper} to Equation~\eqref{eq:exp-expect-upper-1} and taking $\delta = 1/2$,
\begin{equation*} \begin{aligned}
\E\left[ e^\frac{a}{k} \right]
&\leq \exp\left\{a - \frac{2n}{3} \log \left(\frac{n}{3m'}\right)\right\} 
    + \exp\left\{\frac{a}{1+\log n} -Cn\right\} + e^{\frac{2a}{n-(n-1)m'/n}} \\
&\leq \exp\left\{a - Cn \log \left(\frac{n}{3m'}\right)\right\}
	+ \exp\left\{ \frac{a}{\log n} - Cn\right\} + e^{\frac{Ca}{n}},
\end{aligned} \end{equation*}
where $C > 0$ is a universal constant that does not depend on $n$ and $m'$ as long as $n$ is sufficiently large (and again using our assumption that $n > 3m'$).
\end{proof}

\section{Localized Error Bound for Initial Estimator} \label{sec:localized}
We begin by stating a refined bound on the row-wise estimation error of $\mUstar$.
For any matrix $\mA$ with singular value decomposition $\mA = \mX_1 \mSig \mX_2^\top$, we define the generalized sign matrix as $\sgn{\mA} := \mX_1 \mX_2^\top$.
Adapting the proof of Theorem 4.2 in \cite{chen2021spectral}, we obtain the following refined bound on the row-wise estimation error of $\mUstar$.
\begin{theorem} \label{thm:ustar-row}
    Suppose that $\mYz$ is given by Equation~\eqref{eq:base-model}, that is, $\mYz = \mMstar + \mWz$, where the noise matrix $\mW^{(0)}$ satisfies Assumption~\ref{assump:noise} and Equation~\eqref{eq:L-sigma}. 
    Assume that 
    \begin{equation} \label{eq:lambdastar_r:assumption}
        |\lambdastar_r| \geq C \sigma \kappa \sqrt{n \log n }
    \end{equation} 
    for some sufficiently large constant $C$, and that the condition number $\kappa = O(1)$. 
    Then with probability at least $1 - O(n^{-6})$, the following bound holds uniformly for all $\ell \in [n]$:
    \begin{equation*}
        \left\|\mUz_{\ell,\cdot}\sgn{\mOfrak} - \mU^\star_{\ell,\cdot}\right\|_2 \lesssim \frac{\sigma \kappa \sqrt{n} \|\mU^\star_{\ell,\cdot}\|_2}{|\lambdastar_r|} + \frac{\sigma \sqrt{r \log n}}{|\lambdastar_r|} 
    \end{equation*}
    where 
    \begin{equation} \label{eq:mOfrak:def}
        \mOfrak := \mUzt \mUstar. 
    \end{equation}
\end{theorem}

\subsection{Proof of Theorem~\ref{thm:ustar-row}}

\begin{proof}
For any $\ell \in [n]$, we have %
\begin{equation*} \begin{aligned}
\left\|\mUz_{\ell,\cdot}\sgn{\mOfrak} - \mU^\star_{\ell,\cdot}\right\|_2
&\leq \left\|\mUz_{\ell,\cdot}\mOfrak - \mU^\star_{\ell,\cdot}\right\|_{2} 
  + \left\|\mUz_{\ell,\cdot}\mOfrak - \mUz_{\ell,\cdot} \sgn{\mOfrak}\right\|_2.
\end{aligned} \end{equation*}
Without loss of generality, we assume that $\lambdastar_r > 0$. 
Using Lemma~\ref{lem:UHUstar} to bound the first term,
\begin{equation}\label{eq:U-control:eq0}
\left\|\mUz_{\ell,\cdot}\sgn{\mOfrak} - \mU^\star_{\ell,\cdot}\right\|_2
\le
\left\|\mUz_{\ell,\cdot}\mOfrak - \mUz_{\ell,\cdot} \sgn{\mOfrak}\right\|_2
+
\frac{\sigma\kappa\sqrt{n}}{\lambdastar_r}\left\|\mU^\star_{\ell,\cdot}\right\|_2
	+ \frac{\sigma \sqrt{r \log n}}{\lambdastar_r} .
\end{equation}

Applying submultiplicativity of the norm followed by Lemma~\ref{lem:mOfrak:op},
\begin{equation*} \begin{aligned}
\|\mUz_{\ell, \cdot}\mOfrak - \mUz_{\ell, \cdot} \sgn{\mOfrak} \|_{2}
&\leq \|\mUz_{\ell,\cdot}\|_{2} \|\mOfrak - \sgn{\mOfrak}\|
\leq \frac{2\left\|\mWz\right\|^2}{(\lambdastar_r)^2}
	\left\|\mUz_{\ell,\cdot}\right\|_2.
\end{aligned} \end{equation*}
Since
\begin{equation*}
    \left\|\mUz_{\ell,\cdot}\right\|_{2} \leq \left\|\mUz_{\ell,\cdot} \mOfrak\right\|_{2} \|\mOfrak^{-1}\| \leq 2 \left\|\mUz_{\ell,\cdot} \mOfrak\right\|_{2},
\end{equation*}
where the last inequality follows from Lemma~\ref{lem:mOfrak:op}. 
By the above display and the triangle inequality, 
\begin{equation*}
\begin{aligned}
\left\|\mUz_{\ell, \cdot}\mOfrak - \mUz_{\ell, \cdot} \sgn{\mOfrak} \right\|_{2}
&\leq \frac{4\left\|\mWz\right\|^2}{(\lambdastar_r)^2} \left\|\mUz_{\ell,\cdot}\mOfrak\right\|_2 \\
&\leq \frac{4\left\|\mWz\right\|^2}{(\lambdastar_r)^2}
	\left( \|\mUz_{\ell,\cdot}\mOfrak - \mU^\star_{\ell, \cdot}\|_2 
		+ \|\mU^{\star}_{\ell,\cdot}\|_2 \right) \\
&\lesssim \frac{\sigma^2 n}{(\lambdastar_r)^2}
\left( \left\|\mUz_{\ell,\cdot}\mOfrak - \mU^\star_{\ell, \cdot}\right\|_2 
	+ \left\|\mU^\star_{\ell,\cdot}\right\|_2 \right) ,
\end{aligned} \end{equation*}
where the last inequality follows from $\|\mWz\| \leq C\sigma \sqrt{n}$ holds with probability at least $1 - O(n^{-6})$ for some constant $C > 0$ (see for example, Theorem 3.4 in \cite{chen2021spectral}).
Thus, applying the above bound to Equation~\eqref{eq:U-control:eq0} under the assumption in Equation~\eqref{eq:lambdastar_r:assumption} gives
\begin{equation*} \begin{aligned}
\left\|\mUz_{\ell,\cdot}\sgn{\mOfrak} - \mU^\star_{\ell,\cdot}\right\|_2
&\leq
2 \left\|\mUz_{\ell,\cdot}\mOfrak - \mU^\star_{\ell, \cdot}\right\|_2 
+ \frac{\sigma^2 n}{(\lambdastar_r)^2} \left\|\mU^\star_{\ell,\cdot}\right\|_2 + \frac{\sigma \sqrt{r \log n}}{\lambdastar_r} \\
&\lesssim \frac{\sigma \kappa \sqrt{n} \|\mU^\star_{\ell,\cdot}\|_2}{\lambdastar_r} + \frac{\sigma \sqrt{r \log n}}{\lambdastar_r} 
\end{aligned} \end{equation*}
where the last inequality follows from Lemma~\ref{lem:UHUstar}. 
\end{proof}

\begin{lemma}\label{lem:mOfrak:op}
Under the same assumptions as in Theorem~\ref{thm:ustar-row},
if $|\lambdastar_r| > 2\|\mWz\|$, then
\begin{equation*}
\left\|\mOfrak - \sgn{\mOfrak}\right\|
\leq \frac{2\|\mWz\|^2}{(\lambdastar_r)^2}
\quad \text{and} \quad
\|\mOfrak^{-1}\| \leq 2, 
\end{equation*}
where $\mOfrak = \mUzt \mUstar$ is the matrix defined in Theorem~\ref{thm:ustar-row}.
The same bounds hold when we replace $\mOfrak$ by its leave-one-out version $\mOfrak^{(\ell)}$, for any $\ell \in [n]$.
\end{lemma}
\begin{proof}
Denote the SVD of $\mOfrak$ as $\mOfrak = \mX \left(\cos \mTheta\right) \mY^\top$ and define the matrix $\sgn{\mOfrak} = \mX \mY^\top$.
We write the singular values of $\mOfrak$ as a diagonal matrix $\cos(\mTheta) = \diag(\cos(\theta_1), \cdots, \cos(\theta_r))$, since the singular values of $\mOfrak$ are the cosines of the principal angles between the subspaces spanned by the columns of $\mUz$ and $\mUstar$.
Then we have
\begin{equation*} \begin{aligned}
\|\mOfrak - \sgn{\mOfrak}\|
&= \left\|\mX (\cos \mTheta - \mI) \mY^\top \right\|
\leq \| \mI - \cos \mTheta\| \\
&\leq \|\mI - \cos^2 \mTheta\| = \|\sin^2 \mTheta \| \\
&\leq \frac{2\left\|\mWz\right\|^2} {(\lambdastar_r)^2},
\end{aligned} \end{equation*}
where 
the second inequality follows from the fact that $1 - \cos \theta_i \geq 1 - \cos^2 \theta_i$ for all $i \in [r]$, and the last inequality follows from the Davis-Kahan $\sin \Theta$ theorem \citep[see Theorem 2.7 in][for a reference]{chen2021spectral}.
Thus, we have
\begin{equation*}
\sigma_{\min}(\mOfrak)
\geq \|\sgn{\mOfrak}\| - \|\mOfrak - \sgn{\mOfrak}\|
\geq 1 - \frac{2\|\mWz\|^2}{(\lambdastar_r)^2}
\geq 1/2,
\end{equation*}
as long as $|\lambda_r^\star| > 2 \|\mWz\|$.
Noting that
\begin{equation*}
    \|\mOfrak^{-1}\| = \frac{1}{\sigma_{\min}(\mOfrak)} \leq 2
\end{equation*}
completes the proof.
\end{proof}

With the help of Theorem~\ref{thm:ustar-row}, we can now prove the main result of this section, which provides a bound on the row-wise estimation error of $\mMz$ relative to $\mMstar$.
Some supporting technical results are deferred to Section~\ref{subsec:localized:technical} below.

\subsection{Proof of Theorem~\ref{thm:delta_row}}

\begin{proof}
Recall from Equation~\eqref{eq:mOfrak:def} that $\mOfrak := \mUzt \mUstar$. 
For notational simplicity, we will drop the superscript $(0)$ in $\mUz$ when it is clear from context in this section.
We caution the reader not to confuse this with the matrix $\mU$ that appears in Section~\ref{sec:common-structure}, which serves a distinct role.

By the triangle inequality, we have 
\begin{equation} \label{eq:MMstarl}
\left\|\mMz_{\ell, \cdot} - \mMstar_{\ell,\cdot}\right\|_2
\leq
\underbrace{\left\|\mU_{\ell, \cdot} \left(\mLambdaz - \mOfrak \mLambdastar \mOfrak^\top \right) \mU^\top \right\|_2}_{\gamma_1}
+ \underbrace{\left\| \mU_{\ell, \cdot} \mOfrak \mLambdastar \mOfrak^\top \mU^\top - \mU^{\star}_{\ell, \cdot} \mLambdastar \mUstart\right\|_2}_{\gamma_2}. 
\end{equation}
and, similarly, for any $k \in [n]$, 
\begin{equation} \label{eq:MMstarl:infty}
\left|\mMz_{\ell k} - \mMstar_{\ell k}\right|
\leq \underbrace{\left|\mU_{\ell, \cdot} \left(\mLambdaz - \mOfrak \mLambdastar \mOfrak^\top \right) \mU_{k, \cdot}^\top \right|}_{\gamma^{(\ell, k)}_1}
+ \underbrace{\left| \mU_{\ell, \cdot} \mOfrak \mLambdastar \mOfrak^\top \mU^\top_{k, \cdot} - \mU^{\star}_{\ell, \cdot} \mLambdastar \mUstart_{k, \cdot}\right|}_{\gamma^{(\ell, k)}_2}. 
\end{equation}

Throughout the proof that follows, we will repeatedly use the inequalities
\begin{equation*}
    \left|\vx^\top \mA \vy\right| \leq \left\|\vx\right\|_2 \left\|\mA \vy\right\|_2 \leq \left\|\vx\right\|_2 \left\|\vy\right\|_2 \left\|\mA\right\|
\end{equation*}
and 
\begin{equation*}
    \|\mA \mB\| \leq \|\mA\|\|\mB\|
\end{equation*}
without explicit referencing, where $\vx, \vy$ are real vectors and $\mA, \mB$ are real matrices and they have compatible dimensions. 

\paragraph{Step 1: Bounding $\gamma_1$.}
By definition of $\gamma_1$ above, we have 
\begin{equation} \label{eq:gamma1:bound1}
\gamma_1
\leq
\left\|\mU_{\ell,\cdot}\right\|_2 \left\| \mLambdaz - \mOfrak \mLambdastar \mOfrak^\top \right\| \left\| \mU \right\|
\leq \left\|\mU_{\ell,\cdot}\right\|_2 \left\| \mLambdaz - \mOfrak \mLambdastar \mOfrak^\top \right\|. 
\end{equation}
and, similarly for $\gamma^{(\ell,k)}_1$,
writing $\nu_k = \left\|\mU_{k,\cdot}\right\|_2$ for $k \in [n]$,
\begin{equation} \label{eq:gamma1l:bound1}
\gamma^{(\ell, k)}_1
\leq
\left\|\mU_{\ell,\cdot}\right\|_2 \left\|\mU_{k,\cdot}\right\|_2 \left\| \mLambdaz - \mOfrak \mLambdastar \mOfrak^\top \right\|
= \nu_{\ell} \nu_k \left\| \mLambdaz - \mOfrak \mLambdastar \mOfrak^\top \right\| .
\end{equation}
By Equation (4.139) in \cite{chen2021spectral}, 
\begin{equation*}
\left\|\mOfrak \mLambdastar \mOfrak^\top - \mLambdaz\right\|
\lesssim \frac{\kappa \sigma^2 n}{\lambda_r^{\star}} + \sigma \sqrt{r \log n}
\end{equation*}
holds with probability at least $1 - O(n^{-6})$ by choosing sufficiently large constants in the bound. 
Substituting the above bound into Equations~\eqref{eq:gamma1:bound1} and~\eqref{eq:gamma1l:bound1} yields that 
\begin{equation} \label{eq:gamma1:bound2}
\begin{aligned}
\gamma^{(2)}_1
&\lesssim \nu_\ell \left(\frac{\kappa \sigma^2 n}{\lambdastar_r}
	+ \sigma \sqrt{r \log n}\right) \\
&\leq
\left(\frac{\kappa \sigma^2 n}{\lambda_r^{\star}} 
	+ \sigma \sqrt{r \log n}\right)
\left(\left\|(\mU \mOfrak - \mUstar)_\ld\right\|_2 
	+ \left\|\mUstarl\right\|_2\right)
\end{aligned}
\end{equation}
and
\begin{equation*}
\begin{aligned}
    \gamma^{(\ell, k)}_1 &\lesssim \nu_\ell \nu_k \left(\frac{\kappa \sigma^2 n}{\lambda_r^{\star}} + \sigma \sqrt{r \log n}\right) \\
    &\leq  \left(\frac{\kappa \sigma^2 n}{\lambda_r^{\star}} + \sigma \sqrt{r \log n}\right) \left(\left\|(\mU \mOfrak - \mUstar)_\ld\right\|_2 + \left\|\mUstarl\right\|_2\right) \left(\left\|(\mU \mOfrak - \mUstar)_{k,\cdot}\right\|_2 + \left\|\mU^\star_{k,\cdot}\right\|_2\right).
\end{aligned} \end{equation*}
Let 
\begin{equation*}
    \delta_k = \left\|(\mU \mOfrak - \mUstar)_{k,\cdot}\right\|_2 \text{ and } \nustar_k = \left\|\mU^\star_{k,\cdot}\right\|_2 \text{ for all } k \in [n]. 
\end{equation*}
Then we can rewrite the above bound as 
\begin{equation} \label{eq:gamma1l:bound2}
\begin{aligned}
    \gamma^{(\ell, k)}_1 &\lesssim \left(\frac{\kappa \sigma^2 n}{\lambda_r^{\star}} + \sigma \sqrt{r \log n}\right) \left(\delta_{\ell} + \nustar_\ell\right) \left(\delta_k + \nustar_k\right).
\end{aligned} \end{equation}

\paragraph{Step 2: Bounding $\gamma_2$.}
Since the difference $\mU \mOfrak \mLambdastar \mOfrak^\top \mU - \mUstar \mLambdastar \mUstart$ can be decomposed into
\begin{equation*}
\begin{aligned}
     \left(\mU\mOfrak - \mUstar\right) \mLambdastar \mUstart + \left(\mU\mOfrak - \mUstar\right) \mLambdastar \left(\mU \mOfrak - \mUstar \right)^\top + \mUstar \mLambdastar \left(\mU \mOfrak - \mUstar \right)^\top.
\end{aligned}
\end{equation*}
By triangle inequality, one can upper bound $\gamma_2$ by
\begin{equation*}
    \left\|\left(\mU\mOfrak - \mUstar\right)_{\ell, \cdot} \mLambdastar \mUstart\right\|_2 + \left\|\left(\mU\mOfrak - \mUstar\right)_{\ell, \cdot} \mLambdastar \left(\mOfrak^\top \mU^\top - \mUstart \right)\right\|_2 + \left\|\mU^\star_{\ell, \cdot} \mLambdastar \left(\mOfrak^\top \mU^\top - \mUstart\right)\right\|_2,
\end{equation*}
which is further bounded by
\begin{equation*}
    \lambdastar_1  \delta_\ell + \lambdastar_1 \left( \delta_\ell + \nustar_\ell \right) \left\|\mU \mOfrak - \mUstar\right\| \lesssim \lambdastar_1 \left( \delta_\ell + \nustar_\ell \right) \frac{\sigma \sqrt{n}}{\lambdastar_r} + \lambdastar_1  \delta_\ell \lesssim \kappa \sigma \sqrt{n} \nustar_\ell + \lambdastar_1 \delta_\ell.  
\end{equation*}
up to some constant factor.
Similarly, one can bound $\gamma^{(\ell, k)}_2$ by
\begin{equation*}
    \left|\left(\mU\mOfrak - \mUstar\right)_{\ell, \cdot} \mLambdastar \mUstart_\kd\right| + \left|\left(\mU\mOfrak - \mUstar\right)_{\ell, \cdot} \mLambdastar \left(\mU \mOfrak - \mUstar\right)^\top_\kd\right| + \left|\mU^\star_{\ell, \cdot} \mLambdastar \left(\mU \mOfrak - \mUstar\right)^\top_\kd\right|,
\end{equation*}
and it is further bounded by
\begin{equation*}
    \lambdastar_1 \delta_\ell \nustar_k + \lambdastar_1 \delta_\ell \delta_k + \lambdastar_1 \delta_k \nustar_\ell. 
\end{equation*}

It follows that 
\begin{equation} \label{eq:gamma2:bound1}
\begin{aligned}
\gamma_2
&\lesssim \kappa \sigma \sqrt{n} \left\|\mU^\star_{\ell, \cdot}\right\|_2 
+ \lambdastar_1 \left\|\left(\mU\mOfrak - \mUstar\right)_{\ell,\cdot}\right\|_2
\lesssim \left(\sigma \kappa^2 \sqrt{n}\right) \nustar_\ell 
	+ \sigma \kappa \sqrt{r \log n},
\end{aligned}
\end{equation}
where the second inequality follows from Equation~\eqref{eq:UHUstar:eq0} in Lemma~\ref{lem:UHUstar}. Similarly, one has
\begin{equation*}
\gamma^{(\ell, k)}_2
\lesssim \left(\sigma \kappa^2 \sqrt{n}\right) \nustar_k \nustar_{\ell} 
+ \sigma \kappa \sqrt{r \log n} \left(\nustar_k + \nustar_\ell\right)
+ \frac{\sigma \kappa^3 n \nustar_\ell \nustar_k}{\lambdastar_r} 
+ \frac{\sigma^2 \kappa r \log n}{\lambdastar_r}. 
\end{equation*}

\paragraph{Step 3: Putting everything together.}
Combining Equations~\ref{eq:gamma1:bound2} and~\ref{eq:gamma2:bound1} with Equation~\eqref{eq:MMstarl}, we have 
\begin{equation*}
\left\|\mMz_{\ell, \cdot} - \mMstar_{\ell,\cdot}\right\|_2
\lesssim \left(\sigma \kappa^2 \sqrt{n}\right) \left\|\mUstarl\right\|_2 
	+ \sigma \kappa \sqrt{r \log n},
\end{equation*}
as desired. 
\end{proof}

\subsection{Supporting technical results}
\label{subsec:localized:technical}


\begin{lemma} \label{lem:UHUstar}
Under the setting of Theorem~\ref{thm:delta_row}, 
with probability at least $1 - O(n^{-6})$, it holds for all $\ell \in [n]$ that
\begin{equation} \label{eq:UHUstar:eq0}
\left\|\left(\mU\mOfrak - \mUstar\right)_{\ell, \cdot}\right\|_2
\lesssim \frac{\left(\sigma \kappa \sqrt{n}\right) \nustar_\ell}{\lambdastar_r}
	+ \frac{\sigma \sqrt{r \log n}}{\lambdastar_r},
\end{equation}
where $\nustar_\ell = \left\|\mU^\star_{\ell, \cdot}\right\|_2$ and we abbreviate $\mUz$ to $\mU$ for notational simplicity.
\end{lemma}
\begin{proof}
By the triangle inequality,
\begin{equation*} \begin{aligned}
\left\|\left(\mU\mOfrak - \mUstar\right)_{\ell, \cdot}\right\|_2
&= \left\|\mUl \mOfrak - \mMstarl \mUstar
	\left(\mLambdastar\right)^{-1}\right\|_2 \\
&\leq 
\left\|\mUl \mOfrak - \mMzl \mUstar \left(\mLambdastar\right)^{-1}\right\|_2 
+ \left\|\mWz_{\ell,\cdot} \mUstar \left(\mLambdastar\right)^{-1}\right\|_2 .
\end{aligned} \end{equation*}
By standard concentration inequalities \citep{vershynin2018HDP}, with probability at least $1 - O(n^{-6})$, it holds uniformly over all $\ell \in [n]$ that
\begin{equation*}
\left\|\left(\mU\mOfrak - \mUstar\right)_{\ell, \cdot}\right\|_2
\lesssim
\left\|\mUl \mOfrak - \mMzl \mUstar \left(\mLambdastar\right)^{-1}\right\|_2
+ \frac{\sigma \sqrt{r \log n} + L \log n \left\|\mUstar\right\|_{2, \infty}}
	{\lambdastar_r}  .
\end{equation*}
Applying Lemma~\ref{lem:UHM} and the assumptions in Equations~\eqref{eq:L-sigma} and~\eqref{eq:lambdastar_r:assumption},
\begin{equation*}
\begin{aligned}
\left\|\left(\mU\mOfrak - \mUstar\right)_{\ell, \cdot}\right\|_2
&\lesssim \frac{\left(\sigma \kappa \sqrt{n}\right) \left\|\mUstarl\right\|_2}
		{\lambdastar_r}  
+ \frac{\sigma \sqrt{r \log n}}{\lambdastar_r} 
+ \frac{\sigma L \sqrt{n \log^2 n}}{\left(\lambdastar_r\right)^2} 
	\|\mU\mOfrak - \mUstar\|_{2,\infty}. 
\end{aligned} \end{equation*}
Taking maximum over $\ell \in [n]$ on both sides in the above display and rearranging terms, one has 
\begin{equation*}
\|\mU\mOfrak - \mUstar\|_{2,\infty}
\lesssim
\frac{\left(\sigma\kappa\sqrt{n}\right)\left\|\mUstar\right\|_2}{\lambdastar_r}
+ \frac{\sigma \sqrt{r \log n}}{\lambdastar_r},
\end{equation*}
which implies that 
\begin{equation*}
\left\|\left(\mU\mOfrak - \mUstar\right)_{\ell, \cdot}\right\|_2
\lesssim \frac{\left(\sigma \kappa \sqrt{n}\right) \left\|\mUstarl\right\|_2}
		{\lambdastar_r}  
+ \frac{\sigma \sqrt{r \log n}}{\lambdastar_r},
\end{equation*}
as we set out to show.
\end{proof}


\begin{lemma} \label{lem:UHM} 
Under the assumption of Theorem~\ref{thm:delta_row}, with probability at least $1 - O(n^{-6})$, for all $\ell \in [n]$, the difference
$\|\mUl \mOfrak - \mMzl \mU \mOfrak \left(\mLambdastar\right)^{-1}\|_2$ is bounded by
\begin{equation*}
\frac{\sigma \kappa \sqrt{n}}{\lambdastar_r} \left\|\mUstarl\right\|_2 
+ \frac{\sigma^2 \sqrt{r n \log n}}{\left(\lambdastar_r\right)^2} 
+ \frac{\sigma L \sqrt{n \log^2 n}}{\left(\lambdastar_r\right)^2}
	\|\mU\mOfrak - \mUstar\|_{2,\infty}
+ \frac{\sigma^3 n^{3/2}}{(\lambdastar_r)^3}
	\left\|\mUl \mOfrak - \mUstarl\right\|_2 .
\end{equation*}
\end{lemma}
\begin{proof}
Following essentially the same proof as Equations (4.124)-(4.127) in \cite{chen2021spectral}, one arrives at the bound
\begin{equation*} \begin{aligned}
\left\|\mUl \mOfrak - \mMzl \mU \mOfrak \left(\mLambdastar\right)^{-1}\right\|_2
&\lesssim
\frac{\sigma \sqrt{n}}{(\lambdastar_r)^2}
\left(\left\|\mMzl\left(\mU \mOfrak -\mUstar\right)\right\|_{2}
	+ \left\|\mMzl \mUstar\right\|_{2}\right)\\
&\lesssim
\frac{\sigma \sqrt{n}}{(\lambdastar_r)^2}
\left(\left\|\mMzl\left(\mU \mOfrak -\mUstar\right)\right\|_{2} 
	+ \lambdastar_1 \|\mUstarl\|_2 + \left\|\mWzl \mUstar\right\|_{2}
\right) .
\end{aligned} \end{equation*}
Applying standard concentration inequalities \citep{vershynin2018HDP} to control the $\mWzl$ term, with probability at least $1 - O(n^{-6})$, it holds uniformly over all $\ell \in [n]$ that
\begin{equation} \label{eq:UHM:eq1}
\left\|\mUl \mOfrak - \mMzl \mU \mOfrak \left(\mLambdastar\right)^{-1}\right\|_2
\lesssim 
\frac{\sigma \sqrt{n}}{\lambdastar_r} \|\mUstarl\|_2
+ \frac{ \sigma^2  \sqrt{r n \log n} }{ (\lambdastar_r)^2} 
+
\frac{\sigma \sqrt{n}}{(\lambdastar_r)^2} 
\left\|\mMzl\left(\mU \mOfrak -\mUstar\right)\right\|_{2}  .
\end{equation}

By the triangle inequality, one has
\begin{equation} \label{eq:UHM:eq2}
\left\|\mMzl\left(\mU \mOfrak -\mUstar\right)\right\|_{2}
\leq \left\|\mMstarl\left(\mU \mOfrak -\mUstar\right)\right\|_2 
	+ \left\|\mWzl\left(\mU \mOfrak -\mUstar\right)\right\|_2.
\end{equation}
The first term on the right hand side can be bounded directly by
\begin{equation*} \begin{aligned}
\left\|\mMstarl\left(\mU \mOfrak -\mUstar\right)\right\|_2
&= \left\|\mUstarl \mLambdastar \mUstart 
	\left(\mU\mOfrak - \mUstar\right)\right\|_2
\leq 
\left\|\mUstarl\right\|_2 \|\mLambdastar\| \|\mOfrak^\top \mOfrak - \mI\| \\
&\lesssim \frac{\kappa \sigma^2 n}{\lambdastar_r} \left\|\mUstarl\right\|_2,
\end{aligned} 
\end{equation*}
where the last inequality follows from the Davis-Kahan $\sin \Theta$ theorem. 
Applying this bound to Equation~\eqref{eq:UHM:eq2},
\begin{equation*}
\left\|\mMzl\left(\mU \mOfrak -\mUstar\right)\right\|_{2}
\lesssim  \frac{\kappa \sigma^2 n}{\lambdastar_r} \left\|\mUstarl\right\|_2
	+ \left\|\mWzl\left(\mU \mOfrak -\mUstar\right)\right\|_2.
\end{equation*}
Combined the above bound and Lemma~\ref{lem:loo} with Equation~\eqref{eq:UHM:eq2} yields that 
\begin{equation*}
\begin{aligned}
    \left\|\mMzl\left(\mU \mOfrak -\mUstar\right)\right\|_{2} \lesssim & \frac{\kappa \sigma^2 n}{\lambdastar_r} \left\|\mUstarl\right\|_2 + (L\log n) \|\mU\mOfrak - \mUstar\|_{2,\infty} \\
    & + \frac{\sigma^2 \sqrt{r n \log n}}{\lambdastar_r} + \frac{\sigma^2 n}{\lambdastar_r} \left\|\mUl \mOfrak - \mUstarl\right\|_2,
\end{aligned}
\end{equation*}
where we use the fact that $\kappa \geq 1$ to collect terms. 
Taking the above bound into the right hand side of Equation~\eqref{eq:UHM:eq1} gives
\begin{equation*}
    \frac{\sigma \kappa \sqrt{n}}{\lambdastar_r} \left\|\mUstarl\right\|_2 + \frac{\sigma^2 \sqrt{r n \log n}}{\left(\lambdastar_r\right)^2} + \frac{\sigma L\sqrt{n \log^2 n}}{\left(\lambdastar_r\right)^2} \|\mU\mOfrak - \mUstar\|_{2,\infty} + \frac{\sigma^3 n^{3/2}}{(\lambdastar_r)^3}\left\|\mUl \mOfrak - \mUstarl\right\|_2
\end{equation*}
as desired.
\end{proof}


\begin{lemma} \label{lem:loo}
Under the assumption of Theorem~\ref{thm:delta_row}, with probability at least $1 - O(n^{-6})$, for all $\ell \in [n]$,
\begin{equation*}
\left\|\mWzl\left(\mU\mOfrak - \mUstar\right)\right\|_2
\lesssim
(L\log n) \|\mU\mOfrak - \mUstar\|_{2,\infty} 
+ \frac{\sigma^2 \sqrt{r n \log n}}{\lambdastar_r} 
+ \frac{\sigma^2 n}{\lambdastar_r}
	\left(\left\|\mUl\mOfrak-\mUstarl\right\|_2
		+ \left\|\mUstarl\right\|_2 \right) .
\end{equation*}
\end{lemma}
\begin{proof}
By the triangle inequality and basic properties of the norm,
\begin{equation*}
\begin{aligned}
\left\| \mWzl \left(\mU \mOfrak - \mUstar\right) \right\|_2
&\leq \left\| \mWzl \left(\mUll \mOfrakll - \mUstar\right) \right\|_2
+ \left\| \mWzl \right\|_2\left\|\mUll \mOfrakll - \mU \mOfrak\right\|_{\F}.
\end{aligned} \end{equation*}
Controlling the $\mWzl$ term using basic concentration inequalities \citep{vershynin2018HDP}, it holds with high probability that, uniformly over all $\ell \in [n]$,
\begin{equation} \label{eq:loo:eq1}
\begin{aligned}
\left\| \mWzl \left(\mU \mOfrak - \mUstar\right) \right\|_2
    &\leq \left\| \mWzl \left(\mUll \mOfrakll - \mUstar\right) \right\|_2 + \sigma\sqrt{n} \left\|\mUll \mOfrakll - \mU \mOfrak\right\|_{\F}
\end{aligned}
\end{equation}
where the last inequality holds with probability at least $1 - O(n^{-6})$ uniformly over all $\ell \in [n]$ by choosing sufficiently large constants in the bound.
The first term in Equation~\eqref{eq:loo:eq1} can be directly bounded by 
\begin{equation*}
    (L\log n) \left\|\mUll \mOfrakll - \mUstar\right\|_{2,\infty} + \sigma \sqrt{\log n} \left\|\mUll \mOfrakll - \mUstar\right\|_{\F}.
\end{equation*}
Applying triangle inequality, the above is further bounded by
\begin{equation} \label{eq:loo:eq1.5}
\begin{aligned}
     & (L\log n) \left\|\mUll \mOfrakll - \mUstar\right\|_{2,\infty} + \sigma \sqrt{\log n} \left\|\mUll \mOfrakll - \mU \mOfrak\right\|_{\F} + \sigma \sqrt{\log n} \left\|\mUstar - \mU \mOfrak\right\|_{\F}\\
     \lesssim ~& (L\log n) \left(\left\|\mU\mOfrak - \mUstar\right\|_{2,\infty} + \left\|\mUll \mOfrakll - \mU\mOfrak \right\|_{\F} \right) + \frac{\sigma^2 \sqrt{r n \log n}}{\lambdastar_r},
\end{aligned} \end{equation}
where the last inequality follows from the fact that $L\log n > \sigma \sqrt{\log n}$ and triangle inequality.
Since $L\log n \lesssim \sigma \sqrt{n}$ under the assumption in Equation~\eqref{eq:L-sigma}, the main contribution of Equation~\eqref{eq:loo:eq1.5} is 
\begin{equation}\label{eq:loo:eq1.6}
    (L\log n) \left\|\mU\mOfrak - \mUstar\right\|_{2,\infty} + \frac{\sigma^2 \sqrt{r n \log n}}{\lambdastar_r}. 
\end{equation}
It remains to control the second term in Equation~\eqref{eq:loo:eq1}. 
By Equation (4.105) in \cite{chen2021spectral}, we have 
\begin{equation}\label{eq:loo:eq2}
    \left\|\mUll \mOfrakll - \mU \mOfrak\right\|_{\F} \lesssim \frac{\left\| \left(\mMz - \mMzll\right) \mUll\right\|_{\F}}{\lambdastar_r} . 
\end{equation}
By construction of $\mMzll$, we bound the numerator in Equation~\eqref{eq:loo:eq2} by 
\begin{equation*}
    \left(\mMz - \mMzll\right) \mUll = \ve_\ell \mWzl \mUll + \left(\mWz_{\cdot, \ell} - W_{\ell \ell} \ve_l\right) \mUld.
\end{equation*}
Together with the triangle inequality, we have
\begin{equation} \label{eq:loo:eq3}
\begin{aligned}
    \left\|\left(\mMz - \mMzll\right) \mUll\right\|_{\F} &\leq \left\|\mWzl \mUll\right\|_2 + \left\|\mWz_{\cdot, \ell} - W^{(0)}_{\ell\ell} \ve_{\ell} \right\|_2 \left\|\mUld\right\|_2 \\
    &\lesssim \left\|\mWzl \mUll\right\|_2 + \sigma \sqrt{n} \left\|\mUld\right\|_2
\end{aligned}
\end{equation}
We have 
\begin{equation} \label{eq:loo:eq4}
\begin{aligned}
     \left\|\mWzl \mUll\right\|_2 &\lesssim \sigma \sqrt{\log n} \left\|\mUll\right\|_{\F} + L\log n \left\|\mUll\right\|_{2,\infty} \\
     &\lesssim \sigma \sqrt{r \log n}  + L\log n \left(\left\|\mUll \mOfrakll - \mU \mOfrak\right\|_{2,\infty} + \left\|\mU\mOfrak - \mUstar\right\|_{2,\infty} + \left\|\mUstar\right\|_{2,\infty} \right)
\end{aligned}
\end{equation}
where the first inequality holds with probability at least $1 - O(n^{-6})$ uniformly over all $\ell \in [n]$ by choosing sufficiently large constants in the bound.
To control $\left\|\mUld\right\|_2$, we apply Lemma~\ref{lem:mOfrak:op} and use triangle inequality to obtain
\begin{equation*}
    \left\|\mUld\right\|_2
\lesssim 
\left\|\mUld \mOfrakll\right\|_2 \leq \left\|\mUld \mOfrakll - \mUl \mOfrak\right\|_2 + \left\|\mUl \mOfrak - \mUstarl\right\|_2 + \left\|\mUstarl\right\|_2
\end{equation*}
Taking the above bound and Equations~\eqref{eq:loo:eq3} and~\eqref{eq:loo:eq4} into Equation~\eqref{eq:loo:eq2} yields
\begin{equation*}
\begin{aligned}
    \left\|\mUll \mOfrakll - \mU \mOfrak\right\|_{\F} \lesssim
    & \frac{\sigma \sqrt{r \log n}}{\lambdastar_r} + \frac{L\log n}{\lambdastar_r} \left(\left\|\mU\mOfrak - \mUstar\right\|_{2,\infty} + \left\|\mUstar\right\|_{2,\infty}\right)\\
    & + \frac{\sigma \sqrt{n}}{\lambdastar_r} \left(\left\|\mUld \mOfrakll - \mUl \mOfrak\right\|_2 + \left\|\mUl \mOfrak - \mUstarl\right\|_2 + \left\|\mUstarl\right\|_2\right)
\end{aligned}
\end{equation*}
which after reorganizing, yields 
\begin{equation*}
\begin{aligned}
    \left\|\mUll \mOfrakll - \mU \mOfrak\right\|_{\F} \lesssim
    & \frac{\sigma \sqrt{r \log n}}{\lambdastar_r} + \frac{L\log n}{\lambdastar_r} \left(\left\|\mU\mOfrak - \mUstar\right\|_{2,\infty} + \left\|\mUstar\right\|_{2,\infty}\right)\\
    & ~~~~ + \frac{\sigma \sqrt{n}}{\lambdastar_r} \left(\left\|\mUl \mOfrak - \mUstarl\right\|_2 + \left\|\mUstarl\right\|_2\right)
\end{aligned}
\end{equation*}
By the assumption in Equation~\eqref{eq:L-sigma}, the above bound for $\left\|\mUll \mOfrakll - \mU \mOfrak\right\|_{\F}$ can be simplified as 
\begin{equation} \label{eq:loo:eq5}
    \left\|\mUll \mOfrakll - \mU \mOfrak\right\|_{\F} \lesssim \frac{\sigma \sqrt{r \log n}}{\lambdastar_r} + \frac{\sigma \sqrt{n}}{\lambdastar_r} \left(\left\|\mUl \mOfrak - \mUstarl\right\|_2 + \left\|\mUstarl\right\|_2\right)
\end{equation}
Taking Equations~\eqref{eq:loo:eq5} and~\eqref{eq:loo:eq1.6} into Equation~\eqref{eq:loo:eq1} yields 
\begin{equation*}
    \left\|\mWzl\left(\mU\mOfrak - \mUstar\right)\right\|_2 \lesssim (L\log n) \|\mU\mOfrak - \mUstar\|_{2,\infty} + \frac{\sigma^2 \sqrt{r n \log n}}{\lambdastar_r} + \frac{\sigma^2 n}{\lambdastar_r} \left(\left\|\mUl \mOfrak - \mUstarl\right\|_2 + \left\|\mUstarl\right\|_2\right)
\end{equation*}
as desired. 
\end{proof}

\section{Proof of SDP} \label{sec:sdp}

Here we collect proof details related to SDP.
Throughout this section, we treat the error matrix $\mDelta$, given by Equation~\eqref{eq:mDelta:define}, as fixed, since it comes from $\mYz$ and is independent of $\mYo$. 
Appendices~\ref{sec:proof:sdp:key-lower}-\ref{subsec:proof:cor:sdp-more-general} gives a detailed proof of the SDP formulated in Equation~\eqref{eq:prime:sdp}, along with necessary technical lemmas. 
Specifically, the main result Theorem~\ref{thm:sdp:L1-bound}, is shown in Appendix~\ref{subsec:thm:sdp:L1-bound}, and Corollary~\ref{cor:sdp-more-general} is proved in Appendix~\ref{subsec:proof:cor:sdp-more-general}.
Appendix~\ref{subsec:thm:sdp-minimax:proof} provides the proof of Theorem~\ref{thm:sdp-minimax}, which establishes the minimax optimality of SDP. 
Appendices~\ref{sec:proof:thm:tauhat} and~\ref{subsec:proof:trunc:sdp} establishes results related to the truncated SDP. 
Appendix~\ref{sec:proof:sdp:multiple} handles results related to the SDP when multiple observations are available, specifically Theorem~\ref{thm:sdp:multiple}. 
Finally, Appendix~\ref{subsec:sdp:aux} contains technical results that are used in the proofs in this section.

\subsection{Proof of Lemma \ref{lem:key-lower}}
\label{sec:proof:sdp:key-lower}

\begin{proof}
Recall that the cost matrix $\mC = \mYtil \circ \mYtil$. 
By the optimality of $\mZhat$, we have 
\begin{equation} \label{eq:sdp:basic:basic}
    \langle \mZhat-\mZstar, \mC \rangle \leq 0. 
\end{equation}
Recall that the matrix $\mCo$ defined in Equation~\eqref{eq:mCo:define} removes $\mDeltatil$ from the cost matrix $\mC$, where $\mDeltatil$ denotes the estimation error matrix of the low rank component $\mMstar$.    
Since we can relabel the indices in $\Istar$, without loss of generality, we assume that $\Istar = [m]$. 
Under this assumption, the matrix $\mBstar$ in Equation~\eqref{eq:B-more-general} ican be written as
\begin{equation*}
	\mBstar = \beta^2 \left[\mJ - \begin{pmatrix}
        \mJ_1 & \\
        & \mJ_2
    \end{pmatrix}\right].
\end{equation*}
The expectation of $\mCo$ is given by
\begin{equation*}
    \E \mC^{(1)} = \beta^2 \left[\mJ - \begin{pmatrix}
        \mJ_1 & \\
        & \mJ_2 
    \end{pmatrix}\right] + \mSig = \left(\sigma_0^2 + \beta^2\right)\mJ - \beta^2  \begin{pmatrix}
    \mJ_1 & \\
    & \mJ_2
\end{pmatrix} + \left(\mSig - \sigma_0^2 \mJ\right),
\end{equation*}
where $\mJ_1 \in \R^{m\times m}$ and $\mJ_2 \in \R^{K\times K}$ are matrices of all ones. 
By Equation~\eqref{eq:sdp:basic:basic}, 
\begin{equation*}
\left\langle \widehat{\mZ}-\mZstar, \mC - \E \mCo\right\rangle
\le
\left\langle \mZstar - \mZhat, \E \mCo\right\rangle .
\end{equation*}
Expanding $\E \mCo$ and
using the SDP constraint $\langle \mJ, \mZhat \rangle = \langle \mJ, \mZstar \rangle = K^2$, it follows that
\begin{equation} \label{eq:sdp-i:I}
\left\langle \widehat{\mZ}-\mZstar, \mC - \E \mCo\right\rangle
\le
-\beta^2 \left\langle \mZstar - \mZhat,  \begin{pmatrix}
    \mJ_1 & \\
    & \mJ_2
\end{pmatrix} \right\rangle + \langle \mZstar - \widehat{\mZ}, \mSig - \sigma^2_0 \mJ\rangle .
\end{equation}
By the definition of $\mZstar$ in Equation~\eqref{eq:def:mZstar}, since $\mZstar$ is all one over its support $I_\star^c \times I_\star^c$, we have
\begin{equation} \label{eq:sdp-i:II}
    \left\langle \mZstar, \begin{pmatrix} 
    \mJ_1 & \\
    & \mJ_2
\end{pmatrix} - \mJ \right\rangle = 0, \quad \left\langle \mZstar, \mJ \right\rangle = K^2.
\end{equation}
Additionally, by the constraint $\langle \mJ, \mZhat \rangle = K^2$ in Equation~\eqref{eq:prime:sdp}, we have
\begin{equation}\label{eq:sdp-i:III}
	\langle \mJ, \mZhat - \mZstar \rangle = 0. 
\end{equation}
it follows from Equations~\eqref{eq:sdp-i:I}-\eqref{eq:sdp-i:III} that %
\begin{equation*}
\left\langle \widehat{\mZ}-\mZstar, \mC - \E \mCo\right\rangle
\le
\beta^2
\left\langle \mZhat,
	\begin{pmatrix} \mJ_1 & \\
    			& \mJ_2 \end{pmatrix}
	- \mJ \right\rangle
+ \langle \mZstar - \widehat{\mZ}, 
	\mSig - \sigma^2_0 \mJ\rangle .
\end{equation*}
Rearranging terms, it follows that 
\begin{equation} \label{eq:sdp-key-upper-mid}
\begin{aligned}
    2\sum_{i=1}^m \sum_{j=m+1}^n \widehat{Z}_{ij} &= \left\langle \mZhat, \mJ - \begin{pmatrix}
    \mJ_1 & \\
    & \mJ_2
\end{pmatrix}\right\rangle  \\
&\leq \frac{1}{\beta^2} \langle \mZstar - \widehat{\mZ}, \mC - \E \mCo + \mSig - \sigma^2_0 \mJ\rangle. 
\end{aligned}
\end{equation}
Applying Lemma~\ref{lem:Zhat-l1} to the left hand side of Equation~\eqref{eq:sdp-key-upper-mid} 
\begin{equation*}
\|\mZhat - \mZstar\|_1 
\leq \frac{5}{2\beta^2} \left\langle\mZstar-\mZhat, \mC-\E \mCo + \mSig - \sigma^2_0 \mJ\right\rangle,
\end{equation*}    
which yields Equation~\eqref{eq:sdp-key-upper2}. 

Expanding $\mC - \E \mCo + \mSig - \sigma^2_0 \mJ$ as %
\begin{equation*}
\mC - \E \mCo + \mSig - \sigma^2_0 \mJ
=
\mCo - \E \mCo + 2\mDeltatil \circ \mWtilo + 2\mDeltatil \circ \mBstar 
	+ \mDeltatil \circ \mDeltatil + \mSig - \sigma^2_0 \mJ,
\end{equation*}
it follows that $\|\mZhat - \mZstar\|_1$ is upper bounded by
\begin{equation*} \begin{aligned}
&\|\mZhat - \mZstar\|_1 \\
&\le
\frac{5}{2\beta^2} \left| \left\langle\!\mZstar\!\!-\!\mZhat, \mCo - \E \mCo + 2\mDeltatil \circ \mWtilo \right\rangle \right| 
+ \frac{5}{2\beta^2} \left| \left\langle\!\mZstar\!\!-\!\mZhat, 2\mDeltatil \circ \mBstar + \mDeltatil \circ \mDeltatil + \mSig - \sigma^2_0 \mJ\right\rangle \right| .
\end{aligned} \end{equation*} 
Applying H\"{o}lder's inequality, and using the fact that $\|\mZhat - \mZstar\|_{\infty} \leq 1$ trivially,
\begin{equation} \label{eq:l1-bound-mid} \begin{aligned}
\|&\mZhat - \mZstar\|_1 \\
&\leq
\frac{5}{2\beta^2} \left| \left\langle\mZstar-\mZhat, \mCo - \E \mCo + 2\mDeltatil \circ \mWtilo \right\rangle \right| 
+ \frac{10}{\beta^2}\|\mZhat - \mZstar\|_{\infty} \| \mDeltatil\circ \mBstar + \mDeltatil\circ \mDeltatil + \mSig-\sigma^2_0\mJ\|_{1} \\ 
&\leq \frac{5}{2\beta^2} \left| \left\langle\mZstar-\mZhat, \mC^{(1)} - \E \mCo + 2\mDeltatil \circ \mWtilo \right\rangle \right|
+ \frac{10}{\beta^2}
\left\|\mDeltatil\circ \mBstar + \mDeltatil\circ \mDeltatil 
	+ \mSig-\sigma^2_0\mJ\right\|_{1} .
\end{aligned} \end{equation}

By construction of $\mBstar$, we have 
\begin{equation*}
\|\mDeltatil \circ \mBstar\|_1
\leq \beta \|\mDeltatil\|_1
\leq \beta n \|\mDeltatil\|_{\F} \lesssim \sigma \beta r^{1/2}n^{3/2},
\end{equation*}
where the last inequality follows from the conditioning on $\{\|\mDeltatil\|_{\F} \leq C \sigma \sqrt{r n}\}$.
Additionally, we have
$\|\mDeltatil \circ \mDeltatil\|_{1} = \|\mDeltatil\|_\F^2
	\lesssim \sigma^2 n r$
and, by Assumption~\ref{assump:homoscedastic}, 
$\|\mSig - \sigma^2_0 \mJ\|_1 \lesssim \sigma^2 n^{3/2}$ .
Applying the triangle inequality, it follows that
\begin{equation*}
\left\|\mDeltatil\circ \mBstar + \mDeltatil\circ \mDeltatil 
        + \mSig-\sigma^2_0\mJ\right\|_{1}
\le 
\beta n \|\mDeltatil\|_{\F} \lesssim \sigma \beta r^{1/2}n^{3/2},
+ \sigma^2 n r
+ \sigma^2 n^{3/2} .
\end{equation*}

Applying this bound to Equation~\eqref{eq:l1-bound-mid},
\begin{equation*}
    \|\mZhat - \mZstar\|_1 \lesssim \frac{1}{\beta^2} \left| \left\langle\mZstar-\mZhat, \mCo - \E \mCo + 2\mDeltatil \circ \mWtilo \right\rangle \right| + \frac{\sigma^2 n^{3/2} + \sigma^2 n r}{\beta^2} + \frac{\sigma r^{1/2} n^{3/2}}{\beta} ,
\end{equation*}
as we set out to show.
\end{proof}

\subsection{Proof of Lemma \ref{lem:Grothendieck:bound}}
\label{sec:proof:sdp:grothendieck}

\begin{proof}
By the triangle inequality,
\begin{equation} \label{eq:Grothendieck:triangle}
\left|\langle \mZhat-\mZstar,                                                                   \mCo - \E \mCo + 2 \mDeltatil \circ \mWtilo\rangle\right|
\le \alpha_1 + \alpha_2 ,
\end{equation}
where
\begin{equation*}
    \alpha_1 = \left|\langle \mZhat-\mZstar, \mCo - \E \mCo\rangle\right| \quad 
\text{and} \quad
    \alpha_2 =\left|\langle \mZhat-\mZstar, \mDeltatil \circ \mWtilo \rangle\right|.
\end{equation*}
We will thus bound these two quantities in turn.
Both bounds are based on Bernstein's inequality and Grothendieck's inequality \citep{vershynin2018HDP}.

Applying Grothendieck’s inequality as quoted in Equation~\eqref{eq:Grothendieck},
\begin{equation*}
\alpha_1 
\leq 2c_0 \max _{\substack{\vx, \vy \in\{ \pm 1\}^n}}\left|\sum_{i, j}\left(\Co_{i j}-\E \Co_{i j}\right) x_i y_j\right| .
\end{equation*}
Expanding $\mCo$ using Equation~\eqref{eq:mCo:define} 
and bounding $(a+b)^2 \le 2(a^2 + b^2)$,
\begin{equation} \label{eq:C1-triangle}
\alpha_1
\leq
4c_0 \max _{\substack{\vx, \vy \in\{ \pm 1\}^n}}
  \left|\sum_{i, j}\left((\Wtilo_{i j})^2-\sigma^2_{ij}\right) x_i y_j\right|
+ 4c_0 \max_{\substack{\vx, \vy \in\{ \pm 1\}^n}}
	\left|\sum_{i, j} B_{i j}\Wtilo_{i j} x_i y_j\right| .
\end{equation}
We will bound the two right-hand terms in Equation~\eqref{eq:C1-triangle} separately. 

For $i,j \in [n]$, define $X_{ij} = [(\Wtilo_{i j})^2-\sigma_{ij}^2] x_i y_j$.
By Assumption~\ref{assump:noise} on the magnitude of $\mWtilo$, we have
\begin{equation*}
    |X_{ij}| = \left|(\Wtilo_{i j})^2-\sigma_{ij}^2\right| 
    \lesssim L^2, \quad \text{and} \quad \E X_{ij}^2 
    \leq \E (\Wtilo_{i j})^4 \leq L^2 \sigma^2. 
\end{equation*}
Applying Bernstein's inequality stated in Lemma~\ref{lem:bern-ineq}, for any $t < n^2 \sigma^2$,
\begin{equation*}
\Pr\left[ \left|\sum_{1\leq i\leq j \leq n} X_{ij}\right| \geq t\right]
\leq
2 \exp\left( -c \min \left\{\frac{t^2}{n^2 \sigma^2 L^2}, 
			\frac{t}{L^2}\right\} \right) 
= 2 \exp\left\{ -\frac{c t^2}{n^2 \sigma^2 L^2}\right\} .
\end{equation*}
Applying a union bound over all possible choices of $\vx,\vy \in \{\pm1\}^n$,
\begin{equation*} \begin{aligned}
\Pr\left[ \max_{\substack{\vx, \vy \in\{ \pm 1\}^n}}
	\left|\sum_{i, j} \left((\Wtilo_{i j})^2-\sigma^2_{ij}\right) x_i y_j
		\right|
	\geq t\right]
&\leq ~\sum_{\vx, \vy \in \{\pm 1\}^n}
\Pr\left(\left|\sum_{i, j}\left((\Wtilo_{i j})^2-\sigma_{ij}^2\right) x_i y_j\right| \geq t\right) \\
    &\lesssim 
     ~ 4^n \exp\left(-\frac{c t^2}{n^2 \sigma^2 L^2}\right) \\
     &\leq \exp\left(-\frac{c t^2}{n^2 \sigma^2 L^2} + n\right) .
\end{aligned}
\end{equation*}
Noting that $n^{3/2} \sigma L \ll n^2 \sigma^2$ by Equation~\eqref{eq:L-sigma} and taking $t = C n^{3/2} \sigma L$ for some sufficiently large constant $C > 0$, we obtain that with probability at least $1 - O(n^{-40})$,
\begin{equation}\label{eq:W-C1-groth}
\max_{\substack{\vx, \vy \in\{ \pm 1\}^n}} 
\left|\sum_{i, j}\left((\Wtilo_{i j})^2-\sigma_{ij}^2\right) x_i y_j\right|
\lesssim \sigma L n^{3/2} .
\end{equation}

Using the same arguement as above, for any $\vx, \vy \in \{\pm 1\}^n$, under Assumption~\ref{assump:noise} we have
\begin{equation*}
\left|B^\star_{i j} \Wtilo_{i j} x_i y_j\right| \leq \beta L,
\quad \text{and} \quad \E \left(B^\star_{i j} \Wtilo_{i j} x_i y_j\right)^2 \leq \beta^2 \sigma^2.
\end{equation*}
Applying Bernstein's inequality as stated in Lemma~\ref{lem:bern-ineq} followed by a union bound over $\vx,\vy \in \{\pm1\}$, it holds with probability at least $1 - O(n^{-40})$ that
\begin{equation} \label{eq:BW-C1-groth}
\max_{\substack{\vx, \vy \in\{ \pm 1\}^n}} 
\left| \sum_{1\leq i, j\leq n} B^\star_{i j} \Wtilo_{i j} x_i y_j\right|
\lesssim
\max\left\{ m^{1/2} n \beta \sigma, \beta L n \right\}
\lesssim n^{3/2} \beta \sigma, 
\end{equation}
where the last inequality follows from Equation~\eqref{eq:L-sigma} and the fact that $m \leq n$.

Applying Equations~\eqref{eq:W-C1-groth} and~\eqref{eq:BW-C1-groth} to Equation~\eqref{eq:C1-triangle} and using the assumption that $\beta \leq L$, 
\begin{equation}\label{eq:C1-groth-concen}
\alpha_1 
\lesssim \sigma n^{3/2} L. 
\end{equation}

Again applying Grothendieck's inequality,
\begin{equation} \label{eq:alpha2:groth}
\alpha_2
\lesssim \max_{\vx, \vy \in \{\pm\}^{n}} 
  \left|\sum_{1\leq i,j \leq n} \Deltatil_{ij} \Wtilo_{ij} x_i y_j\right|.
\end{equation}
Noting that 
\begin{equation*}
\left|\Deltatil_{ij} \Wtilo_{ij} x_i y_j\right|
\leq \|\mDeltatil\|_{\infty} L 
\quad \text{and} \quad
\E \left(\Deltatil_{ij} \Wtilo_{ij} x_i y_j\right)^2 
\leq \Deltatil_{ij}^2 \sigma^2,
\end{equation*}
applying Bernstein's inequality as stated in Lemma~\ref{lem:bern-ineq} yields that for any $t > 0$,
\begin{equation*}
\Pr\left[ \left|\sum_{1\leq i\leq j \leq n} 
		\Deltatil_{ij} \Wtilo_{ij} x_i y_j\right| \geq t\right]
\leq 2 \exp \left( -c \min\left\{ \frac{t}{\|\mDeltatil\|_{\infty} L}, 
			\frac{t^2}{\sigma^2 \|\mDelta\|_\F^2}\right\} \right).
\end{equation*}
Applying the union bound and choosing a suitable $t$, it follows that with probability at least $1 - O(n^{-40})$ 
\begin{equation*} \begin{aligned}
    \max_{\vx, \vy \in \{\pm 1\}^n}\left|\sum_{1\leq i\leq j \leq n} \Deltatil_{ij} \Wtilo_{ij} x_i y_j\right| &\lesssim \max \left\{ L n \|\mDeltatil\|_{\infty}, \sigma \|\mDeltatil\|_\F \sqrt{n} \right\} \lesssim \max \left\{ \sigma L n, \sigma^2 n r^{1/2} \right\}\\
    &\lesssim \sigma^2 n^{3/2}
\end{aligned} \end{equation*}
where the second inequality holds as we conditioned on the event $\{\|\mDeltatil\|_{\F} \leq C\sigma \sqrt{r n}\} \cap \{\|\mDeltatil\|_{\infty} \leq C \sigma\}$, the last inequality follows from Equation~\eqref{eq:L-sigma} and the fact that $r \leq n$.
Applying this bound to Equation~\eqref{eq:alpha2:groth} yields
\begin{equation} \label{eq:alpha2:bound}
\alpha_2
\lesssim \sigma^2 n^{3/2}. 
\end{equation}

Applying Equations~\eqref{eq:C1-groth-concen} and~\eqref{eq:alpha2:bound} to Equation~\eqref{eq:Grothendieck:triangle}
\begin{equation} \label{eq:Grothendieck:final}
\left|\langle \mZhat-\mZstar,                                                                   \mCo - \E \mCo + 2 \mDeltatil \circ \mWtilo\rangle\right|
\lesssim 
\sigma n^{3/2} L + \sigma^2 n^{3/2}
\le \sigma n^{3/2} L, 
\end{equation}
with probability at least $1 - O(n^{-40})$ conditioned on the event $\{\|\mDeltatil\|_{\F} \leq C\sigma \sqrt{r n}\} \cap \{\|\mDeltatil\|_{\infty} \leq C \sigma\}$, as we set out to show, where the last inequality follows from the assumption in Equation~\eqref{eq:L-sigma}.
\end{proof}


\subsection{Lemma \ref{lem:S2:bound} and Proof}

\label{subsec:lem:S2:bound}

\begin{lemma} \label{lem:S2:bound}
Suppose that $\beta \leq L$, $n > 2m$ and that the conditions in Lemma~\ref{lem:key-lower} hold.
Conditional on the event
\begin{equation} \label{eq:event:S2:delta}
\left\{\|\mDelta\| \leq C\sigma \sqrt{n} \right\}
\cap
\left\{ \|\mDelta\|_{\infty}
\leq \frac{C\sigma \mut^{1/2} r (\mut \vee \log n)^{1/2}}{\sqrt{n}}\right\},
\end{equation}
with probability at least $1 - O(n^{-7})$, we have
\begin{equation*}
|S_2| \lesssim \left(\frac{\sigma^2 \mut^{1/2} r^{3/2} \left(\mut \vee \log n\right)^{1/2}}{n} + \frac{\sigma L}{\sqrt{n}}\right) \|\mZhat - \mZstar\|_1,
\end{equation*}
where $S_2$ is defined in Equation~\eqref{eq:S2:define}. 
\end{lemma}
\begin{proof}
	Recall that $\calP_{T^\perp}$ is defined in Equation~\eqref{eq:calP:define}.
Expanding the definition of $S_2$,
\begin{equation} \label{eq:S2:start} \begin{aligned}
|S_2| &= \left| \left \langle \calP_{T^\perp}(\mZstar-\mZhat),
		\mC -\E \mCo + \mSig - \sigma^2_0 \mJ\right \rangle \right| 
= \left| \left \langle \calP_{T^\perp}(\mZhat), 
	\mC -\E \mCo + \mSig - \sigma^2_0 \mJ\right \rangle\right|\\
&\leq \tr \left(\calP_{T^\perp}(\mZhat)\right) 
	\left(\left\|\mC -\E \mCo \right\| 
+ \left\|\mSig - \sigma_0^2 \mJ\right\|\right),
\end{aligned}
\end{equation}
where the last inequality holds by the matrix Hölder inequality and the fact that $\calP_{T^\perp}(\mZhat) \succeq 0$. 
By the triangle inequality,
\begin{equation} \label{eq:CEC1:bigtri}
\left\|\mC - \E \mCo\right\| 
\lesssim \|\mDeltatil \circ \mDeltatil\| + \|\mWtilo \circ \mWtilo - \mSig\| 
+ \|\mBstar \circ \mDeltatil\| + \|\mBstar \circ \mWtilo\| 
+ \|\mDeltatil \circ \mWtilo\| .
\end{equation}

By basic properties of matrix norms,
\begin{equation} \label{eq:CEC1:term1}
\begin{aligned}
\|\mDeltatil \circ \mDeltatil\| &\leq \|\mDeltatil \circ \mDeltatil\|_\F 
= \left(\sum_{1\leq i,j \leq n} \Deltatil_{ij}^4 \right)^{1/2}
\leq \|\mDeltatil\|_{\infty} \|\mDeltatil\|_\F \\
&\lesssim \sigma^2 \mut^{1/2} r^{3/2} \left(\mut \vee \log n\right)^{1/2}.
\end{aligned} \end{equation}
By Assumption~\ref{assump:noise},
\begin{equation*}
\left|(\Wtilo_{ij})^2 - \sigma_{ij}^2\right| \leq 2L^2 \quad \text{and} \quad \E \left((\Wtilo_{ij})^2 - \sigma_{ij}^2\right)^2 \leq \sigma^2 L^2 .
\end{equation*}
The matrix version of Bernstein's inequality, as stated in Theorem~\ref{thm:bernstein}, implies that for any $t > 0$,
\begin{equation*}
\Pr\left[ \|\mWtilo \circ \mWtilo - \mSig\| \geq 4\sigma L \sqrt{n} + t\right]
\leq n \exp\left(-\frac{c t^2}{L^4}\right).
\end{equation*}
Taking $t = C L^2 \log^{1/2} n$ for suitably large constant $C$, with probability at least $1-O(n^{-8})$,
\begin{equation} \label{eq:CEC1:term2}
\left\|\mWtilo \circ \mWtilo - \mSig\right\| 
\lesssim \sigma L \sqrt{n} + L^2 \log^{1/2} n \lesssim \sigma L \sqrt{n}. 
\end{equation}

By basic properties of the norm and Equation~\eqref{eq:event:S2:delta},
\begin{equation} \label{eq:CEC1:term3}
\|\mBstar \circ \mDeltatil\| \leq \beta \|\mDeltatil\| 
\leq \beta \sigma \sqrt{n}. 
\end{equation}
By Lemma~\ref{lem:A-circ-W-op}, we have
\begin{equation} \label{eq:CEC1:term4}
\left\|\mBstar \circ \mWtilo\right\| \lesssim \sigma \|\mBstar\|_{2, \infty} + L \|\mBstar\|_{\infty} \sqrt{\log n} = \sigma \beta \sqrt{n} + L \beta \sqrt{\log n}
\end{equation}
and 
\begin{equation} \label{eq:CEC1:term5} 
\begin{aligned}
\left\|\mDeltatil \circ \mWtilo\right\|
&\lesssim \sigma \|\mDeltatil\|_{2, \infty} 
	+ L \|\mDeltatil\|_{\infty} \sqrt{\log n}  \\
&\lesssim \sigma^2 r^{1/2} \left(\mut \vee \log n\right)^{1/2} 
+ \frac{\sigma L \mut^{1/2} r \log^{1/2} n}{\sqrt{n}} 
	\left(\mut \vee \log n\right)^{1/2}.
\end{aligned} \end{equation}

Applying Equations~\eqref{eq:CEC1:term1} through~\eqref{eq:CEC1:term5} to Equation~\eqref{eq:CEC1:bigtri}, under the assumption that $\beta \leq L$ and the growth rate assumption on $\sigma, L$ in Equation~\eqref{eq:L-sigma}
\begin{equation*}
\|\mC - \E \mCo\| 
\lesssim \sigma^2 \mut^{1/2} r^{3/2} \left(\mut \vee \log n\right)^{1/2} + \sigma L \sqrt{n}.
\end{equation*}
Additionally, under Assumption~\ref{assump:homoscedastic}, we have 
\begin{equation*}
    \left\|\mC - \E \mCo\right\| + \|\mSig - \sigma_0^2 \mJ\| \lesssim \sigma^2 \mut^{1/2} r^{3/2} \left(\mut \vee \log n\right)^{1/2} + \sigma L \sqrt{n}. 
\end{equation*}
Applying this bound to Equation~\eqref{eq:S2:start},
\begin{equation*}
|S_2| \le
\left[ \sigma^2 \mut^{1/2} r^{3/2} \left(\mut \vee \log n\right)^{1/2} 
	+ \sigma L \sqrt{n} \right] \tr \left(\calP_{T^\perp}( \mZhat ) \right).
\end{equation*}
Applying Lemma~\ref{lem:S2-aux}, 
\begin{equation*}
|S_2| \le
\left[ \sigma^2 \mut^{1/2} r^{3/2} \left(\mut \vee \log n\right)^{1/2} 
	+ \sigma L \sqrt{n} \right]
\frac{ \|\mZhat - \mZstar\|_1 }{ K }.
\end{equation*}
By $K = n-m$ and $m < n/3$, the above bound implies the result stated in Lemma~\ref{lem:S2:bound}.  
\end{proof}

\subsection{Lemma \ref{lem:S1:bound} and Proof} \label{subsec:lem:S1:bound}

\begin{lemma}\label{lem:S1:bound}
Under the conditions in Lemmas~\ref{lem:key-lower} and~\ref{lem:Grothendieck:bound} hold, suppose that the event
\begin{equation} \label{eq:event:S1:delta}
\left\{\|\mDelta\|_{\F} \leq C\sigma \sqrt{n r}, \; \|\mDelta\|_{\infty} \leq \frac{C\sigma \mut^{1/2} r (\mut \vee \log n)^{1/2}}{\sqrt{n}}, 
\; \|\mDelta\|_{2,\infty} \leq C \sigma r^{1/2} \left(\mut \vee \log n\right)^{1/2}\right\},
\end{equation}
holds for some sufficiently large constant $C > 0$.
Then with probability at least $1 - O(n^{-7})$, the quantity $S_1$ defined in Equation~\eqref{eq:S1:define} satisfies
\begin{equation*} \begin{aligned}
|S_{1}|
&\lesssim \sigma \|\mZhat - \mZstar\|_1 
\left(\beta \sqrt{\frac{r(\mut \vee \log n)}{n}} + \sigma \frac{r (\mut \vee \log n)}{n}
+ L \sqrt{\frac{\log \left(n^2 \|\mZhat - \mZstar\|_1^{-1}\right)}{n}}\right) 
\end{aligned}
\end{equation*}
\end{lemma}
\begin{proof}
Recalling the definition of $S_1$ and applying the triangle inequality, we have
\begin{equation}\label{eq:S1:start}
|S_1|
\le \left| \left \langle \mZstar-\mZhat, \calP_T\left(\mC - \E \mCo\right)
	 \right \rangle \right|
+ \left| \left\langle \mZstar-\mZhat, \calP_T\left(\mSig - \sigma_0^2 \mJ\right)
        \right\rangle \right| .
\end{equation}
We remind the reader that $\calP_T$ is defined in Equation~\eqref{eq:calP:define}.


Consider two upper-triangular matrices:
$\mXi^{(1)} = [\xi^{(1)}_{ij}] \in \R^{n \times n}$ with entries
\begin{equation} \label{eq:def:xi1}
\xi^{(1)}_{ij} =
\begin{cases}
    2(\Deltatil_{ij} + B^\star_{ij}) \Wtilo_{ij} & \text{ if } 1\leq i < j \leq n\\
    0 & \text{ otherwise } \\
\end{cases}
\end{equation}
and $\mXi^{(2)} = [\xi^{(2)}_{ij}] \in \R^{n \times n}$ with entries
\begin{equation} \label{eq:def:xi2}
\xi^{(2)}_{ij} =
\begin{cases}
    (\Wtilo_{ij})^2 - \sigma^2_{ij} & \text{ if } 1\leq i < j \leq n\\
    \frac{1}{2}(\Wtilo_{ii})^2 - \frac{1}{2}\sigma^2_{ii} & \text{ if } 1\leq i = j \leq n \\
    0 & \text{ if } 1\leq j < i \leq n .
\end{cases}
\end{equation}
Expanding out the Hadamard products, one may verify that
\begin{equation*}
2(\mBstar + \mDeltatil) \circ \mWtilo 
+ \left(\mWtilo \circ \mWtilo - \mSig\right)
= \mXi^{(1)} + \mXi^{(1)\top} + \mXi^{(2)} + \mXi^{(2)\top},
\end{equation*}
and we can decompose $\mC - \E \mCo$ as
\begin{equation} \label{eq:CEC1:decomp}
\mC - \E \mCo
= \mXi^{(1)} + \mXi^{(1)\top} 
	+ \mXi^{(2)} + \mXi^{(2)\top} 
	+ \mDeltatil \circ \mDeltatil 
	+ 2\mBstar \circ \mDeltatil.
\end{equation}

Defining the quantities
\begin{equation} \label{eq:def:S1parts}
\begin{aligned}
S_{11} &=
\left\langle\mZstar-\mZhat, 
\calP_{T}(\mXi^{(1)} + \mXi^{(1)\top})\right\rangle \\
S_{12} &=
\left\langle\mZstar-\mZhat, 
\calP_{T}(\mXi^{(2)} + \mXi^{(2)\top})\right\rangle ~~~\text{ and } \\
S_{13} &=
\left\langle\mZstar-\mZhat, 
\calP_{T}(\mDeltatil \circ \mDeltatil + 2\mBstar \circ \mDeltatil)
\right\rangle ,
\end{aligned} \end{equation}
applying the decomposition in Equation~\eqref{eq:CEC1:decomp} and applying the triangle inequality,
\begin{equation} \label{eq:S1parts:tri}
\left| \left \langle \mZstar-\mZhat, \calP_T\left(\mC - \E \mCo\right)
         \right \rangle \right|
\le |S_{11}| + |S_{12}| + |S_{13}| .
\end{equation}

We will proceed to bound these three terms separately.
Throughout the proof, we will frequently use the equalities
\begin{equation*}
    \left\langle \mA_1, \mA_2 \mA_3\right\rangle = \tr(\mA_1^\top \mA_2 \mA_3) = \tr(\mA_3 \mA_1^\top \mA_2) = \left\langle \mA_1 \mA_3^\top, \mA_2 \right\rangle 
\end{equation*}
and
\begin{equation*}
    \left \langle \mA_1, \mA_2 \right \rangle = \tr\left(\mA_1^\top \mA_2\right) = \tr\left(\mA_2^\top \mA_1\right) = \left \langle \mA_2, \mA_1 \right \rangle
\end{equation*}
which holds for any compatible matrices $\mA_1, \mA_2$ and $\mA_3$.

Recalling the definition of $S_{11}$ from Equation~\eqref{eq:def:S1parts}, the definition of $\calP_T$ in Equation~\eqref{eq:calP:define} and $K=n-m$, we have
\begin{equation*} \begin{aligned}
S_{11}
&= \left\langle\mZstar-\mZhat, 
	\calP_{T}(\mXi^{(1)} + \mXi^{(1)\top})\right\rangle \\
&= 2\left\langle\mZstar-\mZhat,
	\frac{ \mZstar \mXi^{(1)}}{K}
	+ \frac{\mZstar\mXi^{(1)\top}}{K} \right\rangle 
- \left\langle\mZstar-\mZhat, 
	\frac{ \mZstar\mXi^{(1)}\mZstar}{K^2} 
		+ \frac{\mZstar\mXi^{(1)\top} \mZstar}{K^2}
	\right\rangle.
\end{aligned} \end{equation*}

Since
\begin{equation*} \begin{aligned}
\left\langle\mZstar-\mZhat, 
	 \frac{1}{K^2} \mZstar \mXi^{(1)} \mZstar\right\rangle
= \left\langle \left(\mZstar-\mZhat\right) \frac{ \mZstar }{K},
		 \frac{1}{K} \mZstar \mXi^{(1)} \right\rangle
\end{aligned}
\end{equation*}
and
\begin{equation*} \begin{aligned}
\left\langle\mZstar-\mZhat, \frac{1}{K^{2}}\mZstar \mXi^{(1) \top} \mZstar
	\right\rangle 
&=\left\langle\left(\mZstar-\mZhat\right) \frac{ \mZstar }{K},
		\frac{ \mZstar }{K} \mXi^{(1)} \right\rangle,
\end{aligned} \end{equation*}
we have 
\begin{equation*}
S_{11}
= 
2 \left\langle\mZstar-\mZhat,
	\frac{\mZstar}{K} \mXi^{(1)} + \frac{\mZstar}{K}\mXi^{(1)\top} 
	\right\rangle 
-2\left\langle\left(\mZstar-\mZhat\right) \frac{\mZstar}{K}, 
		\frac{\mZstar}{K} \mXi^{(1)} \right\rangle.
\end{equation*}
Hence, by the triangle inequality, it holds that 
\begin{equation} \label{eq:S1parts:S11:expand}
|S_{11}| \lesssim 
\left|\left\langle\mZstar-\mZhat, 
	\frac{\mZstar \mXi^{(1)}}{K} \right\rangle\right| 
+
\left|\left\langle\mZstar-\mZhat, 
	\frac{\mZstar\mXi^{(1)\top}}{K} \right\rangle\right| 
+
\left|\left\langle\left(\mZstar-\mZhat\right) \frac{\mZstar}{K},
	\frac{\mZstar \mXi^{(1)}}{K} \right\rangle\right|.
\end{equation}

By a similar argument,
\begin{equation} \label{eq:S1parts:S12:expand} \begin{aligned}
|S_{12}| &\lesssim 
\left|\left\langle\mZstar-\mZhat, 
	\frac{\mZstar}{K} \mXi^{(2)} \right\rangle\right| 
+ \left|\left\langle\mZstar-\mZhat, 
	\frac{ \mZstar\mXi^{(2)\top}}{K} \right\rangle\right|
+ \left|\left\langle\left(\mZstar-\mZhat\right) \frac{\mZstar}{K}, 
		\frac{ \mZstar \mXi^{(2)}}{K} \right\rangle\right|.
\end{aligned}
\end{equation}

Recalling the definitions of the matrices $\mXi^{(1)}$ and $\mXi^{(2)}$ from Equations~\eqref{eq:def:xi1} and~\eqref{eq:def:xi2}, define the random variables
$X^{(\ell)}_j = \left|\sum_{k=m+1}^n \xi^{(\ell)}_{kj}\right|$ 
for $j \in [n]$ and $\ell \in \{1,2\}$.
Let 
$$
\varepsilon_0 = \|\mZstar - \mZhat\|_{1}. 
$$
Since $\|\mZstar - \mZhat\|_{\infty} \leq 1$,
it follows from Lemma~\ref{lem:reduce-to-order} that 
\begin{equation} \label{eq:S1:S11S12:intermezzo}
\left|\left\langle\mZstar-\mZhat, 
		\frac{\mZstar \mXi^{(\ell)}}{K} \right\rangle\right| 
\leq \sum_{j=1}^{\lceil \varepsilon_0 / K\rceil} X^{(\ell)}_{(j)}, \quad
    \text{and} \quad  
\left|\left\langle\left(\mZstar-\mZhat\right) \frac{\mZstar}{K},
	\mZstar \mXi^{(\ell)}/K \right\rangle\right| 
\leq \sum_{j=1}^{\lceil \varepsilon_0 / K\rceil} X^{(\ell)}_{(j)}.
\end{equation}
The terms involving $\mXi^{(\ell)\top}$ in $S_{11}$ and $S_{12}$ can be controlled similarly.
Following the assumptions in Lemma \ref{lem:key-lower} and \ref{lem:Grothendieck:bound}, we have $\varepsilon_0/K = o_{\Pr}(n)$. 
Hence, the conditions in Lemma~\ref{lem:order-statistics-bound} are satisfied with high probability. 
Two applications of Lemma~\ref{lem:order-statistics-bound}, one with
\begin{equation*}
    \nu_{\xi^{(1)}} \lesssim n\beta^2 \sigma^2 + \sigma^2 \|\mDeltatil\|_{2,\infty}^2, \quad L_{\xi^{(1)}} \leq (\|\mDeltatil\|_{\infty} + \beta) L 
\end{equation*}
and the other with
\begin{equation*}
    \nu_{\xi^{(2)}} \lesssim n L^2 \sigma^2, \quad L_{\xi^{(2)}} \lesssim L^2.
\end{equation*}
yield that with probability at least $1 - O(n^{-40})$ that 
\begin{equation*}
\begin{aligned}
    \sum_{j=1}^{\lceil \varepsilon_0 / K\rceil} X^{(1)}_{(j)}
&\leq C_0 \lceil \varepsilon_0 / K\rceil
\max \left\{ \sqrt{\nu_{\xi^{(1)}} \log (K^2 / \varepsilon_0)}, 
	L_{\xi} \log \left(K^2 / \varepsilon_0 \right)\right\} \\
&\lesssim \left( \frac{\varepsilon_0}{n} \right) \max\left\{ \sigma \left(\beta \sqrt{n} + \|\mDeltatil\|_{2,\infty}\right) \sqrt{\log\left(\frac{n^2}{\varepsilon_0}\right)}, (\|\mDeltatil\|_{\infty}+\beta) L \log\left(\frac{n^2}{\varepsilon_0}\right) \right\}
\end{aligned}
\end{equation*}
and
\begin{equation*}
\begin{aligned}
    \sum_{j=1}^{\lceil \varepsilon_0 / K\rceil} X^{(2)}_{(j)} 
    &\lesssim \left( \frac{\varepsilon_0}{n} \right) \max\left\{ \sigma L \sqrt{n \log\left(\frac{n^2}{\varepsilon_0}\right)}, L^2 \log\left(\frac{n^2}{\varepsilon_0}\right) \right\}.
\end{aligned}
\end{equation*}
Applying the above two displays to Equation~\eqref{eq:S1:S11S12:intermezzo} and applying the result in turn to Equations~\eqref{eq:S1parts:S11:expand} and~\eqref{eq:S1parts:S12:expand},
\begin{equation} \label{eq:S1:S11plusS12:done}
|S_{11}| + |S_{12}| 
\lesssim \varepsilon_0 L
\max\left\{ \frac{\sigma \log^{1/2} \frac{n^2}{\varepsilon_0} }{\sqrt{n}}, 
	\frac{L \log \frac{n^2}{\varepsilon_0} }{n} \right\} 
\lesssim \sigma \varepsilon_0 L
	\frac{ \log^{1/2} \frac{n^2}{\varepsilon_0} }{\sqrt{n}},
\end{equation}
where the last inequality follows from the growth rate conditions in Lemma~\ref{lem:key-lower}.

%
%
By the triangle inequality, we have
\begin{equation} \label{eq:S1:S13:tri} \begin{aligned}
\left|S_{13}\right| \lesssim 
\left|\left\langle\mZstar-\mZhat,
	\calP_{T}(\mDeltatil \circ \mDeltatil)\right\rangle\right| 
+ \left|\left\langle\mZstar-\mZhat, 
	\calP_{T}(\mDeltatil \circ \mBstar)\right\rangle\right|.
\end{aligned} \end{equation}

To bound the first of these terms, we expand the definition of $\calP_T$ and apply the triangle inequality to obtain
\begin{equation*}
\begin{aligned}
&\left|\left\langle\mZstar-\mZhat,
	\calP_{T}(\mDeltatil \circ \mDeltatil)\right\rangle\right|  \\
&~~~~~~\le
\left| \left\langle\mZstar-\mZhat,
	\frac{\mZstar}{K} (\mDeltatil \circ \mDeltatil)\right\rangle \right|
+ \left| \left\langle\mZstar-\mZhat,
	(\mDeltatil \circ \mDeltatil)\frac{\mZstar}{K} \right\rangle \right|
+ \left| \left\langle\mZstar-\mZhat,
	\frac{\mZstar(\mDeltatil \circ \mDeltatil)\mZstar}{K^2} \right\rangle
	\right| \\
&~~~~~~= 2
\left| \left\langle\mZstar-\mZhat, 
		\frac{\mZstar}{K}(\mDeltatil \circ \mDeltatil) 
	\right\rangle \right|
+ \left| \left\langle \left(\mZstar-\mZhat\right)\frac{\mZstar}{K},
	\frac{\mZstar}{K}(\mDeltatil \circ \mDeltatil) \right\rangle \right|.
\end{aligned}
\end{equation*}
Let $X^{(3)}_j = \sum_{i=m+1}^n \Deltatil_{ij}^2$ for $j \in [n]$. 
Applying Lemma~\ref{lem:reduce-to-order}, we have
\begin{equation} \label{eq:S1:S13:part1} \begin{aligned}
\left|\left\langle\mZstar-\mZhat, \calP_{T}(\mDeltatil \circ \mDeltatil)\right\rangle\right| 
&\lesssim \sum_{j=1}^{\lceil \varepsilon_0 / K\rceil} X_{(j)}^{(3)} 
\lesssim \lceil \varepsilon_0 / K\rceil \|\mDeltatil\|_{2,\infty}^2  \\
&\lesssim \sigma^2 r (\mut \vee \log n) \lceil \varepsilon_0 / K\rceil,
\end{aligned}
\end{equation} 
where the last inequality follows from Equation~\eqref{eq:event:S1:delta}. 

Similarly, for the second term in Equation~\eqref{eq:S1:S13:tri},
\begin{equation*}
\left|\left\langle\mZstar-\mZhat, 
        \calP_{T}(\mDeltatil \circ \mBstar)\right\rangle\right|
\le
2 \left| \left\langle\mZstar-\mZhat, 
	\frac{\mZstar}{K} (\mDeltatil \circ \mBstar) \right\rangle \right|
+
\left| \left\langle \left(\mZstar-\mZhat\right) \frac{\mZstar}{K},
	\frac{\mZstar}{K} (\mDeltatil \circ \mBstar) \right\rangle \right| 
\end{equation*}

Following the same argument as above, we let $X^{(4)}_j = \left|\sum_{i=m+1}^n \Deltatil_{ij} \Bstar_{ij}\right|$ for $j \in [n]$.
Applying Lemma~\ref{lem:reduce-to-order} 
\begin{equation*} \begin{aligned}
\left|\left\langle\mZstar-\mZhat, 
	\calP_{T}(\mDeltatil \circ \mBstar)\right\rangle\right| 
&\lesssim 
\left|\left\langle\mZstar-\mZhat, 
	\frac{\mZstar}{K} (\mDeltatil \circ \mBstar) \right\rangle\right| 
+ \left|\left\langle \left(\mZstar-\mZhat\right) \frac{\mZstar}{K}, 
		\frac{\mZstar}{K} (\mDeltatil \circ \mBstar) 
	\right\rangle\right|\\
&\lesssim \sum_{j=1}^{\lceil \varepsilon_0 / K\rceil} X_{(j)}^{(4)} 
\lesssim \beta \lceil \varepsilon_0 / K\rceil \max_{j \in [m]} \left|\sum_{i=m+1}^{n} \Deltatil_{ij}\right|.
\end{aligned} \end{equation*}
Conditioned on the event in Equation~\eqref{eq:event:S1:delta}, we apply the bound 
\begin{equation*}
\left|\sum_{i=m+1}^{n} \Deltatil_{ij}\right| 
\leq \sqrt{n} \|\mDeltatil\|_{2, \infty} 
\lesssim \sigma \sqrt{n r (\mut \vee \log n)},
\end{equation*}
and obtain
\begin{equation} \label{eq:S1:S13:part2} 
\left|\left\langle\mZstar-\mZhat, \calP_{T}(\mDeltatil \circ \mBstar)
	\right\rangle\right| 
\lesssim  \sigma \beta \lceil \varepsilon_0 / K\rceil
		\sqrt{n r (\mut \vee \log n)}.
\end{equation}

Applying Equations~\eqref{eq:S1:S13:part1} and~\eqref{eq:S1:S13:part2} to Equation~\eqref{eq:S1:S13:tri},
\begin{equation} \label{eq:S1:S13:done}
|S_{13}|
\lesssim \sigma \beta \lceil \varepsilon_0 / K\rceil
		\sqrt{n r (\mut \vee \log n)} 
+ \sigma^2 \lceil \varepsilon_0 / K\rceil r (\mut \vee \log n) .
\end{equation}

Applying Equations~\eqref{eq:S1:S11plusS12:done} and~\eqref{eq:S1:S13:done} to Equation~\eqref{eq:S1parts:tri} and using the fact that $K = \Theta(n)$ yields
\begin{equation} \label{eq:S1:pcc}\begin{aligned}
&\left| \left \langle \mZstar-\mZhat, \calP_T\left(\mC - \E \mCo\right)
         \right \rangle \right| \\
&~~~\lesssim
\sigma L \varepsilon_0 
        \frac{ \log^{1/2} \frac{n^2}{\varepsilon_0} }{\sqrt{n}}
+
\sigma \beta \varepsilon_0
                \sqrt{ \frac{r (\mut \vee \log n)}{n} }
+ \sigma^2 \varepsilon_0 \frac{r (\mut \vee \log n)}{n} .
\end{aligned} \end{equation}

By Assumption~\ref{assump:homoscedastic}, for all $j \in [n]$,
\begin{equation*}
\left|\sum_{i=m+1}^n \sigma^2_{ij} - \sigma_0^2\right| \lesssim \sigma^2 \sqrt{n}
\end{equation*}
Using the same argument as above, we can also show that
\begin{equation}\label{eq:S1:sigma}
\left|\left\langle\mZstar-\mZhat, \calP_{T}(\mSig - \sigma_0^2 \mJ)\right\rangle\right| 
\lesssim  \sigma^2  \lceil \varepsilon_0 / K\rceil \sqrt{n} \lesssim \frac{\sigma^2 \varepsilon_0}{\sqrt{n}}.
\end{equation}
We omit the details here, as they are similar to the arguments above.

Since $\sigma \leq L$ by Equation~\eqref{eq:L-sigma}, applying Equations~\eqref{eq:S1:sigma} and~\eqref{eq:S1:pcc} to Equation~\eqref{eq:S1:start} yields
\begin{equation*}
\left| S_1 \right|
\lesssim
\sigma L \varepsilon_0
        \frac{ \log^{1/2} \frac{n^2}{\varepsilon_0} }{\sqrt{n}}
+
\sigma \beta \varepsilon_0
                \sqrt{\frac{r (\mut \vee \log n)}{n}} 
+ \sigma^2 \varepsilon_0 \frac{r (\mut \vee \log n)}{n}.
\end{equation*} 
Taking $\varepsilon_0 = \|\mZstar - \mZhat\|_1$ into the above bound yields the desired result.




\end{proof}

\begin{lemma}\label{lem:order-statistics-bound}
    Consider independent, mean zero random variables $(\xi_{ij})_{i\in[K], j\in[n]}$, where $K > n/2$. 
    Assume that 
    \begin{equation*}
        \max_{i\in [K], j\in[n]} |\xi_{ij}| \leq L_{\xi}, \quad \text{and} \quad \sum_{i=1}^K \E \xi_{ij}^2 \leq \nu_{\xi}.
    \end{equation*}
    Define $X_j:=\left|\sum_{i=1}^K \xi_{i j}\right|$ for $j \in [n]$. For a constant $C_0$ sufficiently large and $n$ sufficiently large, we have 
    \begin{equation}
        \Pr\left\{\exists \rho \in\left[1,\left\lceil c_u n\right\rceil\right]: \sum_{j=1}^{\lceil\rho\rceil} X_{(j)}>C_0 \lceil \rho \rceil \max\left[\sqrt{\nu_{\xi} \log(K/\lceil \rho \rceil)}, L_{\xi} \log (K/\lceil \rho \rceil) \right] \right\} = O(n^{-40}),
    \end{equation}
    where $c_u > 0$ is a sufficiently small constant. 
\end{lemma}
\begin{proof}
For any $u \geq 0$, consider an integer $0 < T < n$ and any integers $1\leq k_1 < k_2 < \cdots < k_T \leq n$.
For any $\vs \in \{-1, +1\}^T$, by Bernstein's inequality as stated in Lemma~\ref{lem:bern-ineq}, we have 
\begin{equation*}
\begin{aligned}
\Pr\left[\sum_{i=1}^K \sum_{j=1}^T  \xi_{i k_j} s_j > u\right]
\leq \exp\left(-c \min \left\{ \frac{u^2}{T \nu_{\xi}}, \frac{u}{L_\xi}\right\}
	\right).
\end{aligned}
\end{equation*}
Taking 
\begin{equation*}
u = u_T
= C_0 T \max\left[\sqrt{\nu_{\xi} \log(K/T)}, L_{\xi} \log (K/T) \right],
\end{equation*}
where $C_0 = (C/c)^{1/2}$ for $C$ sufficiently large, it follows that for sufficiently large $n$,
\begin{equation} \label{eq:orderstat:probbound:single}
\begin{aligned}
\Pr\left[\sum_{i=1}^K \sum_{j=1}^T  \xi_{i k_j} s_j > u_T \right]
\leq
\left(\frac{T}{K}\right)^{CT} \leq \left(\frac{2T}{n}\right)^{CT}.
\end{aligned}
\end{equation}
Hence, taking a union bound over all $1 \leq k_1 < k_2 < \cdots < k_T \leq n$, we have 
\begin{equation*} \begin{aligned}
\Pr \left[ \sum_{j=1}^T X_{(j)} > u_T\right]
&= \Pr\left[ \bigcup_{k_1<\cdots<k_T} \left\{\sum_{j=1}^T X_{k_j} > u_T\right\}
	\right]
\leq \sum_{k_1<\cdots<k_T} \Pr \left[ \sum_{j=1}^T X_{k_j} > u_T\right] \\
&= \sum_{k_1<\cdots<k_T}
\Pr\left[ \max_{\vs \in \{\pm 1\}^T} 
		\sum_{j=1}^T \sum_{i=1}^K \xi_{i k_j} s_j  > u_T\right] \\
&\leq \sum_{k_1<\cdots<k_T} \sum_{\vs \in \{\pm 1\}^{T}} 
	\Pr\left[ \sum_{j=1}^T  \sum_{i=1}^K \xi_{i k_j} s_j  > u_T\right] .
\end{aligned} \end{equation*}
Applying Equation~\eqref{eq:orderstat:probbound:single} to the summands,
\begin{equation*} \begin{aligned}
\Pr \left[\sum_{j=1}^T X_{(j)} > u_T\right]
&\leq 2^T  \binom{n}{T} \left(\frac{2T}{n}\right)^{CT}
\leq \left(\frac{n e}{T}\right)^T e^{(C+1)T} \left(\frac{T}{n}\right)^{CT} \\
&\leq e^{(C+2)T}\left(\frac{T}{n}\right)^{(C-1)T}
= \exp\left\{ -(C-1)T \log \frac{n}{T} + (C+2)T\right\}. 
\end{aligned} \end{equation*}
For $T \leq c_u n$ where $c_u = \exp\{-(C+2)/(C-41)\}$, we have 
\begin{equation*}
(C+2)T - (C-1)T \log \frac{n}{T} \leq -40T\log \frac{n}{T},
\end{equation*}
and it follows that
\begin{equation*}
\Pr \left[ \sum_{j=1}^T X_{(j)} > u_T\right]
\leq \left(\frac{T}{n}\right)^{7T}. 
\end{equation*}
Hence, taking a union bound over $\lceil \rho \rceil \in [1, \lceil c_u n \rceil]$ yields
\begin{equation*} \begin{aligned}
&\Pr\left[ \exists \rho \in [1, \lceil c_u n\rceil]: 
	\sum_{j=1}^{\lceil \rho \rceil} X_{(j)} > C_0 \lceil \rho \rceil 
	\max\left\{ \sqrt{\nu_{\xi} \log(K/\lceil \rho \rceil)}, 
	L_{\xi} \log (K/\lceil \rho \rceil) \right\} \right] \\
&~~~\leq
\sum_{T=1}^{\lceil c_u n\rceil}\left(\frac{T}{n}\right)^{40T}
= \sum_{t=1}^{41} \left(\frac{T}{n}\right)^{40T}
	+ \sum_{T=42}^{\lceil c_u n\rceil}\left(\frac{T}{n}\right)^{40T} \\
&~~~=  O(n^{-40}) ,
\end{aligned} \end{equation*}
as we set out to show.
\end{proof}

\subsection{Proof of Theorem~\ref{thm:sdp:L1-bound}}
\label{subsec:thm:sdp:L1-bound}
 
\begin{proof}
Define the events 
\begin{equation*}
\calE_1 =
\left\{\|\mDelta\|_{\F} \! \leq \! C\sigma \sqrt{n r} \right\}
\bigcap
\left\{ \|\mDelta\|_{\infty}
	\leq \frac{C\sigma \mut^{1/2} r (\mut \vee \log n)^{1/2}}{\sqrt{n}}
\right\}
\bigcap
\left\{ \|\mDelta\|_{2,\infty} \!\leq\!
	C \sigma r^{1/2}\! \left(\mut \vee \log n\right)^{1/2}\right\},
\end{equation*}
and 
\begin{equation*}
    \calE_2 = \left\{\|\mDelta\| \!\leq\! C\sigma \sqrt{n}\right\}.
\end{equation*}
We first condition on the event $\calE_1 \cap \calE_2$. 
Let $\varepsilon_0 := \|\mZhat - \mZstar\|_1$ and recall the definitions of $S_1$ and $S_2$ from Equations~\eqref{eq:S1:define} and~\eqref{eq:S2:define}, respectively.
Combining Lemmas \ref{lem:key-lower}, \ref{lem:S2:bound} and \ref{lem:S1:bound} yields 
\begin{equation*} \begin{aligned}
\varepsilon_0\beta^2 &\lesssim |S_1| + |S_2| \\
&\lesssim \sigma \varepsilon_0 \left(\beta \sqrt{\frac{r(\mut \vee \log n)}{n}} + L \sqrt{\frac{\log n^2/\varepsilon_0}{n}}\right)
    +  \sigma \varepsilon_0 \left(\frac{\sigma \mut^{1/2} r^{3/2} \left(\mut \vee \log n\right)^{1/2}}{n} + \frac{L}{\sqrt{n}}\right).
\end{aligned}
\end{equation*}
If $\varepsilon_0 = 0$, then the bound in Equation~\eqref{eq:tau-bound} holds trivially, so we focus on the event $\{\varepsilon_0 \geq 0\}$. 
Dividing both sides by $\varepsilon_0$ and rearranging, we obtain
\begin{equation} \label{eq:tau-bound0}
\beta^2 \lesssim \sigma \left(\beta \sqrt{\frac{r(\mut \vee \log n)}{n}} + L \sqrt{\frac{\log n^2/\varepsilon_0}{n}}\right) + \frac{\sigma^2 \mut^{1/2} r^{3/2} \left(\mut \vee \log n\right)^{1/2}}{n}.
\end{equation}
When $\mut \leq \log n$, the last term is absorbed by the others. 
In the case $\mut > \log n$, to ensure the bound is non-trivial, we require that
\begin{equation*}
c\beta > \sigma \sqrt{\frac{\mut r}{n}} \quad \text{and} \quad \frac{\sigma^2 \mut r^{3/2}}{n} \leq c \beta^2
\end{equation*}
for some sufficiently small constant $c$, which is implied by the assumption $\mut r^{3/2} \leq c n \beta^2 / \sigma^2$.
Applying these bounds to Equation~\eqref{eq:tau-bound0} and solving for $\varepsilon_0$ yields the claimed exponential bound in Equation~\eqref{eq:tau-bound}.
Standard low rank perturbation theory implies that the event $\calE_1 \cap \calE_2$ holds with probability at least $1 - O(n^{-6})$ provided that $C$ is chosen sufficiently large. 
Combining this with the high-probability bounds in Lemmas~\ref{lem:S1:bound} and~\ref{lem:S2:bound} completes the proof.
\end{proof}

\subsection{Proof of Corollary \ref{cor:sdp-more-general}}
\label{subsec:proof:cor:sdp-more-general}
\begin{proof}
Define the events
\begin{equation*}
\calE_1 =
\left\{\|\mDelta\|_{\F} \! \leq \! C\sigma \sqrt{n r} \right\}
\bigcap
\left\{ \|\mDelta\|_{\infty}
	\leq \frac{C\sigma \mut^{1/2} r (\mut \vee \log n)^{1/2}}{\sqrt{n}}
\right\}
\bigcap
\left\{ \|\mDelta\|_{2,\infty} \!\leq\!
	C \sigma r^{1/2}\! \left(\mut \vee \log n\right)^{1/2}\right\},
\end{equation*}
and 
\begin{equation*}
    \calE_2 = \left\{\|\mDelta\| \!\leq\! C\sigma \sqrt{n}\right\}.
\end{equation*}
We first fix $\mDelta$ to be any matrix satisfying the conditions of
\begin{equation*}
    \calE = \calE_1 \cap \calE_2. 
\end{equation*}
    Suppose that $\mBstar$ satisfies Equation~\eqref{eq:B-more-and-more-general} and $\mB^{(0)}$ satisfies Equation~\eqref{eq:B-more-general} with the same $\Istar$. 
For all $i,j \in [n]$, let the entries of $\mC$ and $\mC^{(0)}$ be
\begin{equation} \label{eq:C-and-C0}
C_{ij} = (\Wtilo_{ij} + \Deltatil_{ij} + \Bstar_{ij})^2 \quad \text{and} \quad C^{(0)}_{ij} = (\Wtilo_{ij} + \Deltatil_{ij} + B^{(0)}_{ij})^2.
\end{equation}
    
For any $1\leq i,j \leq n$ such that $\left|\Deltatil_{ij} + \Bstar_{ij}\right| < \left|\Deltatil_{ij} + B^{(0)}_{ij}\right|$, one can instead substitue $B^{(0)}_{ij}$ for $\Bstar_{ij}$ and shrink $|\Deltatil_{ij}|$ to $|\Deltatil^{(0)}_{ij}|$, so that the new $\mDeltatil^{(0)} \in \calE$ and satisfies
\begin{equation*}
    \mDeltatil^{(0)} + \mB^{(0)} = \mDeltatil + \mBstar
\end{equation*}
Repeating the proof of Theorem~\ref{thm:sdp:L1-bound}
on $\mDeltatil^{(0)} + \mB^{(0)}$ instead shows that Theorem~\ref{thm:sdp:L1-bound} still holds for the new $\mB^{(0)}$. 
    
Based on the above observation, without loss of generality, we assume that $\left|\Deltatil_{ij} + \Bstar_{ij}\right| \geq \left|\Deltatil_{ij} + B^{(0)}_{ij}\right|$ holds for all $1\leq i,j \leq n$. 
Under Assumption~\ref{assump:mono} and by Lemma~\ref{lem:stoc-dom}, we have that the term $C_{ij}$ has first-order stochastic dominance over $C^{(0)}_{ij}$. 
By Strassen's theorem \citep[see, e.g., Theorem 17.58 in][]{klenke2013probability}, we can construct a matrix $\mC'$ with entries $C'_{ij} \eqdist C^{(0)}_{ij}$, such that $C'_{ij} \leq C_{ij}$ holds almost surely for all $i,j \in [n]$. 
Noting that for $(i,j) \in I_\star^c \times I_\star^c$, $B^\star_{ij} = B^{(0)}_{ij} = 0$, we can pick $C'_{ij} = C_{ij}$ for these entries.


Let $\mC^{(1)} := \left(\mB^{(0)} + \mW\right) \circ \left(\mB^{(0)} + \mW\right)$. 
Let $\mZhat$ be the solution to the SDP in Equation~\eqref{eq:prime:sdp:equal} with $\mC$ defined in Equation~\eqref{eq:C-and-C0}.
Following the same proof as Lemma~\ref{lem:key-lower}, Equation~\eqref{eq:sdp-key-upper2} still holds for the $\mC^{(1)}$ given above, since it only relies on the structure of $\mB^{(0)}$ given in Equation~\eqref{eq:B-more-general}.
We have
\begin{equation*}
\begin{aligned}
\left\|\mZhat-\mZstar\right\|_1
&\leq \frac{5}{2 \beta^2}\left\langle\mZstar-\mZhat, 
	\mC-\E \mC^{(1)} + \mSig - \sigma_0^2 \mJ\right\rangle\\
&\leq \frac{5}{2 \beta^2}\left\langle\mZstar-\mZhat, 
		\mC'-\E \mC^{(1)} + \mSig - \sigma_0^2 \mJ\right\rangle,
\end{aligned}
\end{equation*}
where the second inequality follows the fact that the support of $\mC'-\mC$ and $\mZstar$ as given in Equation~\eqref{eq:def:mZstar} are disjoint, the matrix $\mZhat$ is entrywise nonnegative and 
\begin{equation*}
\left\langle \mZstar - \mZhat, \mC' - \mC\right\rangle = \left\langle \mZhat, \mC - \mC'\right\rangle \geq 0.
\end{equation*}
Noting that the other results from Lemmas \ref{lem:Grothendieck:bound}, \ref{lem:S2:bound} and \ref{lem:S1:bound} holds for any $\mZhat$ satisfying Equation~\eqref{eq:sdp-key-upper2}, 
we conclude that all results proved for $\mBstar$ satisfying Equation~\eqref{eq:B-more-general} remain valid for $\mBstar$ satisfying Equation~\eqref{eq:B-more-and-more-general}. 
Finally, noting that the event $\calE$ holds with probability at least $1 - O(n^{-6})$ by Thereom~\ref{thm:delta_row} completes the proof.
\end{proof}

\subsection{Proof of Theorem \ref{thm:sdp-minimax}}
\label{subsec:thm:sdp-minimax:proof}
\begin{proof}
The result follows by an adaptation of the proof techniques from Theorems~\ref{thm:symmetric:minimax} and~\ref{thm:multiple:minimax-recoverable}, with the key modification being that we apply Lemma~\ref{lem:likelihood:second-moment:beta} in place of Lemma~\ref{lem:likelihood:second-moment}, and Lemma~\ref{lem:minimax:risk:beta} in place of Lemma~\ref{lem:minimax:risk}.
The essential change occurs in the construction of the prior over $\mB$.
Specifically, in the computation of the likelihood ratio as in Equation~\eqref{eq:likelihood:ratio:multiple}, 
we replace the uniform distribution on the sphere $\Unif(b\bbS^{k-1})$ with the uniform distribution over the discrete hypercube $\Unif(\beta\{\pm 1\}^{k})$, 
reflecting the entrywise lower-bound condition in $\calB'_S(\beta, n, m)$.
The remainder of the argument proceeds identically, and thus details are omitted.
\end{proof}

\subsection{Proof of Theorem~\ref{thm:tauhat}}
\label{sec:proof:thm:tauhat}

\begin{proof} 
By Theorem~\ref{thm:delta_row}, when $\kappa r = O(1)$, the estimation error on the subset $\calS$ satisfies
\begin{equation*}
    \left\|\mDelta_{\calS}\right\|_{\infty} \lesssim \frac{\sigma \log n}{\sqrt{n}}.
\end{equation*}
Consequently, the Frobenius norm is bounded as
\begin{equation} \label{eq:DeltaS:frob}
    \left\|\mDelta_{\calS}\right\|_{\F} \leq n \left\|\mDelta_{\calS}\right\|_{\infty} \lesssim \sigma \sqrt{n} \log n.
\end{equation}

By the triangle inequality, the residual matrix $\mWcheck$ defined in Equation~\eqref{eq:mWcheck:define} obeys 
\begin{equation*}
\left\|\mW_{\calS}\right\|_{\F} - \left\|\mDelta_\calS\right\|_{\F}
\leq \left\|\mWcheck\right\|_{\F}
\leq \left\|\mW_{\calS}\right\|_{\F} + \left\|\mDelta_{\calS}\right\|_{\F}, 
\end{equation*}
and Equation~\eqref{eq:DeltaS:frob} thus implies
\begin{equation}\label{eq:WcalS:tri}
\left\|\mW_{\calS}\right\|_{\F} - \sigma \sqrt{n} \log n 
\lesssim \left\|\mWcheck\right\|_{\F}
\lesssim \left\|\mW_{\calS}\right\|_{\F} + \sigma \sqrt{n} \log n.
\end{equation}
To bound $\|\mW_{\calS}\|_{\F}$, note that
\begin{equation*}
    \sum_{i,j \in \calS} \E \left(W_{\calS, i, j}\right)^4 \leq L^2 \sum_{i,j \in \calS} \sigma^2_{ij}, 
    \quad \text{and} \quad \left\|\mW_{\calS} \circ \mW_{\calS}\right\|_{\infty} \leq L^2.
\end{equation*}
Thus, applying Bernstein's inequality as stated in Lemma~\ref{lem:bern-ineq}, it holds with probability at least $1 - O(n^{-6})$ that
\begin{equation*}
    \left|\sum_{i,j \in \calS} (W_{\calS, i, j})^2 - \E \sum_{i,j \in \calS} (W_{\calS, i, j})^2\right| 
    \lesssim \max\left\{L \sqrt{\sum_{i, j \in \calS} \sigma^2_{ij} \log n}, L^2 \log n \right\} 
    \lesssim \sigma^2 n \log n  ,
\end{equation*}
so that 
\begin{equation*}
\left\|\mW_{\calS}\right\|_{\F}^2
\lesssim \sigma^2 n\log n + |\calS|^2 \sigma^2 .
\end{equation*}
Taking square roots and applying this to Equation~\eqref{eq:WcalS:tri}, we obtain
\begin{equation*}
    |\calS| \sigma_{\min} - \sigma\sqrt{n} \log n \lesssim \left\|\mWcheck\right\|_{F} \lesssim \sigma |\calS| + \sigma \sqrt{n} \log n.
\end{equation*}
Since by construction, $|\calS| \geq n/2$, it follows that for $n$ sufficiently large, the threshold $\tau$ defined in Equation~\eqref{eq:tauhat:define} satisfies
\begin{equation*}
    c_0 \sigma \leq \tau \leq C_0 \sigma
\end{equation*}
for some constants $c_0, C_0 > 0$, as desired.
\end{proof}

\subsection{Proof of Theorem~\ref{thm:trunc:sdp}}
\label{subsec:proof:trunc:sdp}

\begin{proof}
We follow the same strategy used in the analysis of the vanilla SDP estimator. 
We begin by considering the case where $\mBstar$ satisfies the structured form in Equation~\eqref{eq:B-more-general}, and proceed in four steps.
In Step 1, analogous to the basic inequality obtained in Lemma~\ref{lem:key-lower}, we obtain the basic inequality for the truncated SDP in Equation~\eqref{eq:basic-ineq:trunc}, 
which yields the key inequality in Equation~\eqref{eq:trunc:important} that serves as the foundation of our analysis. 
In Step 2, we extend the results in Lemma~\ref{lem:S2:bound} to the truncated SDP, adapting to the truncated setting our bound for the projection onto the orthogonal complement space $T^\perp$.
In Step 3, we similarly extend Lemma~\ref{lem:S1:bound} to the truncated SDP.
In Step 4, we finally extend Corollary~\ref{cor:sdp-more-general} to the truncated SDP setting, thereby concluding the proof under the more general signal model in Equation~\eqref{eq:B-more-and-more-general}.

We condition our argument on the event 
\begin{equation*}
\calE
= \! \left\{\|\mDeltatil\|_{\F} \! \leq \! C \sigma \sqrt{r n} \right\}
\bigcap
\left\{ \|\mDeltatil\| \! \leq \! C\sigma \sqrt{n} \right\}
\bigcap
\left\{ \|\mDeltatil\|_{\infty} 
	\! \leq \!
	\frac{C \sigma \mut^{1 / 2} r(\mut \!\vee\! \log n)^{1 / 2}}{\sqrt{n}}
\right\}
\bigcap \left\{c_0 \sigma \!\leq\! \tau \!\leq\! C_0 \sigma\right\}
\end{equation*}
and we first consider the case where 
\begin{equation} \label{eq:beta:trunc}
    \beta = o(\sigma). 
\end{equation}

\paragraph{Step 1: Basic Inequality and Grothendieck's Inequality.}
By the optimality of $\mZhat$, we have
\begin{equation*}
    \left\langle \mZhat-\mZstar, \mCbar\right\rangle \leq 0.
\end{equation*}
Subtracting the expectation $\E \mCbar$ from both sides  yields
\begin{equation} \label{eq:basic-ineq:trunc}
    \left\langle \mZhat-\mZstar, \mCbar - \E \mCbar\right\rangle \leq \left\langle \mZstar-\mZhat, \E \mCbar\right\rangle.
\end{equation}

We now analyze the structure of the expectation $\E \Cbar_{ij}$, 
depending on the location of $(i,j)$ in the support pattern of $\mBstar$.
First, for $i, j \in I^{c}_\star$, we have 
\begin{equation*}
\Cbar_{ij} = \begin{cases}
    \left(\Deltatil_{ij}+\Wtilo_{ij}\right)^2, & \mbox{ if } \left|\Deltatil_{ij}+\Wtilo_{ij}\right| \leq \tau\\
    \tau^2, & \mbox{ otherwise. }
\end{cases} \end{equation*}
By Lemma~\ref{lem:truncated-upper}, we have for some constant $C_1 > 0$,
\begin{equation*}
    \E \Cbar_{ij} \leq C_1 \Deltatil_{ij}^2 + \E \Wbar_{ij}^2,
\end{equation*}
which yields that 
\begin{equation} \label{eq:up-bound1:trunc}
    (\Zstar_{ij} - \Zhat_{ij}) \E \Cbar_{ij} \leq (\Zstar_{ij} - \Zhat_{ij}) \left(C_1 \Deltatil_{ij}^2 + \E \Wbar_{ij}^2\right). 
\end{equation}
Next, for $i \in \Istar$ and $j \in [n]$, under Equation~\eqref{eq:beta:trunc} and the assumption that $\mutil r^{3/2} = O(\sqrt{n})$, we have
\begin{equation*}
    \left|\Bstar_{ij}+\Deltatil _{ij}\right| \leq \beta + \frac{C \sigma \mut^{1/2} r^{3/2} (\mut \wedge \log n)^{1/2}}{\sqrt{n}} = o(\sigma) \leq \tau,
\end{equation*}
by Lemma~\ref{lem:truncated-lower}, there exists a constant $0 < \alpha \leq \frac{1}{2}$ such that
\begin{equation*}
    \E \Cbar_{ij} \geq \alpha^2 (\Bstar_{ij}+\Deltatil _{ij})^2 + \E \Wbar^2_{ij}.
\end{equation*}
Since $\Zstar_{ij} = 0$ for $i \in \Istar$, we have
\begin{equation} \label{eq:up-bound2:trunc}
    (\Zstar_{ij} - \Zhat_{ij}) \E \Cbar_{ij} = - \Zhat_{ij}\E \Cbar_{ij} \leq - \Zhat_{ij} \left[\alpha^2 (\Bstar_{ij}+\Deltatil _{ij})^2 + \E \Wbar^2_{ij}\right].
\end{equation}
For notational convenience, we define
\begin{equation*}
    \Theta_{ij} = \begin{cases}
        C_1 \Deltatil _{ij}^2, & \mbox{ if } i, j \in I_{\star}^c \\
        \alpha^2 \Deltatil _{ij}^2, & \mbox { otherwise. }
    \end{cases}
\end{equation*}
Expanding the inner product in Equation~\eqref{eq:basic-ineq:trunc} and applying Equations~\eqref{eq:up-bound1:trunc} and~\eqref{eq:up-bound2:trunc}, we have
\begin{equation} \label{eq:trunc:important}
\begin{aligned}
\left\langle \mZhat-\mZstar, \mCbar - \E \mCbar\right\rangle
&\leq \left\langle \mZstar-\mZhat, \mTheta + \alpha^2 \mBstar \circ \mBstar + 2\alpha^2 \mBstar \circ \mDeltatil  + \E \mWbar \circ \mWbar\right\rangle. 
\end{aligned}
\end{equation}

Following the same proof as Lemma~\ref{lem:key-lower}, 
one arrives at the bound
\begin{equation} \label{eq:l1-bound:truncated}
\begin{aligned}
    \left\|\mZhat - \mZstar\right\|_1 &\leq \frac{5}{2\alpha^2\beta^2} \left\langle \mZstar - \mZhat, \mCbar - \E \mCbar\right\rangle \\
    &~~~~+ \frac{5}{2\alpha^2\beta^2} \left\langle \mZstar - \mZhat, 2\alpha^2 \mBstar \circ \mDeltatil  
    + \mTheta + \E \mWbar \circ \mWbar - \sigbar^2 \mJ\right\rangle. 
\end{aligned}
\end{equation}
The first term above is bounded by combining Grothendieck's inequality (see Theorem 3.5.1 in \cite{vershynin2018HDP} for a detailed reference) %
followed by Lemma~\ref{lem:Cbar-bern}, which yields that with probability at least $1 - O(n^{-10})$,
\begin{equation} \label{eq:l1-bound:trunc:term1}
\left|\left\langle \mZstar - \mZhat, \mCbar - \E \mCbar \right\rangle\right|
\lesssim \tau^2 n^{3/2} \lesssim \sigma^2 n^{3/2} .
\end{equation}
The second term in Equation~\eqref{eq:l1-bound:truncated} can be handled by the same argument as in Lemma~\ref{lem:key-lower}, with Assumption~\ref{assump:strong:homoscedastic} in place of Assumption~\ref{assump:homoscedastic}. 
We omit the details. 
Applying these bounds to Equation~\eqref{eq:l1-bound:truncated}, we conclude that when either $\beta = \omega(\sigma n^{-1/4})$ or $m = o(n)$, the estimation error satisfies
\begin{equation*}
    \frac{1}{n^2}\left\|\mZhat - \mZstar\right\|_1 = o_{\Pr}(1). 
\end{equation*}

\paragraph{Step 2: Extending Lemma~\ref{lem:S2:bound}.}
To control the contribution from the projection onto the orthogonal complement of the tangent space, we define the term
\begin{equation} \label{eq:def:Sbar2}
\Sbar_2 = \left\langle \calP_{T^\perp} \left(\mZstar - \mZhat\right), 
    \mCbar - \E \mCbar + 2\alpha^2 \mBstar \circ \mDeltatil
	+ \mTheta + \E \mWbar \circ \mWbar - \sigbar^2 \mJ\right\rangle,
\end{equation}
which mirrors the role of $S_2$, defined in Equation~\eqref{eq:S2:define} for the non-truncated setting.
Our goal is to bound $|\Sbar_2|$ in terms of the estimation error $\|\mZhat - \mZstar\|_1$.

Following the same proof as Lemma~\ref{lem:S2:bound}, we first establish bounds on 
\begin{equation*}
    \|\mCbar - \E \mCbar\| \quad \text{and} \quad \left\|2\alpha^2 \mBstar \circ \mDeltatil  + \mTheta + \E \mWbar \circ \mWbar - \sigbar^2 \mJ\right\|.
\end{equation*}
Observe that since $\Cbar_{ij}$ are bounded by $\tau^2$ due to truncation, 
\begin{equation*}
\left|\Cbar_{ij} - \E \Cbar_{ij}\right| \leq 2\tau^2
\quad \text{and} \quad
\sum_{j} \E \Cbar^2_{ij} \leq \tau^4 n, 
\end{equation*}
an application of the matrix Bernstein inequality (see Theorem~\ref{thm:matrix:bernstein}) yields 
\begin{equation*}
\left\|\mCbar - \E \mCbar\right\|
\lesssim \tau^2 \sqrt{n} \lesssim \sigma^2 \sqrt{n}
\end{equation*}
with probability at least $1 - O(n^{-7})$.
Additionally, we invoke the same bounds as in Lemma~\ref{lem:S2:bound}, adapted to the truncated setting:
\begin{equation*} \begin{aligned}
\left\|2\alpha^2 \mBstar \circ \mDeltatil \right\|
&\leq 2\alpha^2 \beta \|\mDeltatil \| \leq 2\alpha^2 \beta \sigma \sqrt{n} \\
\end{aligned} \end{equation*}
and 
\begin{equation*} \begin{aligned}
\left\|\mTheta\right\|
&\leq \left\|\mTheta\right\|_{\F}
\leq \frac{3}{2} \|\mDeltatil \circ\mDeltatil \|_{\F}
\lesssim \sigma^2 \mut^{1/2} r^{3/2} \left(\mut \wedge \log n\right)^{1/2}.
\end{aligned} \end{equation*}
By Assumption~\ref{assump:strong:homoscedastic}, we have
\begin{equation*}
    \left\|\E \mWbar \circ \mWbar - \sigbar^2 \mJ\right\| \lesssim \sigma^2 \sqrt{n}. 
\end{equation*}

Taking absolute values in Equation~\eqref{eq:def:Sbar2}, applying the triangle inequality followed by the above four displays yields 
\begin{equation} \label{eq:Sbar2:bound}
\left|\Sbar_2\right|
\lesssim
\left[ \frac{\sigma^2 \mut^{1/2} r^{3/2}
	\left(\mut \wedge \log n\right)^{1/2}}{n} 
	+ \frac{\sigma^2}{\sqrt{n}}\right] \left\|\mZhat - \mZstar\right\|_1
\end{equation}
with probability at least $1 - O(n^{-7})$.

\paragraph{Step 3: Extending Lemma~\ref{lem:S1:bound}.}
We now extend the analysis of the tangent space component to the truncated setting. 
Similar to Lemmma~\ref{lem:S1:bound}, we define 
\begin{equation} \label{eq:def:Sbar1}
\Sbar_1
= \left\langle \calP_{T} \left(\mZstar - \mZhat\right),
	\mCbar - \E \mCbar + 2\alpha^2 \mBstar \circ \mDeltatil 
	+ \mTheta + \E \mWbar \circ \mWbar - \sigbar^2 \mJ\right\rangle. 
\end{equation}
Following the same proof as Lemma~\ref{lem:S1:bound}, we let
\begin{equation*}
    X_j = \sum_{i=m+1}^n \left( \Cbar_{ij} - \E \Cbar_{ij} \right), \quad \text{for all } j \in [n].
\end{equation*}
By Lemma~\ref{lem:order-statistics-bound}, 
it holds with probability at least $1 - O(n^{-40})$ that
\begin{equation*}
\sum_{j=1}^{\lceil \varepsilon_0 / K \rceil} X_{(j)} 
\leq \left(\frac{\varepsilon_0}{n}\right) \sigma^2
	\sqrt{n \log \left(\frac{n^2}{\varepsilon_0}\right)} ,
\end{equation*}
where $\varepsilon_0 := \|\mZhat - \mZstar\|_1$ and $X_{(j)}$ is the $j$-th order statistic of the random variables $X_1, X_2, \ldots, X_n$. 
The following the same argument as Lemma~\ref{lem:S1:bound} and recalling the definition of $\epsilon_0 = \| \mZhat - \mZstar \|_1$, we can bound $\Sbar_1$ in Equation~\eqref{eq:def:Sbar1} as
\begin{equation} 
\label{eq:Sbar1:bound}
|\Sbar_1|
\lesssim
\left\| \mZhat - \mZstar \right\|_1
\left(\beta \vee \sigma\right) 
\frac{\sigma\log^{1/2}\left(\frac{n^2}{\varepsilon_0}\right)}{\sqrt{n}}. 
\end{equation}

Combining Equations~\eqref{eq:Sbar1:bound} and~\eqref{eq:Sbar2:bound}, we obtain a bound for the truncated SDP equivalent to that given in Theorem~\ref{thm:sdp:L1-bound} for the non-truncated case.

\paragraph{Step 4: Extending Corollary~\ref{cor:sdp-more-general}.}

Our final step is to show that Corollary~\ref{cor:sdp-more-general} also holds for the truncated SDP. 
Suppose that $\mBstar$ satisfies Equation~\eqref{eq:B-more-and-more-general} and $\mB^{(0)}$ satisfies Equation~\eqref{eq:B-more-general}. 
Following the same proof as Corollary~\ref{cor:sdp-more-general}, we assume that without loss of generality, 
\begin{equation*}
    \left|\Deltatil_{ij} + \Bstar_{ij}\right| \geq \left|\Deltatil _{ij} + B_{ij}^{(0)}\right|
\end{equation*}
holds for all $i,j \in [n]$. 
Let 
\begin{equation*}
C_{i j}=\left(\Wtilo_{i j}+\Deltatil _{i j}+B_{i j}^{\star}\right)^2
\quad \text { and } \quad
C_{i j}^{(0)}=\left(\Wtilo_{i j}+\Deltatil _{i j}+B_{i j}^{(0)}\right)^2.
\end{equation*}
Following the same argument as in the proof of Corollary~\ref{cor:sdp-more-general}, there exists a $C'_{ij} \stackrel{d}{=} C^{(0)}_{ij}$, such that $C'_{ij} \leq C_{ij}$ holds almost surely for all $(i,j) \notin I^c_{\star}\times I^c_{\star}$. 
Taking $C'_{ij} = C_{ij}$ for $i,j \in I^c_\star$, set 
\begin{equation*}
    \Cbar'_{ij} = C'_{ij} \wedge \tau^2,
\quad \Cbar_{ij} =  C_{ij} \wedge \tau^2 
\end{equation*}
for all $i,j \in [n]$. 
Following the same argument as in Step 1, one arrives at a bound similar to Equation~\eqref{eq:l1-bound:truncated}, namely
\begin{equation*} \begin{aligned}
\left\|\mZhat - \mZstar\right\|_1
&\leq \frac{5}{2\alpha^2 \beta^2} \left\langle \mZstar - \mZhat, 
				\mCbar - \mE \mCbar'\right\rangle \\
&~~~~ + \frac{5}{2\alpha^2 \beta^2} \left\langle \mZstar - \mZhat,
			2\alpha^2 \mBstar \circ \mDeltatil  + \mTheta 
			+ \E(\mWbar \circ \mWbar) - \sigbar^2\mJ\right\rangle.
\end{aligned} \end{equation*}
Noting that by construction, 
\begin{equation*}
\left\langle \mZstar - \mZhat, \mCbar' - \mCbar \right\rangle 
= \left\langle\mZhat, \mCbar - \mCbar' \right\rangle \geq 0,
\end{equation*}
it follows that 
\begin{equation*} \begin{aligned}
\left\|\mZhat - \mZstar\right\|_1
&\leq \frac{5}{2\alpha^2 \beta^2} \left\langle \mZstar - \mZhat, 
			\mCbar' - \mE \mCbar'\right\rangle \\ 
&~~~~ + \frac{5}{2\alpha^2 \beta^2}\left\langle \mZstar - \mZhat, 
		2\alpha^2 \mBstar \circ \mDeltatil  + \mTheta 
		+ \E (\mWbar \circ \mWbar) - \sigbar^2 \mJ\right\rangle.
\end{aligned} \end{equation*}
Thus, applying the previous analysis in Step 1 to 3 with $\mCbar'$ %
yields the desired bound on $\left\|\mZhat - \mZstar\right\|_1$. 
Finally, by standard low rank perturbation theory and Theorem~\ref{thm:tauhat}, the event $\calE$ holds with probability at least $1 - O(n^{-6})$, completing the proof. 
\end{proof}

\subsubsection{Technical Lemmas for Truncated SDP}
Here we collect technical results related to Theorem~\ref{thm:trunc:sdp}, which concerns the truncated SDP introduced in Section~\ref{sec:truncated-sdp}. 

\begin{lemma} \label{lem:Cbar-bern}
Let $\mCbar = \calT_{\tau^2}(\mYtil \circ \mYtil)$, where $\calT_{\tau^2}(a) = about \wedge \tau^2$ is the entrywise truncation operator. 
Then, the deviation of $\mCbar$ obeys, with probability at least $1-O(n^{-10})$,
\begin{equation*}
\max_{\substack{\vx, \vy \in \{\pm 1\}^n}}
\left| \sum_{i,j} \left(\Cbar_{ij} - \E \Cbar_{ij}\right) x_i y_j \right| 
\lesssim \tau^2 n^{3/2} .
\end{equation*}
\end{lemma}
\begin{proof}
We largely follow the proof of Lemma~\ref{lem:Grothendieck:bound}.
Since each entry of $\mCbar$ is truncated at $\tau^2$, we immediately have the pointwise bounds
\begin{equation*}
\left|\Cbar_{ij} - \E \Cbar_{ij}\right| \leq \tau^2
\qquad \text{ and } \qquad
\E\left(\Cbar_{ij} - \E \Cbar_{ij}\right)^2
\leq \E \Cbar_{ij}^2 \leq \tau^4. 
\end{equation*}
Moreover, because the entries $\Cbar_{ij}$ are functions of independent random variables, we can apply Bernstein's inequality to write that for any $t > 0$, 
\begin{equation*}
\Pr\left[ \left|\sum_{1\leq i \leq j \leq n} 
	\left(\Cbar_{ij} - \E \Cbar_{ij}\right) x_i y_j \right| \geq t \right]
\leq 2 \exp\left(-c \min\left\{\frac{t^2}{n^2 \tau^4}, 
			\frac{t}{\tau^2}\right\}\right).
\end{equation*}
Combined with the union bound over all $\vx, \vy \in \{\pm 1\}^n$, it is straightforward to show that 
\begin{equation*}
    \max_{\substack{\vx, \vy \in \{\pm 1\}}} \left|\sum_{1\leq i \leq j \leq n} \left(\Cbar_{ij} - \E \Cbar_{ij}\right) x_i y_j\right| \lesssim \tau^2 n^{3/2}. 
\end{equation*}
Choosing suitable constants in the bounds above yields the desired probability of failure.
\end{proof}

\begin{lemma}\label{lem:truncated-lower}
Under the same setting as Theorem~\ref{thm:trunc:sdp}, suppose that 
\begin{equation}
\label{eq:trunc:assump}
    \left|\Bstar_{ij} + \Deltatil_{ij}\right| = o(\tau)
\end{equation}
there exists a constant $\alpha > 1/2$ such that for all $i,j \in [n]$,
\begin{equation*}
\E \Cbar_{ij}
\geq \alpha^2 \left(\Bstar_{ij} + \Deltatil_{ij}\right)^2 + \E \Wbar^2_{ij} .
\end{equation*}
\end{lemma}
\begin{proof}
For notational simplicity, let $\theta = \Bstar_{ij} + \Deltatil_{ij}$, dropping the dependence on $i,j$. 
Recall that 
\begin{equation*}
\Cbar_{ij} = \left(\left|\Wtilo_{ij}+\theta\right| \wedge \tau\right)^2
= \begin{cases}
        (\Wtilo_{ij}+\theta)^2 & \mbox{ if } -\tau-\theta \leq \Wtilo_{ij} 
						\leq \tau-\theta\\
        \tau^2 & \mbox{ otherwise }
    \end{cases}
\end{equation*}
and
\begin{equation*}
    \Wbar_{ij} = \begin{cases}
        \Wtilo_{ij} & \mbox{ if } -\tau \leq \Wtilo_{ij} \leq \tau\\
        \tau \sgn{\Wtilo_{ij}} & \mbox{ otherwise. }
    \end{cases}
\end{equation*}
Denote the probability density function of $\Wtilo_{ij}$ by $f$ and the corresponding cumulative distribution function as $F$.
We aim to lower bound $\E \Cbar_{ij} - \E \Wbar_{ij}^2$.
Under the assumption that $\Wtilo_{ij}$ has a symmetric distribution around zero, this differencecan be expressed as 
\begin{equation}\label{eq:diff:trunc}
\begin{aligned}
\E \Cbar_{ij} - \E \Wbar_{ij}^2 &=
   \int_{-\tau-\theta}^{\tau-\theta} (u+\theta)^2 f(u) du - \int_{-\tau}^{\tau} u^2 f(u) du 
\\
&~~~~~~~~
+\tau^2 \left[F(-\tau-\theta) - F(-\tau) + F(\tau) -F(\tau-\theta)\right] \\
&= \int_{-\tau-\theta}^{\tau-\theta} (u+\theta)^2 f(u) du - \int_{-\tau}^{\tau} u^2 f(u) du \\
&~~~~~~~
+\tau^2 \left[ F(\tau) - F(\tau+\theta) + F(\tau) -F(\tau-\theta)\right] \\
\end{aligned}
\end{equation}
We analyze the right-hand quantity term by term.
Since by the Assumption in Equation~\ref{eq:trunc:assump}, $\tau > 2\theta > 0$, we have 
\begin{equation}\label{eq:diff:lip}
\begin{aligned}
&\int_{-\tau-\theta}^{\tau-\theta} (u+\theta)^2 f(u) du 
	- \int_{-\tau}^{\tau} u^2 f(u) du \\
&~~~~= \int_{\tau}^{\tau+\theta} u^2 f(u) du 
	- \int_{\tau-\theta}^{\tau} u^2 f(u) du 
	- 2\theta \int_{\tau-\theta}^{\tau+\theta} u f(u) du 
	+ \theta^2 \left[F(\tau-\theta) - F(-\tau-\theta)\right] .
\end{aligned}
\end{equation}
By Assumption~\ref{assump:mono}, we have  
\begin{equation}\label{eq:lip:1}
\int_{\tau}^{\tau+\theta} u^2 f(u) du - \int_{\tau-\theta}^{\tau} u^2 f(u) du 
\geq \theta\tau^2 \left[ f(\tau+\theta) - f(\tau-\theta)\right],
\end{equation}
\begin{equation}\label{eq:lip:2}
\int_{\tau-\theta}^{\tau+\theta} u f(u) du
\leq 4\theta \tau f(\tau-\theta)
= 4\theta \tau f(\tau) + 4\theta\tau \left[f(\tau-\theta) - f(\tau)\right],
\end{equation}
and 
\begin{equation}\label{eq:lip:3}
\begin{aligned}
F(\tau-\theta) - F(-\tau-\theta)
&= F(\tau) + \left[ F(\tau-\theta) - F(\tau) \right] 
	- \left[1 - F(\tau) + F(\tau) - F(\tau+\theta)\right] \\
&= F(\tau) - F(-\tau) + \left[ F(\tau-\theta) - F(\tau) \right]
	+ \left[F(\tau+\theta) - F(\tau)\right] \\
&\geq F(\tau) - F(-\tau) + \theta f(\tau) 
	+ \theta \left[f(\tau+\theta) - f(\tau)\right]
\end{aligned} \end{equation}
Under the Lipschitz continuous condition given in Assumption~\ref{assump:lipschitz},
\begin{equation*}
    f(\tau+\theta) - f(\tau-\theta) \leq 2\theta \gamma,  \quad F(\tau+\theta) - F(\tau-\theta) \leq 2\theta f(\tau) + 2\theta^2 \gamma.
\end{equation*}
Applying this to Equations~\eqref{eq:lip:1},~\eqref{eq:lip:2} and~\eqref{eq:lip:3} yields, respectively,
\begin{equation*} 
\int_{\tau}^{\tau+\theta} u^2 f(u) du - \int_{\tau-\theta}^{\tau} u^2 f(u) du 
\geq 2\theta^2 \tau^2 \gamma, 
\end{equation*}
\begin{equation*} 
2\theta \int_{\tau-\theta}^{\tau+\theta} u f(u) du
\leq 8\theta^2 \tau f(\tau) + 4\theta^4\tau\gamma, 
\end{equation*}
and
\begin{equation*} 
\theta^2(F(\tau-\theta) - F(-\tau-\theta))
\geq \theta^2(F(\tau) - F(-\tau)) + \theta^3 f(\tau) + \theta^4 \gamma.  
\end{equation*}
Applying the above three bounds to Equation~\eqref{eq:diff:lip}, it follows that
\begin{equation*}
\int_{-\tau-\theta}^{\tau-\theta} (u+\theta)^2 f(u) du
	- \int_{-\tau}^{\tau} u^2 f(u) du 
\geq \theta^2 \left[ F(\tau) - F(-\tau) \right] - 3\theta^2 \tau^2 \gamma 
	- 9\theta^2 \tau f(\tau) - 4\theta^4 \tau \gamma. 
\end{equation*}
Applying this bound to Equation~\eqref{eq:diff:trunc},
\begin{equation*} \begin{aligned}
\E \Cbar_{ij} - \E \Wbar_{ij}^2 
&\ge \theta^2 \left[ F(\tau) - F(-\tau) \right] - 3\theta^2 \tau^2 \gamma 
        - 9\theta^2 \tau f(\tau) - 4\theta^4 \tau \gamma \\
&~~~~~~
+\tau^2 \left[ F(\tau) - F(\tau+\theta) + F(\tau) -F(\tau-\theta)\right] .
\end{aligned} \end{equation*}
By Assumption~\ref{assump:mono}, 
\begin{equation*}
    F(\tau) - F(\tau+\theta) + F(\tau) -F(\tau-\theta) \geq -\theta f(\tau) + \theta f(\tau) \geq 0,
\end{equation*}
it follows that
\begin{equation*}
\E \Cbar_{ij} - \E \Wbar_{ij}^2 
\ge \theta^2 \left[ F(\tau) - F(-\tau) \right] - 3\theta^2 \tau^2 \gamma 
        - 9\theta^2 \tau f(\tau) - 4\theta^4 \tau \gamma .
\end{equation*}
By Assumption~\ref{assump:lipschitz}, we have
\begin{equation*}
3\tau^2 \gamma + 4\theta^2\tau \gamma + 9\tau f(\tau)
\leq \frac{1}{2} (F(\tau) - F(-\tau)),
\end{equation*}
whence
\begin{equation*}
\E \Cbar_{ij} - \E \Wbar_{ij}^2 
\ge \frac{ \theta^2 }{2 } \left[ F(\tau) - F(-\tau) \right].
\end{equation*}
Choosing $\alpha$ so that
\begin{equation*}
    \alpha^2 > \frac{1}{4} \geq \frac{1}{2} (F(\tau) - F(-\tau))
\end{equation*}
and recalling the definition of $\theta$ completes the proof.
\end{proof}

\begin{lemma}\label{lem:truncated-upper}
Under the same setting as Theorem~\ref{thm:trunc:sdp}, suppose that $c_0\sigma \leq \tau \leq C_0 \sigma$, then for all $i,j \in [n]$,
\begin{equation*}
\E \Cbar_{ij} - \E \Wbar_{ij}^2 
\leq C_1 \left(\Bstar_{ij} + \Deltatil_{ij}\right)^2
\end{equation*}
for some constant $C_1 > 0$.
\end{lemma}
\begin{proof}
Denote the probability density function of $\Wtilo_{ij}$ by $f$ and its cumulative distribution function is $F$.
If $\tau \leq 2|\Bstar_{ij} + \Deltatil_{ij}|$, then the bound in Lemma~\ref{lem:truncated-upper} holds trivially, since $\Cbar_{ij} \leq \tau^2$ and $\E \Wbar_{ij}^2 \geq 0$.
Letting $\theta = \Bstar_{ij} + \Deltatil_{ij}$ for notational convenience,
\begin{equation} \label{eq:Cbar:diff:expand} \begin{aligned}
\E \Cbar_{ij} - \E \Wbar_{ij}^2 
&= \int_{-\tau-\theta}^{\tau-\theta} (u+\theta)^2 f(u) du \\
&~~~~~~- \int_{-\tau}^{\tau} u^2 f(u) du +
    \tau^2 \left[F(\tau) - F(\tau+\theta) + F(\tau) -F(\tau-\theta)\right]. 
\end{aligned} \end{equation}
By Assumptions~\ref{assump:mono} and~\ref{assump:lipschitz}, we have
\begin{equation*} \begin{aligned}
F(\tau) - F(\tau+\theta) + F(\tau) - F(\tau-\theta)
&\leq \theta \left[f(\tau-\theta) - f(\tau+\theta)\right] \\
&\leq 2\theta^2 \gamma. 
\end{aligned} \end{equation*}
Applying this to Equation~\eqref{eq:Cbar:diff:expand}, 
\begin{equation} \label{eq:Cbar:diff:inter} \begin{aligned}
\E \Cbar_{ij} - \E \Wbar_{ij}^2 
\le \int_{-\tau-\theta}^{\tau-\theta} (u+\theta)^2 f(u) du 
	- \int_{-\tau}^{\tau} u^2 f(u) du + 2\theta^2 \tau^2 \gamma.
\end{aligned} \end{equation}

Suppose that $\theta > 0$. 
The case where $\theta < 0$ folows similarly, since the distribution of $\Wtilo_{ij}$ is symmetric around zero.
For $\tau > 2\theta > 0$, by Assumptions~\ref{assump:mono} and~\ref{assump:lipschitz}, we have 
\begin{equation*} \begin{aligned}
&\int_{-\tau-\theta}^{\tau-\theta} (u+\theta)^2 f(u) du 
	- \int_{-\tau}^{\tau} u^2 f(u) du \\
&~~~~= \int_{\tau}^{\tau+\theta} u^2 f(u) du 
	- \int_{\tau-\theta}^{\tau} u^2 f(u) du 
	- 2\theta \int_{\tau-\theta}^{\tau+\theta} u f(u) du
	+ \theta^2 \left[F(\tau-\theta) - F(-\tau-\theta)\right]\\
&~~~~\leq \int_{\tau}^{\tau+\theta} u^2 f(u) du 
	- \int_{\tau-\theta}^{\tau} u^2 f(u) du + \theta^2
\leq \theta^2 (1+ 4\tau f(\tau) ).
\end{aligned}
\end{equation*}
Applying this to Equation~\eqref{eq:Cbar:diff:inter},
\begin{equation} \label{eq:Cbar:diff:final}
    \E \Cbar_{ij} - \E \Wbar_{ij}^2 \leq (1+2\gamma \tau^2+4\tau f(\tau)) \theta^2.  
\end{equation}
By Assumption~\ref{assump:mono}, we have
\begin{equation}
    \label{eq:trunc:upper:bound:f}
    \frac{1}{2}\tau f(\tau) \leq \int_{\frac{\tau}{2}}^{\tau} f(u) du = F(\tau) - F\left(\frac{\tau}{2}\right) \leq 1. 
\end{equation}
By Assumption~\ref{assump:lipschitz} and the assumption that $\tau \geq c_0 \sigma$, we have 
\begin{equation*}
    \tau^2 \gamma \leq c' \frac{\sigma}{\tau} \leq \frac{c'}{c_0}
\end{equation*}
for some sufficiently small constant $c'$. 
Combining this with Equations~\eqref{eq:Cbar:diff:final} and~\eqref{eq:trunc:upper:bound:f} yields 
\begin{equation*}
    \E \Cbar_{ij} - \E \Wbar_{ij}^2 \leq (1+2\gamma \tau^2 +4\tau f(\tau)) \theta^2 \leq C_1 \theta^2, 
\end{equation*}
for some constant $C_1 > 0$.
\end{proof}

\subsection{Proof of Theorem~\ref{thm:sdp:multiple}}
\label{sec:proof:sdp:multiple}

\begin{proof}
The proof follows the same structure as Theorem~\ref{thm:trunc:sdp}. 
We highlight the key differences, omitting some details.
We remind the reader that $\mYtil_1$ and $\mYtil_2$ are residual matrices introduced in Equation~\eqref{eq:Ytil:define}.
The matrix $\mDeltatil$ represents the estimation error matrix $\mDelta$ restricted to rows and columns indexed by $\calU$.
The matrices $\mWtilo_1$ and $\mWtilo_2$ are the noise matrices restricted to the same index set.
Following the discussion at the beginning of Section~\ref{sec:one-step-recovery}, the results are obtained under $\calU = [n]$ without loss of generality.
Under this notation, we have the following decomposition:
\begin{equation*}
    \mYtil_l = \mBstar + \mWtilo_l + \mDeltatil, \quad l = 1,2.
\end{equation*}
All of the following analysis will be conditioned on the event 
\begin{equation*}
\left\{\|\mDeltatil\|_{\F} \leq C\sigma \sqrt{n r} \right\}
\bigcup
\left\{ \|\mDeltatil\| \leq C\sigma \sqrt{n} \right\}
\bigcup
\left\{ \|\mDeltatil\|_{\infty}
	\leq \frac{C \sigma \mutil^{1/2} r(\mutil \vee \log n)^{1/2}}{\sqrt{n}}
\right\}.
\end{equation*} 
One can then use a high probability bound obtained in Theorem~\ref{thm:delta_row} to show that the above event holds with the desired probability. 

We follow the same four steps as in Theorem~\ref{thm:trunc:sdp}.
The first step is to extend the results in Lemmas~\ref{lem:key-lower} and~\ref{lem:Grothendieck:bound}.
The second step is to extend Lemma~\ref{lem:S2:bound} to the case where the noise matrices are not independent.
The third step is to extend Lemma~\ref{lem:S1:bound} to the case where the noise matrices are not independent.
The final step is to extend Corollary~\ref{cor:sdp-more-general} to the case where $\mBstar$ follows a more general form given in Equation~\eqref{eq:B-more-and-more-general}.

\paragraph{Step 1: Basic Inequality and Grothendieck's Inequality.}
By the optimality of $\mZhat$,
\begin{equation} \label{eq:Ctil:basic}
    \left\langle \mZhat - \mZstar, \mCtil \right\rangle \leq 0.
\end{equation}
Letting $\mCtil^{(1)} = \left(\mBstar + \mWtilo_1\right) \circ \left(\mBstar + \mWtilo_2\right)$, we have 
\begin{equation*}
    \E \mCtil^{(1)} = \mBstar \circ \mBstar = \beta^2 \left[\mJ - \begin{pmatrix}
        \mJ_1 & \\
        & \mJ_2
    \end{pmatrix}\right]
\end{equation*}
and it follows from Equation~\eqref{eq:Ctil:basic} that
\begin{equation*}
\left\langle \mZhat - \mZstar, \mCtil - \E \mCtil^{(1)} \right\rangle 
\leq \left\langle \mZstar - \mZhat, \E \mCtil^{(1)} \right\rangle 
= \beta^2 \left\langle\mZhat, \begin{pmatrix}
    \mJ_1 & \\
        & \mJ_2    
    \end{pmatrix} - \mJ
        \right\rangle. 
\end{equation*}
Thus, following the same argument as in Theorem~\ref{lem:key-lower}, we have 
\begin{equation*}
    \left\|\mZhat - \mZstar\right\|_1 \leq \frac{5}{2\beta^2} \left\langle \mZstar - \mZhat, \mCtil - \E \mCtil^{(1)} \right\rangle. 
\end{equation*}
Expanding $\mCtil$ in the right hand side of the above display gives 
\begin{equation*}
    \left\|\mZhat - \mZstar\right\|_1 \leq \frac{5}{2\beta^2} \left\langle \mZstar - \mZhat, \mCtil^{(1)} - \E \mCtil^{(1)} + \mDeltatil \circ \mWtilo_2 + \mDeltatil \circ \mWtilo_1 + 2\mBstar \circ \mDeltatil + \mDeltatil \circ \mDeltatil\right\rangle.
\end{equation*}
Among the terms at the right hand side of the above display, we focus on bounding the leading term 
\begin{equation*}
\left\langle \mZhat - \mZstar, \mCtil^{(1)} - \E \mCtil^{(1)}\right\rangle.
\end{equation*}
The other terms can be handled in the same manner as in Lemmas~\ref{lem:key-lower} and~\ref{lem:Grothendieck:bound}, and thus details are omitted.
To control the leading term, it boils down to bounding 
\begin{equation*}
\left\langle \mZhat - \mZstar, \mWtilo_1 \circ \mWtilo_2 \right\rangle 
\leq 2 c_0 \max_{\vx, \vy \{\pm 1\}^n }
	\left|\sum_{i,j} \Wtilo_{1,ij} \Wtilo_{2,ij} x_i y_j\right|,
\end{equation*}
where we have used Grothendieck's inequality as stated in Theorem 3.5.1 in \cite{vershynin2018HDP}.  
Letting $X_{ij} = \Wtilo_{1,ij} \Wtilo_{2,ij} x_i y_j$ for $i,j \in [n]$, we have
\begin{equation*}
    |X_{ij}| \leq L^2, \quad \text{and} \quad \E X_{ij}^2 = \sigma_{1,ij}^2 \sigma_{2,ij}^2 \leq \sigma^4.  
\end{equation*}
Applying Bernstein's inequality in Lemma~\ref{lem:bern-ineq}, it holds for all $t \leq n^2 \sigma^4 / L^2$ that 
\begin{equation*}
\Pr\left[\left| \sum_{1\leq i\leq j\leq n} X_{ij}\right| \geq t\right] 
\leq 2 \exp\left(-c \min\left\{\frac{t^2}{n^2 \sigma^4}, 
			\frac{t}{L^2}\right\}\right) 
= 2\exp \left\{ - \frac{c t^2}{\sigma^4 n^2}\right\} .
\end{equation*}
Applying the union bound and following the same proof as Lemma~\ref{lem:Grothendieck:bound}, with probability at least $1 - O(n^{-10})$,
\begin{equation*} 
\max_{\vx, \vy \in \{\pm 1\}^n}
\left|\sum_{1\leq i \leq j\leq n} X_{ij}\right| 
\lesssim \sigma^2 n^{3/2} .
\end{equation*}
Combining all the results and following the same argument as in Lemmas~\ref{lem:key-lower} and~\ref{lem:Grothendieck:bound}, we have
\begin{equation*}
\left\|\mZhat - \mZstar\right\|_1 
\lesssim \frac{\sigma^2 n^{3/2}}{\beta^2} 
	+ \frac{\sigma r^{1/2} n^{3/2}}{\beta}.  
\end{equation*}

\paragraph{Step 2: Extending Lemma~\ref{lem:S2:bound}.}

Define 
\begin{equation*}
\Stil_2
= \left\langle \calP_{T^{\perp}}\left(\mZstar - \mZhat\right), 
			\mCtil - \E \mCtil^{(1)} \right\rangle.
\end{equation*}
We follow the same argument as in Lemma~\ref{lem:S2:bound}, the key difference being that we need to bound the term 
\begin{equation*}
        \left\|\mWtilo_1 \circ \mWtilo_2\right\|
\end{equation*}
instead of $\|\mWtilo \circ \mWtilo\|$.
Since 
\begin{equation*}
\left|\Wtilo_{1,ij} \Wtilo_{2,ij}\right| \leq L^2, \quad \text{and} \quad 
\sum_{j} \left(\Wtilo_{1,ij} \Wtilo_{2,ij}\right)^2 \leq \sigma^4 n,
\end{equation*}
the matrix Bernstein inequality in Theorem~\ref{thm:matrix:bernstein} yields 
\begin{equation*}
    \left\|\mWtilo_1 \circ \mWtilo_2\right\| \lesssim \sigma^2 \sqrt{n} + L^2 \sqrt{\log n}.
\end{equation*}
By the same argument as in the proof of Lemma~\ref{lem:S2:bound}, we have 
\begin{equation*}
\left\|\mCtil - \E \mCtil^{(1)}\right\|
\lesssim \sigma^2 \mutil^{1/2} r^{3/2} (\mutil \vee \log n)^{1/2} 
+ \sigma^2 \sqrt{n} + L^2 \sqrt{\log n}
\end{equation*}
and 
\begin{equation*} \begin{aligned}
|\Stil_2| &\leq \tr\left(\calP_{T^\perp}(\mZhat)\right) 
	\left\|\mCtil - \E \mCtil^{(1)}\right\|
\leq \frac{\left\|\mCtil - \E \mCtil^{(1)}\right\|}{K}
		\left\|\mZhat - \mZstar\right\|_1 \\
&\lesssim 
\left\|\mZhat - \mZstar\right\|_1\
  \left(\frac{\sigma^2 \mutil^{1/2} r^{3/2} (\mutil\vee\log n)^{1/2}}{n} 
+ \frac{\sigma^2}{\sqrt{n}} + \frac{L^2 \sqrt{\log n}}{n}\right).  
\end{aligned} \end{equation*}

\paragraph{Step 3: Extending Lemma~\ref{lem:S1:bound}.}

Let 
\begin{equation*}
    \Stil_1 = \left\langle \mZstar - \mZhat, \calP_{T}\left(\mCtil - \E\mCtil^{(1)}\right)\right\rangle
\end{equation*}
and define $\mXitil = \mW_1 \circ \mW_2$.
We follow the proof of Lemma~\ref{lem:S1:bound}, with the difference that we need to bound 
\begin{equation*}
\Stil_{10}=\left\langle\mZhat - \mZstar, \calP_{T}\left(\mXitil 
	+ \mXitil^\top\right)\right\rangle
\end{equation*}
instead of $\mXitil$ being $\mWtilo \circ \mWtilo$.
Following the same argument as in Lemma~\ref{lem:S1:bound}, we let
\begin{equation*}
    X_{j} := \left|\sum_{k=m+1}^n \xitil_{kj}\right|
\end{equation*}
for $j \in [n]$. The key quantity to control is the sum of order statistics
\begin{equation*}
    \sum_{j=1}^{\lceil \varepsilon_0 / K \rceil} X_{(j)},
\end{equation*}
where $\varepsilon_0 := \left\|\mZhat - \mZstar\right\|_1$.
To apply Lemma~\ref{lem:order-statistics-bound}, we note that the variance and magnitude parameters are given by
\begin{equation*}
\nu_{\xitil} 
\lesssim n \sigma^4, \quad \text{and} \quad L_{\xitil}^2 \lesssim L^2,
\end{equation*}
which implies that with probability at least $1 - O(n^{-40})$,
\begin{equation*}
\sum_{j=1}^{\lceil \varepsilon_0 / K \rceil} X_{(j)} 
\lesssim \frac{\varepsilon_0}{n} 
\max\left\{ \sigma^2 \sqrt{n \log \left(\frac{n^2}{\varepsilon_0}\right)}, 
    	L^2 \log \left(\frac{n^2}{\varepsilon_0}\right) \right\} .
\end{equation*}
The remaining terms in $\Stil_1$ follow similarly to Lemma~\ref{lem:S1:bound}, leading to
\begin{equation*}
|\Stil_1|
\lesssim 
\varepsilon_0 \left( \sigma^2 \sqrt{\frac{\log (n^2/\varepsilon_0)}{n}} 
+ \frac{L^2}{n} \log \left(\frac{n^2}{\varepsilon_0}\right) 
+ \sigma \beta \sqrt{\frac{r (\mutil \vee \log n)}{n}}\right). 
\end{equation*}
Under the setting where $\mutil r^{3/2} \leq c n \beta^2 / \sigma^2$, 
it follows that 
\begin{equation*}
\beta^2 \left\|\mZhat - \mZstar\right\|_1 
\lesssim \varepsilon_0 \left( \sigma^2 \sqrt{\frac{\log(n^2/\varepsilon_0)}{n}}
	+ \frac{L^2}{n} \log \left(\frac{n^2}{\varepsilon_0}\right)\right),
\end{equation*}
which implies that
\begin{equation*}
\left\|\mZhat - \mZstar\right\|_1 
\lesssim n^2 \max\left\{ \exp\left(-\frac{n \beta^2}{L^2}\right), 
		\exp\left(-\frac{n \beta^4}{\sigma^4}\right)\right\}. 
\end{equation*}

\paragraph{Step 4: Extending Corollary~\ref{cor:sdp-more-general}.}
Our final step is to extend the previous results to the case where $\mBstar$ follows Equation~\eqref{eq:B-more-and-more-general}. 
For all $i,j \in [n]$, 
let 
\begin{equation*}
\Ctil_{ij} 
= \left(\Wtilo_{1,ij} + \Deltatil_{ij} 
+ \Bstar_{ij}\right) \left(\Wtilo_{2,ij} + \Deltatil_{ij} + \Bstar_{ij}\right)
\end{equation*}
and 
\begin{equation*}
\Ctil^{(0)}_{ij}
= \left(\Wtilo_{1,ij} + \Deltatil_{ij} + B^{(0)}_{ij}\right) 
	\left(\Wtilo_{2,ij} + \Deltatil_{ij} + B^{(0)}_{ij}\right),
\end{equation*}
where $\mB^{(0)}$ follows Equation~\eqref{eq:B-more-general}. 
By construction of $\mBstar$, we have 
\begin{equation} \label{eq:increase:Bstar}
    \Deltatil_{ij} + \Bstar_{ij} \geq \Deltatil_{ij} + B^{(0)}_{ij}
\end{equation}
for all $1\leq i,j \leq n$.
Let 
\begin{equation*}
\Theta_{ij} := \Deltatil_{ij} + \Bstar_{ij} \quad \text{and} \quad
\Theta_{ij}^{(0)} := \Deltatil_{ij} + B^{(0)}_{ij}.
\end{equation*}
Our goal is then to show that for any $t \in \R$, 
\begin{equation*}
\Pr\left[\left(\Wtilo_{1,ij}+\Theta_{ij}\right)
		\left(\Wtilo_{2,ij}+\Theta_{ij}\right) \geq t\right]
\geq
\Pr\left[\left(\Wtilo_{1,ij}+\Theta^{(0)}_{ij}\right)
		\left(\Wtilo_{2,ij}+\Theta^{(0)}_{ij}\right) \geq t\right].
\end{equation*}
Let $B_{k, ij} := \Wtilo_{k,ij} + \Theta_{ij}$ and $A_{k, ij} := \Wtilo_{k,ij} + \Theta^{(0)}_{ij}$. 
By Equation~\eqref{eq:increase:Bstar}, we have $A_{k,ij} \preceq B_{k,ij}$ for $k = 1,2$.
By Lemma~\ref{lem:stoc:order}, it follows that $A_{1,ij} A_{2,ij} \preceq B_{1,ij} B_{2,ij}$. 
The remainder of the argument follows the same argument as in the proof of Corollary~\ref{cor:sdp-more-general}, and so details are omitted.
\end{proof}

\subsection{Auxilary Lemmas for SDP Proofs} \label{subsec:sdp:aux}

\begin{lemma}\label{lem:Zhat-l1}
Under the same assumptions as in Lemma~\ref{lem:key-lower}, 
\begin{equation*}
\left\|\mZhat - \mZstar \right\|_1 < 5 \sum_{i=1}^m \sum_{j=m+1}^n \Zhat_{ij}
\end{equation*}
\end{lemma}
\begin{proof}
Let 
\begin{equation} \label{eq:zhat-blocks}
    \mZhat_1 = \left(\Zhat_{ij}\right)_{i,j \in \Istar} \quad \text{and} \quad \mZhat_2 = \left(\Zhat_{ij}\right)_{i,j \in I_\star^c}.
\end{equation}
Without loss of generality, we assume that $\Istar = [m]$ as we can always relabel the indices.
Let $\alpha := \tr \mZhat_1$. 
By the constrain $\Zhat_{ii} \leq 1$ for all $i \in [n]$ in Equation~\eqref{eq:prime:sdp}, we have 
\begin{equation} \label{eq:zhat-block-trace}
    \alpha = \tr \mZhat_1 = \sum_{i=1}^m \Zhat_{ii} \leq m \quad \text{and} \quad \tr \mZhat_2 = \sum_{i=m+1}^n \Zhat_{ii} = K - \alpha. 
\end{equation}
We note that if 
\begin{equation} \label{eq:zhat-block-large}
\sum_{i=1}^m \sum_{j=m+1}^n \Zhat_{ij}
\leq \frac{1}{2}\sum_{i=1}^m \sum_{j=1}^m \Zhat_{ij},
\end{equation}
then expanding the inner product between $\mZhat$ and $\mJ$ yields
\begin{equation*}
\begin{aligned}
    \left\langle \mZhat, \mJ\right\rangle &= \sum_{i=1}^m \sum_{j=1}^m \Zhat_{ij} + 2\sum_{i=1}^m \sum_{j=m+1}^n \Zhat_{ij} + \sum_{i=m+1}^n \sum_{j=m+1}^n \Zhat_{ij} \\
    &\leq 2 \sum_{i=1}^m \sum_{j=1}^m \frac{\Zhat_{ii} + \Zhat_{jj}}{2} + \sum_{i=m+1}^n \sum_{j=m+1}^n \frac{\Zhat_{ii} + \Zhat_{jj}}{2} \\
    &\stackrel{(i)}{=} 2 m \operatorname{tr}(\mZhat_1) + K \operatorname{tr}(\mZhat_2)\\
    &\stackrel{(ii)}{=} 2 m\alpha + K(K-\alpha) = K^2 - \alpha(n - 3m),
\end{aligned}
\end{equation*}
where the inequality follows from the semidefiniteness of $\mZhat$ that $2\Zhat_{ij} \leq \Zhat_{ii}+\Zhat_{jj}$ for all $i,j \in [n]$ and Equation~\eqref{eq:zhat-block-large}, equality $(i)$ follows from the definition of $\mZhat_1$ and $\mZhat_2$ in Equation~\eqref{eq:zhat-blocks} and equality $(ii)$ follows from Equation~\eqref{eq:zhat-block-trace}.
Hence, under the constraint that $\langle \mZhat, \mJ \rangle = K^2$ and the assumption that $3m < n$, we must have
\begin{equation} \label{eq:sdp-important-observe}
\sum_{i=1}^m \sum_{j=m+1}^n \Zhat_{ij} > \frac{1}{2}\left(\sum_{i=1}^m \sum_{j=1}^m \Zhat_{ij}\right).
\end{equation}
It follows that 
\begin{equation*}
\begin{aligned}
\|\mZhat - \mZstar\|_1 &= \sum_{i=1}^m  \sum_{j=1}^m \Zhat_{ij} +  2\sum_{i=1}^m  \sum_{j=m+1}^n \Zhat_{ij} + \sum_{i=m+1}^n \sum_{j=m+1}^n (1 - \Zhat_{ij}) \\
    &= \sum_{i=1}^m  \sum_{j=1}^m \Zhat_{ij} +  2\sum_{i=1}^m  \sum_{j=m+1}^n \Zhat_{ij} + K^2 - \left(K^2 - \sum_{i=1}^m \sum_{j=1}^m \Zhat_{ij} -  2\sum_{i=1}^m  \sum_{j=m+1}^n \Zhat_{ij}\right) \\
    &= 2 \sum_{i=1}^m  \sum_{j=1}^m \Zhat_{ij} + 4\sum_{i=1}^m  \sum_{j=m+1}^n \Zhat_{ij} < 5\sum_{i=1}^m  \sum_{j=m+1}^n \Zhat_{ij},
\end{aligned}
\end{equation*}
where the inequality follows from Equation~\eqref{eq:sdp-important-observe}. 
\end{proof}

\begin{lemma}\label{lem:S2-aux}
Under the same assumptions as in Lemma~\ref{lem:S2:bound},
\begin{equation*}
\tr \left(\calP_{T^\perp}(\mZhat)\right) 
\leq \frac{1}{K} \|\mZhat - \mZstar\|_1.
\end{equation*}
\end{lemma}
\begin{proof}
Noting that $\mZhat \succeq 0$, we have 
\begin{equation*}
\calP_{T^\perp} (\mZhat)
= \left(\mI-\frac{1}{K}\mZstar\right)\mZhat\left(\mI-\frac{1}{K}\mZstar\right) 
\succeq 0.
\end{equation*}
Since $\calP_{T^\perp} (\mZstar) = 0$ by definition of $\calP_{T^\perp}$ in Equation~\eqref{eq:calP:define}, we have 
\begin{equation*}
\tr\left(\calP_{T^\perp}(\mZhat)\right)
= \tr\left(\calP_{T^\perp}(\mZhat-\mZstar)\right)
=
\tr \left(\mI-\frac{1}{K}\mZstar\right)(\mZhat - \mZstar)
		\left(\mI-\frac{1}{K}\mZstar\right) .
\end{equation*}
By cyclicity of the trace and the fact that $\mI-\frac{1}{K}\mZstar$ is a projection matrix, it follows that
\begin{equation*} \begin{aligned}
\tr \calP_{T^\perp}(\mZhat)
&= \tr \left[\left(\mI-\frac{1}{K}\mZstar\right)(\mZhat - \mZstar)\right] \\
&= \frac{1}{K}\tr \left[\mZstar (\mZstar-\mZhat)\right] .
\end{aligned}
\end{equation*}
where the last equality holds since $\tr(\mZstar) = \tr(\mZhat) = K$. 
Since $\mZstar_{ik} = 1$ for $i, k \in I_\star^c$ and $\mZstar_{ik} = 0$ otherwise, we have
\begin{equation*}
\begin{aligned}
\tr \calP_{T^\perp}(\mZhat)
&=
\frac{1}{K}\sum_{i=1}^n\sum_{k=1}^n \Zstar_{ik} (\Zstar_{ik} - \Zhat_{ik})\\
    &= \frac{1}{K} \sum_{i=m+1}^n \sum_{k=m+1}^n (\Zstar_{ik} - \Zhat_{ik})
    \leq \frac{1}{K} \|\mZhat - \mZstar\|_1,
\end{aligned}
\end{equation*}
as we set out to show.
\end{proof}

\begin{lemma} \label{lem:reduce-to-order}
For any matrix $\mA$ such that $\|\mA\|_{\infty} \leq 1$ and $\|\mA\|_1 \leq \varepsilon_0$ and any matrix $\mXi$, we have 
\begin{equation*}
\frac{1}{K} \left|\left \langle \mA, \mZstar \mXi \right \rangle\right| 
\leq \sum_{j=1}^{\lceil \varepsilon_0 / K\rceil} X_{(j)}
\quad \text{and} \quad
\frac{1}{K} \left|\left \langle \mA \frac{ \mZstar }{ K }, 
	\mZstar \mXi \right \rangle\right|
\leq \sum_{j=1}^{\lceil \varepsilon_0 / K\rceil} X_{(j)},
\end{equation*}
where $\mZstar$ is given in Equation~\eqref{eq:def:mZstar}, $K = n - m$, $X_j := \left|\sum_{k \in I^c_\star} \Xi_{kj}\right|$ for $j \in [n]$ and $X_{(j)}$ denotes the $j$-th order statistic of the random variables $X_1, \ldots, X_n$.
\end{lemma}
\begin{proof}
Without loss of generality, we assume that $\Istar = [m]$. 
By definition of $\mZstar$, we have $\Zstar_{ik} = 0$ if either $i \in \Istar = [m]$ or $k \in [m]$. 
It follows that
\begin{equation*}
\begin{aligned}
\frac{1}{K}\langle \mA, \mZstar \mXi \rangle 
&= \frac{1}{K} \sum_{1\leq i,j \leq n} A_{ij} \sum_{k=1}^n \Zstar_{ik} \Xi_{kj}
= \sum_{j=1}^n \left(\sum_{i=m+1}^n \frac{ A_{ij} }{ K }\right) 
	\sum_{k=m+1}^n \Xi_{kj}.
\end{aligned} \end{equation*}
Let $a_j := \left|\sum_{i=m+1}^n A_{ij}\right| / K$ for $j \in [n]$. 
By the assumption that $\|\mA\|_{\infty} \leq 1$, we have $a_j \leq 1$ for all $j \in [n]$.
It follows from the above display and the definition of $X_j$ that 
\begin{equation}\label{eq:reduce-to-order}
\begin{aligned}
\frac{1}{K}\langle \mA, \mZstar \mXi \rangle 
&\leq \sum_{j=1}^n a_j X_j 
\leq \sum_{j=1}^{\lceil \sum_{i=1}^n a_i \rceil} X_{(j)} 
\leq \sum_{j=1}^{\lceil \varepsilon_0 / K\rceil} X_{(j)}, 
\end{aligned} \end{equation}
where the second inequality follows from $a_j \leq 1$ for all $j \in [n]$ and properties of order statistics, while the last inequality follows from the assumption that $\|\mA\|_1 \leq \varepsilon_0$.
We also have that 
\begin{equation*}
\left(\frac{1}{K} \mZstar \mA\right)_{ij} 
= \frac{1}{K} \sum_{k=1}^n \Zstar_{ik} A_{kj} \leq \|\mA\|_{\infty} \leq 1,
\end{equation*}
by assumption on $\|\mA\|_{\infty}$, and
\begin{equation*} \begin{aligned}
\sum_{i=1}^n \sum_{j=1}^n \left(\frac{1}{K} \mZstar \mA\right)_{ij} 
&= \frac{1}{K}\sum_{i=1}^n \sum_{j=1}^n \sum_{k=1}^n Z^\star_{ik} A_{kj} 
= \sum_{j=1}^n \sum_{k=1}^n A_{kj} 
	\left(\sum_{i=1}^n \frac{ \Zstar_{ik} }{ K } \right)\\
&\leq \|\mA\|_1 \leq \varepsilon_0,
\end{aligned} \end{equation*}
from which, respectively, $\|\mZstar\mA/K\|_{\infty} \leq 1$ and $\|\mZstar\mA/K\|_{1} \leq \varepsilon_0$.
Hence, the matrix $\mZstar \mA / K$ satisfies the same assumptions as $\mA$.
Replacing $\mA$ by $\mZstar \mA / K$ in Equation~\eqref{eq:reduce-to-order} yields that
\begin{equation*}
\frac{1}{K} \left|\left \langle \mA \frac{ \mZstar }{ K }, 
		\mZstar \mXi \right \rangle\right| 
\leq \sum_{j=1}^{\lceil \varepsilon_0 / K\rceil} X_{(j)} ,
\end{equation*} 
as we set out to show.
\end{proof}

\begin{lemma}\label{lem:stoc-dom}
Under Assumption~\ref{assump:mono} and assuming that the distribution of $W_{ij}$ is symmetric about $0$ for all $i,j\in[n]$, for any $a, \delta > 0$, for any $1\leq i, j \leq n$, both $(W_{ij}+a+\delta)^2$ and $(W_{ij}-a-\delta)^2$ have first-order stochastic dominance over $(W_{ij}-a)^2$ and $(W_{ij}+a)^2$, respectively. 
\end{lemma}
\begin{proof}
Denote the cumulative distribution function of $W_{ij}$ by $F(x) = \Pr(W_{ij} \leq x)$ for a given pair of $i,j \in [n]$.
We first show that 
\begin{equation}\label{eq:dom:eq1}
    F(x+a+\delta) + F(x-a-\delta) \leq F(x+a) + F(x-a) .
\end{equation}
Let $(v)_{+} := \max\{0, v\}$.
If $x \geq a$, then using the symmetry about zero assumption of $W_{ij}$, we have
\begin{equation*} \begin{aligned}
    F(x+a+\delta) - F(x+a) &= \int_{x+a}^{x+a+\delta} f(u) du \\
    &\leq \int_{(x-a-\delta)_{+}}^{x-a} f(u) du + \int_{0}^{(a+\delta-x)_{+}} f(u) du\\
    &= \int_{x-a-\delta}^{x-a} f(u) du = F(x-a) - F(x-a-\delta). 
\end{aligned} \end{equation*} 
Rearranging the above equation yields that Equation~\eqref{eq:dom:eq1} holds when $x \geq a$.
Similarly, if $x<a$, we have
\begin{equation*}
\begin{aligned}
F(x-a) - F(x-a-\delta) &= \Pr(x-a-\delta \leq W_{ij} \leq x-a) \\
&= \Pr(a-x \leq W_{ij} \leq a+\delta-x) = \int_{a-x}^{a+\delta-x} f(u) du\\
&\geq \int_{a+x}^{a+x+\delta} f(u) du = F(x+a+\delta) - F(x+a),
\end{aligned} \end{equation*}
which establishes that Equation~\eqref{eq:dom:eq1} also holds for $x < a$.
    
Noting that 
\begin{equation*} \begin{aligned}
\Pr\left[\left(W_{ij}+a\right)^2 \geq x^2\right] 
&= \Pr\left(W_{ij} \geq x-a\right) + \Pr\left(W_{ij} \leq -a-x\right)\\
    &= \Pr\left(W_{ij} \geq x-a\right) + \Pr\left(W_{ij} \geq x+a\right) = 2 - (F(x-a) + F(x+a))
\end{aligned} \end{equation*}
and
\begin{equation*}
\Pr\left[\left(W_{ij}+a+\delta\right)^2 \geq x^2\right]
= 2 - (F(x-a-\delta) + F(x+a+\delta))
\end{equation*}
It follows from Equation~\eqref{eq:dom:eq1} that
\begin{equation*}
\Pr\left[ \left(W_{ij}+a\right)^2 \geq x^2\right]
\leq \Pr\left[ \left(W_{ij}+a+\delta\right)^2 \geq x^2\right] .
\end{equation*}
Since 
\begin{equation*}
 \Pr\left[ \left(W_{ij}-a\right)^2 \geq x^2\right]
= \Pr\left[ \left(W_{ij}+a\right)^2\geq x^2\right]
\end{equation*}
and 
\begin{equation*}
\Pr\left[ \left(W_{ij}-a-\delta\right)^2 \geq x^2\right]
 = \Pr\left[ \left(W_{ij}+a+\delta\right)^2 \geq x^2\right], 
\end{equation*}
the proof is now complete. 
\end{proof}

\section{The Solution Path of Group Lasso} \label{sec:glasso}

Let $\mY \in \Sym(n)$ denote an observed symmetric matrix.
In this section, our goal is to analyze the theoretical properties of applying the overlapping group lasso estimator to $\mY$, with a particular focus on its solution path. 
We defer the analysis of support recovery properties to Section~\ref{sec:glasso-support}.
Specifically, we analyze the solution path of the overlapping group lasso optimization problem
\begin{equation*} \begin{aligned}
\min_{\mB \in \Sym(n)} & \frac{1}{2} \sum_{1\leq i \leq j\leq n} 
	(y_{ij} - B_{ij})^2 + \lambda \Omega(\mB)\\
&= \min_{\mV \in \R^{n\times n}}\frac{1}{2} \sum_{1\leq i \leq j\leq n}
	(y_{ij} - v_{ij} - v_{ji})^2 
	+ \lambda \sum_{i=1}^n \sqrt{\sum_{j=1}^n v_{ij}^2},
\end{aligned} \end{equation*} 
where $\mY = [y_{ij}]_{1\leq i,j \leq n}$ denotes the observed matrix, 
and the penalty term $\Omega(\mB)$ enforces node-level group sparsity via the overlapping group Lasso, where $\mB = \mV + \mV^\top$: 
\begin{equation*}
    \Omega(\mB) := \sum_{i=1}^n \sqrt{\sum_{j=1}^n v_{ij}^2}
\end{equation*}
To efficiently solve this optimization problem, we employ the Alternating Direction Method of Multipliers \citep[ADMM; see][for an overview]{boyd2011distributed}. 
The corresponding iterative procedure is detailed in Algorithm~\ref{alg:admm-glasso}.
The operator $\scrT_2$ denotes the row-wise soft-thresholding operator, defined for each row $i \in [n]$ by
\begin{equation*} 
\scrT_2(\mA, q)_{i,\cdot} =  
\left(1 - \frac{q}{\left\|\mA_{i,\cdot}\right\|_2}\right)_{+} 
\left(\mA_{i,\cdot}\right)
\end{equation*} 
where $( x )_+ = \min\{ x, 0 \}$ and $q$ is a positive scalar parameter. 

\begin{algorithm}[ht]
\caption{ADMM for Group Lasso}\label{alg:admm-glasso}
\begin{algorithmic}[1]
    \Require: $\mY \in \R^{n\times n}$, $\lambda > 0$, $\rho > 0$, $T > 0$; 
	Initialize $\mV^{(0)}$, $\mZ^{(0)}$, $\mU^{(0)}$ randomly. 
    \Ensure: $\mV^{(T)}$.
    \For{$t=0, \cdots, T-1$}
        \State $\mV$\textbf{-step}: $\mV^{(t+1)} \gets \frac{1}{\rho+2} \left(\mY - (\mZ^{(t)} - \mU^{(t)})^\top\right) 
        + \frac{\rho+1}{\rho+2} (\mZ^{(t)} - \mU^{(t)})$
    
        \State $\mZ$\textbf{-step}: $\mZ^{(t+1)} \gets \scrT_2\left(\mV^{(t+1)} + \mU^{(t)}, \frac{\lambda}{\rho}\right)$, 
    
        \State $\mU$\textbf{-step}: $\mU^{(t+1)} \gets \mU^{(t)} + \mV^{(t+1)} - \mZ^{(t+1)}$
    \EndFor
\end{algorithmic}
\end{algorithm}

\subsection{KKT Condition and Reformulation}

Let $\vv_i = (v_{i1}, v_{i2}, \ldots, v_{in})^\top$ denote the $i$-th row of $\mV$. 
Suppose that, for a given index set $I \subset [n]$, we obtain a solution with the property that $\vv_i \neq \vo$ for every $i \in I$, while $\vv_{j} = \vo$ for all $j \notin I$. 
For such a solution to be optimal for the overlapping group lasso problem, 
it must satisfy the Karush-Kuhn-Tucker (KKT) conditions. 
Specifically, these conditions can be stated as follows:
\begin{equation} \label{eq:node-glasso-KKT}
    \begin{aligned}
        v_{ij} + v_{ji} + \lambda \frac{v_{ij}}{\|\vv_i\|_2} &=  y_{ij},~~~~~~\text{ if } i,j \in I, i\neq j;\\
        v_{ij} + \lambda \frac{v_{ij}}{\|\vv_i\|_2} &= y_{ij},~~~~~~\text{ if } i \in I, j \in I^c \text{ or } i = j;\\
        \sum_{i\in I} (y_{\ell i} - v_{i\ell})^2 + \sum_{i\in I^c} y_{\ell i}^2 &\leq \lambda^2,~~~~~~~\text{if } j \in I^c. 
    \end{aligned}
\end{equation}
We remind the reader that $\lambda$ is a nonnegative regularization parameter.
The KKT reformulation offers valuable structural insight into the solution of the overlapping group lasso problem. 
Let us denote 
\begin{equation}\label{eq:alpha_i:define}
    \alpha_i = \|\vv_i\|_2 \quad \text{ for } i \in [n],
\end{equation}
and by the property of $\vv_i$ and $I$, we have $\alpha_i > 0$ for $i \in I$ and $\alpha_i = 0$ for $i \in I^c$. 
Let $\valpha = (\alpha_1, \alpha_2, \dots, \alpha_n)^\top$.
    
From the first equality in Equation~\eqref{eq:node-glasso-KKT} and the symmetry of the observation matrix $\mY$, we have 
\begin{equation*} 
\frac{v_{ij}}{\alpha_i} = \frac{v_{ji}}{\alpha_j} \quad \text{for } i, j \in I. 
\end{equation*} 
Using this identity, we can express the KKT conditions more explicitly in terms of $\valpha$ and the observed data:
\begin{equation} \label{eq:node-glasso-kkt2-eq}
\begin{aligned}
    v_{ij} &=  \frac{\alpha_i}{\alpha_i + \alpha_j + \lambda}~y_{ij} &\text{ if } i,j \in I, i\neq j;\\
    v_{ij} &= \frac{\alpha_i}{\alpha_i + \lambda} ~y_{ij} & \text{ if } i \in I, j \in I^c, \text{ or } i = j;\\
\end{aligned}
\end{equation}
and
\begin{equation} \label{eq:node-glasso-kkt2-ineq}
    \sum_{k \in I} \left(\frac{\lambda}{\alpha_k + \lambda}\right)^2 y_{kj}^2 + \sum_{k \in I^c} y_{kj}^2 \leq \lambda^2,~~~~~~~~~~~~~~~~~\text{ if } j \in I^c. 
\end{equation}
Combining the equalities in Equation~\eqref{eq:node-glasso-kkt2-eq} with the definition $\alpha_i = \|\vv_i\|_2$ yields that
\begin{equation} \label{eq:node-glasso-eql}
    \frac{y_{ii}^2 + \sum_{j \in I^c} y_{ij}^2}{(\alpha_i + \lambda)^2} + \sum_{j \in I, j\neq i} \frac{y_{ij}^2}{(\alpha_i + \alpha_j + \lambda)^2} = 1, ~~~~~~~\text{ for } i\in I. 
\end{equation}
This non-linear system characterizes the $\valpha$ values as implicit functions of $\lambda$ and the data $\mY$.
Importantly, the set of equations~\eqref{eq:node-glasso-eql} and inequalities~\eqref{eq:node-glasso-kkt2-ineq} always has at least one solution, 
due to the continuity and coercivity of the objective function in Equation~\eqref{eq:node-glasso}.
This ensures the existence of a valid minimizer for the group lasso problem. 

\subsection{Local Analysis of the Solution Path}

To further understand how the solution evolves with $\lambda$, 
for any fixed $\lambda \geq 0$, we may partition the index set $[n]$ into three disjoint subsets $E_1(\lambda), E_2(\lambda), J(\lambda)$ such that:
\begin{itemize}
    \item $\alpha_i > 0$ for $i \in E_1(\lambda)$;
    \item $\alpha_i = 0$ but the KKT equality holds for $i \in E_2(\lambda)$;
    \item the KKT inequality is strict for $i \in J(\lambda)$.
\end{itemize}
More precisely, this partition gives rise to a set of equations and inequalities:
\begin{equation} \label{eq:node-glasso-partition}
\begin{aligned}
\sum_{j \in E_1(\lambda), i\neq j} \frac{y_{ij}^2}{(\alpha_i + \alpha_j + \lambda)^2} + \frac{y_{ii}^2 + \sum_{j\in E_1(\lambda)^c} y_{ij}^2}{(\alpha_i + \lambda)^2}&= 1, ~~~~~~~\text{ if } i\in E_1(\lambda);\\ 
\sum_{i \in E_1(\lambda)} \frac{y_{ij}^2}{(\alpha_i + \lambda)^2} + \frac{ \sum_{i\in E_1(\lambda)^c} y_{ij}^2}{\lambda^2} &= 1,~~~~~~~~\text{ if } i \in E_2(\lambda);\\
\sum_{j \in E_1(\lambda)} \frac{y_{\ell j}^2}{(\alpha_\ell + \lambda)^2}  + \frac{\sum_{j \in E_1(\lambda)^c} y_{\ell j}^2}{\lambda^2} &< 1,~~~~~~~~\text{ if } \ell \in J(\lambda).
\end{aligned}
\end{equation}
For indices $i \in E(\lambda) = E_1(\lambda) \cup E_2(\lambda)$, we can compactly express the conditions as:
\begin{equation} \label{eq:partition-eq}
    \sum_{j \in E(\lambda), j\neq i} \frac{y_{ij}^2}{(\alpha_i + \alpha_j + \lambda)^2} + \frac{y_{ii}^2 + \sum_{j\in E(\lambda)^c} y_{ij}^2}{(\alpha_j + \lambda)^2} = 1, ~~~~~~~\text{ if } i\in E(\lambda).
\end{equation}
We next show that under mild conditions, each $\alpha_i(\lambda)$ evolves smoothly and monotonically in $\lambda$. 

\begin{lemma}\label{lem:decreasing}
Suppose that for some $\lambda_0 > 0$, Equation~\eqref{eq:partition-eq} has a solution $\valpha_0$ with $m_0 := |E(\lambda_0)| < n$. 
Suppose that the observed matrix $y_{ij} > 0$ holds for all $1 \leq i \neq j \leq n$. 
Then there exists a neighborhood $(\lambda_0 - \delta, \lambda_0 + \delta)$ for some $\delta > 0$, such that:
\begin{enumerate}
    \item A unique solution $\valpha(\lambda)$ exists in this neighborhood and continuously extends $\valpha_0$;
    \item For every $i \in E(\lambda_0)$, the function $\alpha_i(\lambda)$ is continuous and strictly decreasing in $\lambda$.
\end{enumerate}
\end{lemma}
\begin{proof}
Without loss of generality, assume $E(\lambda_0) = [m_0]$.
We denote by $\valpha_E = (\alpha_1, \alpha_2, \dots, \alpha_{m_0})^\top$ the vector of active group norms, but we will drop the subscript $E$ for simplicity in the remainder of the proof.
For each $i \in [m_0]$, define the function
\begin{equation}\label{eq:Fi}
F_i(\valpha, \lambda) = \sum_{j=1, j\neq i}^{m_0} \frac{y_{ij}^2}{(\alpha_i + \alpha_j + \lambda)^2} + \frac{y_{ii}^2 + \sum_{j=m_0+1}^n y_{ij}^2}{(\alpha_i + \lambda)^2} - 1.
\end{equation}
Let $F(\valpha, \lambda) = (F_1, F_2, \dots, F_{m_0})^\top$. 
We compute the Jacobian matrix $\partial F / \partial \valpha$ according to
\begin{equation*}
\frac{\partial F_i}{\partial \alpha_j} (\valpha, \lambda) = 
\begin{cases}
    -\frac{2y_{ij}^2}{(\alpha_i+\alpha_j+\lambda)^3} = -2 a_{ij} & \text{ if } j\neq i; \\
    -\sum_{\ell=1, \ell \neq i}^{m_0}\frac{2y_{i\ell}^2}{(\alpha_i+\alpha_\ell+\lambda)^3} - 
    \frac{2y_{ii}^2 + 2\sum_{\ell=m_0+1}^{n} y_{i\ell}^2}{(\alpha_i + \lambda)^3} = -2 \sum_{\ell=1, \ell\neq i}^{m_0} a_{i\ell} - 2 a_{ii}, & \text{ if } j = i,
\end{cases}
\end{equation*}
where we have set, for $i \neq j$,
\begin{equation*}
a_{ij} = \frac{y_{ij}^2}{(\alpha_i+\alpha_j+\lambda)^3}
~\text{ and }~
a_{ii} = \frac{y_{ii}^2 + \sum_{\ell=m_0+1}^{n} y_{i\ell}^2}{(\alpha_i + \lambda)^3}.
\end{equation*}
Therefore, the Jacobian matrix has the form
\begin{equation} \label{eq:partF-partA}
    \frac{\partial F}{\partial \boldsymbol{\alpha}}\left(\boldsymbol{\alpha}, \lambda\right)=-2\left(
    \begin{array}{cccc}
    \sum_{j=1}^{m_0} a_{1 j} & a_{12} & \ldots & a_{1 m_0} \\
    a_{21} & \sum_{j=1}^{m_0} a_{2 j} & \ldots & a_{2 m_0} \\
    & \ldots & & \\
    a_{m_0 1} & a_{m_0 2} & \ldots & \sum_{j=1}^{m_0} a_{2 j}
    \end{array}
    \right).
\end{equation} 
Since the observed matrix satisfies $y_{ij} > 0$ for all $1 \leq i \neq j \leq n$ and $m_0 < n$, we have $a_{ii} > 0$ holds almost surely for all $i \in [m_0]$.
It follows that the Jacobian matrix is strictly diagonally dominant, and thus $\partial F/\partial\valpha$ is almost surely invertible at all $(\valpha, \lambda)$ such that $\lambda > 0$ and $\alpha_j \geq 0$ for $j \in [m_0]$.
By the Implicit Function Theorem, since $F(\valpha_0, \lambda_0) = 0$ and the Jacobian is nonsingular at $(\valpha_0, \lambda_0)$, there exists $\delta > 0$ and a unique continuously differentiable function $\valpha(\lambda)$ on the interval $(\lambda_0 - \delta, \lambda_0 + \delta)$ such that $F(\valpha(\lambda), \lambda) = 0$ and $\valpha(\lambda_0) = \valpha_0$.
    


To establish the monotonicity of $\valpha(\lambda)$, compute
\begin{equation} \label{eq:partalpha-partlambda}
\frac{d\valpha}{d\lambda} = -\left(\frac{\partial F}{\partial \boldsymbol{\alpha}}\left(\boldsymbol{\alpha}(\lambda), \lambda\right)\right)^{-1} \frac{\partial F}{\partial \lambda} \left( \valpha(\lambda), \lambda \right).
\end{equation}
Recalling the definition of $F_i$ from Equation~\eqref{eq:Fi}, taking the partial derivative with respect to $\lambda$, we have
\begin{equation} \label{eq:partF-partlambda}
\frac{\partial F_i}{\partial \lambda} 
= \frac{\partial F_i}{\partial \alpha_i} 
= -2 \sum_{\ell=1}^{m_0} a_{i\ell} ,\quad \text{for } i \in [m_0]. 
\end{equation}
Let $\mD = \diag\left(\sum_{j=1}^{m_0} a_{1 j}, \sum_{j=1}^{m_0} a_{2 j}, \ldots, \sum_{j=1}^{m_0} a_{2 j}\right)$ and define the matrix $\mS$ with entries 
\begin{equation*}
S_{ij} = \begin{cases}
        \frac{a_{ij}}{\sum_{k=1}^{m_0} a_{ik}} & \text{ if } i\neq j;\\
        0 & \text{ if } i = j.
    \end{cases}
\end{equation*}
Then, by Equation~\eqref{eq:partF-partA} we have 
\begin{equation*}
    \frac{\partial F}{\partial \boldsymbol{\alpha}} = -2 \mD \left(\mI_{m_0} + \mS\right).
\end{equation*}
Combining this with Equations~\eqref{eq:partalpha-partlambda} and~\eqref{eq:partF-partlambda}, we have
\begin{equation*}
    \frac{d \valpha}{d \lambda} =\frac{1}{2}(\mI_{m_0}+\mS)^{-1} \mD^{-1} \frac{\partial F}{\partial \lambda}(\valpha(\lambda), \lambda)
    = - (\mI_{m_0} + \mS)^{-1} \mathbf{1}_{m_0}
\end{equation*}

Because $\mS \geq 0$ elementwise and has row sums less than $1$, its spectral radius is strictly less than $1$ by Gershgorin circle theorem.
Thus, the matrix $\mI + \mS$ is invertible and
\begin{equation} \label{eq:dalpha-dlambda}
    \begin{aligned}
    \frac{d \valpha}{d \lambda} 
    =-\sum_{k=0}^{\infty} \mathbf{S}^{2k}(\mI_{m_0}-\mathbf{S}) \mathbf{1}_{m_0} < 0 ,
    \end{aligned}
\end{equation}
where the last inequality holds elementwise.
This shows that each $\alpha_i(\lambda)$ is strictly decreasing in $\lambda$ for $i \in E(\lambda_0)$, completing the proof.
\end{proof}
    
The preceding lemma establishes that the group norm $\alpha_i(\lambda)$ evolves continuously and strictly decreases as the regularization strength $\lambda$ increases. 
The following lemma quantifies the rate at which $\alpha_i(\lambda)$ decreases with $\lambda$, providing both upper and lower bounds on the derivative $d\alpha_i/d\lambda$.

\begin{lemma} \label{lem:grad-bound}
Let $\valpha(\lambda)$ be a solution to Equation~\eqref{eq:partition-eq} and let $m_0 = |E(\lambda)|$ as defined in Equation~\eqref{eq:partition-eq}.
Under the same assumptions as in Lemma~\ref{lem:decreasing}, for all $i \in E(\lambda)$, the derivative of $\alpha_i$ with respect to $\lambda$ satisfies
\begin{equation}\label{eq:grad-bound}
     \frac{\sum_{\ell \in E(\lambda)^c} y_{i \ell}^2 }{\sum_{\ell=1}^{n}  y_{i \ell}^2} \leq \left| \frac{d\alpha_i}{d\lambda} \right| \leq 1.
\end{equation}
\end{lemma}
\begin{proof}
Without loss of generality, assume that $E(\lambda) = [m_0]$ and denote $\valpha_E = (\alpha_i)_{i \in [m_0]}$. 
For notational convenience, we omit the subscript and simply write $\valpha = \valpha_E$ throughout this proof. 
From the expression derived in Equation~\eqref{eq:dalpha-dlambda}, we have
\begin{equation*}
\begin{aligned}
    \left|\frac{d\valpha}{d\lambda}\right| &= \sum_{k=0}^{\infty} \mS^{2k}(\mI_{m_0}-\mS) \mathbf{1}_{m_0} 
    = \mathbf{1}_{m_0} - \sum_{k=0}^{\infty} \mS^{2k+1} (\mI_{m_0} - \mS) \mathbf{1}_{m_0}.
\end{aligned}
\end{equation*}
Since $\mS \geq 0$ elementwise and has row sums strictly less than $1$, 
the term subtracted on the right-hand side is nonnegative. 
This immediately yields the upper bound:
\begin{equation*}
    \left|\frac{d\alpha_i}{d\lambda}\right| \leq 1 \quad \text{for all } i \in [m_0].
\end{equation*}
To obtain the lower bound, observe that 
\begin{equation*}
\begin{aligned}
    \left|\frac{d\valpha}{d\lambda}\right| &= \mathbf{1}_{m_0} - \mS \mathbf{1}_{m_0} + \sum_{k=0}^{\infty} \mS^{2k+2} (\mI_{m_0} - \mS) \mathbf{1}_{m_0}
    \geq \mathbf{1}_{m_0} - \mS \mathbf{1}_{m_0},
\end{aligned}
\end{equation*}
again holding elementwise. 
For each $i \in [m_0]$,
\begin{equation*} \begin{aligned}
(\mS \mathbf{1}_{m_0})_i 
&= \sum_{\ell=1, \ell \neq i}^{m_0} \frac{a_{i\ell}}{\sum_{j=1}^{m_0} a_{i j}} =  \frac{\sum_{\ell=1, \ell \neq i}^{m_0} a_{i\ell}}{a_{ii} + \sum_{j=1, j \neq i}^{m_0} a_{i j}} 
= 1 - \frac{a_{ii}}{\sum_{j=1}^{m_0} a_{i j}} ,
\end{aligned} \end{equation*}
which implies 
\begin{equation} \label{eq:dalpha-dlambda:ratio}
\left|\frac{d \alpha_i}{d \lambda}\right| 
\geq \frac{a_{i i}}{\sum_{j=1}^{m_0} a_{i j}}. 
\end{equation}
Plugging in the definition of $a$ into the right-hand of Equation~\eqref{eq:dalpha-dlambda:ratio}, where 
\begin{equation*}
a_{i i}=\frac{y_{i i}^2+\sum_{\ell=m_0+1}^n y_{i \ell}^2}{\left(\alpha_i+\lambda\right)^3}, 
\quad a_{i \ell}=\frac{y_{i \ell}^2}{\left(\alpha_i+\alpha_{\ell}+\lambda\right)^3},
\end{equation*}
we obtain
    \begin{equation*}
    \begin{aligned}
        \left|\frac{d\alpha_i}{d\lambda}\right| &\geq 
        \frac{\left(y_{ii}^2 + \sum_{\ell=m_0+1}^n y_{i \ell}^2\right) / \left(\alpha_i+\lambda\right)^3}{\sum_{\ell=1, \ell\neq i}^{m_0} y_{i \ell}^2/\left(\alpha_i+\alpha_{\ell}+\lambda\right)^3 + (y_{ii}^2+\sum_{\ell=m_0+1}^n y_{i \ell}^2)/\left(\alpha_i+\lambda\right)^3} \\
        &=  \frac{y_{ii}^2 + \sum_{\ell=m_0+1}^n y_{i \ell}^2 }{\sum_{\ell=1, \ell\neq i}^{m_0}  \frac{\left(\alpha_i+\lambda\right)^3}{\left(\alpha_i+\alpha_{\ell}+\lambda\right)^3} y_{i \ell}^2 + y_{ii}^2 + \sum_{\ell=m_0+1}^n y_{ii \ell}^2} .
\end{aligned} \end{equation*}
Observing that all coefficients multiplying $y_{i\ell}^2$ in the denominator are bounded by $1$, it follows that
\begin{equation*} \begin{aligned}
    \left|\frac{d\alpha_i}{d\lambda}\right| &\geq 
    &\geq \frac{y_{ii}^2 + \sum_{\ell=m_0+1}^n y_{i \ell}^2 }{\sum_{\ell=1}^{n}  y_{i \ell}^2},
\end{aligned} \end{equation*}
which yields our desired lower bound.
\end{proof}

With Lemmas~\ref{lem:decreasing} and~\ref{lem:grad-bound}, we have shown that when the active set $E(\lambda)$ contains fewer than $n$ indices, the group norm vector $\valpha(\lambda)$ evolves smoothly ad strictly decreases with respect to the regularization parameter $\lambda$.
A natural next step is to examine what happens when the active set saturates the entire index set, i.e., when $E(\lambda) = [n]$.
The following lemma establishes that the monotonicity of $\valpha(\lambda)$ persists in this regime as well, and even admits a closed-form derivative when the diagonal entries of $\mY$ vanish. 

\begin{lemma}\label{lem:keep-decreasing}
Assume that $y_{ij} \neq 0$ for $i, j \in [n], i\neq j$. 
Let $\valpha(\lambda_J)$ be a solution to Equation~\eqref{eq:node-glasso-partition} such that $J(\lambda_J) = \emptyset$.
Then there exists a unique solution path $\valpha(\lambda) \in \R^n$, continuous on the interval $[0, \lambda_J]$, such that 
\begin{equation*}
    \frac{d\alpha_i(\lambda)}{d\lambda} < 0, \quad \text{ for all } i \in [n]. 
\end{equation*}
Moreover, if $y_{ii} = 0$ for all $i \in [n]$, then this derivative is constant and satisfies
\begin{equation*}
    \frac{d\alpha_i(\lambda)}{d\lambda} = -\frac{1}{2}
	\quad \text{ for all } i \in [n]. 
\end{equation*}
\end{lemma}
\begin{proof}
We follow the setup of Lemmas~\ref{lem:decreasing} and~\ref{lem:grad-bound}. 
Define the matrix $\mS \in \R^{n \times n}$ by
\begin{equation*}
S_{i j}
= \begin{cases}\frac{a_{i j}}{\sum_{k=1}^n a_{i k}}, &\mbox{ if } i \neq j, \\
		0, &\mbox{ if } i=j\end{cases}\\
\end{equation*}
where 
\begin{equation*}
a_{i j}=\frac{y_{i j}^2}{\left(\alpha_i+\alpha_j+\lambda_J\right)^3}, \quad \text { for } i \neq j, \quad a_{i i}=\frac{y_{i i}^2}{\left(\alpha_i+\lambda_J\right)^3}
\end{equation*}
Under the assumption that $y_{ij} \neq 0$ for all $i \neq j$, we have $S_{ij} > 0$ for all off-diagonal entries.
If not all $y_{ii} = 0$ for $i \in [n]$, then there exists at least one row (say, the first) such that 
\begin{equation*}
    \sum_{j=1}^n S_{1j} < 1.
\end{equation*}
It follows that for any $k \in [n]$,
\begin{equation*}
    (\mS^2 \mathbf{1}_n)_k = \sum_{j=1}^n \sum_{i=1}^n S_{kj} S_{ji} = S_{k1}\sum_{i=1}^n S_{1i} + \sum_{j\neq 1} S_{kj} \sum_{i=1}^n S_{ji} < S_{k1} + \sum_{j\neq 1} S_{kj} \leq 1,
\end{equation*}
which indicates that the row sums of $\mS^2$ are all strictly less than 1. 
Hence, by Gershgorin circle theorem, all eigenvalues of $\mS^2$ lie strictly within the unit circle.
Consequently, $\lim _{k \rightarrow \infty} \mS^k=0$, and the Neumann series expansion of $(\mI + \mS)^{-1}$ is valid.
By the same argument as in Lemma~\ref{lem:decreasing}, this implies that for some $\delta > 0$, the solution path $\valpha(\lambda)$ is well-defined and differentiable on the interval $(\lambda_J - \delta, \lambda_J]$, and satisfies
\begin{equation*}
    -1 \leq \frac{d\alpha_i(\lambda)}{d\lambda} < 0, \quad \text{ for all } i \in [n]. 
\end{equation*}
Because $\valpha(\lambda)$ is continuous and its derivative is uniformly bounded, we may take the right-handed limit 
\begin{equation*}
\valpha(\lambda_J - \delta)
= \lim_{\lambda \to (\lambda_J-\delta)^+} \valpha(\lambda),
\end{equation*} 
which exists and satisfies Equation~\eqref{eq:partition-eq} by continuity.
Therefore, $\valpha(\lambda_J - \delta)$ is also a solution to Equation~\eqref{eq:partition-eq} with $\lambda = \lambda_J - \delta$.
Moreover, $\valpha(\lambda_J - \delta)$ is the unique extension of $\valpha(\lambda_J)$ to the interval $[\lambda_J - \delta, \lambda_J]$ by the implicit function theorem.
By repeating this argument, we can uniquely extend $\valpha(\lambda)$ to the entire interval $[0, \lambda_J]$.

For the special case where $y_{ii} = 0$ for all $i$, the diagonal entries $a_{ii} = 0$, and $\mS$ becomes row-stochastic with strictly positive off-diagonal entries. 
Therefore, $\mS^2 > 0$ and $\mI_n + \mS$ is invertible by Lemma~\ref{lem:IS-invertible}. 
Applying Lemma~\ref{lem:eigenvalue-IS} yields
\begin{equation*}
\frac{d \valpha}{d\lambda} = -(\mI_n + \mS)^{-1} \mathbf{1}_n 
= -\frac{1}{2} \mathbf{1}_n,
\end{equation*}
completing the proof.
\end{proof}

\subsection{Global Extension of the Solution Path}

Building on the local continuity and monotonicity results established in Lemmas~\ref{lem:decreasing} through~\ref{lem:keep-decreasing}, we now turn to the global behavior of the group lasso solution path.
The preceding lemmas guarantee that for any set of active indices $E(\lambda_0) \subseteq [n]$, the corresponding solution vector $\valpha_{E(\lambda_0)}(\lambda)$ evolves continuously and strictly decreases as $\lambda$ decreases, at least within a neighborhood of any given $\lambda_0$.
Moreover, Lemma~\ref{lem:keep-decreasing} ensures that this solution path remains well-defined even when all indices are active.

We now show that starting from any valid solution $\valpha(\lambda_0)$, this local behavior can be extended inductively to yield a continuous and elementwise decreasing solution path $\valpha(\lambda)$ for all $\lambda \geq 0$.
This result provides a global structural characterization of the overlapping group lasso estimator and underlies its interpretation as a greedy algorithm.
\begin{lemma} \label{lem:unique}
Assume that $y_{ij} \neq 0$ for $i, j \in [n], i\neq j$. 
Given a solution $\valpha(\lambda_0)$ to the group lasso at $\lambda = \lambda_0 > 0$, 
there exists a continuous, 
elementwise decreasing function $\valpha(\lambda)$ defined for all $\lambda \geq 0$, 
such that $\valpha(\lambda)$ solves the group lasso problem for each $\lambda$.
\end{lemma}
\begin{proof}
First, observe that for any $j \in [n]$, 
the left-hand side of the KKT inequality condition in Equation~\eqref{eq:node-glasso-kkt2-ineq} is upper bounded by $\|\vy_j\|_2^2$.
Thus, for $\lambda > \lambda_{\max} := \max_{j \in [n]} \|\vy_j\|_2$, the inequality cannot be satisfied with $\alpha_j > 0$, 
and hence the only solution is the all-zero vector $\valpha(\lambda) = 0$. 
Therefore, we only need to construct a valid solution path over the compact interval $\lambda \in [0, \lambda_{\max}]$.

Fix $\lambda_0 \in (0, \lambda_{\max}]$ and assume we are given a solution $\valpha(\lambda_0)$. 
We aim to extend this solution path continuously and monotonically to the left over the interval $[0, \lambda_0]$.
Recall that we partition $[n]$ into $E_1(\lambda)$, $E_2(\lambda)$ and $J(\lambda)$ according to Equation~\eqref{eq:node-glasso-partition}. 
We call the indices $E(\lambda_0) = E_1(\lambda_0) \cup E_2(\lambda_0)$ the set of \emph{active groups} at $\lambda_0$. 
We let $J(\lambda_0)$ be the complement of $E(\lambda_0)$.
We distinguish two cases based on whether $J(\lambda_0)$ is empty.
    

If $J(\lambda_0) = \emptyset$, then all groups are already active, and the entire index set participates in the solution. 
In this case, we invoke Lemma~\ref{lem:keep-decreasing} to construct the full solution path on $[0, \lambda_0]$. 
Otherwise, as we decrease $\lambda$ from $\lambda_0$, eventually some inequality constraints in Equation~\eqref{eq:node-glasso-kkt2-ineq} become active, and new group of coordinates may enter the support.
Let $\lambda_1 < \lambda_0$ be the first such value where this occurs.
We refer to $\lambda_1$ as a \emph{breaking point}. 

We now proceed to extend the solution from $\lambda_0$ to $\lambda_1$ so that the solution path is continuous over the interval $[\lambda_1,\lambda_0]$.
 We consider the solution $\valpha_{E(\lambda_0)}(\lambda)$ defined on the active set $E(\lambda_0)$.
By Lemma~\ref{lem:decreasing}, there exists a neighborhood of $\lambda_0$ over which $\valpha_{E(\lambda_0)}(\lambda)$ is unique, continuous, and strictly decreasing. 
Moreover, Lemma~\ref{lem:grad-bound} ensures that this solution path has bounded derivatives, so the limit
\begin{equation*}
    \valpha_{E(\lambda_0)}(\lambda_0-\delta) := \lim_{\lambda \to (\lambda_0-\delta)^+} \valpha_{E(\lambda_0)}(\lambda)
\end{equation*}
exists and satisfies the partition KKT conditions at $\lambda = \lambda_0 - \delta$.
Thus, the solution can be further extended.
By repeating this argument, we can continue to extend $\valpha_{E(\lambda_0)}(\lambda)$ leftward until reaching $\lambda_1$, the first point at which a new group of coordinates is activated. 
By setting $\valpha_{E(\lambda_0)^c}(\lambda) = 0$ for $\lambda \in [\lambda_1, \lambda_0]$, we obtain a complete solution $\valpha(\lambda)$ over $[\lambda_1, \lambda_0]$ that is continuous and decreasing in each coordinate.

Now, we aim to iterate this process to continue the solution path to values smaller than $\lambda_1$.
At $\lambda = \lambda_1$, new group become active, corresponding to $E_2(\lambda_1) \neq \emptyset$.
We now recursively apply the argument just given above with $\lambda_1$ in place of $\lambda_0$.
Each such application of this argument adds new group into the active set, and we continue until eventually $J(\lambda) = \emptyset$.
    
Once $J(\lambda) = \emptyset$, the entire index set participates in the solution.
In this regime, Lemma~\ref{lem:keep-decreasing} guarantees the existence of a globally defined, strictly decreasing solution path $\valpha(\lambda)$ over $[0, \lambda]$. 
This completes the extension over the entire interval $[0, \lambda_0]$.
\end{proof}
 

To summarize, we have shown that for the overlapping group lasso estimator, the solution path $\valpha(\lambda)$ is continuous and strictly decreasing in $\lambda$, starting from the maximal regularization level $\lambda_{\max} := \max_{i \in [n]} \|\vy_i\|_2$. 
At this level, the unique solution is $\valpha(\lambda_{\max}) = \mathbf{0}$, and Lemma~\ref{lem:unique} guarantees that this solution can be extended to all $\lambda \geq 0$ as a continuous and elementwise decreasing function.
    
For theoretical purposes, we focus on this canonical solution path initialized at $\lambda_{\max}$ and interpret group lasso as a selection procedure: once a group becomes active at some $\lambda = \bar{\lambda}$, it remains active for all smaller values of $\lambda$. 
This viewpoint allows us to formalize support recovery guarantees based on the structure of $\valpha(\lambda)$.

\subsection{Empirical Results of the Solution Path}

To validate the theoretical results developed in Section~\ref{sec:glasso}, we conduct a simple simulation study illustrating the behavior of the group lasso solution path as a function of the regularization parameter $\lambda$.
We fix $n=500$ and generate a symmetric matrix $\mY \in \Sym(n)$ according to the additive model $\mY = \mBstar + \mW$, where $\mW$ is a symmetric noise matrix with i.i.d.~entries from $\calN(0,1)$ on and above the diagonal.
The signal matrix $\mBstar$ is constructed as $\mB_0 + \mB_0^\top$, where $\mB_0$ contains $m=5$ randomly selected nonzero rows, each with entries independently drawn from $\calN(0, \sigma_B^2)$ and $\sigma_B = 1.9 \, n^{-1/4} \log^{1/4} n$. 
Let $\Istar$ denote the indices of the nonzero rows of $\mB_0$. 
By construction, $\mBstar$ is a node-sparse symmetric matrix, and with high probability, the row norms corresponding to indices in $\Istar$ are of order $n^{1/4} \log^{1/4} n$, matching the minimax signal strength threshold (see Remark~\ref{rem:mle-snr}). 

We compute the group lasso estimator over a grid of $T = 60$ regularization parameters $\lambda$ spanning the interval
\begin{equation*}
    \lambda_{\min} := 0.85 \cdot \min_{i \in [n]} \|\vy_{i}\|_2, \quad \text{and} \quad \lambda_{\max} := \max_{i \in [n]} \|\vy_{i}\|_2,
\end{equation*}
where $\vy_i$ denotes the $i$-th row of $\mY$.
For each value of $\lambda$, we solve the optimization problem in Equation~\eqref{eq:node-glasso} using the ADMM algorithm described in Algorithm~\ref{alg:admm-glasso}. 
We then compute $\alpha_i(\lambda)$ for each $i \in [n]$, where $\alpha_i$ is defined as the $\ell_2$ norm of the $i$-th row of the estimated matrix $\mV(\lambda)$ (see Equation~\eqref{eq:alpha_i:define}). 
This yields the full solution path $\valpha(\lambda)$, which characterizes how the magnitude and support of the estimated rows evolve with the penalty parameter $\lambda$.

Figure~\ref{fig:glasso-path} displays the resulting solution paths. 
Each curve represents $\alpha_i(\lambda)$ for a fixed row $i \in [n]$. 
Orange curves correspond to signal nodes ($i \in I^\star$), while blue curves correspond to non-signal nodes ($i \notin I^\star$). 
For visual reference, we include two guide lines: a dashed red line with slope $-1$ and a dotted red line with slope $-\frac{1}{2}$. 
The $x$-axis represents the regularization parameter $\lambda$, while the $y$-axis represents the $\ell_2$ norm $\alpha_i(\lambda)$ of the estimated rows $\vv_i$.
The empirical findings are in strong agreement with our theoretical analysis:
\begin{enumerate}
    \item The solution paths are continuous and strictly decreasing in $\lambda$, as evidenced by the monotonic decay of all curves. 
    This behavior reflects the selection mechanism of the group lasso estimator: once a group becomes active at some threshold $\lambda = \bar{\lambda}$, it remains active for all smaller values of $\lambda$ along a solution path.
    This observation is consistent with the global monotonicity result established in Lemma~\ref{lem:unique}.
    \item For small values of $\lambda$, the solution paths tend to follow a slope of $-\frac{1}{2}$, closely tracking the dotted reference line in Figure~\ref{fig:glasso-path}. 
    This matches the theoretical prediction in Lemma~\ref{lem:keep-decreasing}, which shows that when all groups are active and the diagonal entries vanish, the group norms decay linearly with slope $-\frac{1}{2}$.
    For large values of $\lambda$, the paths approximately follow a slope of $-1$, as indicated by the dashed line, which agrees with Lemma~\ref{lem:grad-bound}. 
    This is expected since, with high probability, the left-hand side of Equation~\eqref{eq:grad-bound} concentrates near one.
    \item Finally, we observe that the curves corresponding to rows in the true support set $\Istar$ are the first to become active as $\lambda$ decreases.
    This supports the conclusion that, when the signal strength is sufficiently large, the group lasso estimator can correctly recover the ground truth node support $\Istar$.
    In particular, the empirical behavior suggests that recovery is achievable at the minimax signal strength threshold, in line with the sufficient conditions established in Theorem~\ref{thm:glaso-sufficient-crude}.
    We provide a detailed analysis of support recovery guarantees in Section~\ref{sec:glasso-support}.
\end{enumerate}

In summary, these findings provide compelling empirical support for the theoretical results derived in Section~\ref{sec:glasso}, and set the stage for our subsequent investigation into the theoretical guarantees for support recovery in Section~\ref{sec:glasso-support}.

\begin{figure}[H]
    \centering
    \includegraphics[width=0.9\textwidth]{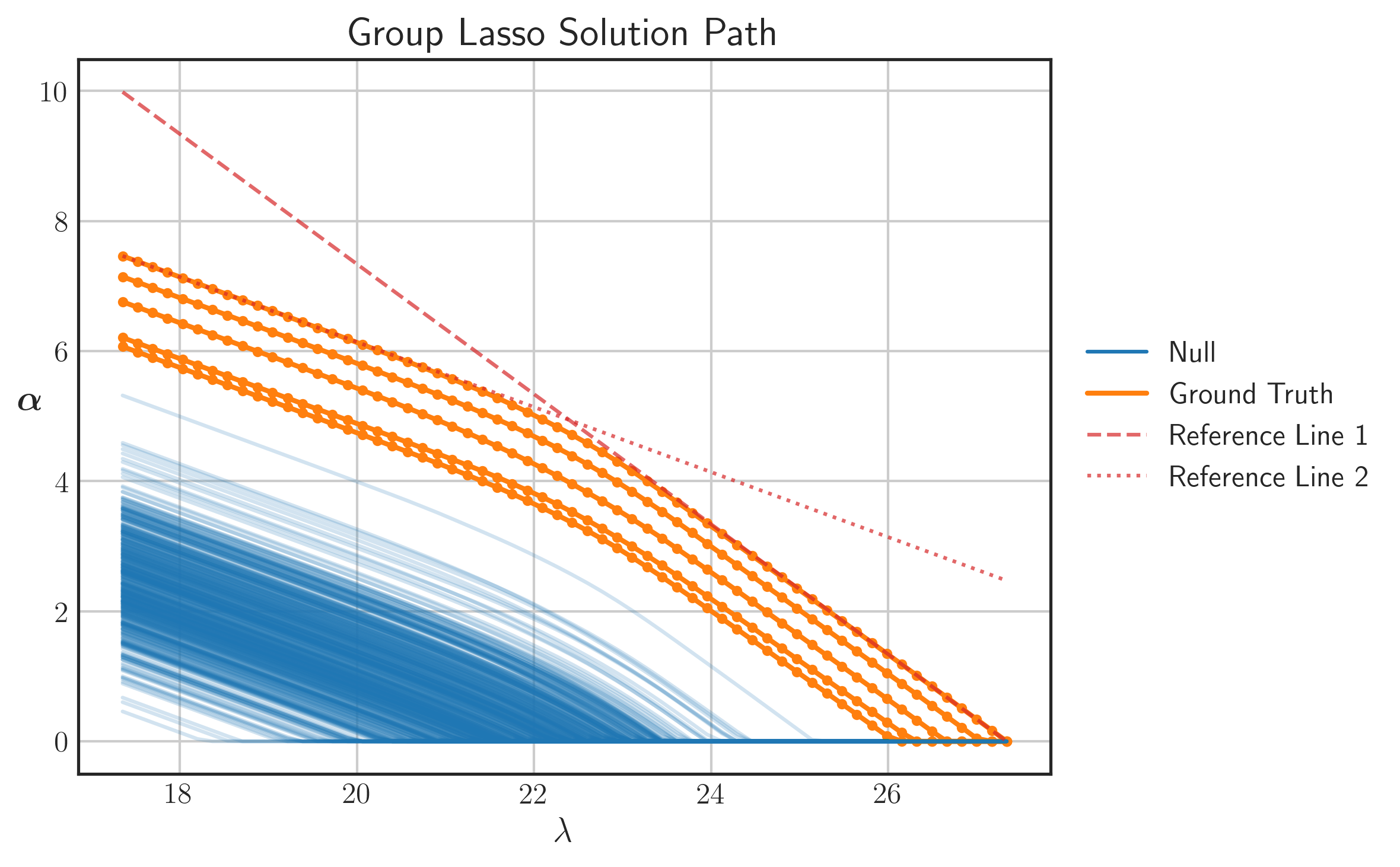}
    \caption{Solution path of the overlapping group lasso estimator. Each curve corresponds to a group of coordinates. 
    The $y$-axis shows the value of each $\alpha_i$ while the $x$-axis shows the value of the regularization parameter $\lambda$.
     Orange curves correspond to signal nodes ($i \in I^\star$), while blue curves correspond to non-signal nodes ($i \notin I^\star$). 
     The dashed red line has a slope of $-1$ and the dotted red line has a slope of $-\frac{1}{2}$. }
    \label{fig:glasso-path}
\end{figure}


\section{Group Lasso Support Recovery} \label{sec:glasso-support}

With the solution path results established in Section~\ref{sec:glasso}, we now turn to establishing conditions under which the group Lasso correctly recovers the support of the underlying signal.
Throughout this section, we retain the notation introduced earlier in Section~\ref{sec:glasso} and, 
without loss of generality, assume the true support is given by $\Istar = [m]$.
We consider the case of applying the group Lasso estimator to a single observed matrix
\begin{equation} \label{eq:single-observed}
    \mYtil = \mB^\star + \mDeltatil + \mWtilo,
\end{equation}
where $\mYtil$ is the output from Step~\ref{alg:general:step:select} in Algorithm~\ref{alg:general}. 
We will write $\mYtil = (y_{ij})_{1\leq i, j \leq n}$, where $y_{ij}$ is the $(i,j)$-th entry of $\mYtil$, to keep the notation simple and consistent with Section~\ref{sec:glasso}.
As before, we assume without loss of generality that the selected index set is $\calU = [n]$ as discussed at the beginning of Section~\ref{sec:one-step-recovery} for a discussion on this. 

Our goal is to prove the existence of a regularization parameter $\lambda \geq 0$ such that the corresponding group Lasso solution $\valpha(\lambda)$ satisfies the following:
\begin{itemize}
    \item {\bf Exact support recovery.}  $\alpha_i > 0$ for all $i\in [m]$
    \item {\bf KKT equalities for active groups.} For all $i \in [m]$,
    \begin{equation}\label{eq:KKT-equal-active-m}
        \sum_{k \in [m], k\neq i}\left(\frac{\lambda}{\alpha_k+\alpha_i+\lambda}\right)^2 y_{i k }^2+\left(\frac{\lambda}{\alpha_i+\lambda}\right)^2 \left(y_{ii}^2 + \sum_{k \in [m]^c} y_{ik }^2\right) = \lambda^2 .
    \end{equation}
    \item {\bf KKT strict inequalities for inactive groups.}
	For all $\ell \in [m]^c$,
    \begin{equation}\label{eq:KKT-ineq-inactive}
        \sum_{k \in [m]}\left(\frac{\lambda}{\alpha_k+\lambda}\right)^2 y_{\ell k}^2+\sum_{k \in [m]^c} y_{\ell k}^2 < \lambda^2 .
    \end{equation}
\end{itemize}
Establishing these conditions will allow us to guarantee that the support recovered by the group Lasso estimator exactly matches the true support of the underlying signal. 

We start with a useful upper bound on $\alpha_i(\lambda)$. 
By Lemma \ref{lem:grad-bound}, the group norm satisfies the derivative bound
\begin{equation*}
    \frac{ d\alpha_i }{ d\lambda } \geq -1,
\end{equation*}
For $i\in[n]$, let $\lambdabar_i$ denote the largest value of $\lambda$ such that the $i$-th group is active, i.e., 
\begin{equation*}
\alpha_i(\lambda) > 0 \; \text{ for all } \lambda < \lambdabar_i,
\quad \text{and} \quad
\alpha_i(\lambdabar_i) = 0, \; \text{ for all } \lambda \geq \lambdabar_i.
\end{equation*}
Then, for any $0 \leq \lambda < \lambdabar_i$, we have
\begin{equation} \label{eq:alpha_j-upper}
    \alpha_i(\lambda) = -\int_{\lambda}^{\lambdabar_{i}} \frac{d\alpha_i(u)}{d\lambda} du \leq \lambdabar_{i} - \lambda .
\end{equation}

To establish the existence of $(\alpha_i)_{i\in [n]}$ satisfying the KKT conditions for support recovery, we analyze a simplified setting.
Instead of directly solving the group Lasso problem on the full matrix $\mYtil$ as given in Equation~\eqref{eq:single-observed}, we consider an auxiliary problem with group Lasso applied to $\mYtil_{\Istar} \in \R^{n \times n}$, where
\begin{equation*}
\Ytil_{\Istar, ij} = \begin{cases}
    \Ytil_{ij}, & \text{ if } i \in \Istar, j \in [n] \text{ or } i \in [n], j \in \Istar, \\
    0, & \text{ otherwise.}
\end{cases}
\end{equation*}
The following lemma provides a sufficient condition under which the group Lasso solution applied to $\mYtil_{\Istar}$ selects all and only selects the first $m$ groups.
The proof is provided in Section~\ref{sec:proof-lem-alter-lasso}.
\begin{lemma}\label{lem:alter-lasso}
Suppose that $n > 3m$ and $\kappa^2 \sqrt{\mut r} = O(n^{1/4} \log^{1/4} n)$ and assume without loss of generality that $\Istar = [m]$. 
If 
\begin{equation} \label{eq:easy-glasso-recovery-cond}
\min_{i \in [m]} \sum_{k=m+1}^n B_{ik}^{\star 2} + \frac{1}{6} n\sigma^2 
\geq \frac{n^{1/4}+1}{n^{1/4} - 1}
	\max_{\ell \in [m]} \sum_{k=1}^m B_{\ell k}^{\star 2},
\end{equation}
then for sufficiently large $n$, 
there exists $\lambda > 0$ such that the group Lasso applied to $\mYtil_{\Istar}$ selects exactly the first $m$ groups with probability at least $1 - O(n^{-6})$.
That is, there exists a solution $(\alpha_1(\lambda), \alpha_2(\lambda), \ldots, \alpha_n(\lambda))$ such that satisfies the KKT equalities in Equation~\eqref{eq:KKT-equal-active-m} for all $i \in [m]$ and the KKT inequalities in Equation~\eqref{eq:KKT-ineq-inactive} for all $\ell \in [m]^c$ when applied to $\mYtil_{\Istar}$.
\end{lemma}
Let $\lambdabar_{\pi(m)}$ denote the largest regularization parameter for which all of the first $m$ groups are selected by the group Lasso applied to $\mYtil_{\Istar}$. 
In other words, $\lambdabar_{\pi(m)}$ is the largest value of $\lambda$ such that equation~\eqref{eq:partition-eq} holds and $[m] \subseteq E(\lambda)$ when applied to $\mYtil_{\Istar}$. 
More generally, for each $k \in [m]$, we define $\lambdabar_{\pi(k)}$ as the largest value of $\lambda$ such that at least $k$ groups from the true support $[m]$ are selected, and the newly selected group is $\pi(k)$.
Here, the index function $\pi: [m] \to [m]$ records the order in which the true support groups enter the active set along the solution path: 
$\pi(k)$ identifies the group that is selected when $\lambda$ decreases to $\lambdabar_{\pi(k)}$.
To better understand this entry sequence, we next establish that the largest and smallest regularization levels at which groups are selected, $\lambdabar_{\pi(1)}$ and $\lambdabar_{\pi(m)}$, are close to one another. 

\begin{lemma}\label{lem:alpha-lower}
Suppose $\kappa^2 \sqrt{\mut r}  = o(n^{1/4} \log ^{-1/4}n)$ and $m = \Omega(\log n)$.
Further, assume that
\begin{equation} \label{eq:growth-condition-on-Bstar}
\min_{i\in [m]}\|\mB^\star_{i, \cdot}\|^2_2
\asymp \max_{i \in [m]} \|\mB^\star_{i, \cdot}\|^2_2 
= \Theta(\sigma^2 \sqrt{n \log n}).
\end{equation}
Then with probability at least $1 - O(n^{-6})$, we have
\begin{equation*}
    \left(\frac{\lambdabar_{\pi(m)}}{\lambdabar_{\pi(1)}}\right)^2 \geq 1 - C\sqrt{\frac{\log n}{n}}. 
\end{equation*}
\end{lemma}
The proof is provided in Section~\ref{sec:proof-lem-alpha-lower}.

\subsection{Proof of Theorem \ref{thm:glaso-sufficient-crude}}
\label{sec:proof-thm-glaso-sufficient-crude}

By Lemma~\ref{lem:alter-lasso}, we have constructed a candidate solution $(\alpha_1(\lambdabar_{\pi(m)}), \alpha_2(\lambdabar_{\pi(m)}),\dots, \alpha_m(\lambdabar_{\pi(m)}))$ that satisfies the KKT equalities in Equation~\eqref{eq:KKT-equal-active-m} for the group Lasso applied to the restricted matrix $\mYtil_{\Istar}$, we now seek to verify that this solution also satisfies the KKT inequality conditions in Equation~\eqref{eq:KKT-ineq-inactive} when extended to the full matrix $\mYtil$.
Our approach mirrors the primal-dual witness (PDW) framework, a widely used technique for establishing exact support recovery in structured sparsity problems.
For example, see Chapter 7 of \cite{wainwright2019high} for details of the PDW framework in the context of sparse linear regression.
In short, we show that if the KKT inequalities are satisfied, then the group Lasso solution on $\mYtil$ correctly recovers the support $\Istar$.
We formalize this result in Theorem~\ref{thm:glaso-sufficient-crude}, whose proof follows below.

\begin{proof}[Proof of Theorem~\ref{thm:glaso-sufficient-crude}]
Without loss of generality, assume that the true support is $\Istar = [m]$.
Under the conditions of Theorem~\ref{thm:glaso-sufficient-crude}, one can verify that the sufficient condition in Equation~\eqref{eq:easy-glasso-recovery-cond} holds. 
Let $\lambdabar_{\pi(m)}$ denote the largest regularization parameter such that the group Lasso applied to the restricted matrix $\mYtil_{\Istar}$ selects all $m$ groups as defined above. 
By Lemma~\ref{lem:alter-lasso}, with probability at least $1 - O(n^{-6})$, the solution at $\lambda = \lambdabar_{\pi(m)}$ selects only the groups in $[m]$ when applied to $\mYtil_{\Istar}$.
Define the event
\begin{equation*}
\calE_1 := \left\{ \text{Group Lasso with } \lambda = \lambdabar_{\pi(m)} \text{ selects exactly the first } m \text{ groups} \right\}.
\end{equation*}
Since $\lambdabar_{\pi(m)}$ is the largest value of $\lambda$ for which all $m$ groups are selected, the KKT condition in Equation~\eqref{eq:KKT-equal-active-m} implies that for all $i \in [m]$,
\begin{equation*}
\sum_{k \in[m]}
\left( \frac{\lambdabar_{\pi(m)}}
	{\alpha_k(\lambdabar_{\pi(m)})+\lambdabar_{\pi(m)}}\right)^2 y_{i k}^2
+\sum_{k \in[m]^c} y_{i k}^2 
\geq \lambdabar_{\pi(m)}^2.
\end{equation*}
In particular, by definition of $\lambdabar_{\pi(m)}$, equality holds for $i = \pi(m) \in [m]$.  
Thus, to ensure exact support recovery for $\mYtil$, it suffices to show that for all $\ell \in [m]^c$ and $i \in [m]$,
\begin{equation} \label{eq:pdw-success}
\begin{aligned} 
\sum_{k \in [m]^c} (y_{ik}^2 - y_{\ell k}^2) 
&> \sum_{k \in [m]}
\left(\frac{\lambdabar_{\pi(m)}}
	{\alpha_k(\lambdabar_{\pi(m)}) + \lambdabar_{\pi(m)}}\right)^2 
	\left( y^2_{\ell k} - y^2_{ik} \right),
\end{aligned} \end{equation}
so that the KKT conditions in Equations~\eqref{eq:KKT-equal-active-m} and~\eqref{eq:KKT-ineq-inactive} still hold for the group Lasso applied to the full matrix $\mYtil$.

By Equation~\eqref{eq:alpha_j-upper}, we know for all $k \in [m]$,
\begin{equation*}
\left(\frac{\lambdabar_{\pi(m)}}{\alpha_k(\lambdabar_{\pi(m)}) 
	+ \lambdabar_{\pi(m)}}\right)^2 
\geq \left(\frac{\lambdabar_{\pi(m)}}{\lambdabar_{\pi(1)}}\right)^2.
\end{equation*}
Lemma~\ref{lem:alpha-lower} ensures that with high probability,
\begin{equation*}
    \left(\frac{\lambdabar_{\pi(m)}}{\lambdabar_{\pi(1)}}\right)^2 \geq 1 - Cn^{-1/2} \log^{1/2} n.
\end{equation*}
Define the event
\begin{equation*}
\calE_2
:=\left\{ \left(\frac{\lambdabar_{\pi(m)}}{\lambdabar_{\pi(1)}}\right)^2 
\geq 1 - Cn^{-1/2} \log^{1/2} n \right\} .
\end{equation*}
Fix $a = Cn^{-1/2} \log^{1/2} n$ and define $a_k \in [1 - 2a, 1]$ for $k \in [m]$.
To upper bound the right-hand side of Equation~\eqref{eq:pdw-success}, we consider the term
\begin{equation} \label{eq:to-bound}
\begin{aligned}
    &\sup_{a_k \in [1-2a, 1], k \in [m]} \sum_{k=1}^{m} a_k (y^2_{\ell k} - y^2_{ik}) \\
\end{aligned}
\end{equation}
and upper bound it by decomposition into contributions from noise, signal-noise interaction, and signal as
\begin{equation} \label{eq:bound-to-bound}
\begin{aligned} 
&\sup_{a_k \in [1-2a, 1], k \in [m]} \sum_{k=1}^{m} a_k (\left(\Wtilo_{\ell k}\right)^2 - \left(\Wtilo_{i k}\right)^2), \\
+ &\sup_{a_k \in [1-2a, 1], k \in [m]} 2\sum_{k=1}^{m} a_k \left[ (B^{\star }_{\ell k}+\Deltatil_{\ell k}) \Wtilo_{\ell k} - (B^{\star }_{i k}+\Deltatil_{i k}) \Wtilo_{i k}\right]\\
+ &\sup_{a_k \in [1-2a, 1], k \in [m]} \sum_{k=1}^{m-1} a_k \left[(B^{\star }_{\ell k}+\Deltatil_{\ell k})^2 - (B^{\star}_{i k} + \Deltatil_{i k})^2\right] =: \gamma_1 + \gamma_2 + \gamma_3.
\end{aligned} 
\end{equation}
By Lemma~\ref{lem:ak-Xk-bound}, each of these three terms can be controlled with high probability.
Specifically, with probability at least $1 - O(n^{-6})$, for all $i \in [m]$ and $\ell \in [m]^c$, we have
\begin{equation} \label{eq:g1-bound}
\begin{aligned}
    \gamma_1 =\sup_{a_k \in [1-2a, 1], k \in [m]} & \sum_{k=1}^{m} a_k \left(\left(\Wtilo_{\ell k}\right)^2 - \left(\Wtilo_{i k}\right)^2\right) \\
    &\leq C \sigma^2 m n^{-1/2} \log^{1/2} n  + C\sigma^2 \sqrt{m \log n} 
    \leq C\sigma^2 \sqrt{m \log n}.
\end{aligned}
\end{equation}
and
\begin{equation} \label{eq:g2-bound}
\begin{aligned}
    \gamma_2 := \sup_{a_k \in [1-2a, 1], k \in [m]} &\sum_{k=1}^{m} a_k \left[ (\Bstar_{\ell k}+\Deltatil_{\ell k}) \Wtilo_{\ell k} - (\Bstar_{i k}+\Deltatil_{i k}) \Wtilo_{i k}\right] \\
    &\leq C \sigma \sqrt{\log n} \left[\left(\sum_{k=1}^{m} B^{\star 2}_{\ell k}\right)^{1/2}  + 
    \left(\sum_{k=1}^{m}B^{\star 2}_{i k}\right)^{1/2} + \|\mDeltatil\|_{2, \infty}\right].
\end{aligned}
\end{equation}
For $\gamma_3$, expanding the squares yields that 
\begin{equation} \label{eq:g3-bound}\begin{aligned}
\gamma_3 = \sum_{k=1}^{m} & a_k \left[\left(\Bstar_{\ell k}+\Deltatil_{\ell k}\right)^2-\left(\Bstar_{i k}+\Deltatil_{i k}\right)^2\right] \\
&\leq \sum_{k=1}^{m} B_{\ell k}^{\star 2} - (1-2a)\sum_{k=1}^{m} B_{i k}^{\star 2} + \sum_{k=1}^{m} \Deltatil_{\ell k}^2 + 
    2\|\mDeltatil\|_{2, \infty} \left[\left(\sum_{k=1}^{m} B^{\star 2}_{\ell k}\right)^{1/2}  + \left(\sum_{k=1}^{m}B^{\star 2}_{i k}\right)^{1/2}\right].
\end{aligned} \end{equation}
Applying Theorem~\ref{thm:delta_row} to Equation~\eqref{eq:g3-bound} and combining the upper bounds in Equations~\eqref{eq:g1-bound} \eqref{eq:g2-bound} and~\eqref{eq:g3-bound} with Equation~\eqref{eq:bound-to-bound}, we conclude that the total expression in Equation~\eqref{eq:to-bound} is upper bounded by 
\begin{equation} \label{eq:bound-label-i} \begin{aligned}
    &\sup_{a_k \in [1-2a, 1], k \in [m]} \sum_{k=1}^{m} a_k (y^2_{\ell k} - y^2_{ik}) \leq \\
    & ~~C \sigma  n^{1/4} \log^{1/4} n \left[\left(\sum_{k=1}^{m} B^{\star 2}_{\ell k}\right)^{1/2}  + \left(\sum_{k=1}^{m}B^{\star 2}_{i k}\right)^{1/2}\right]
   + \sum_{k=1}^{m} \left(B_{\ell k}^{\star 2} - (1-2a) B_{i k}^{\star 2}\right) + C \sigma^2 \sqrt{n \log n}
\end{aligned} \end{equation}
with probability at least $1 - O(n^{-6})$.

Meanwhile, for the left-hand side of Equation~\eqref{eq:pdw-success}, Lemma~\ref{lem:portmanteau} ensures that with probability at least $1 - O(n^{-6})$, it holds for all $\ell \in [m]^c$ and all $i \in [m]$ that
\begin{equation} \label{eq:bound-label-ii} \begin{aligned}
    &\sum_{k \in [m]^c} (y_{i k}^2-y_{\ell k}^2)
    \geq  \sum_{k \in [m]^c} B_{i k}^{\star 2} - C \sigma n^{1/4} \log^{1/4} n \left(\sum_{k \in [m]^c} B_{i k}^{\star 2}\right)^{1/2} - C\sigma^2 \sqrt{n \log n} .
\end{aligned} \end{equation}
    
Define the event
\begin{equation*}
    \calE_{3} := \left\{\forall i \in [m], \forall \ell \in [m]^c : \sum_{k \in [m]^c} (y_{i k}^2- y_{\ell k}^2) \geq \sup _{a_k \in[1-2 a, 1], k \in[m]} \sum_{k=1}^{m} a_k\left(y_{\ell k}^2-y_{i k}^2\right)\right\}.
\end{equation*}
When there exists a sufficiently large $C$ such that
\begin{equation} \label{eq:sufficient-glass-condition}
\begin{aligned}
    &\min_{i \in [m]} \left(\sum_{k=1}^n B_{i k}^{\star 2} - C \sigma n^{1/4} \log^{1/4} n \left(\sum_{k=1}^n B_{i k}^{\star 2}\right)^{1/2}\right) \\
    &\geq \max_{\ell \in [m]^c} \left(\sum_{k=1}^{m} B_{\ell k}^{\star 2} + C \sigma n^{1/4} \log^{1/4} n \left(\sum_{k=1}^{m} B^{\star 2}_{\ell k}\right)^{1/2}\right)
    + C \sigma^2 \sqrt{n \log n},
\end{aligned}
\end{equation}
we have that 
\begin{equation*}
    \Pr(\calE_3) \geq 1 - O(n^{-6})
\end{equation*}
by Equations~\eqref{eq:bound-label-i} and~\eqref{eq:bound-label-ii}.
Under the assumptions of Theorem~\ref{thm:glaso-sufficient-crude} in Equation~\eqref{eq:growth-condition-Bstar}, one can verify that condition~\eqref{eq:sufficient-glass-condition} is satisfied for large enough $C$. 
    Therefore, combining with events $\calE_1$ and $\calE_2$, we conclude that 
    \begin{equation*}
        \Pr\left(\calE_1^c \cup \calE_2^c \cup \calE_3^c\right) \leq 
        \Pr(\calE_1^c) + \Pr(\calE_2^c) + \Pr(\calE_3^c) \leq O(n^{-6}),
    \end{equation*}
    which establishes that with high probability, the group Lasso estimator recovers the correct support, completing the proof. 
\end{proof}


\subsection{Technical Lemmas for Group Lasso Support Recovery}

Recall the following definition of the sub-exponential norm (see Definition 2.8.4 in \cite{vershynin2018HDP}). 
\begin{definition}[Sub-exponential Orlicz norm]
For a real-valued random variable $X$, the Orlicz $\psi_1$ norm is defined as 
\begin{equation*}
    \|X\|_{\psi_1}:=\inf \left\{t>0: \mathbb{E}\left[\exp \left(\frac{|X|}{t}\right)\right] \leq 2\right\}
\end{equation*}
\end{definition}
This norm quantifies the tail decay of $X$. 
A finite $\psi_1$-norm implies that $X$ is sub-exponential, i.e., it satisfies the tail bound:
\begin{equation*}
    \Pr\left[ |X|>t \right]
\leq 2 \exp \left\{ -c \frac{t}{\|X\|_{\psi_1}}\right\}
\quad \text { for all } t>0
\end{equation*}
where $c>0$ is an absolute constant.
We will use this definition in the following lemma, which provides a bound on the sum of independent sub-exponential random variables.

\begin{lemma}\label{lem:ak-Xk-bound}
    Let $\left\{X_k\right\}_{k=1}^{m}$ be independent mean-zero sub-exponential random variables, 
    and suppose $a_k \in [1-2a, 1]$ for some $a \in (0,1/2)$ and all $k \in [m]$. 
    Assume $\|X_k\|_{\psi_1}=C_k$ for $k \in [m]$. 
    Then for some constant $c>0$ we have
    \begin{equation*}
        \sup_{a_l \in [1-2a, 1], l \in [m]} \sum_{k=1}^{m} a_k X_k \leq c a \left(\sqrt{\log n} +  
        m^{1 / 2}\right)\left(\sum_{k=1}^{m} C_k^2\right)^{1 / 2} + c \sqrt{\log n}\left(\sum_{k=1}^{m} C_k^2\right)^{1 / 2}
    \end{equation*}
    with probability at least $1 - O(n^{-7})$.
\end{lemma}
\begin{proof}
    Decompose the sum as
    \begin{equation} \label{eq:aXsum:decomp}
    \begin{aligned}
        \sum_{k=1}^{m} a_k X_k = \underbrace{\sum_{k=1}^{m} (a_k - 1 + a) X_k}_{A_1} + \underbrace{(1-a) \sum_{k=1}^{m} X_k}_{A_2}. 
    \end{aligned}  
    \end{equation}
    To bound $A_1$, we note that $a_k + a - 1 \in [-a, a]$. 
    Hence,
    \begin{equation*}
    \begin{aligned}
        A_1 &\leq \sup_{a_k \in [1-2a, 1]} \sum_{k=1}^{m} (a_k - 1 + a) X_k \leq a \sum_{k=1}^{m} |X_k| = \max_{\vu \in \{\pm 1\}^{m}} a \sum_{k=1}^{m} X_k u_k.
    \end{aligned}
    \end{equation*}
Applying a union bound over the $2^m$ sign vectors and using Bernstein's inequality for sub-exponential variables \citep[see Theorem 2.9.1 in][]{vershynin2018HDP}), we have that for any $t > 0$,
\begin{equation*}
\begin{aligned}
    \Pr\left[ A_1 > a t \right]
&\leq 2^{m}  \exp \left( -c \min \left\{ \frac{t^2}{\sum_{k=1}^{m} C_k^2}, \frac{t}{\max_{k\in[m]} C_k}\right\} \right)\\
    &\leq \exp \left(-c \min \left\{\frac{t^2}{\sum_{k=1}^{m} C_k^2}, \frac{t}{\max_{k\in[m]} C_k}\right\} + m\right).
\end{aligned}
\end{equation*}
Setting $t = 7c^{-1}\log^{1/2} n \sqrt{\sum_{k=1}^{m} C_k^2} + c^{-1} \sqrt{m\sum_{k=1}^{m} C_k^2}$ yields that 
\begin{equation} \label{eq:A1:bound}
    A_1 \leq ac^{-1} (7\sqrt{\log n} + c^{-1} m^{1/2})\left(\sum_{k=1}^{m} C_k^2\right)^{1/2}
\end{equation}
with probability $1 - O(n^{-7})$.
    
    The term $A_2$ can be bounded via applying Bernstein's inequality to $\sum_{k=1}^m X_k$, and it follows that
    \begin{equation} \label{eq:A2:bound}
        A_2 \leq 7 c^{-1} \sqrt{\log n}\left(\sum_{k=1}^{m} C_k^2\right)^{1/2}
    \end{equation}
    with probability at least $1 - O(n^{-7})$. 

Applying Equations~\eqref{eq:A1:bound} and~\eqref{eq:A2:bound} to Equation~\eqref{eq:aXsum:decomp} completes the proof.
\end{proof}

\begin{lemma}\label{lem:portmanteau}
Suppose that the assumptions of Theorem~\ref{thm:glaso-sufficient-crude} holds.
Let $\mYtil$ be the matrix defined in Equation~\eqref{eq:single-observed} and for each $i \in [n]$, let $\vy_i$ be the $i$-th row of $\mYtil$.
Then, with probability at least $1 - O(n^{-6})$, for all $i \in [n]$ and $\ell \in [m]^c$, it holds that
With probability at least $1 - O(n^{-6})$, for some sufficiently large constant $C > 0$, it holds for all suitably large $n$ that for all $i \in [n]$,
\begin{equation} \label{eq:all-y-lower}
    \|\vy_{i}\|_2^2 \geq
    \left\|\mBstar_{i, \cdot}\right\|_2^2+n\sigma^2 -C\sigma^2 \sqrt{n\log n} - C\sigma n^{1/4} \log^{1/4} n \|\mBstar_{i,\cdot}\|_{2},
\end{equation}
and
\begin{equation} \label{eq:all-y-upper}
    \|\vy_{i}\|_2^2 \leq
    \left\|\mBstar_{i, \cdot}\right\|_2^2+n\sigma^2 +C\sigma^2 \sqrt{n\log n} + C\sigma n^{1/4} \log^{1/4} n \|\mBstar_{i,\cdot}\|_{2}.
\end{equation}
Further, if $m = \Omega(\log n)$, it holds with probability at least $1-O(n^{-6})$ that for all $i \in [n]$,
\begin{equation} \label{eq:all-y-lower-m}
    \sum_{j=1}^m y_{ij}^2 \geq
    \sum_{j=1}^m B^{\star 2}_{ij} + m\sigma^2 -C\sigma^2 \sqrt{n\log n} - C\sigma n^{1/4} \log^{1/4} n \left(\sum_{j=1}^m B^{\star 2}_{ij}\right)^{1/2},
\end{equation}
and
\begin{equation} \label{eq:all-y-upper-m}
    \sum_{j=1}^m y_{ij}^2 \leq
    \sum_{j=1}^m B^{\star 2}_{ij} + m\sigma^2 + C\sigma^2 \sqrt{n\log n} + C\sigma n^{1/4} \log^{1/4} n \left(\sum_{j=1}^m B^{\star 2}_{ij}\right)^{1/2} .
\end{equation}
\end{lemma}
\begin{proof}
We only present the proof for Equation~\eqref{eq:all-y-lower}, as the other bouds follow by similar arguments.
Straightforward calculation yields that for $i \in [n]$, 
\begin{equation} \label{eq:vyi:expandLB} \begin{aligned}
    \|\vy_{i}\|_2^2 &\geq
    \left\|\mBstar_{i, \cdot}\right\|_2^2+\left\|\mWtilo_{i, \cdot}\right\|_2^2-2\left|\mDeltatil_{i, \cdot}^\top \mBstar_{i, \cdot}\right|-2\left|\mWtil_{i,\cdot}^{(1)\top} \mDeltatil_{i, \cdot}\right|-2\left|\mWtil_{i,\cdot}^{(1)\top} \mBstar_{i, \cdot}\right|\\
    &\geq \left\|\mBstar_{i, \cdot}\right\|_2^2+\left\|\mWtilo_{i, \cdot}\right\|_2^2-2\|\mDeltatil_{i, \cdot}\|_{2}\| \mBstar_{i, \cdot}\|_2 - 2\left|\mWtil_{i,\cdot}^{(1)\top} \mDeltatil_{i, \cdot}\right|-2\left|\mWtil_{i,\cdot}^{(1)\top} \mBstar_{i, \cdot}\right|,
\end{aligned} \end{equation}
where the inequality follows from the Cauchy-Schwarz inequality.
Conditioned on $\mDelta$, we apply Gaussian tail bounds to control $\mWtil_{i,\cdot}^{(1)\top} \mDeltatil_{i, \cdot}$, from which
\begin{equation*}
\Pr\left[\left| \mWtil_{i,\cdot}^{(1)\top} \mDeltatil_{i, \cdot}\right| \geq t \right]
\leq 2\exp\left\{-\frac{t^2}{2\sigma^2 \|\mDeltatil_{i, \cdot}\|_2^2}\right\}.
\end{equation*}
Taking $t = C\sigma\|\mDeltatil\|_{2,\infty}\sqrt{\log n}$ for $C$ sufficiently large, it holds with probability at least $1-O(n^{-7})$ that
\begin{equation} \label{eq:WtilDeltil:UB}
\left| \mWtil_{i,\cdot}^{(1)\top} \mDeltatil_{i, \cdot}\right| 
\leq  C\sigma\|\mDeltatil\|_{2,\infty}\sqrt{\log n} .
\end{equation}
Similarly, also with probability at least $1-O(n^{-7})$,
\begin{equation} \label{eq:WtilBstar:UB}
\left| \mWtil_{i,\cdot}^{(1)\top} \mBstar_{i, \cdot}\right| \leq  C\sigma\|\mBstar_{i,\cdot}\|_{2}\sqrt{\log n} .
\end{equation} 
Applying Equations~\eqref{eq:WtilDeltil:UB} and~\eqref{eq:WtilBstar:UB} to Equation~\eqref{eq:vyi:expandLB} along with Lemma~\ref{lem:chi-sq-concen} to control the squared $\ell_2$ norm of rows of $\mWtilo$, 
it follows that
\begin{equation} \label{eq:all-y-mid-lower}
\begin{aligned}
    \|\vy_{i}\|_2^2 \geq
    \left\|\mBstar_{i, \cdot}\right\|_2^2+n\sigma^2 -C\sigma^2 \sqrt{n\log n}-C\sigma\|\mDeltatil\|_{2,\infty}\sqrt{\log n}- (C\sigma\sqrt{\log n} + \|\mDeltatil\|_{2,\infty})\|\mBstar_{i,\cdot}\|_{2}.
\end{aligned}
\end{equation}
Applying a union bound over $i \in [n]$ yields that Equation~\eqref{eq:all-y-mid-lower} holds with probability at least $1 - O(n^{-6})$ for all $i \in [n]$, conditioned on $\mDelta$. 
Conditioned on the event $\{\|\mDelta\|_{2,\infty} \leq C\sigma n^{1/4} \log^{1/4} n\}$, we have that 
\begin{equation*}
    \|\vy_{i}\|_2^2 \geq
    \left\|\mBstar_{i, \cdot}\right\|_2^2+n\sigma^2 -C\sigma^2 \sqrt{n\log n} - C\sigma n^{1/4} \log^{1/4} n \|\mBstar_{i,\cdot}\|_{2}
\end{equation*} 
from Equation~\eqref{eq:all-y-mid-lower}. 
By Theorem \ref{thm:delta_row} and the assumption that $\kappa^2 \sqrt{\mut r} = o(n^{1/4} \log^{-1/4} n)$, we have that 
$\{\|\mDelta\|_{2,\infty} \leq C\sigma n^{1/4} \log^{1/4} n\}$ holds with probability at least $1 - O(n^{-6})$ for some $C$ sufficiently large. 
We then conclude that for all $i \in [n]$, Equation~\eqref{eq:all-y-lower}
holds with probability at least $1 - O(n^{-6})$. 
\end{proof}

\subsubsection{Proof of Lemma \ref{lem:alter-lasso}}
\label{sec:proof-lem-alter-lasso}
\begin{proof}
Write $\mYtil_{\Istar} = (y_{ij})_{1\leq i \leq j \leq n}$, where $\mYtil$ is given in Equation~\eqref{eq:single-observed} and consider the condition
\begin{equation}\label{eq:easy-glasso-condition}
    \min_{i \in [m]} \sum_{k=m+1}^n y_{ik}^2 > \max_{\ell \in [m]^c} \sum_{k=1}^m y_{\ell k}^2.
\end{equation}
If this condition holds, then we can choose $\lambdatil$ such that
\begin{equation*}
\lambdatil^2 \in \left(\max_{\ell \in [m]^c} \sum_{k=1}^m y_{\ell k}^2, \min_{i \in [m]} \sum_{k=m+1}^n y_{ik}^2\right).
\end{equation*}
With such a choice, the KKT conditions in Equation~\eqref{eq:KKT-ineq-inactive} imply
\begin{equation*}
    \sum_{k=1}^m \left(\frac{\lambdatil}{\alpha_k+\lambdatil}\right)^2 y_{i k}^2 
    < \lambdatil^2, \quad \text { if } i \in[m]^c
\end{equation*}
and 
\begin{equation*}
    \sum_{k=1}^m \left(\frac{\lambdatil}{\alpha_k+\lambdatil}\right)^2 y_{i k}^2 + \sum_{k=m+1}^n  y_{i k}^2 
    > \lambdatil^2, \quad \text { if } i \in [m]. 
\end{equation*}
Hence, with regularization parameter $\lambda = \lambdatil$, group Lasso will recover exactly the first $m$ groups provided that Equation~\eqref{eq:easy-glasso-condition} holds. 

To ensure that Equation~\eqref{eq:easy-glasso-condition} holds with high probability, we apply the bounds from Lemma~\ref{lem:portmanteau}. 
Specifically, using Equation~\eqref{eq:all-y-lower} and~\eqref{eq:all-y-upper-m}, the left hand side of Equation~\eqref{eq:easy-glasso-condition} is lower bounded with probability at least $1 - O(n^{-6})$ by
\begin{equation*}
\begin{aligned}
    &\min_{i \in [m]} \left(\sum_{k=m+1}^n B_{ik}^{\star 2} - C \sigma n^{1/4} \log^{1/4} n \left(\sum_{k=m+1}^n B_{ik}^{\star 2}\right)^{1/2} \right) + (n-m)\sigma^2 - C \sigma^2 \sqrt{n \log n} \\
    &~~\geq (1 - n^{-1/4}) \min_{i \in [m]} \sum_{k=m+1}^n B_{ik}^{\star 2} + (n-m)\sigma^2 - C \sigma^2 n^{3/4} \log^{1/2} n ,
\end{aligned}
\end{equation*}
where the inequality holds from the fact that $2ab < n^{-1/4}a^2 + n^{1/4}b^2$. 

Using Equations~\eqref{eq:all-y-upper} and~\eqref{eq:all-y-lower-m}, the right hand side of Equation~\eqref{eq:easy-glasso-condition} is upper bounded by
\begin{equation*}
\begin{aligned}
    \max_{\ell \in [m]} & \left(\sum_{k=1}^m B_{\ell k}^{\star 2} + C \sigma n^{1/4} \log^{1/4} n \left(\sum_{k=1}^m B_{\ell k}^{\star 2}\right)^{1/2} \right) + m\sigma^2 + C \sigma^2 \sqrt{n \log n}\\
    &~~\leq (1+n^{-1/4}) \max_{\ell \in [m]} \sum_{k=1}^m B_{\ell k}^{\star 2} + m\sigma^2 + C \sigma^2 n^{3/4} \log^{1/2} n
\end{aligned}
\end{equation*}
with probability at least $1 - O(n^{-6})$, where the last inequality follows similar from the fact that $2ab < n^{-1/4}a^2 + n^{1/4}b^2$.

Therefore, under condition~\eqref{eq:easy-glasso-recovery-cond} and for sufficiently large $n$, Equation~\eqref{eq:easy-glasso-condition} holds with probability at least $1 - O(n^{-6})$, completing the proof.
\end{proof}

\subsubsection{Proof of Lemma \ref{lem:alpha-lower}}
\label{sec:proof-lem-alpha-lower}

\begin{proof}
Under the given assumptions, the assumptions in Lemma \ref{lem:alter-lasso} hold.
Thus, with high probability, applying group Lasso to $\mYtil_{\Istar}$ with $\lambdabar_{\pi(m)}$ only selects the first $m$ groups. 
By the KKT equality condition in Equation~\eqref{eq:KKT-equal-active-m}, this implies that for $\pi(m) \in [m]$, 
\begin{equation*}
    \sum_{k \in[m]}\left(\frac{\lambdabar_{\pi(m)}}{\alpha_k(\lambdabar_{\pi(m)})+\lambdabar_{\pi(m)}}\right)^2 y_{\pi(m) k}^2+\sum_{k \in[m]^c} y_{\pi(m) k}^2 = \lambdabar_{\pi(m)}^2.
\end{equation*}
In other words, $\pi(m)$ corresponds to the last group in $[m]$ to be selected as $\lambda$ decreases. 
Without loss of generality, we relabel so that $\pi(m) = m$ and $\pi(1) = 1$.

We now use the upper bound on $\alpha_k(\lambda)$ from Equation~\eqref{eq:alpha_j-upper} and note that since $\lambdabar_k$ is the largest regularization parameter at which the $k$-th group is selected, we have $\lambdabar_k \leq \lambdabar_1$ for all $k \in [m]$. 
Plugging this into the KKT inequality condition~\eqref{eq:KKT-ineq-inactive}, we obtain
\begin{equation*}
\begin{aligned}
    &\sum_{k=1}^{m-1}\left(\frac{\lambdabar_m}{\lambdabar_k}\right)^2 y_{m k}^2+\sum_{k = m}^n y_{m k}^2 \leq \lambdabar_m^2
\end{aligned}
\end{equation*}
Using the bound $\lambdabar_k \leq \lambdabar_1$, this implies 
\begin{equation*} 
    \sum_{k=m}^n y_{m k}^2 \leq \left(1 - \sum_{k =1}^{m-1}\frac{y_{m k}^2}{\lambdabar_1^2}\right) \lambdabar_m^2,
\end{equation*}
which futther implies that 
\begin{equation}\label{eq:equation-needs-to-be-referenced}
    \left(\frac{\lambdabar_m}{\lambdabar_1}\right)^2 \geq \frac{\sum_{k=m}^n y_{m k}^2}{\lambdabar^2_1 - \sum_{k=1}^{m-1} y_{k j}^2} = \left(\frac{\lambdabar_m}{\lambdabar_1}\right)^2 \geq \frac{\sum_{k=m}^n y_{k j}^2}{\sum_{k=1}^n y_{1 k}^2 - \sum_{k=1}^{m-1} y_{m k}^2}
\end{equation}
where in the last equality, we use the assumption that the first selected group is $\pi(1) = 1$ so that $\lambdabar_1 = \sum_{k=1}^n y_{1k}^2$.  
To further bound this ratio from below, we invoke Equations~\eqref{eq:all-y-lower}, \eqref{eq:all-y-upper}, and \eqref{eq:all-y-lower-m} from Lemma~\ref{lem:portmanteau} to ensure that with probability at least $1 - O(n^{-6})$, we have
\begin{equation} \label{eq:beta_j-lower}
\begin{aligned}
    \frac{\sum_{k=m}^n y_{k j}^2}{\sum_{k=1}^n y_{1 k}^2 - \sum_{k=1}^{m-1} y_{m k}^2} &\geq \frac{(n-m) \sigma^2-C \sigma^2 \sqrt{n \log n} + \|\mB^\star_{m, \cdot}\|_2^2 - \sum_{k=1}^{m-1} B_{m k}^{\star 2} - \delta_m}{(n-m) \sigma^2 + C\sigma^2 \sqrt{n\log n} + \|\mB^\star_{1, \cdot}\|_2^2 - \sum_{k=1}^{m-1} B_{m k}^{\star 2} + \delta_1}\\
    &= 1 - \frac{C \sigma^2 \sqrt{n \log n}}{(n-m) \sigma^2 + C\sigma^2 \sqrt{n\log n}}\\
\end{aligned}
\end{equation}
where we define the deviation terms 
\begin{equation*}
    \delta_1 =  C\sigma n^{1/4} \log^{1/4} n \left(\|\mB^\star_{1, \cdot}\|_2 + \left(\sum_{k=1}^{m-1} B_{m k}^{\star 2}\right)^{1/2} \right) \asymp C\sigma^2 \sqrt{n \log n}
\end{equation*}
and
\begin{equation*}
    \delta_m = C\sigma n^{1/4} \log^{1/4} n \left(\|\mB^\star_{m, \cdot}\|_2 + \left(\sum_{k=1}^{m-1} B_{m k}^{\star 2}\right)^{1/2} \right) \asymp C\sigma^2 \sqrt{n \log n}.
\end{equation*}
The rates of $\delta_1$ and $\delta_m$ follow from Equation~\eqref{eq:growth-condition-on-Bstar}
Substituting the lower bound in Equation~\eqref{eq:beta_j-lower} into the right hand side of Equation~\eqref{eq:equation-needs-to-be-referenced} yields
\begin{equation*}
\begin{aligned}
    \left(\frac{\lambdabar_m}{\lambdabar_1}\right)^2 &\geq 1 - \frac{C \sigma^2 \sqrt{n \log n}}{(n-m) \sigma^2 + C\sigma^2 \sqrt{n\log n}} .
\end{aligned}
\end{equation*}
Applying the fact that $m = o(n)$ and $\sigma^2 > 0$ is fixed, 
\begin{equation*}
    \left(\frac{\lambdabar_m}{\lambdabar_1}\right)^2 \geq 1 - C\sqrt{\frac{\log n}{n}},
\end{equation*}
completing the proof.
\end{proof}

\subsection{Proof of Theorem \ref{thm:hard-thresh-suffice}}
\label{sec:hard-thresh-suffice}
\begin{proof}
Let $\calH_m$ denote the hard-thresholding operator that selects the $m$ rows (i.e., groups) of a symmetric matrix with the largest $\ell_2$ norms.
Formally, for any matrix $\mA \in \Sym(n)$, let $S \subseteq [n]$ be the index set of the $m$ largest values of $\|\mA_{i, \cdot}\|_2$ for $i \in [n]$. 
Then, 
\begin{equation}\label{eq:hard-thresh:define}
    [\calH_m(\mA)]_{ij} := \begin{cases}
        A_{i j}, & \text { if } i \in S \text{ or } j \in S\\ 
        0, & \text { otherwise } .
    \end{cases}
\end{equation}
The hard-thresholding estimator arises from applying the hard-thresholding operator, defined in Equation~\eqref{eq:hard-thresh:define}, to the entries of the matrix 
\begin{equation*}
    \mYtil = \mB^\star + \mDeltatil + \mWtilo, 
\end{equation*}
where $\mYtil$ is the output from Step~\ref{alg:general:step:select} in Algorithm~\ref{alg:general}. 
As before, we assume without loss of generality that the selected index set is $\calU = [n]$ and $\Istar = [m]$. 
Let $\mYtil = (y_{ij})_{1\leq i, j \leq n}$. 

The hard-thresholding estimator selects the $m$ rows of $\mYtil$ with the largest row-wise $\ell_2$ norms. 
Therefore, to show it recovers the node support $\Istar$, it suffices to show that with high probability,
\begin{equation*}
\min _{i \in[m]} \sum_{j=1}^n y_{i j}^2>\max _{\ell \in[m]^c} \sum_{j=1}^n y_{\ell j}^2.
\end{equation*}

By Equation~\eqref{eq:all-y-lower} in Lemma~\ref{lem:portmanteau}, we have that with probability at least $1 - O(n^{-6})$,
\begin{equation*} \begin{aligned}
    \sum_{j=1}^n y_{ij}^2
    &\geq \left\|\mBstar_{i, \cdot}\right\|_2^2+n\sigma^2 -C\sigma^2 \sqrt{n\log n} - C\sigma \|\mBstar_{i,\cdot}\|_{2} \cdot n^{1/4} \log^{1/4} n ,
\end{aligned} \end{equation*}
holds for all $i \in [m]$.
Similarly, by Equation~\eqref{eq:all-y-upper} in Lemma~\ref{lem:portmanteau}, it also holds with probability at least $1 - O(n^{-6})$ that 
\begin{equation*} \begin{aligned}
    \sum_{j=1}^n y_{\ell j}^2 &\leq \left\|\mBstar_{\ell, \cdot}\right\|_2^2+n\sigma^2 +C\sigma^2 \sqrt{n\log n} + C\sigma \|\mBstar_{\ell,\cdot}\|_{2} \cdot n^{1/4} \log^{1/4} n ,
\end{aligned} \end{equation*}
for all $\ell \in [m]^c$.
Under the assumptions of Theorem~\ref{thm:glaso-sufficient-crude}, 
particularly the signal strength and separation conditions, we have
\begin{equation*}
    \min_{i \in [m]} \left\|\mBstar_{i,\cdot}\right\|_2^2 \geq \max_{\ell \in [m]^c} \left\|\mBstar_{\ell, \cdot}\right\|_2^2 + C\sigma^2 \sqrt{n \log n},
\end{equation*}
for some sufficiently large constant $C > 0$.
It then follows that 
\begin{equation}  \label{eq:hard-thresh:success}
    \min_{i \in [m]} \sum_{j=1}^n y_{ij}^2 > \max_{\ell \in [m]^c} \sum_{j=1}^n y_{\ell j}^2,
\end{equation}
with probability at least $1 - O(n^{-6})$.
Equation~\eqref{eq:hard-thresh:success} implies that the top-$m$ hard-thresholding rule will select the first $m$ indices with high probability, and thus, correctly recovers the true support $\Istar$. 
\end{proof}

\subsection{Proof of Theorem \ref{thm:failure}}
\label{sec:proof:failure}

\begin{proof}
Suppose that without loss of generality, $\Istar = [m]$. 
Since both group Lasso and hard thresholding are selection-based procedures, it suffices to show that with high probability, the $(m+1)$-th group, a non-signal group, has the largest row norm, and thus gets selected before any group in $\Istar$. 
We first bound the squared row norms of the signal and non-signal groups using Lemma~\ref{lem:portmanteau}.
Applying Equation~\eqref{eq:all-y-lower} shows that for the $(m+1)$th row,
\begin{equation*}
    \|\vy_{m+1}\|_2^2 \geq m\beta_1^2 + n\sigma^2 - C\sigma^2 \sqrt{n \log n} - C\beta_1 \sigma \sqrt{m \log n},
\end{equation*}
with high probability.
Applying Equation~\eqref{eq:all-y-upper}, for any $\ell \in [m+1]^c$ (i.e., a non-signal row beyond $m+1$):
\begin{equation*}
    \max_{\ell \in [m+1]^c}\|\vy_{\ell}\|_2^2 \leq m\beta_2^2 + n\sigma^2 + C\sigma^2 \sqrt{n \log n} + C\beta_2 \sigma \sqrt{m \log n},
\end{equation*}
and for any $i \in [m]$ (true signal row):
\begin{equation*}
    \max_{i \in [m]}\|\vy_{i}\|_2^2 \leq n\beta_2^2 + n\sigma^2 + C\sigma^2 \sqrt{n \log n} + C\beta_2 \sigma \sqrt{n \log n} + \beta_1^2 + C\sigma \beta_1 \sqrt{\log n}
\end{equation*}
both with high probability.

Substituting $\beta_2 = C \sigma n^{-1/4} \log^{1/4} n$,
the signal contribution to $\vy_i$ is $O(n^{1/2} \log^{1/2} n)$,
while the contribution from $\beta_1$ in $\vy_{m+1}$ is $m \beta_1^2$.
If $\beta_1 \geq C \sigma (n \log n)^{1/4} / \sqrt{m}$ for large enough $C$, we have
\begin{equation*}
    m \beta_1^2 \gg \sigma^2 \sqrt{n \log n},
\end{equation*}
and thus $\|\vy_{m+1}\|_2^2 > \|\vy_i\|_2^2$ and $\|\vy_{m+1}\|_2^2 > \|\vy_\ell\|_2^2$ for all $i \in [m]$ and $\ell \in [m+1]^c$ with high probability.
Therefore, both group Lasso and hard thresholding will incorrectly prioritize the $(m+1)$-th group, leading to failure in exact support recovery.
\end{proof}

\section{Proofs of Main Results on Recovering the Common Structure}

In this section, we prove the upper bounds on eigenspace and entrywise estimation of the low rank component $\mMstar$ obtained in Section~\ref{sec:common-structure} for Theorems~\ref{thm:eigenvec:l2}, \ref{thm:improved:linear-form-bound}, \ref{thm:eigenspace:l2infty}, \ref{thm:improved:entrywise} and Corollaries~\ref{cor:single:eigenspace} and~\ref{cor:single:entrywise}. 

\subsection{Proof of Theorem~\ref{thm:eigenvec:l2}}
\label{sec:proof:thm:eigenvec:l2}

\begin{proof}
The result largely follows from adapting results in \cite{cheng2021tackling} and using new concentration results introduced in \cite{yan2025improved}. 
We will establish the three main results of Theorem~\ref{thm:eigenvec:l2}, namely $\ell_2$-perturbation, perturbation of linear forms and $\ell_\infty$-perturbation, in sequence.
We will only provide some key steps in the proof, and refer the reader to \cite{cheng2021tackling} for more details.
As a reminder, we remind the reader that $\mH$ is an asymetric noise matrix given in Equation~\eqref{eq:M:asym}, $\lambda_l$ and $\vu_l$ are the $l$-th eigenvalue and eigenvector of $\mM$, respectively, as given in Equation~\eqref{eq:left-right-eigenvc}. 

\paragraph{$\ell_2$-perturbation.} 
By Equations (83) and (84) in \cite{cheng2021tackling}, we have 
\begin{equation} \label{eq:l2:I}
\sqrt{1 - \sum_{k=1}^r \left|\vustart_k \vu_l\right|^2} 
\leq \left\|\vu_l - \sum_{k=1}^r \frac{\lambdastar_k \vustart_k \vu_l}{\lambda_l} \vustar_k\right\|_2 \leq \frac{8\kappa}{3} \frac{\|\mH\|}{\lambdastar_{\min}}.
\end{equation} 
The orthonormality of $\{\vustar_i\}_{i\in[r]}$ implies that 
\begin{equation*} 
\vustart_k \vu_l 
= \left(1 - \lambdastar_k/\lambda_l\right)^{-1} \vustart_k \left(\vu_l - \sum_{i=1}^r \frac{\vustart_i \vu_l}{\lambda_l / \lambdastar_i} \vustar_i\right).
\end{equation*}
Taking $\va = \vustar_k$ in Theorem~\ref{thm:linear-form-rank-r}, with probability at least $1-O(n^{-17})$, we have for all $k \neq l$, 
\begin{equation} \label{eq:l2:III}
\begin{aligned}
    \left|\vustart_k \vu_l\right| &\lesssim \left|1 - \lambdastar_k/\lambda_l\right|^{-1} \sqrt{\frac{\kappa^2 r \log^4 n}{n}} \max\left\{\frac{\mub^{1/2} L \log^3 n}{\lambdastar_{\min}} \|\vustar_k\|_{\infty}, \left(\frac{\sigma \sqrt{n \log^3 n}}{\lambdastar_{\min}}\right)^{k_0}\right\}\\
    &\lesssim \left|1 - \lambdastar_k/\lambda_l\right|^{-1} \frac{\sigma \sqrt{\kappa^2 r \log^7 n}}{\lambdastar_{\min}} \lesssim \frac{\sigma \sqrt{\kappa^4 r \log^7 n}}{\delta^\star_l},
\end{aligned}
\end{equation} 
which yields the bound in \ref{eq:l2:perturb:inner:diff}.
Putting the above in Equations~\eqref{eq:l2:I} and~\eqref{eq:l2:III} together (which are improved results of Equation (84) and (89) in \cite{cheng2021tackling}), we obtain the bound in Equation~\eqref{eq:l2:perturb:norm} (see details in Equation (90) in \cite{cheng2021tackling}). 
The bound in Equation~\eqref{eq:l2:perturb:inner:same} then follows from the identity $\left|\vustart_l \vu_l\right| =1-\frac{1}{2} \min \left\{\left\|\vu_l \pm \vustar_l \right\|_2^2\right\}$. 

\paragraph{Perturbation of linear forms.}
By the triangle inequality,
\begin{equation} \label{eq:IV}
\left|\boldsymbol{a}^{\top}\left(\vu_l
	-\frac{\vustart_l \vu_l}{\lambda_l / \lambdastar_l} \vustar_l \right)
\right|
\leq
\left|\boldsymbol{a}^{\top}\left(\vu_l-\sum_{k=1}^r \frac{\vustart_k \vu_l}{\lambda_l / \lambdastar_k} \vustar_k \right)\right|
+\sum_{k \neq l, k=1}^r\left|\frac{\vustart_k \vu_l}{\lambda_l / \lambdastar_k}
\left(\boldsymbol{a}^{\top} \vustar_k \right)\right| .
\end{equation}
Applying Theorem~\ref{thm:linear-form-rank-r} to the first term in the above display, with probability at least $1-O(n^{-17})$, we have for all $k \neq l$,
\begin{equation} \label{eq:first-term}
    \left|\boldsymbol{a}^{\top}\left(\boldsymbol{u}_l-\sum_{k=1}^r \frac{\boldsymbol{u}_k^{\star \top} \boldsymbol{u}_l}{\lambda_l / \lambda_k^{\star}} \boldsymbol{u}_k^{\star}\right)\right| \lesssim \frac{\sigma\sqrt{\kappa^2 r \log^7 n}}{\lambdastar_{\min}}.
\end{equation}
By Equation~\eqref{eq:l2:perturb:inner:diff},
\begin{equation} \label{eq:inner:diff:over:ev}
\left|\frac{\vustart_k \vu_l}{\lambda_l / \lambdastar_k}\right|
\lesssim \frac{\sigma}{\left|\lambdastar_l-\lambdastar_k\right|}
		\sqrt{\kappa^4 r \log^7 n},
\end{equation}
and it follows that the second term in Equation~\eqref{eq:IV} is bounded as
\begin{equation*}
    \sum_{k\neq l, k=1}^r \left|\frac{\boldsymbol{u}_k^{\star \top} \boldsymbol{u}_l}{\lambda_l / \lambda_k^{\star}}(\va^\top \vustar_k)\right| \lesssim \sigma r \sqrt{\kappa^4 r \log^7 n} \left\{ \max_{k: k\neq l} \frac{\va^\top \vustar_k}{\left|\lambda_l^{\star}-\lambda_k^{\star}\right|} \right\}
\end{equation*}
Combining the bound in the above display and Equation~\eqref{eq:first-term} with Equation~\eqref{eq:IV}, it follows that 
\begin{equation} \label{eq:first:order:bound}
\left|\boldsymbol{a}^{\top}
\left(\vu_l-\frac{\vustart_l \vu_l}{\lambda_l / \lambdastar_l}\vustar_l \right)
\right| 
\lesssim \sigma \sqrt{\kappa^2 r \log^7 n}
\left\{\frac{1}{\lambdastar_{\min }}
	+\kappa r \max_{k: k \neq l} 
	\frac{\left|\boldsymbol{a}^{\top} \vustar_k\right|}
	{\left|\lambdastar_l-\lambdastar_k\right|}\right\} .
\end{equation}
Since
\begin{equation*}
\min \left|1 \pm \frac{\vustart_l \vu_l}{\lambda_l / \lambdastar_l}\right|
\lesssim \frac{\kappa r^2 \sigma \sqrt{ \log n}}{\lambdastar_{\min}}
+\frac{\sigma^2 \kappa^4 r^2 \log ^7 n}{\left(\delta_l^{\star}\right)^2}
+\frac{\kappa^2 \sigma^2 n \log n}{\left(\lambdastar_{\min }\right)^2},
\end{equation*}
we have
\begin{equation*} \begin{aligned}
    \min \left|\boldsymbol{a}^{\top}\left(\boldsymbol{u}_l \pm \boldsymbol{u}_l^{\star}\right)\right| &\leq \left|\boldsymbol{a}^{\top}\left(\boldsymbol{u}_l-\frac{\boldsymbol{u}_l^{\star \top} \boldsymbol{u}_l}{\lambda_l / \lambda_l^{\star}} \boldsymbol{u}_l^{\star}\right)\right|+\min \left|1 \pm \frac{\boldsymbol{u}_l^{\star \top} \boldsymbol{u}_l}{\lambda_l / \lambda_l^{\star}}\right| \cdot\left|\boldsymbol{a}^{\top} \boldsymbol{u}_l^{\star}\right| \\
    &\lesssim \sigma \sqrt{\kappa^2 r \log^7 n}\left\{\frac{1}{\lambda_{\min }^{\star}}+\kappa r \max _{k: k \neq l} \frac{\left|\boldsymbol{a}^{\top} \boldsymbol{u}_k^{\star}\right|}{\left|\lambda_l^{\star}-\lambda_k^{\star}\right|}\right\} \\
    &+\left(\frac{\sigma^2 \kappa^4 r^2 \log ^7 n}{\left(\delta_l^{\star}\right)^2}+\frac{\kappa^2 \sigma^2 n \log n}{\left(\lambda_{\min }^{\star}\right)^2}\right) \left|\va^\top \vustar_l\right|
\end{aligned} \end{equation*}
which completes the proof of Equation~\eqref{eq:naive:linear-form}. 

\paragraph{$\ell_\infty$-perturbation.} 
Setting $\va$ to the standard bases yields the bound in \eqref{eq:naive:entrywise}, which completes the proof.
\end{proof}

\subsection{Proof of Theorem~\ref{thm:improved:linear-form-bound}}
\label{sec:proof:thm:improved:linear-form-bound}

\begin{proof}
We remind the reader that $\vw_l$ is the $l$-th left eigenvector of $\mM$ given in Equation~\eqref{eq:left-right-eigenvc}, and they satisfy the same results as stated in Theorem~\ref{thm:eigenvec:l2} since one can repeat the proof on $\mM^\top$ to obtain the same bounds.
By Equation (101) in \cite{cheng2021tackling} (and also viewing from Equation~\eqref{eq:first-order-approx}), the square of $\uhat_{\va, l}$ defined in Equation~\eqref{eq:linear-form:est} is given by 
\begin{equation}\label{eq:square:linear-form}
\left|\frac{\left(\boldsymbol{a}^{\top} \vustar_l \cdot 
	\frac{\vustart_l \vu_l}{\lambda_l / \lambdastar_l} +\gamma_1\right)
	\left(\boldsymbol{a}^{\top} \vustar_l \cdot 
	\frac{\vustart_l \boldsymbol{w}_l}{\lambda_l / \lambdastar_l}
	+\gamma_2\right)}
	{\frac{\vustart_l \boldsymbol{w}_l}{\lambda_l / \lambdastar_l} \cdot
	\frac{\vustart_l \vu_l}{\lambda_l / \lambdastar_l}+\gamma_3}\right|,
\end{equation}
where 
\begin{equation} \label{eq:delta123}
\gamma_1 := \va^\top \left( \vu_l - \frac{\vustart_l \vu_l}{\lambda_l / \lambdastar_l} \vustar_l\right), \quad 
\gamma_2 := \va^\top \left( \vw_l - \frac{\vustart_l \vw_l}{\lambda_l / \lambdastar_l} \vustar_l\right), \quad
\gamma_3 := \calU_1 + \calU_2
\end{equation}
and the quantities $\calU_1$ and $\calU_2$ are given by
\begin{equation*} \begin{aligned}
    \calU_1 &:= \sum_{k: k\neq l} \frac{\vustart_k \vw_l}{\lambda_l / \lambdastar_k} \cdot \frac{\vustart_k \vu_l}{\lambda_l / \lambdastar_k} ~\text{ and } \\
    \calU_2 &:= \sum_{k_1=1}^r \sum_{k_2 = 1}^r \sum_{s_1, s_2: s_1 + s_2 \geq 1} \frac{\vustart_{k_1} \vw_l}{\lambda_l / \lambdastar_{k_1}} \cdot \frac{\vustart_{k_2} \vu_l}{\lambda_l / \lambdastar_{k_2}} \cdot \frac{\vustart_{k_1} \mH^{s_1 + s_2} \vustar_{k_2}}{\lambda_l^{s_1 + s_2}}. 
\end{aligned} \end{equation*}

We will bound $\gamma_1$, $\gamma_2$ and $\gamma_3$ separately, and then combine the results to obtain the desired bound in Equation~\eqref{eq:improved:linear-form-bound}.

For $\gamma_1$ and $\gamma_2$, it follows from Equation~\eqref{eq:first:order:bound} that with probability at least $1 - O(n^{-17})$,
\begin{equation} \label{eq:delta1:delta2}
\max\left\{ \left|\gamma_1\right|, \left|\gamma_2\right| \right\}
\lesssim \sigma \sqrt{\kappa^2 r \log ^7 n}
\left\{\frac{1}{\lambda_{\min }^{\star}}+\kappa r \max _{k: k \neq l} \frac{\left|\boldsymbol{a}^{\top} \vustar_k\right|}{\left|\lambdastar_l-\lambdastar_k\right|}\right\} .
\end{equation}

To bound $\gamma_3$, we bound $\calU_1$ and $\calU_2$ separately.
For $|\calU_1|$, it follows from Equation~\eqref{eq:inner:diff:over:ev} that with probability at least $1 - O(n^{-17})$,
\begin{equation} \label{eq:calU1:bound} 
|\calU_1| \lesssim r \left(\frac{\sigma}{\delta^\star_l} \sqrt{\kappa^4 r \log^7 n}\right)^2. 
\end{equation}
For $|\calU_2|$, using Theorem~\ref{thm:eigenvalue:rank-r} and the naive bound that for all $k_1,k_2,l \in [r]$ the inner products $|\vustart_{k_1} \vw_l|$ and $|\vustart_{k_1} \vu_l|$ are upper bounded by $1$, with probability at least $1 - O(n^{-17})$, 
we have
\begin{equation*} \begin{aligned}
\left|\calU_2\right| 
&\lesssim \sum_{k_1=1}^r \sum_{k_2=1}^r \sum_{s_1,s_2:s_1+s_2\geq 1} \frac{1}{\left|\lambdastar_l/\lambdastar_{k_1}\cdot \lambdastar_l / \lambdastar_{k_2}\right|} \left|\frac{\vustart_{k_1} \boldsymbol{H}^{s_1+s_2} \vustar_{k_2}}{\left(\lambda_l^{\star} / 2\right)^{s_1+s_2}}\right| .
\end{aligned} \end{equation*}
Recalling the definition of $\kappa = \lambdastar_{\max} / \lambdastar_{\min}$, it follows that
\begin{equation*} \begin{aligned}
\left|\calU_2\right| 
    &\leq \kappa^2 r^2 \sum_{s_1, s_2: s_1+s_2 \geq 1}\left(\frac{2}{\lambda_{\min }^{\star}}\right)^{s_1+s_2}\left|\boldsymbol{u}_{k_1}^{\star \top} \boldsymbol{H}^{s_1+s_2} \vustar_{k_2} \right| \\
    &= \kappa^2 r^2 \sum_{s=1}^\infty \sum_{s_1, s_2: s_1+s_2 = s}\left(\frac{2}{\lambda_{\min }^{\star}}\right)^{s}\left|\boldsymbol{u}_{k_1}^{\star \top} \boldsymbol{H}^{s} \boldsymbol{u}_{k_2}^{\star}\right| .
\end{aligned} \end{equation*}
Applying the na\"{i}ve bound $s \le 2^s$,
\begin{equation*}
\left|\calU_2\right| 
\leq
\kappa^2 r^2 \sum_{s=1}^\infty \left(\frac{4}{\lambda_{\min }^{\star}}\right)^{s}\left|\boldsymbol{u}_{k_1}^{\star \top} \boldsymbol{H}^{s} \boldsymbol{u}_{k_2}^{\star}\right| .
\end{equation*}
Applying Lemma~\ref{lem:inf:series:bound},
\begin{equation*}
\left|\calU_2\right| 
\lesssim \frac{\sigma\kappa^2 r^2\sqrt{\log^7 n}}{\lambdastar_{\min}}.
\end{equation*}
Applying this and Equation~\eqref{eq:calU1:bound} to bound
\begin{equation*}
\left| \gamma_3 \right| \le \left| \calU_1 \right| + \left| \calU_2 \right|,
\end{equation*}
we obtain that with probability at least $1 - O(n^{-17})$,
\begin{equation} \label{eq:delta3}
\left|\gamma_3\right| 
\lesssim
\frac{\sigma^2 \kappa^4 r^2 \log^7 n}{\left(\delta^\star_{l}\right)^2} 
+ \frac{\sigma\kappa^2 r^2\sqrt{\log^7 n}}{\lambdastar_{\min}}. 
\end{equation}

Before applying the bounds on $\gamma_1$, $\gamma_2$ and $\gamma_3$ to Equation~\eqref{eq:square:linear-form}, we simplify the proof by following the same strategy as in \cite{cheng2021tackling} to analyze two cases separately, based on the magnitude of $\va^\top \vustar_l$.
\begin{itemize}
    \item[(a)] When $\left|\va^\top \vustar_l\right| \leq \max\{|\gamma_1|, |\gamma_2|\}$, it follows from Equation (108) in \cite{cheng2021tackling} that
    \begin{equation*}
        \left(\widehat{u}_{\boldsymbol{a}, l}\right)^2 \lesssim \left(\max \left\{\left|\gamma_1\right|,\left|\gamma_2\right|\right\}\right)^2 
    \end{equation*}
    and we have 
    \begin{equation} \label{eq:case-a-linear-form}
        \min \left|\widehat{u}_{\boldsymbol{a}, l} \pm \boldsymbol{a}^{\top} \vustar_l\right| \leq\left|\widehat{u}_{\boldsymbol{a}, l}\right|+\left|\boldsymbol{a}^{\top} \vustar_l\right|
\lesssim \max \left\{\left|\gamma_1\right|,\left|\gamma_2\right|\right\}
    \end{equation}
    which can then be combined with Equation~\eqref{eq:delta1:delta2}. 
\item[(b)] Otherwise, repeating Equation (111) in \cite{cheng2021tackling} and replacing their upper bounds on $|\gamma_1|, |\gamma_2|$ and $|\gamma_3|$ with the refined bounds given in Equation~\eqref{eq:delta1:delta2} and Equation~\eqref{eq:delta3}, we have 
\begin{equation*}
\begin{aligned}
    \left| \frac{(\va^\top \vu_l) (\va^\top \vw_l)}{\vu^\top_l \vw_l} - \left(\va^\top\vustar_l\right)^2\right| &\lesssim \frac{\left|\boldsymbol{a}^{\top} \boldsymbol{u}_l^{\star}\right| \max \left\{\left|\gamma_1\right|,\left|\gamma_2\right|\right\}+\left(\max \left\{\left|\gamma_1\right|,\left|\gamma_2\right|\right\}\right)^2+\left(\boldsymbol{a}^{\top} \boldsymbol{u}_l^{\star}\right)^2\left|\gamma_3\right|}{\left(\frac{1}{2}\right)^2-\frac{1}{8}} \\
    &\lesssim \left|\boldsymbol{a}^{\top} \boldsymbol{u}_l^{\star}\right| \sigma \sqrt{\kappa^2 r \log ^7 n}\left\{\frac{1}{\lambda_{\min }^{\star}}+\kappa r \max _{k: k \neq l} \frac{\left|\boldsymbol{a}^{\top} \boldsymbol{u}_k^{\star}\right|}{\left|\lambda_l^{\star}-\lambda_k^{\star}\right|}\right\} \\
    &+ \left(\boldsymbol{a}^{\top} \boldsymbol{u}_l^{\star}\right)^2 \left(\frac{\sigma^2 \kappa^4 r^2 \log^7 n}{\left(\delta^\star_{l}\right)^2} + \frac{\sigma\kappa^2 r^2\sqrt{\log^7 n}}{\lambdastar_{\min}}\right).
\end{aligned}
\end{equation*} 
Taking the above display into Equation (113) in \cite{cheng2021tackling} yields that 
\begin{equation*}
\begin{aligned}
    \min \left|\widehat{u}_{\boldsymbol{a}, l} \pm \boldsymbol{a}^{\top} \boldsymbol{u}_l^{\star}\right| &\leq \frac{\left|\frac{\left(\boldsymbol{a}^{\top} \boldsymbol{u}_l\right)\left(\boldsymbol{a}^{\top} \boldsymbol{w}_l\right)}{\boldsymbol{u}_l^{\top} \boldsymbol{w}_l}-\left(\boldsymbol{a}^{\top} \boldsymbol{u}_l^{\star}\right)^2\right|}{\left|\boldsymbol{a}^{\top} \boldsymbol{u}_l^{\star}\right|}\\
    &\lesssim  \sigma \sqrt{\kappa^2 r \log ^7 n}\left\{\frac{1}{\lambda_{\min }^{\star}}+\kappa r \max _{k: k \neq l} \frac{\left|\boldsymbol{a}^{\top} \boldsymbol{u}_k^{\star}\right|}{\left|\lambda_l^{\star}-\lambda_k^{\star}\right|}\right\} \\
    &~~~~~~+ \left(\boldsymbol{a}^{\top} \boldsymbol{u}_l^{\star}\right) \left(\frac{\sigma^2 \kappa^4 r^2 \log^7 n}{\left(\delta^\star_{l}\right)^2} + \frac{\sigma\kappa^2 r^2\sqrt{\log^7 n}}{\lambdastar_{\min}}\right). 
\end{aligned}
\end{equation*}
For ease of presentation, 
Under the assumption $\kappa r = O(1)$, it follows that 
\begin{equation} \label{eq:above-bound}
    \min \left|\widehat{u}_{\boldsymbol{a}, l} \pm \boldsymbol{a}^{\top} \boldsymbol{u}_l^{\star}\right| \lesssim \sigma \sqrt{\log^7 n} \left\{\frac{1}{\lambda_{\min }^{\star}}+ \max _{k: k \neq l} \frac{\left|\boldsymbol{a}^{\top} \boldsymbol{u}_k^{\star}\right|}{\left|\lambda_l^{\star}-\lambda_k^{\star}\right|}\right\} + \left|\boldsymbol{a}^{\top} \boldsymbol{u}_l^{\star}\right| \frac{\sigma^2 \log^7 n}{\left(\delta^\star_{l}\right)^2}.
\end{equation}
\end{itemize}
Taking the maximum over Equation~\eqref{eq:above-bound} and Equation~\eqref{eq:case-a-linear-form} to cover the two cases, we obtain the bound in Equation~\eqref{eq:improved:linear-form-bound}.
\end{proof}

\begin{lemma}\label{lem:inf:series:bound}
For any unit vector $\va$, under the same assumptions as Theorem~\ref{thm:improved:linear-form-bound}, with probability at least $1 - O(n^{-17})$, for the asymmetric noise matrix $\mH$ given in Equation~\eqref{eq:M:asym}, we have
\begin{equation}\label{eq:the-sum}
\max_{j \in [r]} \sum_{s=1}^\infty \frac{\left|\va^\top \mH^s \vustar_j\right|}{|\lambda_l|^s} 
\lesssim \frac{\sigma \sqrt{\log^7 n}}{\lambdastar_{\min}}. 
\end{equation}
\end{lemma}
\begin{proof}
We break the sum on the left hand side of Equation~\eqref{eq:the-sum} into two parts. 
For $s \geq 20\log n$, we have 
\begin{equation*} 
\sum_{s\geq 20\log n} \frac{\left|\va^\top \mH^s \vustar_j\right|}{|\lambda_l|^s} \leq \sum_{s\geq 20\log n} \frac{\|\mH\|^s}{|\lambda_l|^s}.
\end{equation*}
By Theorem~\ref{thm:eigenvalue:rank-r}, we have 
\begin{equation*}
    \|\mH\| \leq \frac{|\lambda_l|}{3} \quad \text{and} \quad \frac{1}{2} \leq\left|\frac{\lambda_l}{\lambda_l^{\star}}\right| \leq 2,
\end{equation*}
and it follows that 
\begin{equation} \label{eq:part-I}
\begin{aligned}
    \sum_{s\geq 20\log n} \frac{\left|\va^\top \mH^s \vustar_j\right|}{|\lambda_l|^s}  &\leq \sum_{s\geq 20\log n} \frac{\|\mH\|^s}{|\lambda_l|^s} \\
    &\lesssim \frac{\|\mH\|}{\lambdastar_{\min}} \sum_{s\geq 20\log n - 1} \left(\frac{1}{3}\right)^s \lesssim \frac{1}{\lambdastar_{\min}} \max\left\{L \log n, \sigma \sqrt{n \log n}\right\} \cdot n^{-10}. 
\end{aligned}
\end{equation}

On the other hand, for $s \leq 20\log n$, Theorem~\ref{thm:yan:main} implies
\begin{equation} \label{eq:part-II}
\begin{aligned}
\sum_{s=1}^{20\log n} \frac{\left|\va^\top \mH^s \vustar_j\right|}{|\lambda_l|^s}
&\lesssim \max\left\{ \frac{L \log^5 n}{|\lambda_l|} \sqrt{\frac{\mub}{n}} \|\va\|_{\infty}, \frac{\sigma \sqrt{\log^7 n}}{|\lambda_l|} \right\}\\
&\lesssim \max\left\{ \frac{L \log^5 n}{|\lambdastar_{\min}|} \sqrt{\frac{\mub}{n}} \|\va\|_{\infty}, \frac{\sigma \sqrt{\log^7 n}}{|\lambdastar_{\min}|} \right\},
\end{aligned}
\end{equation}
Adding the two parts in equations~\eqref{eq:part-I} and~\eqref{eq:part-II} together and using the assumption on the growth rate of $L$ and $\sigma$ in Equation~\eqref{eq:L-sigma} completes the proof.  
\end{proof}

\subsection{Proof of Corollary~\ref{cor:single:eigenspace}}
\label{sec:proof:cor:single:eigenspace}

\begin{proof}
The bound in Equation~\eqref{eq:improved:tti:eigenspace:bound} follows directly from Equation~\eqref{eq:improved:linear-form-bound}.
To estabish Equation~\eqref{eq:improved:infty:eigenvec:bound}, suppose that Equations~\eqref{eq:naive:linear-form} and~\eqref{eq:improved:linear-form-bound} hold for all $\va = \ve_1, \ve_2, \cdots, \ve_n$, which is the case with probability at least $1 - O(n^{-16})$ by Theorems~\ref{thm:eigenvec:l2} and~\ref{thm:improved:linear-form-bound}.
Then, for each $l \in [r]$, we take $\va = \ve_i$ in Equation~\eqref{eq:naive:linear-form}, there is a binary sign $s \in \{\pm 1\}$ such that 
\begin{equation*} \begin{aligned}
\left|u_{l,i} - s u^\star_{l,i}\right| 
&\leq C \sigma \sqrt{\log^7 n} \left( \frac{1}{\lambdastar_{\min}} 
	+ \max_{k: k\neq l} \frac{\left|u^\star_{k,i}\right|}{|\lambdastar_l - \lambdastar_k|}\right) + \frac{1}{2} \left|u^\star_{l,i}\right|. 
\end{aligned} \end{equation*}
Define the index set
\begin{equation*}
    \calA = \left\{i \in [n] : |u^\star_{l,i}| \leq C \sigma \sqrt{\log^7 n} \left\{\frac{1}{\lambdastar_{\min}} + \max_{k: k\neq l} \frac{\left|u^\star_{k,i}\right|}{|\lambdastar_l - \lambdastar_k|}\right\}\right\}  
\end{equation*}
Recalling the definition of $\uhat_{l,i}$ given in Equation~\eqref{eq:vuhat:def} and applying the triangle inequality, we have 
\begin{equation*} \begin{aligned}
    \max_{i \in \calA} \left|\uhat_{l,i} - s u^\star_{l,i}\right| 
    &= \max_{i \in \calA} \left|\sgn{u_{l,i}} \left|\uhat_{\ve_i,l}\right| - s u^\star_{l,i}\right| \\
    &\leq \max_{i \in \calA} \left|\sgn{u_{l,i}} \left|\uhat_{\ve_i,l}\right| - \sgn{u_{l,i}} \left|u^\star_{l,i}\right|\right| + 2\max_{i \in \calA} |u^\star_{i,l}| .
\end{aligned} \end{equation*}
Applying the definition of $\calA$ and Equation~\eqref{eq:improved:linear-form-bound}, it follows that
\begin{equation} \label{eq:bound:I}
\begin{aligned}
    \max_{i \in \calA} \left|\uhat_{l,i} - s u^\star_{l,i}\right| 
    &\leq C \sigma \sqrt{\log^7 n} \left\{\frac{1}{\lambdastar_{\min}} + \max_{k: k\neq l} \frac{\left|u^\star_{k,i}\right|}{|\lambdastar_l - \lambdastar_k|}\right\}.
\end{aligned}
\end{equation}

For $i \in \calA^c$, we have 
\begin{equation*}
    \left|u_{l,i} - s u^\star_{l,i}\right| \leq \frac{3}{4} |u^\star_{l,i}|,
\end{equation*}
which implies that 
\begin{equation*}
    \sgn{u_{l,i}} = s\sgn{u^\star_{l,i}}, \quad \text{for all } i \in \calA^c.
\end{equation*}
It follows that 
\begin{equation*}
\begin{aligned}
    \max_{i \in \calA^c} \left|\uhat_{l,i} - s u^\star_{l,i}\right| &= \max_{i \in \calA^c} \left|s\sgn{u^\star_{l,i}} \left|\uhat_{\ve_i,l}\right| - s u^\star_{l,i}\right|\\
    &= \max_{i \in \calA^c} \left|\left|\uhat_{\ve_i,l}\right| - \left|u^\star_{l,i}\right|\right|\\
    &\leq C \sigma \sqrt{\log^7 n} \left\{\frac{1}{\lambdastar_{\min}} + \max_{k: k\neq l} \frac{\left|u^\star_{k,i}\right|}{|\lambdastar_l - \lambdastar_k|}\right\} + \frac{C|u^\star_{l,i}| \sigma^2 \log^7 n}{\left(\deltastar_{l}\right)^2},
\end{aligned} 
\end{equation*}
where the inequality follows from taking $\va = \ve_i$ for $i \in \calA^c$ in Equation~\eqref{eq:improved:linear-form-bound}. 
Combining Equation~\eqref{eq:bound:I} with the above display yields the desired bound in Equation~\eqref{eq:improved:infty:eigenvec:bound}.
\end{proof}

\subsection{Proof of Theorem~\ref{thm:eigenspace:l2infty}}
\label{sec:proof:thm:eigenspace:l2infty}

\begin{proof}
By the first order approximation in Equation~\eqref{eq:first-order-approx}, we have 
\begin{equation*}
\begin{aligned}
     \mU &= \mUstar \mLambdastar \mUstar^\top \mU \mLambda^{-1} + \mE 
     = \mUstar \mR \mSig \mGamma^\top + \mE,
\end{aligned}
\end{equation*}
where $\mE \in \R^{n\times r}$ is given by 
\begin{equation} \label{eq:E:define} 
E_{i\ell} 
= \sum_{j=1}^r \frac{\lambda_j^{\star}}{\lambda_\ell}\left(\boldsymbol{u}_j^{\star \top} \boldsymbol{u}_\ell\right)
  \sum_{s=1}^{\infty} \frac{\ve_{i}^\top \boldsymbol{H}^s \boldsymbol{u}_j^{\star}}{\lambda_\ell^s}
\end{equation}
and $\mR \mSig \mGamma^\top$ is the singular value decomposition of 
\begin{equation} \label{eq:Q:define}
    \mQ := \mLambdastar \mUstart \mU \mLambda^{-1}. 
\end{equation}

Our proof proceeds in four steps.
In Step 1, we show that $\mG^{-1}$ is close to $\mGamma \mSig^2 \mGamma^\top$ under the operator norm. 
Denote the singular value of $\mG^{-1}$ as $\mSighat^{-2}$. 
To show that $\mSighat^{-1}$ is close to $\mSig$ under the operator norm, 
we provide an lower bound on $\mSig$ in Step 2.
In Step 3, we then symmetrize $\mG$ to get $\mG_{\symm}$ and show that $\mG_{\symm}$ can produce a bias-corrected matrix $\mPsi$ in Step 3.
In Step 4, we obtain the bound in Equation~\eqref{eq:eigenspace:l2infty}.
Some technical details are left to the end of the proof, after these four steps.

\paragraph{Step 1: Bounding $\|\mG^{-1} - \mGamma \mSig^2 \mGamma^\top\|$.}

By the definition of $\mG$ in Equation~\eqref{eq:G:defin}, we have
\begin{equation*} \begin{aligned}
\mG^{-1}
&= \mLambda^{-1} \mU^\top \mM^{(1)} \mM^{(2)\top} \mU \mLambda^{-1} \\
&= \mLambda^{-1} \mU^\top \left[\mUstar \left(\mLambdastar\right)^2 \mUstart 
	+ \mMstar \mW^{(2)} + \mW^{(1)} \mMstar 
	+ \mW^{(1)} \mW^{(2)}\right] \mU \mLambda^{-1}\\
&= \mQ^\top \mQ + \mLambda^{-1} \mU^\top \left(\mW^{(1)} \mMstar + \mMstar \mW^{(2)}\right) \mU \mLambda^{-1} + \mLambda^{-1} \mU^\top \mW^{(1)} \mW^{(2)} \mU \mLambda^{-1}, 
\end{aligned} \end{equation*}
where $\mQ$ is defined in Equation~\eqref{eq:Q:define}.
Defining
\begin{equation*}
\mE_1 = \mLambda^{-1} \mU^\top \left(\mW^{(1)} \mMstar + \mMstar \mW^{(2)}\right) \mU \mLambda^{-1}
~\text{ and }~
\mE_2 = \mLambda^{-1} \mU^\top \mW^{(1)} \mW^{(2)} \mU \mLambda^{-1},  
\end{equation*} 
and rearranging, we have
\begin{equation*}
\mG^{-1} - \mGamma \mSig^2 \mGamma^{\top} = \mE_1 + \mE_2,
\end{equation*}
where $\mQ^\top \mQ = \mGamma \mSig^2 \mGamma^{\top}$ follows from the SVD of $\mQ$ given by $\mQ = \mR \mSig \mGamma^\top$.  

Thus, to bound $\| \mG^{-1} - \mGamma \mSig^2 \mGamma^{\top} \|$, it will suffice to control the operator norms of $\mE_1$ and $\mE_2$. 
By the triangle inequality, 
\begin{equation*}
\|\mE_1\| \leq \left\|\mLambda^{-1} \mU^\top \mW^{(1)} \mMstar \mU \mLambda^{-1}\right\| + \left\|\mLambda^{-1} \mU^\top  \mMstar \mW^{(2)} \mU \mLambda^{-1}\right\|.
\end{equation*}
It suffices to control the first term in the above display, as we can handle the other term in exactly the same manner. 
Taking the spectral decomposition of $\mMstar$, we have
\begin{equation} \label{eq:LUHULUUL}
\begin{aligned}
     \left\|\mLambda^{-1} \mU^\top \mW^{(1)} \mMstar \mU \mLambda^{-1}\right\| &=
    \left\|\mLambda^{-1}\mU^\top \mW^{(1)} \mUstar \mLambdastar \mUstart \mU\mLambda^{-1}\right\| \\
    &\leq \left\|\mLambda^{-1}\right\|^2 \left\|\mU^\top \mW^{(1)} \mUstar\right\| \left\|\mLambdastar\right\| \left\|\mUstart \mU\right\|.
\end{aligned}
\end{equation}

To bound $\| \mE_2 \|$, we note that by submultiplicativity of the norm,
\begin{equation} \label{eq:LUH1H2UL}
\| \mE_2 \|
\leq 
\left\|\mLambda^{-1}\right\|^2 
\left\|\mU^\top \mW^{(1)} \mW^{(2)} \mU \right\|.
\end{equation}
We show below that with probability at least $1 - O(n^{-17})$,
\begin{equation}\label{eq:UH1Ustar}
\left\|\mU^\top \mW^{(1)\top} \mUstar\right\| \lesssim \sigma \sqrt{r \log n}
\end{equation}
and 
\begin{equation}\label{eq:UWWU}
\left\|\mU^\top \mW^{(1)} \mW^{(2)\top} \mU\right\| \lesssim \sigma^2 \sqrt{r n \log n} .
\end{equation}
Therefore, applying Equation~\eqref{eq:UH1Ustar} to Equation~\eqref{eq:LUHULUUL} yields
\begin{equation*}
\left\|\mLambda^{-1}\mU^\top \mW^{(1)\top} \mUstar \mLambdastar \mUstart \mU\mLambda^{-1}\right\| 
\lesssim \frac{\sigma \kappa \sqrt{r \log n}}{\lambdastar_{\min}} 
\end{equation*}
and applying Equation~\eqref{eq:UWWU} in Equation~\eqref{eq:LUH1H2UL} yields
\begin{equation*}
\left\|\mLambda^{-1} \mU^\top \mW^{(1)} \mW^{(2) \top} \mU \mLambda^{-1}\right\|
\lesssim \frac{\sigma^2 \sqrt{r n \log n}}{\left(\lambdastar_{\min}\right)^2}. 
\end{equation*}
It follows from the above two displays that 
\begin{equation}\label{eq:Ginv:op}
    \left\|\mG^{-1} - \mGamma \mSig^2 \mGamma^\top\right\| \leq \|\mE_1\| + \|\mE_2\| \lesssim \frac{\sigma \kappa \sqrt{r \log n}}{\lambdastar_{\min}}.
\end{equation}
Additionally, denote the singular values of $\mG^{-1}$ as $\mSighat^2$, we have by Weyl's inequality, 
\begin{equation}\label{eq:Sighat2:op}
    \left\|\mSighat^2 - \mSig^2\right\| \lesssim \frac{\sigma \kappa \sqrt{r \log n}}{\lambdastar_{\min}}.
\end{equation}

\paragraph{Step 2: Upper and lower bounds on singular values of $\mQ$.}

To obtain an upper bound on $\|\mSighat - \mSig\|$, we need to provide a lower bound on the singular values of $\mQ$.
Recall that by Theorem~\ref{thm:eigenvec:l2},  
\begin{equation*}
    |Q_{kk}| = \left|\frac{\lambdastar_k}{\lambda_k} \vustart_k \vu_k\right| \geq 1 - O\left(\frac{\sigma^2 \kappa^4 r^2 \log^7 n}{\left(\delta^\star_k\right)^2} + \frac{\kappa^2 \sigma^2 n \log n}{\left(\lambdastar_{\min}\right)^2}\right)
\end{equation*}
and for any $k \neq \ell$, we have
\begin{equation*}
    \sum_{\ell \neq k} \left| Q_{k\ell} \right| = \sum_{\ell\neq k} \left|\frac{\lambdastar_k}{\lambda_\ell} \vustart_k \vu_\ell\right| \lesssim \sum_{\ell \neq k} \frac{\sigma \kappa^3 \sqrt{r \log^7 n}}{\delta^\star_{\ell}}.
\end{equation*}
Thus, for $\delta^\star_{\ell} \gg \sigma \kappa^3 r\sqrt{r \log^7 n}$ and $n$ sufficiently large, we have
\begin{equation*}
\min_{k \in [r]} ~ \min \left\{\left|Q_{kk}\right| - \sum_{\ell \neq k} \left| Q_{\ell k} \right|, ~ \left|Q_{kk}\right| - \sum_{\ell \neq k} \left| Q_{\ell k} \right| \right\} \geq \frac{1}{2},
\end{equation*}
and thus by Theorem~\ref{thm:dd:singular:lower}, the smallest singular value of $\mQ$ is lower-bounded as
\begin{equation}\label{eq:lower:srQ}
    \sigma_r(\mQ) \geq \frac{1}{2}
\end{equation}
when $n$ is sufficiently large. 
Thus, we have 
\begin{equation} \label{eq:sig-diff} 
\begin{aligned}
    \left\|\mSighat - \mSig\right\| &= \left\|\mSighat + \mSig\right\|^{-1} \left\|\mSighat^2 - \mSig^2\right\|\\
    &\leq \left\|\mSig\right\|^{-1} \left\|\mSighat^2 - \mSig^2\right\| \lesssim \frac{\sigma \kappa \sqrt{r \log n}}{\lambdastar_{\min}},
\end{aligned}
\end{equation} 
where the first inequality follows from the fact that both $\mSighat$ and $\mSig$ are nonnegative diagonal matrices and the last inequality combines Equations~\eqref{eq:Sighat2:op} and~\eqref{eq:lower:srQ}.
Additionally, we can also obtain an upper bound on $\|\mSig\|$. 
By Lemma~\ref{lem:dd:singular:upper}, we have
\begin{equation*} \begin{aligned}
\|\mQ\| &\leq 2\max_{k\in[r]}\{|Q_{kk}|\} = 2\max_{k\in[r]}\left\{\left|\frac{\lambdastar_k}{\lambda_k} \vustart_k \vu_k \right|\right\} \\
    &\leq 2 \max_{k \in [r]} \left\{\left|\frac{\lambdastar_k}{\lambda_k} \|\vustar_k\|_2 \|\vu_k\|_2 \right|\right\} = 2 \max_{k \in [r]} \left\{\left|\frac{\lambdastar_k}{\lambda_k}\right|\right\}. 
\end{aligned} \end{equation*}
By Theorem~\ref{thm:eigenvalue:rank-r}, when $n$ is sufficiently large, we have
\begin{equation*}
    \frac{1}{2} \leq \max_{k \in [r]} \left\{\left|\frac{\lambdastar_k}{\lambda_k}\right|\right\} \leq 2.
\end{equation*}
Thus, it follows that for $n$ sufficiently large,
\begin{equation}\label{eq:upper:srQ}
    \|\mSig\| = \|\mQ\| \leq 4. 
\end{equation}
Combining the fact that $\|\mA^{-1}-\mB^{-1}\| \leq \|\mA^{-1}\|\|\mA-\mB\|\|\mB^{-1}\|$ with Equations~\eqref{eq:lower:srQ},~\eqref{eq:Ginv:op} and~\eqref{eq:Sighat2:op}, we have
\begin{equation}\label{eq:G:op}
\begin{aligned}
    \left\|\mG - \mGamma \mSig^{-2} \mGamma^\top\right\| &\lesssim \left\|\mSighat^{-2}\right\| \left\|\mG^{-1} - \mGamma \mSig^2 \mGamma^\top \right\|\left\|\mSig^{-2}\right\| 
    \lesssim \frac{\sigma \kappa \sqrt{r \log n}}{\lambdastar_{\min}}. 
\end{aligned}
\end{equation}

\paragraph{Step 3: Bounds on $\mG_{\symm}$ and $\mPsihat$.}

Recall the definition of $\mG_{\symm}$ in Equation~\eqref{eq:Gsymm:def}.
By the triangle inequality, it follows from Equation~\eqref{eq:G:op} that
\begin{equation} \label{eq:Gsymm:op}
    \left\|\mG_{\symm} - \mGamma \mSig^{-2} \mGamma^\top\right\| = \left\|\frac{1}{2}(\mG+\mG^\top) - \mGamma \mSig^{-2} \mGamma^\top\right\| \lesssim \frac{\sigma \kappa \sqrt{r \log n}}{\lambdastar_{\min}}.
\end{equation}
Combining the above display with Weyl's inequality, we have 
\begin{equation*}
\left\|\mSighat^{-2}_{\symm} - \mSig^{-2}\right\| 
\leq \left\|\mG_{\symm} - \mGamma \mSig^{-2} \mGamma^\top\right\| \lesssim \frac{\sigma \kappa \sqrt{r \log n}}{\lambdastar_{\min}}.
\end{equation*}
It is straightforward to see that $\mG_{\symm}$ is a positive semidefinite matrix as its eigenvalues are nonnegative combining the above display with the upper bound in Equation~\eqref{eq:upper:srQ}.
Taking the matrix square root of $\mG_{\symm}$ given by
\begin{equation*}
    \mPsihat = \mGammahat_{\symm} \mSighat^{-1}_{\symm} \mGammahat_{\symm}^\top,
\end{equation*}
it follows from Lemma~\ref{lem:square-root-lipschitz} and Equation~\eqref{eq:Gsymm:op} that
\begin{equation}\label{eq:Gsymm:sqrt:op}
    \left\|\mPsihat - \mGamma \mSig^{-1} \mGamma^\top\right\| \lesssim \left\|\mG_{\symm} - \mGamma \mSig^{-2} \mGamma^\top\right\| \lesssim \frac{\sigma \kappa \sqrt{r \log n}}{\lambdastar_{\min}}.
\end{equation}

\paragraph{Step 4: Bounding the $\ell_{2,\infty}$ norm.}
Now consider the bias-corrected estimator $\mU \mPsihat$, we have 
\begin{equation} \label{eq:UhatPsihat:bound} \begin{aligned} 
    \left\|\mU \mPsihat - \mUstar \mR \mGamma^\top\right\|_{2,\infty} 
    &\leq \left\|\mUstar \mR \mSig \mGamma^\top \mPsihat - \mUstar \mR \mGamma^\top\right\|_{2,\infty}  + \|\mE \mPsihat\|_{2,\infty}\\
    &= \left\|\mUstar \mR \mSig \mGamma^\top \left(\mPsihat - \mGamma \mSig^{-1} \mGamma^\top\right)\right\|_{2,\infty} +  \|\mE \mPsihat\|_{2,\infty}
\end{aligned} \end{equation} 
For the first right-hand term, it follows from Equation~\eqref{eq:Gsymm:sqrt:op} that 
\begin{equation} \label{eq:firsttermbound}\begin{aligned}
    \left\|\mUstar \mR \mSig \mGamma^\top \left(\mPsihat - \mGamma \mSig^{-1} \mGamma^\top\right)\right\|_{2,\infty} &\leq \left\|\mUstar \right\|_{2,\infty} \left\|\mR \mSig \mGamma^\top\right\| \left\|\mPsihat - \mGamma \mSig^{-1} \mGamma^\top\right\|\\
    &\leq \sqrt{\frac{\mub r}{n}} \left\|\mSig\right\| \frac{\sigma \kappa \sqrt{r \log n}}{\lambdastar_{\min}} \lesssim \frac{\sigma \kappa r \sqrt{\log n}}{\lambdastar_{\min}} \sqrt{\frac{\mub}{n}},
\end{aligned} \end{equation}
where the last inequality follows from Equation~\eqref{eq:upper:srQ}. 

For the second right-hand term in Equation~\eqref{eq:UhatPsihat:bound}, recall the definition of $\mE$ in Equation~\eqref{eq:E:define}, we have for all $i \in [n]$ and $l \in [r]$,
\begin{equation} \label{eq:Eil:entry:bound}
\begin{aligned}
    |E_{il}| &\leq \left(\sum_{j=1}^r \frac{\left|\lambda_j^{\star}\right|}{\left|\lambda_l\right|}\left|\boldsymbol{u}_j^{\star \top} \boldsymbol{u}_l\right|\right)\left\{\max _{j \in [r]} \sum_{s=1}^{\infty} \frac{1}{\left|\lambda_l\right|^s}\left|\ve_i^{\top} \boldsymbol{H}^s \boldsymbol{u}_j^{\star}\right|\right\}\\
    &\stackrel{(i)}{\lesssim} \kappa \sqrt{r} \left\{\max _{j \in [r]} \sum_{s=1}^{\infty} \frac{1}{\left|\lambda_l\right|^s}\left|\ve_i^{\top} \boldsymbol{H}^s \boldsymbol{u}_j^{\star}\right|\right\} \stackrel{(ii)}{\lesssim} \frac{\sigma \kappa \sqrt{r \log^7 n}}{\lambdastar_{\min}}, 
\end{aligned}
\end{equation}
where inequality $(i)$ follows from Equation (50) in \cite{chen2021asymmetry} and inequality $(ii)$ follows from Lemma~\ref{lem:inf:series:bound}. 
We obtain that with probability at least $1 - O(n^{-17})$, 
\begin{equation*}
\begin{aligned}
\|\mE \mPsihat\|_{2,\infty} &\leq \left\|\mE\right\|_{2,\infty} \left\|\mPsihat\right\| \lesssim \frac{\sigma \kappa r\sqrt{\log^7 n}}{\lambdastar_{\min}},
\end{aligned}
\end{equation*}
where the second inequality follows from Equation~\eqref{eq:Eil:entry:bound}. 
Noting that $\mR \mGamma^\top \in \bbO_{r}$, we conclude that 
\begin{equation*}
    \min_{\mO \in \bbO_r} \left\|\mU \mPsihat - \mUstar \mO\right\|_{2,\infty} \leq \left\|\mU \mPsihat - \mUstar \mR \mGamma^\top\right\|_{2,\infty} \lesssim \frac{\sigma \kappa r\sqrt{ \log^7 n}}{\lambdastar_{\min}},
\end{equation*}
obtaining the desired bound.

\paragraph{Supporting technical bounds.}
It remains for us to return the bounds in Equations~\eqref{eq:UH1Ustar}  and~\eqref{eq:UWWU}. 

To bound the term $\|\mU^\top \mW^{(1)\top} \mUstar\|$, we instead consider bounding $\|\mU^\top \mHtilde \mUstar\|$ for any $\mHtilde$ that satisfies the same assumptions as $\mH$. 
The reason is that for any symmetric random matrix $\mW$, we can decompose $\mW = \mH_1 + \mH_2$, where $\mH_1$ contains the upper-triangular entries of $\mW$, and $\mH_2$ contains the strict lower-triangular entries of $\mW$. Both $\mH_1$ and $\mH_2$ satisfy the same set of assumptions as $\mH$ and one can use triangle inequality to control bounds involving $\mW$. 
Thus, it suffices to bound $\mHtilde$. 

To bound $\|\mU^\top \mHtilde \mUstar\|$, the idea here is to express $\mU^\top \mHtilde \mUstar$ as a sum of independent matrices and apply the matrix Bernstein inequality stated in Theorem~\ref{thm:matrix:bernstein}. 
Noting that 
\begin{equation*}
    \mHtilde = \sum_{i=1}^n \sum_{j=1}^n \Htilde_{ij} \ve_i \ve_j^\top,
\end{equation*}
we have
\begin{equation*}
\begin{aligned}
    \mU^\top \mHtilde \mUstar &= \sum_{i=1}^n \sum_{j=1}^n \Htilde_{ij} \mU^\top_{i,\cdot} \mU^{\star}_{j,\cdot} =: \sum_{i=1}^n \sum_{j=1}^n \mZ_{ij}. 
\end{aligned}
\end{equation*}
Since
\begin{equation*}
    \E \left[\mZ_{ij} \mZ^\top_{ij} \right] = \sigma^2_{ij} \mU^\top_{i,\cdot} \mU_{i,\cdot} \left\|\mU^{\star}_{j,\cdot}\right\|_2^2, \quad \E \left[\mZ^\top_{ij} \mZ_{ij} \right] = \sigma^2_{ij} \mU^{\star\top}_{j,\cdot} \mU^\star_{j,\cdot} \left\|\mU_{i,\cdot}\right\|_2^2 ,
\end{equation*}
we have
\begin{equation} \label{eq:ZZtop}
\begin{aligned}
    \sum_{i=1}^n \sum_{j=1}^n \E \left[\mZ_{ij} \mZ^\top_{ij} \right] &\preceq \sum_{i=1}^n \sum_{j=1}^n \sigma^2 \mU^\top_{i,\cdot} \mU_{i,\cdot} \left\|\mU^{\star}_{j,\cdot}\right\|_2^2 \\
    &= \sigma^2 \sum_{i=1}^n \mU_{i,\cdot}^\top \mU_{i,\cdot} \sum_{j=1}^n \left\|\mU^{\star}_{j,\cdot}\right\|_2^2 = r \sigma^2 \mU^\top \mU 
\end{aligned}
\end{equation}
and
\begin{equation}\label{eq:ZtopZ} 
\begin{aligned}
    \sum_{i=1}^n \sum_{j=1}^n \E \left[\mZ_{ij}^\top \mZ_{ij} \right] &\preceq \sum_{i=1}^n \sum_{j=1}^n \sigma^2 \mU^{\star\top}_{j,\cdot} \mU^{\star}_{j,\cdot} \left\|\mU_{i,\cdot}\right\|_2^2 \\
    &= \sigma^2 \mUstart \mUstar \left\|\mU\right\|^2_{\F} = r \sigma^2 \mI,
\end{aligned}
\end{equation}
where the last equality follows from the fact that $\mU$ has unit norm columns as $\vu_1, \vu_2, \cdots, \vu_r$ are unit eigenvectors of $\mM$ as given in Equation~\eqref{eq:left-right-eigenvc}. 
Thus, combining Equations~\eqref{eq:ZZtop} and~\eqref{eq:ZtopZ},
\begin{equation*}
    v:= \max \left\{\left\|\sum_{i=1}^n \sum_{j=1}^n \E \left[\mZ_{ij} \mZ^\top_{ij} \right]\right\|, \left\|\sum_{i=1}^n \sum_{j=1}^n \E \left[\mZ_{ij}^\top \mZ_{ij} \right]\right\|\right\} \lesssim \sigma^2 r.
\end{equation*}
Additionally, we have
\begin{equation*}
    \max_{i,j} \left\|\mZ_{ij}\right\| \lesssim \frac{L \mub r}{n} =: L_0. 
\end{equation*}
By matrix Bernstein inequality given in Theorem~\ref{thm:matrix:bernstein}, it follows that 
\begin{equation*}
    \left\|\mU^\top \mHtilde \mUstar\right\| \lesssim \sigma \sqrt{r \log n} + L_0 \log n \lesssim \sigma \sqrt{r \log n},
\end{equation*}
holds with probability at least $1 - O(n^{-17})$ when choosing sufficiently large constants in the bound, completing the proof of Equation~\eqref{eq:UH1Ustar}. 

To control the other term $\|\mU^\top \mW^{(1)} \mW^{(2)} \mU\|$ in Equation~\eqref{eq:UWWU}, let us first fix 
\begin{equation} \label{eq:W2U}
    \mA := \mW^{(2)} \mU \in \R^{n \times r}.
\end{equation}
Similar to the previous case, we decompose $\mW^{(1)}$ into two copies with independent entries. Without loss of generality, we denote one copy to be $\mH^{(1)}$ and only bound $\|\mU^\top \mH^{(1)} \mA\|$.
Set 
\begin{equation*}
    \mZ_{ij} = H^{(1)}_{ij} \mU_{i,\cdot}^\top \mA_{j,\cdot}, \text{ for all } i,j \in [n]. 
\end{equation*}
It follows that 
\begin{equation*}
    \sum_{i=1}^n \sum_{j=1}^n \E \left[\mZ_{ij} \mZ^\top_{ij} \right] \preceq \sum_{i=1}^n \sum_{j=1}^n \sigma^2 \mU^\top_{i,\cdot} \mU_{i,\cdot} \left\|\mA_{j,\cdot}\right\|_2^2 = \sigma^2 \|\mA\|_{\F}^2 \mU^\top \mU 
\end{equation*}
and
\begin{equation*}
    \sum_{i=1}^n \sum_{j=1}^n \E \left[\mZ_{ij}^\top \mZ_{ij} \right] \preceq \sum_{i=1}^n \sum_{j=1}^n \sigma^2 \mA^{\top}_{j,\cdot} \mA_{j,\cdot} \left\|\mU_{i,\cdot}\right\|_2^2 = \sigma^2 \left\|\mU\right\|^2_{\F} \mA^\top \mA.
\end{equation*}
Therefore, we have
\begin{equation*}
    v:= \max \left\{\left\|\sum_{i=1}^n \sum_{j=1}^n \E \left[\mZ_{ij} \mZ^\top_{ij} \right]\right\|, \left\|\sum_{i=1}^n \sum_{j=1}^n \E \left[\mZ_{ij}^\top \mZ_{ij} \right]\right\|\right\} \leq \sigma^2 \max\left\{\|\mA\|_{\F}^2 \|\mU\|^2, \|\mA\|^2 \|\mU\|_{\F}^2\right\}
\end{equation*}
and 
\begin{equation*}
    \max_{i,j} \|\mZ_{ij}\| \leq L \sqrt{\frac{\mub r}{n}} \left\|\mA\right\|_{2,\infty}
\end{equation*}
Thus, conditioned on $\mW^{(2)}$, it follows from the matrix Bernstein inequality given in Theorem~\ref{thm:matrix:bernstein} that 
\begin{equation*}
    \left\|\mU^\top \mW^{(1)} \mW^{(2) \top} \mU\right\| \lesssim \sigma \sqrt{\log n} \max\left\{\|\mA\|_{\F} \|\mU\|, \|\mA\| \|\mU\|_{\F}\right\} + L \sqrt{\frac{\mub r}{n}} \left\|\mA\right\|_{2,\infty} \log n .
\end{equation*}
holds with probability at least $1 - O(n^{-17})$. 
Now turning our attention to controlling $\mA$ as given in Equation~\eqref{eq:W2U}, we have
\begin{equation*}
    \left\|\mA\right\| \leq \left\|\mW^{(2)}\right\|\|\mU\| \lesssim \sigma \sqrt{n}, \quad \|\mA\|_{\F} \leq \sqrt{r} \|\mA\| \lesssim \sigma \sqrt{r n}.
\end{equation*}
For $\|\mA\|_{2,\infty}$, 
\begin{equation*}
    \|\mA\|_{2,\infty} = \max_{\ell \in [n]} \left\|\mW^{(2)}_{\ell, \cdot} \mU \right\|_2 \leq \sqrt{r} \max_{\ell \in [n]} \max_{k \in [r]} \left|\mW^{(2)}_{\ell, \cdot} \vu_{k}\right|.
\end{equation*}
Since
\begin{equation*}
    \sum_{i=1}^n \E \left(W_{\ell i} U_{i k}\right)^2 \leq \sum_{i=1}^n U_{i \ell}^2 \sigma^2 = \|\vu_{\ell}\|^2 \sigma^2 = \sigma^2
 \quad \text{and} \quad
 \left|W_{\ell i} U_{i k}\right| \leq |U_{i k}| L \lesssim L\sqrt{\frac{\mub r}{n}}, 
\end{equation*}
applying Bernstein's inequality as in Lemma~\ref{lem:bern-ineq} yields for all $\ell \in [n]$
\begin{equation*}
    \left\|\mW^{(2)}_{\ell, \cdot} \mU \right\|_2 \lesssim \sigma \sqrt{r \log n} + L r\sqrt{\frac{\mub}{n}} \lesssim \sigma r \sqrt{r \log n},
\end{equation*}
holds with probability at least $1 - O(n^{-17})$ and it follows that 
\begin{equation*}
    \|\mA\|_{2,\infty} \lesssim \sigma r \sqrt{r \log n}.
\end{equation*}
Hence, we obtain
\begin{equation*}
    \left\|\mU^\top \mW^{(1)} \mW^{(2) \top} \mU\right\| \lesssim \sigma^2 \sqrt{r n \log n},
\end{equation*}
holds with probability at least $1 - O(n^{-17})$,
which establishes Equation~\eqref{eq:UWWU} and completes the proof.
\end{proof}

\subsection{Proof of Corollary~\ref{cor:single:entrywise}}
\label{sec:proof:cor:single:entrywise}

\begin{proof} 
Recall that $\mUhat$ is defined in Equation~\eqref{eq:mUhat} with its entries defined in Equation~\eqref{eq:vuhat:def}. 
Let $\mS = \diag\left(s_1, s_2, \cdots, s_r\right)$ be a diagonal matrix, where for all $l \in [r]$, 
\begin{equation*}
    \|s_l \vuhat_{l} - \vustar_l\|_{\infty} = \min \left\{\left\|\vuhat_{l} \pm \vustar_l\right\|_{\infty}\right\}, \; s_{l} \in \{1, -1\}.
\end{equation*} 
By the triangle inequality and basic properties of the norm,
\begin{equation*}
\begin{aligned}
    \left\|\mUhat \mLambda \mUhat^\top - \mUstar \mLambdastar \mUstart\right\|_{\infty} &\leq \left\|\mUhat \mS \left(\mLambda - \mLambdastar\right) \mS^\top \mUhat^\top\right\|_{\infty} + \left\|\mUhat \mS \mLambdastar \mS^\top \mUhat^\top - \mUstar \mLambdastar \mUstart\right\|_{\infty}\\
    &\leq \left\|\mUhat \mS \left(\mLambda - \mLambdastar\right) \mS^\top \mUhat^\top\right\|_{\infty} + \left\|(\mUhat \mS - \mUstar)\mLambdastar \mUstart\right\|_{\infty} \\
    &~~~+ \left\| \mUhat \mS \mLambdastar \left(\mUhat \mS - \mUstar\right)^\top\right\|_{\infty}\\
    &\leq \left\|\mUhat\right\|_{2,\infty}^2 \left\|\mLambda - \mLambdastar\right\| + 2 \left\|\mUhat \mS - \mUstar\right\|_{2,\infty} \left\|\mLambdastar\right\| \left\|\mUstar\right\|_{2, \infty}. 
\end{aligned}
\end{equation*}
Applying Corollary~\ref{cor:single:eigenspace} and Theorem~\ref{thm:eigenvalue:rank-r} to the above diplay, it holds with probability at least $1 - O(n^{-16})$ that
\begin{equation*}
\begin{aligned}
\left\|\mUhat \mLambda \mUhat^\top - \mUstar \mLambdastar \mUstart\right\|_{\infty} &\lesssim \sigma \sqrt{\frac{\mub \log^{7} n}{n}} +  \frac{\sigma \lambdastar_{\max} \sqrt{\log^7 n}} {\delta^\star_{\min}} \frac{\mub}{n} ,
\end{aligned}
\end{equation*}
as desired.
\end{proof}

\subsection{Proof of Theorem~\ref{thm:improved:entrywise}}
\label{sec:proof:thm:improved:entrywise}
 
\begin{proof}
    We remind the reader that $\mPsihat$ is defined in Equation~\eqref{eq:Psihat:def}, $\mQ$ is defined in Equation~\eqref{eq:Q:define}, and $\mR \mSig \mGamma^\top$ is the singular value decomposition of $\mQ$.
By the triangle inequality and basic properties of the norm,
\begin{equation}\label{eq:UpsiM:infty}
\begin{aligned}
\left\|\mU\mPsihat \mLambda \mPsihat^\top \!\mU^\top \!-\! \mMstar\right\|_{\infty} 
\!&\leq \left\| \mU \mPsihat \!\left(\mLambda \!-\! \mLambdastar\right) \!\mPsihat^\top \mU^\top \right\|_{\infty} 
    + \left\|\mU \mPsihat \mLambdastar \mPsihat^\top \!\mU^\top \!-\! \mUstar \mR \mGamma^\top \mLambdastar \mPsihat^\top \mU^\top\right\|_{\infty}\\
    &~~~+\left\|\mUstar \mR \mGamma^\top \mLambdastar \mPsihat^\top \mU^\top - \mUstar \mR \mGamma^\top \mLambdastar \mGamma \mR^\top \mUstart \right\|_{\infty} \\
    &~~~+ \left\|\mUstar \left(\mR \mGamma^\top \mLambdastar \mGamma \mR^\top - \mLambdastar\right) \mUstart \right\|_{\infty}\\
    &\leq \left\|\mU\mPsihat\right\|_{2, \infty}^2 \left\|\mLambda - \mLambdastar\right\| 
	+ \left\|\mUstar \left(\mR \mGamma^\top \mLambdastar \mGamma \mR^\top - \mLambdastar\right) \mUstart\right\|_{\infty} \\
&~~~+ \left\|\mU \mPsihat - \mUstar \mR \mGamma^\top\right\|_{2, \infty} \left\|\mLambdastar\right\| \left(\left\|\mU \mPsihat\right\|_{2, \infty}+ \left\|\mUstar \mR \mGamma^\top\right\|_{2,\infty}\right) .
\end{aligned}
\end{equation}

Combining results in Theorems~\ref{thm:eigenspace:l2infty} and~\ref{thm:eigenvalue:rank-r}, it is straightforward to show that with probability at least $1 - O(n^{-16})$,
\begin{equation}\label{eq:eta:bound:infty} 
\begin{aligned}
\left\|\mU\mPsihat\right\|_{2, \infty}^2 \left\|\mLambda - \mLambdastar\right\|
&+  \left\|\mU \mPsihat - \mUstar \mR \mGamma^\top\right\|_{2, \infty} \left\|\mLambdastar\right\| \left(\left\|\mU \mPsihat\right\|_{2, \infty}+ \left\|\mUstar \mR \mGamma^\top\right\|_{2,\infty}\right) \\
&\lesssim \frac{\mub r}{n} \frac{\sigma r^2 \sqrt{\log^7 n}}{\lambdastar_{\min}} + \sigma \kappa^2 r \sqrt{\frac{\mub \log ^7 n}{n}} 
\lesssim \sigma \kappa^2 r^2 \sqrt{\frac{\mub \log ^7 n}{n}} 
\end{aligned}
\end{equation}
where the second inequality holds since $\lambdastar_{\min}$ satisfies Assumption~\ref{assump:ev:lower}. 

By basic properties of the norm,
\begin{equation}\label{eq:last:infty:polar}
    \left\|\mUstar \left(\mR \mGamma^\top \mLambdastar \mGamma \mR^\top - \mLambdastar\right) \mUstart \right\|_{\infty} 
    \leq \left\|\mUstar\right\|_{2,\infty}^2 \left\|\mR \mGamma^\top \mLambdastar \mGamma \mR^\top - \mLambdastar\right\|.
\end{equation}
For ease of notation,
Denote $\mO := \mR\mGamma^\top$ 
\begin{equation} \label{eq:relate:OQ-to-orig}
\mR \mGamma^\top \mLambdastar \mGamma \mR^\top - \mLambdastar
= \mO \mLambdastar \mO^\top - \mLambdastar.
\end{equation}
Define $\mQ := \mLambdastar \mUstart \mU \mLambda^{-1}$ and observe that by construction, $\mO$ is the unitary polar factor of $\mQ$.

Plugging in the definition of $\mQ$ and applying the triangle inequality,
\begin{equation} \label{eq:Q:triangle}
    \left\|\mQ - \diag\left(\mUstart \mU\right)\right\| \leq \left\|\mLambdastar \mUstart \mU \left(\mLambda^{-1} - 
    \left(\mLambdastar\right)^{-1}\right) \right\| + \left\|\mLambdastar \mUstart \mU \left(\mLambdastar\right)^{-1} - \diag\left(\mUstart \mU\right)\right\|,
\end{equation} 
where for a matrix $A$, we use $\diag A$ to denote the matrix obtained by setting all elements of $A$ to zero aside from the diagonal.
Using Theorem~\ref{thm:eigenvalue:rank-r} with Equation~\eqref{eq:L-sigma:weaker}, it is straightforward to show that with probability at least $1 - O(n^{-17})$,
\begin{equation*}
    \left\|\mLambdastar \mUstart \mU \left(\mLambda^{-1} - \left(\mLambdastar\right)^{-1}\right) \right\| \lesssim \frac{\kappa\sigma r^2 \sqrt{\log^7 n}}{\lambdastar_{\min}} .
\end{equation*}
By Equation~\eqref{eq:l2:perturb:inner:diff} in Theorem~\ref{thm:eigenvec:l2}, since we assume the same assumptions as Theorem~\ref{thm:eigenvec:l2}, 
it holds with probability at least $1 - O(n^{-17})$ that
\begin{equation*} \begin{aligned}
\left\|\mLambdastar \mUstart \mU \left(\mLambdastar\right)^{-1} - \diag\left(\mUstart \mU\right)\right\|
&\leq \kappa \left\|\mUstart \mU - \diag\left(\mUstart \mU\right)\right\|_{\F} \\
&\lesssim \frac{\sigma \kappa^3 r\sqrt{r \log ^7 n}}{\delta_{\min}^{\star}} .
\end{aligned}
\end{equation*}
Applying the above two displays to Equation~\eqref{eq:Q:triangle} and using the fact that $\lambdastar_{\min} \geq \deltastar_{\min}$,
\begin{equation} \label{eq:Q:finalbound}
\left\|\mQ - \diag\left(\mUstart \mU\right)\right\|
\lesssim \frac{\sigma \kappa^3 r\sqrt{r \log ^7 n}}{\deltastar_{\min}} .
\end{equation}

Let $\mS$ be a diagonal matrix whose diagonal entries are given by the signs of $\diag\left(\mUstart \mU\right)$.
Recall that for any real matrix $\mA$, its unitary polar factor is given by
\begin{equation*}
    \mA \left(\mA^\top \mA\right)^{-1/2}.
\end{equation*}
Thus, the polar factor of $\diag\left(\mUstart \mU\right)$ is given by 
\begin{equation*}
    \mS \left|\diag\left(\mUstart \mU\right)\right| \left(\diag\left(\mUstart \mU\right) \diag\left(\mUstart \mU\right)\right)^{-1/2} = \mS. 
\end{equation*}
As observed above, $\mO$ is the unitary polar factor of $\mQ$.
Thus, by perturbation bounds of unitary polar factors \citep[see Theorem VII.5.1 in][]{bhatia2013matrix},
\begin{equation}\label{eq:polar:perturb}
\begin{aligned}
    \left\|\mO - \mS\right\| &\leq \frac{2}{\left\|\mQ^{-1}\right\|^{-1} + \left\|\diag\left(\mUstart \mU\right)^{-1}\right\|^{-1}} \left\|\mQ - \diag\left(\mUstart \mU\right)\right\| \\
    &\lesssim \left\|\mQ - \diag\left(\mUstart \mU\right)\right\|
    \lesssim \frac{\sigma \kappa^3 r\sqrt{r \log ^7 n}}{\delta_{\min}^{\star}},
\end{aligned}
\end{equation}
where the second inequality holds since $\|\mQ^{-1}\|^{-1}$ is lower bounded by $1/2$ as shown in Equation~\eqref{eq:lower:srQ} and $\|\diag(\mUstart \mU)^{-1}\|^{-1}$ is lower bounded by some constant by Equation~\eqref{eq:l2:perturb:inner:same} in Theorem~\ref{thm:eigenvec:l2} under the growth rate assumption on $\deltastar_l$ and $\lambdastar_{\min}$, the last inequality follows from Equation~\eqref{eq:Q:finalbound}.

Noting that for all $i,j \in [r]$,
\begin{equation*}
    \left(\mO \mLambdastar - \mLambdastar \mO\right)_{ij} = O_{ij} \left(\lambdastar_i - \lambdastar_j\right) 
\end{equation*}
and thus by basic properties of the norm,
\begin{equation} \label{eq:OL:spectral}
\left\| \mO \mLambdastar - \mLambdastar \mO \right\| 
\le \deltastar_{\max} \left( \sum_{1\leq i \neq j \leq r} O_{ij}^2 \right)^{1/2}
\le \sqrt{r} \deltastar_{\max} \left\| \mO - \mS \right\|
\end{equation}
where the first inequality holds since $\left(\mO \mLambdastar - \mLambdastar \mO\right)_{ii} = 0$ for all $i \in [r]$. 
Applying Equation~\eqref{eq:polar:perturb},
\begin{equation*}
\left\| \mO \mLambdastar - \mLambdastar \mO \right\|
\lesssim \frac{ \deltastar_{\max} }{ \deltastar_{\min} }
	\sigma \kappa^3 r^2 \sqrt{\log ^7 n} .
\end{equation*}
Since the operator norm is unitarily invariant, we have
\begin{equation*}
\left\|\mO \mLambdastar \mO^\top - \mLambdastar\right\|
= \left\|\mO \mLambdastar - \mLambdastar \mO\right\|
\lesssim \frac{ \deltastar_{\max} }{ \deltastar_{\min} }
	\sigma \kappa^3 r^2 \sqrt{\log ^7 n} .
\end{equation*}
Recalling the equality in Equation~\eqref{eq:relate:OQ-to-orig}, we further bound Equation~\eqref{eq:last:infty:polar} as
\begin{equation*}
 \left\|\mUstar \left(\mR \mGamma^\top \mLambdastar \mGamma \mR^\top - \mLambdastar\right) \mUstart \right\|_{\infty}
\lesssim \frac{ \deltastar_{\max} }{ \deltastar_{\min} }
        \frac{ \mub }{ n } \sigma \kappa^3 r^2 \sqrt{\log ^7 n} .
\end{equation*}
Applying this and Equation~\eqref{eq:eta:bound:infty} to Equation~\eqref{eq:UpsiM:infty} completes the proof.
\end{proof}

\section{Proofs of Minimax Results on Recovering the Common Structure}

In this section, we prove the minimax lower bounds related to the low rank component $\mMstar$, establishing Theorems~\ref{thm:Mstar:2-infty:minimax}, \ref{thm:rank1:linfty:minimax} and Corollary~\ref{cor:more:linfty:minimax}.

\subsection{Proof of Theorem~\ref{thm:Mstar:2-infty:minimax}}
\label{sec:proof:thm:Mstar:2-infty:minimax}

Before stating the proof of Theorem~\ref{thm:Mstar:2-infty:minimax}, we establish a few technical lemmas.

\begin{lemma}[Generalized Fano method]\label{lem:gFano} 
Let $N\geq 1$ be an integer and $\theta_1,\theta_2, \dots, \theta_N \subset \Theta$ index a collection of probability measures $\Pr_{\theta_i}$ on a measurable space $(\calX, \calF)$.
Let $D$ be a pseudometric on $\Theta$ and suppose that for $\alpha_N, \beta_N > 0$ all $i \neq j$,
\begin{equation*}
    D(\theta_i, \theta_j) \geq \alpha_N
\end{equation*}
and 
\begin{equation*}
    \KLD{\Pr_{\theta_i}}{\Pr_{\theta_j}} \leq \beta_N.
\end{equation*}
Then every $\calF$-measurable estimator $\thetahat$ satisfies 
\begin{equation*}
\max_{i \in [N]} \E_{\theta_i} d(\thetahat, \theta_i)
\geq \frac{\alpha_N}{2}\left(1 - \frac{\beta_N + \log 2}{\log N}\right)
\end{equation*} 
\end{lemma}
See Lemma 3 in \cite{yu1997assouad} for a proof of the above lemma.
    
\begin{lemma}[Varshamov-Gilbert Bound] \label{lem:varshamov-gilbert} 
Let $d$ be an integer satisfying $1 \leq d \leq p/4$. 
There exists a set $\Omega^{(p)}_{d} \subseteq \{0,1\}^{p}$ that satisfies the following properties:
\begin{enumerate}
    \item $\|\vomega\|_{0} = d$ for all $\vomega \in \Omega^{(p)}_{d}$;
    \item $\|\vomega - \vomega'\|_{0} \geq d/2$ for all distinct pairs $\vomega, \vomega' \in \Omega^{(p)}_{d}$;
    \item $\log \left|\Omega^{(p)}_{d}\right| \geq cd\log(p/d)$, where $c \geq 0.233$. 
\end{enumerate}
\end{lemma}
See Lemma 4.10 in \cite{massart2007concentration}. 

With the above lemmas, we are  equipped to prove the minimax lower bound in Theorem~\ref{thm:Mstar:2-infty:minimax}.

\begin{proof}
We will make use of the generalized Fano method stated in Lemma~\ref{lem:gFano}.
Let $p = \lceil n/2 \rceil-1$, $d = \lfloor p/4 \rfloor$ and $\Omega^{(p)}_d$ be the set defined in Lemma~\ref{lem:varshamov-gilbert}.
For a pair of given $\lambdastar$ and $N$, and for each $\vomega \in \Omega^{(p)}_d$, we associate them with the vectors
\begin{equation}\label{eq:vvomega:def}
    \vvomega := \frac{c_1}{\sqrt{N}\lambdastar}\left(2\vomega - \vone_{p}\right) \in \R^{p},
\end{equation} 
and
\begin{equation} \label{eq:vuomega:def}
    \vuomega := \left(u_0, \; \vvomegat, \; \frac{c_2}{\sqrt{\lfloor n/2 \rfloor}}\vone_{\lfloor n/2 \rfloor}^\top\right)^\top \in \R^{n},
\end{equation} 
where $u_0 = \sqrt{\mu / n}$, $c_1, c_2 > 0$ are normalization factors to be specified later so that $\left\|\vuomega\right\|_2 = 1$. 
For notational simplicity, we let $q := \lfloor n/2 \rfloor$. 
For each $\vomega \in \Omega_d^{(p)}$, we set 
\begin{equation} \label{eq:Momega:def}
    \mMomega := \lambdastar \vuomega \vuomegat = \lambdastar 
    \begin{pmatrix}
        u_0^2 & u_0 \vvomegat & \frac{c_2 u_0}{\sqrt{q}} \vone_{q}^\top \\
        u_0 \vvomega & \vvomega \vvomegat & \frac{c_2}{\sqrt{q}} \vvomega \vone_q^\top \\
        \frac{c_2 u_0}{\sqrt{q}} \vone_q & \frac{c_2}{\sqrt{q}} \vone_q \vvomegat & \frac{c_2^2}{q} \vone_q \vone_q^\top
    \end{pmatrix}.
\end{equation}
and for distinct pairs of $\omega$ and $\omega'$, we have
\begin{equation} \label{eq:diff:Momega} \begin{aligned}
\left\|\mMomega - \mMomegap\right\|_{2,\infty}
&\geq \left\|\mMomega_{1,\cdot} - \mMomegap_{1,\cdot}\right\|_{2} = \lambdastar u_0 \left\|\vvomega - \vvomegap\right\|_2 \\
&\stackrel{(i)}{=} \lambdastar \frac{2u_0 c_1}{\sqrt{N}\lambdastar} \left\|\vomega - \vomega'\right\|_2 \\
&\stackrel{(ii)}{\geq} c_1 u_0 \sqrt{\frac{2d}{N}}. 
\end{aligned} \end{equation}
where equality $(i)$ follows from the definition of $\vvomega$ in Equation~\eqref{eq:vvomega:def}, and inequality $(ii)$ follows from the fact that $\|\vomega - \vomega'\|_2^2 = \|\vomega - \vomega\|_0 \geq \frac{d}{2}$ by Lemma~\ref{lem:varshamov-gilbert}.
Let $\Promega_N$ be the probability measure over $\Sym(n)$ associated with 
\begin{equation*}
    \mM_l = \mMomega + \mW_l, \quad \text{for } l \in [N], 
\end{equation*}
where $\{\mW_l\}_{l \in [N]}$ are symmetric noise matrices with independent entries $W_{l,ij}\sim \calN(0, \sigma^2_{l,ij})$ on and above diagonal. 
It follows that 
\begin{equation*}
\begin{aligned}
    \KLD{\Promega_N}{\Promegap_N} &= \sum_{l=1}^N \sum_{1\leq i\leq j\leq n} \KLD{\Promega_{N,ij}}{\Promegap_{N,ij}} \leq \frac{N\left\|\mMomega - \mMomegap\right\|^2_{\F}}{\sigma^2_{\min}} \\
    &\stackrel{(i)}{\leq} \frac{N(\lambdastar)^2}{\sigma_{\min}^2} \left\|\vuomega - \vuomegap\right\|_2^2 \stackrel{(ii)}{=} \frac{N(\lambdastar)^2}{\sigma^2_{\min}} \cdot \frac{4c^2_1}{N \left(\lambdastar\right)^2}\left\|\vomega-\vomega'\right\|_2^2\\
    &= \frac{4 c_1^2}{\sigma^2_{\min}} \left\|\vomega - \vomega'\right\|_{0}
    \stackrel{(iii)}{\leq} \frac{8 c_1^2 d}{\sigma^2_{\min}}
\end{aligned}
\end{equation*}
where inequality $(i)$ follows from Equation~\eqref{eq:Momega:def} using the definition of $\mMomega$ and $\mMomegap$, equality $(ii)$ follows from the definition of $\vuomega$ and $\vuomegap$ in Equation~\eqref{eq:vuomega:def}, equality $(iii)$ follows from Lemma~\ref{lem:varshamov-gilbert}. 
Thus, combining the above display with Equation~\eqref{eq:diff:Momega}, by Lemma~\ref{lem:gFano}, we have 
\begin{equation*} \begin{aligned}
\max_{\vomega \in \Omega_d} \E_{\mMomega} \left\|\mMhat - \mMomega\right\|_{2,\infty}
&\geq c_1 u_0 \sqrt{\frac{2 d}{N}} \left(1 - \frac{1}{\log \left|\Omega^{(p)}_d\right|} \left(\frac{8 c_1^2 d}{\sigma^2_{\min}} + \log 2\right)\right)\\
&\geq c_1 u_0 \sqrt{\frac{2d}{N}} \left(1 - \frac{1}{c \log\left(\frac{p}{d}\right)} \left(\frac{8 c_1^2}{\sigma^2_{\min}} + \frac{\log 2}{d}\right)\right)
\end{aligned} \end{equation*}
Recalling that $d = \lfloor p/4 \rfloor$ and taking $c_1 = \min\{\frac{1}{4\sqrt{2\log 4}}\sigma_{\min}, c_1'\}$ where $c_1' > 0$ is a sufficiently small constant that allows us to normalize $\vuomega$ to have unit norm when $\sigma_{\min}$ is large, it holds for sufficiently large $n$ that
\begin{equation*} \begin{aligned}
\max_{\vomega \in \Omega_d} 
	\E_{\mMomega} \left\|\mMhat - \mMomega\right\|_{2,\infty}
&\geq \frac{\sigma_{\min} u_0}{8 \sqrt{\log 4}} \sqrt{\frac{d}{N}} \geq \frac{\sigma_{\min} u_0}{27\sqrt{\log 4}} \sqrt{\frac{n}{N}}. 
\end{aligned} \end{equation*}
Recalling that $u_0 = \sqrt{\mu /n}$ and picking a suitable $c_2$ to make $\vuomega$ to have unit norm finishes the proof. 
\end{proof}




\subsection{Proof of Theorem~\ref{thm:rank1:linfty:minimax}}
\label{sec:proof:thm:rank1:linfty:minimax}

\begin{proof}
Consider testing the following two hypotheses
\begin{equation*}
    \calH_0: \mM = \mM^{(0)} + \mW 
\quad \text{ versus } \quad
 \calH_1: \mM = \mM^{(1)} + \mW,
\end{equation*}
where $\mM^{(0)} = \lambdastar \vu^{(0)} \vu^{(0)\top}$ and  $\mM^{(1)} = \lambdastar \vu^{(1)} \vu^{(1)\top}$ are rank-one matrices and, observed subject to additive noise $\mW$, where $\mW$ is an independent mean-zero Gaussian matrix with $\E W_{ij}^2 = \sigma^2_{ij} \geq \sigma^2_{\min}$.
Let $\Pr_{k,i,j}$ stand for the distribution of $\mM^{(k)}_{ij}$ for $k = 0, 1$ and $\Pr_k$ denote the product measure obtained from $\Pr_{k,i,j}$, i.e., $\Pr_k = \prod_{1\leq i,j \leq n} \Pr_{k,i,j}$ for $k = 0, 1$.
We show how to obtain a lower bound for $\|\mMstar-\mMhat\|_{\infty}$ by choosing suitable $\vu^{(0)}$ and $\vu^{(1)}$ such that the KL divergence $\KLD{\Pr_1}{\Pr_0}$ is small, while  $\|\mM^{(1)}-\mM^{(0)}\|_{\infty}$ is large, and then applying standard tools from \cite{tsybakov2008introduction} to get the minimax rate.
\paragraph{Step 1: calculation of KL divergence.} 
We first calculate the KL divergence of $\Pr_0$ from $\Pr_1$ 
as follows 
\begin{equation*}
    \KLD{\Pr_1}{\Pr_0} 
=
\sum_{1\leq i,j \leq n} \KLD{\Pr_{1,i,j}}{\Pr_{0,i,j}} 
=
\sum_{1\leq i,j \leq n} \frac{\left(M^{(0)}_{ij} - M^{(1)}_{ij}\right)^2}{2\sigma_{ij}^2} 
\leq \frac{\left\|\mM^{(0)} - \mM^{(1)}\right\|^2_{\F}}{2\sigma^2_{\min}},
\end{equation*}
where the first equality  holds by the decoupling property of the KL divergence \citep[see Equation (15.11a) in][]{wainwright2019high} and the second equality fllows from properties of Gaussian KL divergence \citep[see Example 15.13 in][]{wainwright2019high}.

By definition of $\mM^{(0)}$ and $\mM^{(1)}$, 
\begin{equation} \label{eq:KLD:bound}
    \KLD{\Pr_1}{\Pr_0} = \frac{(\lambdastar)^2 \left\|\vu^{(0)}\vu^{(0)\top} - \vu^{(1)}\vu^{(1)\top}\right\|^2_{\F}}{2\sigma^2_{\min}} \leq \frac{(\lambdastar)^2 \left\|\vu^{(0)} - \vu^{(1)}\right\|_2^2}{\sigma^2_{\min}},
\end{equation}
where the last inequality holds due to the fact that for any unit vectors $\vu$ and $\vv$ \citep[here we repeat Equation (179) in][]{cheng2021tackling},
\begin{equation*} \begin{aligned}
\left\|\boldsymbol{u} \boldsymbol{u}^{\top}-\boldsymbol{v} \boldsymbol{v}^{\top}\right\|_{\F}^2 
&=\|\boldsymbol{u}\|_2^4+\|\boldsymbol{v}\|_2^4-2\left\langle\boldsymbol{u} \boldsymbol{u}^{\top}, \boldsymbol{v} \boldsymbol{v}^{\top}\right\rangle=2-2\left(\boldsymbol{u}^{\top} \boldsymbol{v}\right)^2 \\
&=\left(2-2 \boldsymbol{u}^{\top} \boldsymbol{v}\right)\left(1+\boldsymbol{u}^{\top} \boldsymbol{v}\right)=\frac{1}{2}\|\boldsymbol{u}-\boldsymbol{v}\|_2^2 \cdot\|\boldsymbol{u}+\boldsymbol{v}\|_2^2 \\
&\leq 2\|\vu-\vv \|_2^2.
\end{aligned} \end{equation*}

\paragraph{Step 2: bounding minimax probability of error. } 
Define the minimax probability of errors as 
\begin{equation*}
    p_{e,1} := \inf_{\psi} \max \left\{ ~\Pr\left(\psi \text{ rejects } \calH_0 \mid \calH_0\right), \quad \Pr\left(\psi \text{ rejects } \calH_1 \mid \calH_1\right) ~\right\},
\end{equation*}
where the infimum is over all tests.
By Theorem 2.2 in \cite{tsybakov2008introduction}, if $\KLD{\Pr_1}{\Pr_0} \leq \alpha$, then 
\begin{equation*}
p_{e,1} 
\geq \max\left(\frac{1}{4}\exp(-\alpha), \frac{1-\sqrt{\alpha/2}}{2}\right). 
\end{equation*}
Taking $\alpha = 1/16$, it implies that when 
\begin{equation} \label{eq:KL-target}
    \KLD{\Pr_1}{\Pr_0} \leq 1/8
\end{equation}
we have $p_{e,1} \geq 3/8$. 
Therefore, we will construct $\vu^{(0)}$ and $\vu^{(1)}$ such that 
\begin{equation} \label{eq:vu0u1diff}
\left\|\vu^{(0)} - \vu^{(1)}\right\|_2 \leq \frac{\sigma_{\min}}{4\lambdastar}
\end{equation}
so that by Equation~\eqref{eq:KLD:bound}, we have $\KLD{\Pr_1}{\Pr_0} \leq 1/8$.  
Since 
\begin{equation} \label{eq:MMinfty:bound}
    \left\|\mM^{(0)} - \mM^{(1)}\right\|_{\infty} \geq \max_{i \in [n]}\lambdastar \left|u^{(0)}_i - u^{(1)}_i\right| \left|u^{(0)}_i + u^{(1)}_i\right|,
\end{equation}
we consider 
\begin{equation*}
    u_1^{(k)} = \sqrt{\frac{\mu}{n}} \cos \theta_{k}, \quad u_2^{(k)} = \sqrt{\frac{\mu}{n}} \sin \theta_{k}
\end{equation*}
and $u^{(0)}_i = u^{(1)}_i$ for all $2 \leq i \leq n$. 
It is straightforward to guarantee that the constructed $\vu^{(k)}$ satisfies that $\mM^{(k)} \in \calM(\mu, \sigma_{\min})$, where $\calM$ is defined in Equation~\eqref{eq:calM:def}. 
By the construction, we have 
\begin{equation*} \begin{aligned}
\left\|\vu^{(0)} - \vu^{(1)}\right\|^2_{2} 
&= \left(u_1^{(1)}-u_1^{(0)}\right)^2 + \left(u_2^{(1)}-u_2^{(0)}\right)^2 
= \frac{\mu}{n} \left(\left(\cos \theta_1 - \cos \theta_0\right)^2 + \left(\sin \theta_1 - \sin \theta_0\right)^2\right) \\
&= \frac{2\mu}{n} \left(1-\cos(\theta_1 - \theta_0)\right).
\end{aligned} \end{equation*}
By Equation~\eqref{eq:calM:def}, we have 
\begin{equation} \label{eq:lambdastar:lower:bound}
    \lambdastar \geq \sigma_{\min} \sqrt{\frac{n}{\mu}}
\end{equation}
so that we can set 
\begin{equation*}
    \cos(\theta_1 - \theta_0) = 1 - \frac{n \sigma^2_{\min}}{32 \mu \left(\lambdastar\right)^2} \geq \frac{31}{32} 
\end{equation*}
and Equation~\eqref{eq:vu0u1diff} holds. 
Additionally, we have
\begin{equation*}
\begin{aligned}
    \left|u^{(0)}_1 - u^{(1)}_1\right| \left|u^{(0)}_1 + u^{(1)}_1\right| &= \frac{\mu}{n} \left|\cos^2 \theta_1 - \cos^2 \theta_0\right|\\
    &= \frac{\mu}{n} \left|\sin \left(\theta_1+\theta_0\right) \right| \cdot \sqrt{1 - \cos^2 \left(\theta_1-\theta_0\right)}.
\end{aligned}
\end{equation*}
Taking $\theta_1 + \theta_0 = \pi/2$, so that 
\begin{equation*}
    \theta_0 = \frac{\pi}{4} - \frac{1}{2}\arccos\left(1 - \frac{n \sigma^2_{\min}}{32 \mu \left(\lambdastar\right)^2}\right), \quad \theta_1 = \frac{\pi}{4} + \frac{1}{2}\arccos\left(1 - \frac{n \sigma^2_{\min}}{32 \mu \left(\lambdastar\right)^2}\right).
\end{equation*}
One can verify that $\cos \theta_k \geq 1/2$ for $k = 0$ and $1$ under Equation~\eqref{eq:lambdastar:lower:bound}, implying that 
\begin{equation*}
\left\|\vu^{(k)}\right\|_{\infty} \geq \frac{1}{2}\sqrt{\frac{\mu}{n}}, \quad \text{for } k = 0, 1. 
\end{equation*}
It follows that 
\begin{equation*}
    \left|u^{(0)}_1 - u^{(1)}_1\right| \left|u^{(0)}_1 + u^{(1)}_1\right| \geq  \frac{\mu}{n} \sqrt{\frac{3}{2}} \cdot \sqrt{\frac{n \sigma^2_{\min}}{32 \mu \left(\lambdastar\right)^2}} = \frac{\sigma_{\min}}{8\lambdastar} \sqrt{\frac{3\mu}{n}}
\end{equation*}
and by Equation~\eqref{eq:MMinfty:bound}, we have
\begin{equation*}
    \left\|\mM^{(0)} - \mM^{(1)}\right\|_{\infty} \geq \frac{\sigma_{\min}}{8} \sqrt{\frac{3\mu}{n}} .
\end{equation*}
Thus, combining the above bound and Equation~\eqref{eq:KL-target} with Theorem 2.2 in \cite{tsybakov2008introduction} completes the proof.
\end{proof}

\subsection{Proof of Corollary~\ref{cor:more:linfty:minimax}}

\begin{proof}
Consider the following collection of hypothesis tests:
\begin{equation*}
    \calH_0: \mM^{(l)} = \mMtil^{(0)} + \mW^{(l)}
 \quad \text{ versus} \quad
 \calH_1: \mM^{(l)} = \mMtil^{(1)} + \mW^{(l)}, \quad \text{for } l \in [N]. 
\end{equation*}
Let $\Pr^{(l)}_{k,i,j}$ stand for the distribution of $M^{(l)}_{ij}$ under $\calH_k$ for $k = 0,1$, $l \in [N]$ and $1\leq i,j \leq n$. 
We have 
\begin{equation*}
\KLD{\prod_{l=1}^N \Pr^{(l)}}{\prod_{l=1}^N \Pr^{(l)}_0} 
\leq \sum_{l=1}^N \sum_{1\leq i,j \leq n} \KLD{\Pr_{1,i,j}}{\Pr_{0,i,j}} 
\leq \frac{N}{2\sigma^2_{\min}} \left\|\mM^{(0)} - \mM^{(1)}\right\|^2_{\F}. 
\end{equation*}
Following the same proof and construction as Theorem~\ref{thm:rank1:linfty:minimax} with $\sigma_{\min}$ substituted by $\sigma_{\min}/\sqrt{N}$, we have 
\begin{equation*}
    \KLD{\prod_{l=1}^N \Pr^{(l)}}{\prod_{l=1}^N \Pr^{(l)}_0} \leq \frac{N \left(\lambdastar\right)^2}{2\sigma^2_{\min}} \left\|\vu^{(0)} - \vu^{(1)}\right\|_2^2 \leq \frac{1}{8} 
\end{equation*}
and 
\begin{equation*}
    \left\|\mMtil^{(0)} - \mMtil^{(1)}\right\|_{\infty} \geq \frac{\sigma_{\min}}{8} \sqrt{\frac{3\mu}{n N}}.
\end{equation*}
Combining the above two dislay equations and applying Theorem 2.2 in \cite{tsybakov2008introduction} again yields the desired minimax lower bound.
\end{proof}

\section{\texorpdfstring{Technical Results in \cite{yan2025improved}}{Technical Results in [Yan and Levin, 2025]}}

\label{sec:improved}

The following are results related to \cite{yan2025improved}, which we reproduce here for ease of reference.

\begin{lemma}[Neumann Expansion]\label{lem:neumann}
Consider $\mM$ given in Equation~\eqref{eq:M:asym}. Let $\vu_l$ and $\lambda_l$ be the $l$-th leading right eigenvector and the $l$-th leading eigenvalue of $\mM$, respectively. If $\|\mH\| < \left|\lambdastar_l\right|$, then one has
    \begin{equation*}
        \vu_l=\sum_{j=1}^r \frac{\lambda_j^{\star}}{\lambda_l}\left(\vu_j^{\star \top} \vu_l\right)\left\{\sum_{s=0}^{\infty} \frac{1}{\lambda_l^s} \mH^s \vu_j^{\star}\right\} .
    \end{equation*}
\end{lemma}
\begin{proof}
    See Lemma 1 in \cite{cheng2021tackling}. 
\end{proof}

\begin{theorem}[Theorem 3 in \cite{yan2025improved}]\label{thm:yan:main}
Consider a random matrix $\mH \in \R^{n\times n}$ with independent entries satisfying Assumption~\ref{assump:noise} and the bound in Equation~\eqref{eq:L-sigma}, for any fixed unit vectors $\vx$ and $\vy \in \R^n$, it holds with probability at least $1 - \tilde{O}(n^{-40})$ that for all integers $k$ satisfying $2 \leq k \leq 20\log n$,
\begin{equation} \label{eq:main-bound}
\begin{aligned}
    &\left|\vx^\top \mH^k \vy - \E \vx^\top \mH^k \vy \right| \lesssim c_2^k (\log n)^{2}
    \max\Biggl\{\! \sqrt{ \frac{ (\sigma^2 n\log^3 \! n)^{k} }{n}},  (L\log^3 n)^k \|\vx\|_{\infty}\|\vy\|_{\infty} \! \Biggr\} .
\end{aligned} \end{equation}
\end{theorem}

\begin{theorem}[Theorem 6 in \cite{yan2025improved}]\label{thm:linear-form-rank-r}
 Consider a rank-$r$ symmetric matrix $\mMstar \in \R^{n\times n}$ with incoherence parameter $\mu$. 
 Define $\kappa := \lambdastar_{\max} / \lambdastar_{\min}$, where $\lambdastar_{\max}$ and  $\lambdastar_{\min}$ are the largest and smallest eigenvalues of $\mMstar$ in magnitude, respectively.
 Suppose the noise matrix $\mH$ obeys Assumption~\ref{assump:noise}, and assume the existence of some sufficiently large constant $C_2 \geq 0$ such that
 \begin{equation} \label{eq:lambdastar-max}
     \lambdastar_{\min} 
     \geq C_2 \max \left\{\sigma \sqrt{n \log^{3} n}, L \log^3 n\right\} .
 \end{equation}
 Then for any fixed unit vector $\va\in\R^n$ and for any $l \in [r]$, with probability at least $1 - O(r n^{-20})$,
 \begin{equation} \label{eq:master:rank-r}
    \begin{aligned}
        &\left|\va^\top \vu_l - \sum_{j=1}^r \frac{\lambdastar_j \vustart_j \vu_l}{\lambda_l} \va^\top \vustar\right| \lesssim \sqrt{\frac{\kappa^2 r \log^4 n}{n}} 
        \max\Bigg\{\frac{\mu^{\frac{1}{2}} L \log^3 n}{\lambdastar_{\min}} \|\va\|_{\infty}, \frac{\sigma \sqrt{n \log ^3 n}}{\lambdastar_{\min}} \Bigg\}.
    \end{aligned}
    \end{equation}
\end{theorem}

\begin{theorem}\label{thm:eigenvalue:rank-r}
    Suppose that the noise parameters defined in Assumption~\ref{assump:noise} satisfy
    \begin{equation*}
    \begin{aligned}
        &\delta^\star_l \geq
        c_3 r^2 \! \max\!\left\{\! \frac{\kappa \mu L\! \log ^5\! n}{n}, \kappa \sigma \log^{7/2} \right\}
    \end{aligned}
    \end{equation*}
    and $\lambdastar_{\min}$ satisfies condition~\eqref{eq:lambdastar-max} for some sufficiently large constants $c_3 > 0$ and $C_2 > 0$. 
    Then given any $l \in [r]$, with probability $1 - O(r^2 n^{-20})$, the eigenvalue $\lambda_l$ is real-valued, and
    \begin{equation}\label{eq:eigenvalue-rank-r}
    \begin{aligned}
        |\lambda_l - \lambdastar_l| \leq
        c_3 r^2\max\Bigg\{\frac{\kappa \mu L \log ^5 n}{n}, \kappa \sigma \log^{7/2} n\Bigg\}.
    \end{aligned}
    \end{equation}
\end{theorem}

\section{Technical Bounds and Auxiliary Lemmas}
\label{sec:tech}

Here we collect auxiliary technical results related to concentration inequalities and linear algebra, which are used in proving our main results.

\subsection{Probability Inequalities}

\begin{lemma} \label{lem:A-circ-W-op}
    For any symmetric matrix $\mA \in \R^{n\times n}$ and $\mW$ satisfying Assumption~\ref{assump:noise}, it holds with probability at least $1 - O(n^{-40})$ that
    \begin{equation*}
    \left\|\mA \circ \mW\right\| \lesssim \sigma\left\|\mA\right\|_{2, \infty}+ L \left\|\mA\right\|_{\infty} \sqrt{\log n} .
    \end{equation*}
\end{lemma}
\begin{proof}

We use the matrix Bernstein inequality stated in Theorem~\ref{thm:bernstein}. 
Since
\begin{equation*}
    | A_{ij} W_{ij} | \leq \|\mA\|_{\infty} L \quad \text{and} \quad \max_{i} \sum_{j} \E (A_{ij} W_{ij})^2 \leq \sigma^2 \|\mA\|_{2,\infty}^2, 
\end{equation*}
we have
\begin{equation*}
    \Pr\left\{ \left\|\mA \circ \mW \right\| \geq 4 \sigma \|\mA\|_{2,\infty} + t \right\} \lesssim n \exp \left(-\frac{c t^2}{L^2 \|\mA\|_{\infty}^2}\right).
\end{equation*}
Taking $t = C L\|\mA\|_{2,\infty} \sqrt{\log n}$ for some sufficiently large constant $C$ completes the proof. 
\end{proof}

\begin{lemma} [Example 2.11 in \cite{wainwright2019high}]
\label{lem:chi-sq-concen}

For $Z_1, Z_2, \dots, Z_k \iid \calN(0,1)$ and any $t \in (0,1)$,
\begin{equation*}
    \Pr\left(\left|\frac{1}{k} \sum_{j=1}^k Z_j^2 - 1\right| \geq t\right) \leq 2 e^{-kt^2/8}.
\end{equation*}
For $n$ and $k_n$ such that $k_n^{-1} \log n = o(1)$, taking $t = 7\sqrt{k_n/\log n }$ yields that for sufficiently large $n$, with probability at least $1 - O(n^{-6})$,
\begin{equation*}
    \left|\sum_{j=1}^{k_n} Z_{k_n}^2 - k_n\right|
    \leq 7\sqrt{k_n \log n} .
\end{equation*}
\end{lemma}

\begin{lemma}[Bernstein's inequality]
\label{lem:bern-ineq}
    Let $X_1, X_2, \dots, X_k$ be independent zero-mean random variables such that $|X_i| \leq L$ and $\E X_i^2 \leq \sigma^2$ for all $i \in [k]$.
Then
    \begin{equation*}
        \Pr\left\{\left|\sum_{i=1}^k X_i\right| \geq t\right\} \leq 2 \exp\left(-\frac{t^2/2}{k\sigma^2 + L t / 3}\right) \leq 2 \exp \left(-c \min \left\{ \frac{t^2}{k\sigma^2}, \frac{t}{L}\right\}\right). 
    \end{equation*}
    In other words, by choosing a suitable and universal constant $C > 0$, it holds with probability at least $1 - O(n^{-40})$, that
    \begin{equation*}
        \left|\sum_{i=1}^k \left(X_i - \E X_i\right)\right| \lesssim \max \left\{\sigma\sqrt{k\log n}, L\log n\right\} .
    \end{equation*}
\end{lemma}
\begin{proof}
    See Theorem 2.8.4 in \cite{vershynin2018HDP}. 
\end{proof}

\begin{theorem}[Theorem 3.4 in \cite{chen2021spectral}] \label{thm:bernstein}
Consider a symmetric random matrix $\mX=\left[X_{i, j}\right]_{1 \leq i, j \leq n}$ in $\R^{n \times n}$, whose entries are independently generated and obey 
\begin{equation*}
    \E \left[X_{i, j}\right]=0, 
\quad \text { and } \quad
\left|X_{i, j}\right| \leq B, \quad 1 \leq i, j \leq n.
\end{equation*}
Define
    $$
    \nu:=\max _i \sum_j \E \left[X_{i, j}^2\right] .
    $$
    Then there exists some universal constant $c>0$ such that for any $t \geq 0$,
    \begin{equation*}
        \Pr\{\|\mX\| \geq 4 \sqrt{\nu}+t\} \leq n \exp \left(-\frac{t^2}{c B^2}\right).
    \end{equation*}
\end{theorem}

\begin{theorem}[Corollary 3.3 in \cite{chen2021spectral}]\label{thm:matrix:bernstein}
    Let $\{\mX_i\}_{1\leq i \leq m}$ be a set of independent real random matrices of dimension $n_1 \times n_2$. Suppose that
\begin{equation*}
\E\left[\mX_i\right]=\vo, \quad \text { and } \quad\left\|\mX_i\right\| \leq L_0, \quad \text { for all } i
\end{equation*}
and set $n:=\max \left\{n_1, n_2\right\}$.
Define the matrix variance statistic $v$ as 
\begin{equation*} \begin{aligned}
& v:=\max \left\{\left\|\sum_{i=1}^m \E\left[\left(\mX_i-\E\left[\mX_i\right]\right)\left(\boldsymbol{X}_i-\mathbb{E}\left[\boldsymbol{X}_i\right]\right)^{\top}\right]\right\|,\right. \\
&\qquad\qquad~~~~\left.\left\|\sum_{i=1}^m \mathbb{E}\left[\left(\boldsymbol{X}_i-\mathbb{E}\left[\boldsymbol{X}_i\right]\right)^{\top}\left(\boldsymbol{X}_i-\mathbb{E}\left[\boldsymbol{X}_i\right]\right)\right]\right\|\right\}.
\end{aligned} \end{equation*}
For any $a \geq 2$, with probability exceeding $1-2 n^{-a+1}$ one has
\begin{equation*}
\left\|\sum_{i=1}^m \mX_i\right\| 
\leq \sqrt{2 a v \log n}+\frac{2 a}{3} L_0 \log n .
\end{equation*}
\end{theorem}

\begin{lemma}\label{lem:stoc:order} 
For two random variables $A$ and $B$, we say that $A$ is stochastically dominated by $B$ ($A \preceq B$) if for all $t \in \R$, 
$$
\Pr\{B \geq t\} \geq \Pr\{A \geq t\}.
$$
If $u: \R^n \rightarrow \R$ is a function increasing in each variable and $A_i$ and $B_i$ are independent sets of random variables with $A_i \preceq B_i$ for each $i$, 
then $u\left(A_1, \ldots, A_n\right) \preceq u\left(B_1, \ldots, B_n\right)$ and in particular $\sum_{i=1}^n A_i \preceq \sum_{i=1}^n B_i$.
\end{lemma} 
\begin{proof}
    See Theorem 3.3.11 in \cite{muller2002comparison}. 
\end{proof} 

\subsection{Linear Algebra Lemmas}

\begin{lemma}\label{lem:IS-invertible}
If $\mS \in \R^{n\times n}$ is a row stochastic matrix and $\mS^2 > 0$ holds entrywise, then $\mI_n + \mS$ is invertible. 
\end{lemma}
\begin{proof}
By the Gershgorin circle theorem, all of $\mS$'s eigenvalues are bounded between $-1$ and $1$. 
Thus, all eigenvalues of $\mS^2$ are bounded between $0$ and $1$. 
    Noting that $\mS \mathbf{1}_n = \mathbf{1}_n$, we see that $1$ is a eigenvalue of $\mS$. 
    Since $\mS^2 > 0$, $\mS$ is irreducible and by the Perron–Frobenius theorem, we have that $1$ is an algebraically simple eigenvalue of $\mS^2$. 
    Thus, all eigenvalues of $\mS$ must be strictly larger than $-1$,
and it follows that $\mI_n + \mS$ is invertible.  
\end{proof}

\begin{lemma}\label{lem:eigenvalue-IS}
Suppose that $\mS \in \R^{n\times n}$ is a row stochastic matrix and $\mI_n + \mS$ is invertible.
Then $\mathbf{1}_n$ is an eigenvector of $(\mI_n + \mS)^{-1}$ with eigenvalue $1/2$.
\end{lemma}
\begin{proof}
Since $\mS$ is row stochastic, we have
\begin{equation*}
    \mS \mathbf{1}_n = \mathbf{1}_n \quad \text{ and } \quad (\mI_n + \mS) \mathbf{1}_n = 2\mathbf{1}_n.
\end{equation*}
Since $\mI_n + \mS$ is invertible, it follows that
\begin{equation*}
    (\mI_n + \mS)^{-1} \mathbf{1}_n =  \frac{1}{2} \mathbf{1}_n.
\end{equation*}
\end{proof}

\begin{theorem}\label{thm:dd:singular:lower}
For a square matrix $\mA \in \R^{r\times r}$, set $\alpha = \min_{k\in[r]} \left\{|A_{kk}| - \sum_{j\neq k} |A_{kj}|\right\}$ and $\beta = \min_{k\in[r]} \{|A_{kk}| - \sum_{j\neq k} |A_{jk}|\}$.
If $\mA$ is diagonally dominant both by rows and by columns, then 
    \begin{equation*}
        \sigma_r(\mA) \geq \sqrt{\alpha \beta},
    \end{equation*}
    where $\sigma_r(\mA)$ is the smallest singular value of $\mA$. 
\end{theorem}
\begin{proof}
See Theorem 1, Corollary 1 and Corollary 2 in \cite{Varah1975lower}. 
\end{proof}

\begin{lemma}\label{lem:dd:singular:upper}
For a square matrix $\mA \in \R^{r\times r}$, if $\mA$ is diagonally dominant both by rows and by columns,then 
\begin{equation*}
    \|\mA\| \leq 2 \max_{i\in[r]} \left\{|A_{ii}|\right\}.
\end{equation*}
\end{lemma}
\begin{proof}
We have
\begin{equation*}
\|\mA\| \leq \sqrt{\max_{1\leq j \leq r} \sum_{i=1}^r |A_{ij}|} \sqrt{\max_{1\leq i \leq r} \sum_{j=1}^r |A_{ij}|} 
\leq \max_{j \in [r]}\sqrt{2|A_{jj}|} \max_{i \in [r]}\sqrt{2|A_{ii}|} 
= 2 \max_{i\in[r]} \left\{|A_{ii}|\right\},
\end{equation*}
where the first inequality follows from Corollary 2 in \cite{Varah1975lower} and the second inequality follows from the fact that $\mA$ is diagonally dominant both by rows and by columns. 
\end{proof}

\begin{lemma}\label{lem:square-root-lipschitz}
For symmetric positive definite matrices $\mA, \mB \in \R^{n\times n}$, suppose that the eigenvalues of $\mA$ and $\mB$ are all greater than $\lambda_{\min} > 0$, then 
\begin{equation*}
\left\|\mA^{1/2} - \mB^{1/2}\right\| \leq \frac{1}{2\sqrt{\lambda_{\min}}} \left\|\mA - \mB\right\|.
\end{equation*}
\end{lemma}
\begin{proof}
For any unit eigenvector $\vx$ of $\mA^{1/2}-\mB^{1/2}$ with eigenvalue $\lambda$, we have
\begin{equation*} \begin{aligned}
\vx^\top \left(\mA - \mB\right) \vx 
&= \vx^\top \mA^{1/2} \left(\mA^{1/2} - \mB^{1/2}\right) \vx + \vx^\top \mB^{1/2} \left(\mA^{1/2} - \mB^{1/2}\right) \vx \\
&= \lambda \vx^\top \left(\mA^{1/2} + \mB^{1/2}\right) \vx.
\end{aligned} \end{equation*}
Thus, it follows that 
\begin{equation*}
\left|\lambda\right| 
\leq \frac{\left\|\mA - \mB\right\|}{\min_{\vx: \|\vx\|_2=1} \vx^\top \mA^{1/2} \vx + \min_{\vy: \|\vy\|_2=1} \vy^\top \mB^{1/2} \vy} 
\leq \frac{1}{2\sqrt{\lambda_{\min}}} \left\|\mA - \mB\right\|.
\end{equation*}
Taking the maximum on the left hand side over all eigenvalues of $\mA^{1/2} - \mB^{1/2}$ yields the desired bound. 
\end{proof}

\end{document}